\documentclass[reqno,11pt, a4paper]{amsart}

\usepackage[T1]{fontenc}

\usepackage{amssymb}
\usepackage{amsmath}
\usepackage{amsfonts}
\usepackage{amsthm}
\usepackage{mathtools}
\usepackage{mathabx}
\usepackage{fullpage}
\usepackage{comment}
\usepackage[dvipsnames]{xcolor}
\usepackage{graphicx}
\usepackage{float}

\usepackage{enumitem} 

\usepackage[export]{adjustbox}
\usepackage{bbold}
\usepackage[percent]{overpic}

\usepackage{booktabs,array}
\usepackage{url} 
\usepackage{hyperref}

\usepackage{pgfplots}
\pgfplotsset{compat=1.6}

\usetikzlibrary{calc,quotes,angles,arrows.meta,positioning}
\usetikzlibrary{decorations.pathreplacing,decorations.markings}

\usepackage{mdframed}

\usepackage{float}

\usepackage{caption}
\usepackage{subcaption}

\usepackage[nocompress]{cite}

\usepackage{todonotes}
\theoremstyle{plain}%
\newtheorem{theorem}{Theorem}[section]
\newtheorem{lemma}[theorem]{Lemma}
\newtheorem{proposition}[theorem]{Proposition}
\newtheorem{corollary}[theorem]{Corollary}

\newtheorem*{conjecture*}{Conjecture} 

\newtheorem{sublemma}[theorem]{\normalfont Claim}
 \numberwithin{equation}{section}

\theoremstyle{definition}
\newtheorem{definition}[theorem]{Definition}
\newtheorem{example}[theorem]{Example}

\theoremstyle{remark}
\newtheorem{remark}[theorem]{Remark}

\newenvironment{subproof}[1][\proofname]{%
  \proof[Proof of claim]%
}{\endproof}

 \let \leq \leqslant
 \let \geq \geqslant

\DeclareMathOperator{\sgn}{sgn}

\DeclareMathOperator{\Cov}{Cov}

\DeclareMathOperator{\Var}{Var}
\DeclareMathOperator{\support}{supp}

\DeclareMathOperator{\dist}{dist}

\usepackage{fancyhdr}
\usepackage{graphics}
\usepackage{graphicx}
\usepackage{soul}

\definecolor{detailcolor00}{rgb}{0.4405, 0.204, 0.343}
\definecolor{detailcolor01}{rgb}{0.546, 0.215, 0.352}
\definecolor{detailcolor02}{rgb}{0.675, 0.247, 0.387} 
\definecolor{detailcolor03}{rgb}{0.775, 0.317, 0.455}
\definecolor{detailcolor04}{rgb}{0.830, 0.421, 0.553} 
\definecolor{detailcolor05}{rgb}{0.831, 0.533, 0.663}
\definecolor{detailcolor06}{rgb}{0.779, 0.619, 0.775}
\definecolor{detailcolor07}{rgb}{0.724, 0.694, 0.827}
\definecolor{detailcolor08}{rgb}{0.687, 0.770, 0.880}
\definecolor{detailcolor09}{rgb}{0.671, 0.839, 0.904}
\definecolor{detailcolor10}{rgb}{0.659, 0.872, 0.882}

\usepackage{scalerel,stackengine}
\newcommand\pig[1]{\scalerel*[5.5pt]{\Big#1}{%
  \ensurestackMath{\addstackgap[1.5pt]{\big#1}}}}
\newcommand\pigl[1]{\mathopen{\pig{#1}}}
\newcommand\pigr[1]{\mathclose{\pig{#1}}}

\ifdim\overfullrule>0pt 
  \usepackage{environ}

  \NewEnviron{tikzpicture}{%
    \begin{pgfpicture}
    \pgfpathrectanglecorners{\pgfpointorigin}{\pgfpoint{3cm}{3cm}}%
     \pgfusepath{stroke}\end{pgfpicture}%
  }
\fi

\makeatletter
\newcommand{\hathat}[1]{%
\begingroup%
  \let\macc@kerna\z@%
  \let\macc@kernb\z@%
  \let\macc@nucleus\@empty%
  \hat{\raisebox{.3ex}{\vphantom{\ensuremath{#1}}}\smash{\hat{#1}}}%
\endgroup%
}

\newcommand{\smallhathat}[1]{%
\begingroup%
  \let\macc@kerna\z@%
  \let\macc@kernb\z@%
  \let\macc@nucleus\@empty%
  \hat{\raisebox{.05ex}{\vphantom{\ensuremath{#1}}}\smash{\hat{#1}}}%
\endgroup%
}

\newcommand{\smallsmallhathat}[1]{%
\begingroup%
  \let\macc@kerna\z@%
  \let\macc@kernb\z@%
  \let\macc@nucleus\@empty%
  \hat{\raisebox{-.2ex}{\vphantom{\ensuremath{#1}}}\smash{\hat{#1}}}%
\endgroup%
}
\makeatother

\title{Wilson lines in the Abelian lattice Higgs model}

\author{Malin P. Forsstr\"om}
\address[Malin P. Forsstr\"om]{Department of Mathematics, KTH Royal Institute of Technology, 100 44 Stockholm, Sweden.}
\email{malinpf@kth.se}

\begin{document}

\maketitle

\begin{abstract}
     Lattice gauge theories are lattice approximations of the Yang-Mills theory in physics. The abelian lattice Higgs model is one of the simplest examples of a lattice gauge theory interacting with an external field. In a previous paper~\cite{flv2021}, we calculated the leading order term of the expected value of Wilson loop observables in the low-temperature regime of the abelian lattice Higgs model on \( \mathbb{Z}^4 ,\) with structure group \( G = \mathbb{Z}_n \) for some \( n \geq 2. \) In the absence of a Higgs field, these are important observables since they exhibit a phase transition which can be interpreted as distinguishing between regions with and without quark confinement. However, in the presence of a Higgs field, this is no longer the case, and a more relevant family of observables are so-called open Wilson lines. In this paper, we extend and refine the ideas introduced in~\cite{flv2021} to calculate the leading order term of the expected value of the more general Wilson line observables. Using our main result, we then calculate the leading order term of several natural ratios of expected values and confirm the behavior predicted by physicists.
\end{abstract}


\section{Introduction}

\subsection{Background}

Lattice gauge theories are spin models which describe the interaction of elementary particles. These were first introduced by Wilson~\cite{w1974} as lattice approximations of the quantum field theories that appear in the standard model, known as Yang-Mills theory. Since then, lattice gauge theories have been successfully used to understand the corresponding continuous models, and several of the predictions made using these lattice approximations have been verified experimentally.
At about the same time as lattice gauge theories were introduced in the physics literature by Wilson,  Wegner~\cite{w1971} introduced what he then called \emph{generalized Ising models} as an example of a family of models with a phase transition without a local order parameter. In special cases, these generalized Ising models are lattice gauge theories, and as such, they have been used extensively as toy models for the lattice gauge theories that are more relevant for physics.

In the last couple of years, there has been a renewed interest in the rigorous analysis of four-dimensional lattice gauge theories in the mathematical community, see, e.g.,~\cite{c2019, sc2019, flv2020,c2021,f2021, gs2021}. Most relevant for this work are the papers~\cite{c2019, sc2019, flv2020}, in which the leading order term for the expectation of Wilson loop observables was computed for lattice gauge theories with Wilson action and finite structure groups. 

Pure gauge theories model only the gauge field itself, and to advance towards physically relevant theories; it is necessary also to understand models that include external fields interacting with the gauge field, see, e.g.,~\cite{fs1979,s1988}. 
In this paper, we consider a lattice gauge theory that models a gauge field coupled to a scalar Bosonic field with a quartic Higgs potential. The resulting model is called the \emph{lattice Higgs model}.
This model has received significant attention in the physics community. Some examples are the works~\cite{b1974,b1975II,b1975III}, where calculations to obtain critical parameters of these models were performed, and~\cite{jsj1980, ks1984}, in which phase diagrams were sketched. For further background, as well as more references, we refer the reader to~\cite{fs1979} and~\cite{s1988}.

In a recent paper~\cite{flv2021}, we extended the theory developed in~\cite{flv2020, c2019, sc2019} in order to describe the leading order term for the expectation of Wilson loop observables in the fixed length and low-temperature regime of the abelian Higgs model.
Wilson loop expectations are natural observables in lattice gauge theories and were introduced by Wilson as a means to detect whether quark confinement occurs, see~\cite{w1974}. In lattice gauge theories without matter fields, one can show that the expected value of large Wilson loops undergo a phase transition, where it changes from following a so-called area law to following a so-called perimeter law.
However, as discussed in, e.g.,~\cite{m}, in gauge theories with matter fields, the Wilson loop observable obeys a perimeter law for all parameters, and hence one cannot see a relevant phase transition using only the Wilson loop observable. For this reason, alternative observables have been suggested for studying the lattice Higgs model.
One such observable is the open Wilson line observable, in which the loop in the Wilson loop observable is replaced by an open path that is saturated at the end-points by the Higgs field. This type of observable has been relatively well studied in the physics literature (see, e.g., \cite{bf1983, hgjjkn1987, fmf1986, m, s1988, bf1987, g2006, gr2002}). Moreover, the asymptotic behavior of such observables has been argued to be related to, e.g., the absence of bound states of the charged particle in the presence of an external source \cite{hgjjkn1987}, confinement versus deconfinement in lattice gauge theories with matter fields \cite{bf1987}, and binding versus unbinding of dynamical quarks in the field of a static color source \cite{bf1987}. Hence the Wilson line observables are of physical relevance.

\subsection{Preliminary notation}
For \( m \geq 2 \), the graph naturally associated to $\mathbb{Z}^m$ has a vertex at each point \( x \in \mathbb{Z}^m \) with integer coordinates and a non-oriented edge between nearest neighbors. We will work with oriented edges throughout this paper, and for this reason we associate to each non-oriented edge \( \bar e \)  two oriented edges \( e_1 \) and \( e_2 = -e_1 \) with the same endpoints as \( \bar e \) and opposite orientations. 

Let \( d\mathbf{e}_1 \coloneqq (1,0,0,\ldots,0)\), \( d\mathbf{e}_2 \coloneqq (0,1,0,\ldots, 0) \), \ldots, \( d\mathbf{e}_m \coloneqq (0,\ldots,0,1) \) be oriented edges corresponding to the unit vectors in \( \mathbb{Z}^m \). We say that an oriented edge \( e \) is \emph{positively oriented} if it is equal to a translation of one of these unit vectors, i.e.,\ if there is a \( v \in \mathbb{Z}^m \) and a \( j \in \{ 1,2, \ldots, m\} \) such that \( e = v + d{\mathbf{e}}_j \). 
If \( v \in \mathbb{Z}^m \) and \( j_1 <   j_2 \), then \( p = (v +  d\mathbf{e}_{j_1}) \land  (v+ d\mathbf{e}_{j_2}) \) is a positively oriented 2-cell, also known as a  \emph{positively oriented plaquette}. We let \( C_0(\mathbb{Z}^4) \), \( C_1(\mathbb{Z}^4)\), and \( C_2(\mathbb{Z}^4) \) denote the sets of oriented vertices, edges, and plaquettes.
Next, we let \( B_N \) denote the set \(   [-N,N]^m \subseteq \mathbb{Z}^m \), and we let \( C_0(B_N) \), \( C_1(B_N)\), and \( C_2(B_N) \) denote the sets of oriented vertices, edges, and plaquettes, respectively, whose end-points are all in \( B_N \).

Whenever we talk about a lattice gauge theory we do so with respect to some (abelian) group \( (G,+)  \), referred to as the \emph{structure group}. We also fix a unitary and faithful representation \( \rho \) of \( (G,+) \). In this paper, we will always assume that \( G = \mathbb{Z}_n \) for some \( n \geq 2 \) with the group operation \( + \) given by standard addition modulo \( n \). Also, we will assume that \( \rho \) is a one-dimensional representation of \( G \). We note that a natural such representation is given by \( j\mapsto e^{j \cdot 2 \pi i/n} \).

Now assume that a structure group \( (G,+) \), a one-dimensional unitary representation \( \rho \) of \( (G,+) \), and an integer \( N\geq 1 \) are given.
We let \( \Omega^1(B_N,G) \) denote the set of all  \( G \)-valued  1-forms \( \sigma \) on \( C_1(B_N) \), i.e., the set of all \( G \)-valued functions \(\sigma \colon  e \mapsto \sigma(e) \) on \( C_1(B_N) \) such that \( \sigma(e) =  -\sigma(-e) \) for all \( e \in C_1(B_N) \).
Similarly, we let \( \Omega^0(B_N,G) \) denote the set of all \( G\)-valued functions \( \phi \colon x \mapsto \phi(x)\) on \( C_0(B_N) \) which are such that \( \phi(x) = - \phi(-x) \) for all \( x \in C_1(B_N). \) 
When \( \sigma \in \Omega^1(B_N,G) \) and \( p \in C_2(B_N) \), we let \( \partial p \) denote the formal sum of the four edges \( e_1,\) \( e_2,\) \( e_3,\) and \( e_4 \) in the oriented boundary of \( p \) (see Section~\ref{sec: cell boundary}), and define
\begin{equation*}
    d\sigma(p) \coloneqq \sigma(\partial p) \coloneqq \sum_{e \in \partial p} \sigma(e) \coloneqq \sigma(e_1) + \sigma(e_2) + \sigma(e_3) + \sigma(e_4).
\end{equation*} 
Similarly, when \( \phi \in \Omega^0(B_N,G) \) and \( e \in C_1(B_N) \) is an edge from \( x_1 \) to \( x_2 \), we let \( \partial e \) denote the formal sum \( x_2-x_1, \) and define \( d\phi(e) \coloneqq \phi(\partial e) \coloneqq \phi(x_2) - \phi(x_1). \)

\subsection{The abelian lattice Higgs model}

Given \( \beta, \kappa ,\zeta \geq 0 \), the action \( S_{N,\beta,\kappa, \zeta} \) for lattice gauge theory with Wilson action coupled to a Higgs field on \( B_N \) is, for \( \sigma \in \Omega^1(E_N,G), \) \( \phi \in \Omega^0(B_N,G), \) and a symmetric function \( r \colon C_0(B_N) \to \mathbb{R}_+ \),  defined by
\begin{equation}\label{eq: general action}
    \begin{split}
        S_{N,\beta,\kappa,\zeta}(\sigma, \phi,r)  &\coloneqq 
        -\beta \sum_{p \in C_2(B_N)}  \rho\bigl( d\sigma(p)\bigr) 
        - \kappa\sum_{\substack{e  \in C_1(B_N)\mathrlap{\colon}\\ \partial e = y-x}}  r(x)r(y) \rho\bigl( \sigma(e)-\phi(\partial e)\bigr) 
        \\&\qquad\qquad + \zeta \sum_{x \in C_0(B_N)}  \bigl(r(x)^2-1 \bigr)^2 + \sum_{x \in C_0(B_N)}  r(x)^2. 
    \end{split}
\end{equation}
The first term on the right hand side of~\eqref{eq: general action} is referred to as the \emph{Wilson action functional} for pure gauge theory (see, e.g.,~\cite{w1974}), the second term on the right hand side of~\eqref{eq: general action} is referred to as the interaction term, and the third and fourth term on the right hand side of~\eqref{eq: general action} together are referred to as a \emph{sombrero potential}.
Since \( \phi \in \Omega^0(B_N,G) \) and \( \sigma \in \Omega^1(B_N,G)  \), the action \( S_{N,\beta,\kappa,\zeta}(\sigma, \phi,r) \) is real for all \( \sigma,\) \( \phi, \) and \( r \).
Elements \( \sigma \in \Omega^1(B_N,G) \) will be referred to as \emph{gauge field configurations}, and pairs \( (\phi,r), \) with \( \psi \in \Omega^0(B_N,G) \) and \( r \colon C_0(B_N) \to \mathbb{R}_+ \) symmetric, will be referred to as \emph{Higgs field configurations}.
The quantity \( \beta \) is known as the \emph{gauge coupling constant},  \( \kappa \) is known as the \emph{hopping parameter}, and \( \zeta \) is known as the \emph{quartic Higgs self coupling}.

The Gibbs measure corresponding to the action \( S_{N,\beta,\kappa,\zeta} \) is given by
\begin{equation*}
    d\mu_{N,\beta, \kappa, \zeta}(\sigma, \phi,r) = Z^{-1}_{N,\beta,\kappa, \zeta} e^{-S_{N,\beta,\kappa,\zeta}(\sigma, \phi,r)}
    \prod_{e \in C_1(B_N)^+} d\mu_G\bigl(\sigma(e)\bigr) 
    \prod_{x \in C_0(B_N)^+} d\mu_{G}\bigl(\phi(x)\bigr) \,  d\mu_{\mathbb{R}_+}\! \bigl( r(x)\bigr), 
\end{equation*}
where \( C_1(B_N)^+ \) denotes the set of positively oriented edges in \( C_1(B_N) \), \( d\mu_G \) is the uniform measure on \( G \), and \( \mu_{\mathbb{R}_+} \) is the Lebesgue measure on \( \mathbb{R}_+ \).
We refer to this lattice gauge theory as the \emph{abelian lattice Higgs model}.

We will work with the model obtained from this action in the \emph{fixed length limit} \( \zeta \to \infty\), in which the radial component of the Higgs field concentrates at one. In the physics literature, this is sometimes called the \emph{London limit}. We do not discuss the limit of the Gibbs measure corresponding to \(S_{N,\beta, \kappa, \zeta}\) as \(\zeta \to \infty\) here, but simply from the outset adopt the action resulting from only considering \( r \colon C_0(B_N) \to \mathbb{R}_+\) with \( r(x) = 1 \) for all \( x \in C_0(B_N) .\) In this case, for \( \sigma \in \Omega^1(B_N,G),\) and \( \phi \in \Omega^0(B_N,G)\), we obtain the action
\begin{equation*}
    \begin{split}
        &S_{N,\beta,\kappa, \infty}(\sigma, \phi) \coloneqq   -\beta \sum_{p \in C_2(B_N)}  \rho\bigl( d\sigma(p)\bigr) - \kappa\sum_{\substack{e\in C_1(B_N)\mathrlap{\colon}\\ \partial e = y-x}} 
        \rho\bigl(\sigma(e)-\phi(\partial e)\bigr).
    \end{split}
\end{equation*} 
We then consider a corresponding probability measure \(\mu_{N,\beta, \kappa, \infty}\) on \(\Omega^1(B_N,G) \times \Omega^0(B_N,G)\) given by
\begin{equation*}
    \mu_{N,\beta, \kappa, \infty}(\sigma, \phi)  \coloneqq
    Z_{N,\beta,\kappa, \infty}^{-1} e^{-S_{N,\beta,\kappa, \infty}(\sigma, \phi)} , \qquad \sigma \in \Omega^1(B_N,G) ,\, \phi \in \Omega^0(B_N,G),
\end{equation*}
where \( Z_{N,\beta,\kappa, \infty}\) is a normalizing constant. This is the \emph{fixed length lattice Higgs model}. We let \( \mathbb{E}_{N,\beta,\kappa,\infty} \) denote the corresponding expectation. Whenever \( f \colon \Omega^1(B_M,G) \times \Omega^0(B_M,G) \to \mathbb{R}\) for some \( M \geq 1,\) then, as a consequence of the Ginibre inequalities (see Section~\ref{sec: ginibre}), the infinite volume limit 
\begin{equation*}
    \bigl\langle f(\sigma,\phi) \bigr\rangle_{\beta,\kappa,\infty} \coloneqq \lim_{N \to \infty} \mathbb{E}_{N,\beta,\kappa,\infty} \bigl[f(\sigma,\phi) \bigr]
\end{equation*}
exists, and it is this limit that we will use in our main result.

\subsection{Wilson loops and Wilson lines}

For \( k \in \{ 0,1,\dots, m,\) a \( k \)-chain is a formal sum of positively oriented k-cells with integer coefficients, see Section~\ref{sec: chains} below. The support of a \(1\)-chain \(\gamma\), written \(\support \gamma\), is the set of directed edges with non-zero coefficient in \( \gamma.\) 
We say that a \(1\)-chain with finite support is a \emph{generalized loop} if it has coefficients in \(\{-1,0,1\}\) and empty boundary, see Definition~\ref{def: generalized loop}. Roughly speaking, this means that a generalized loop is a disjoint union of a finite number of closed loops, where each closed loop is a nearest-neighbor path in the graph \( \mathbb{Z}^4  \) starting and ending at the same vertex. For example, any rectangular loop, as well as any finite disjoint union of such loops, is a generalized loop.
We say that a \(1\)-chain with finite support is an \emph{open path} from \( x_1\in \Omega_0^+(B_N) \) to \( x_2 \in \Omega_2^+(B_N)\) if it has coefficients in \(\{-1,0,1\}\) and boundary \( \partial \gamma \coloneqq x_2 - x_1. \)
If \( \gamma \) is either an open path or a generalized loop, we refer to \( \gamma \) as a \emph{path}.
%

Given a path \( \gamma \), the \emph{Wilson line observable} \( L_\gamma(\sigma,\phi) \) is defined by 
\begin{equation*}
    L_\gamma(\sigma,\phi)
    \coloneqq  \rho \bigl( \sigma(\gamma) - \phi(\partial \gamma) \bigr),\qquad \sigma\in \Omega^1(B_N,G),\, \phi\in \Omega^0(B_N,G),
\end{equation*}
where \( \sigma(\gamma) \coloneqq \sum_{e \in  \gamma} \sigma(e), \) and \( \phi(\partial \gamma) = \phi(x_2) - \phi(x_1) \) if \( \gamma \) is an open path from \( x_1 \) to \( x_2 \), and \(\phi(\partial \gamma) = 0 \) if the boundary of \( \gamma \) is empty. If \( \gamma \) is a generalized loop, then \(  W_\gamma(\sigma) \coloneqq L_\gamma(\sigma,\phi)\) is referred to as a \emph{Wilson loop observable}.

\subsection{Main results}

\begin{theorem}\label{theorem: main result Z2}
    Consider the fixed length lattice Higgs model on \( \mathbb{Z}^4 \), with structure group \(G = \mathbb{Z}_2\), and representation \( \rho\colon G \to \mathbb{C} \) given by \( \rho(0) = 1 \) and \( \rho(1)=-1 \). 
    Let \( \beta,\kappa \geq 0 \) be such that \( 18^2 e^{-4\kappa}(2 + e^{-4\kappa})< 1\) and \( 6\beta > \kappa . \) Further, let \( \gamma \) be a path along the boundary of a rectangle with side lengths \( \ell_1,\ell_2 \geq 8, \) and assume that \( |\support \gamma| \geq 24.\)
    Finally, let \( e \in C_1(\mathbb{Z}^4) \) be arbitrary.
    Then
    \begin{equation}\label{eq: main result Z2}
        \pigl| \bigl\langle L_\gamma(\sigma,\phi) \bigr\rangle_{\beta,\kappa,\infty} - \Theta'_{\beta,\kappa}(\gamma) H_\kappa(\gamma) \pigr| \leq K_0 \pigl( e^{-4(\beta+\kappa/6)} + |\support \gamma|^{-1/2} \pigr)^{\frac{1}{4}},
    \end{equation}
    where
    \begin{equation*}
        \Theta'_{\beta,\kappa}(\gamma) \coloneqq e^{-2|\support \gamma|e^{-24\beta-4\kappa}\bigl(1+(e^{8\kappa}-1) \langle L_{e}(\sigma,\phi) \rangle_{\infty,\kappa,\infty}\bigr)},
    \end{equation*}
    \begin{equation*}
        H_\kappa(\gamma) \coloneqq  \bigl\langle L_\gamma(\sigma,\phi) \bigr\rangle_{\infty,\kappa,\infty},
    \end{equation*}
    and \( K_0 = K_0(\kappa,\beta,\ell_1,\ell_2,\gamma) \) is a non-negative function with
    \begin{equation*}
        K_0 \leq  2 \cdot 18^3 +|\support \gamma|^{1/2} e^{-4\kappa}\bigl( 18^2 (2+e^{-4\kappa}) \bigr)^{\min(\ell_1,\ell_2)}  + o_\kappa(1).
    \end{equation*} 
\end{theorem}

An exact expression for \( K_0 \) is given in~\eqref%
{eq: constant in main theorem Z2}.

\begin{remark}\label{remark: ising}
    We later show, in Corollary~\ref{corollary: Ising}, that if \( \gamma \) is an open path, then the function \( H_\kappa(\gamma) \) is exactly equal to the spin-spin-correlation of the spins at the endpoints of \( \gamma \) in the Ising model with coupling parameter \( \kappa. \) By the same argument, the term \( \langle L_e(\sigma,\phi) \rangle_{\infty,\kappa,\infty}  \) in the function \( \Theta'_{\beta,\kappa}(\gamma) \) will be equal to the spin-spin-correlation of the spins at the end-points of the (arbitrary) edge \( e. \) 
    It is well known (see, e.g.,~\cite{d2017}) that when \( \kappa \) is larger than the critical parameter for the Ising model, then  \( H_\kappa(\gamma) \) is uniformly bounded from below for all \( \gamma.\) At the same time, by standard arguments, we have \( \bigl\langle L_e(\sigma,\phi) \bigr\rangle_{\infty,\kappa,\infty} = e^{-4 \cdot 8 \kappa} + o_\kappa(1). \)
\end{remark}

\begin{remark}\label{remark: interpretation}
    Using the previous remark, we now interpret our main theorem. To this, end, assume that \( \gamma \) is a loop along the boundary of a rectangle \( R \). Assume further that the two sides of \( R \) are of the same order, so that \( K_0 \) is bounded from above, and that \( \beta \) and \( |\support \gamma| \) are both very large. Then, by Theorem~\ref{theorem: main result Z2}, the following holds. If \( |\support \gamma|e^{-24\beta-4\kappa} \) is very large, then \( \langle L_\gamma(\sigma,\phi)\rangle_{\beta,\kappa,\infty} \) is very close to zero, and if \( |\support \gamma|e^{-24\beta-4\kappa} \) is bounded from above, then \( \langle L_\gamma(\sigma,\phi) \rangle_{\beta,\kappa,\infty} \) will be non-trivial.
\end{remark}

\begin{remark}
    The assumption that \( 18^2 e^{-4\kappa_0}(2 + e^{-4\kappa_0})<1 \) guarantees that the clusters formed by the edges in unitary gauge (see Section~\ref{sec: unitary gauge}) are finite almost surely, and this is one of the main properties of the model which we use in the proof of Theorem~\ref{theorem: main result Z2}. The assumptions that \( 6\beta \geq  \kappa \) and that the path \( \gamma \) is along the boundary of a rectangle is used only to simplify \( \Theta'_{\beta,\kappa}(\gamma) \) and \( K_0, \) and is not needed for any of the main ideas of the proof. 
    In particular, the strategy used to do this also works for more general classes of loops, as long as their shape is not too rough.
\end{remark}

\begin{remark} 
    If \( \gamma \) is a generalized loop, then \( H_\kappa(\gamma) = 1 \), and hence, in this case, we essentially recover Theorem~1.1~\cite{flv2021}.
\end{remark}

In Section~\ref{sec: proof of main result}, we state a more general version of Theorem~\ref{theorem: main result Z2} (Theorem~\ref{theorem: main result}).
While this result is stated for cyclic groups, with minor changes, this paper's arguments should also work in a more general setting. In particular, the proof strategy should work for all finite abelian groups.
Finally, we also mention that alternative versions of our main result, with different error bounds, are given by Proposition~\ref{proposition: short lines} and Proposition~\ref{proposition: first version of main result}.

\subsection{Applications}
We now apply our main result to a few different Wilson lines, and ratios of Wilson lines, which has been considered in the physics literature. 
In all of these examples, we will work under the assumptions of Theorem~\ref{theorem: main result Z2}. We note that, when these hold, if \( \gamma \) is a loop along a rectangle with side lengths of the same order, then, by Theorem~\ref{theorem: main result Z2}, we have
\begin{equation*}
    \bigl\langle L_\gamma(\sigma,\phi) \bigr\rangle_{\beta,\kappa,\infty} = \Theta'_{\beta,\kappa}(\gamma) H_\kappa(\gamma) + o_{\beta}(1) + o_{|\support \gamma|}(1).
\end{equation*}

\begin{example} 
    In~\cite{bf1983}, Bricmont and Fr\"olich consider Wilson line observables \( L_{\gamma}(\sigma,\phi) \) for axis parallel paths~\( \gamma \) which are a shortest path between two points \( x_1 \) and \( x_2 \) (see Figure~\ref{figure: straight Wilson line}).
    \begin{figure}[!htp]
        \centering
        \begin{tikzpicture}
            \draw[detailcolor00] (0,0) -- (8,0); 
            \draw[detailcolor00,-{Straight Barb[length=1.5mm,color=detailcolor00!40!black]}] (3.95,0) -- (4.0,0);
            \fill[detailcolor04] (0,0) circle (2pt) node[anchor=south] {\color{black}\( x_1\)};
            \fill[detailcolor04] (8,0) circle (2pt) node[anchor=south] {\color{black}\( x_2\)};
        \end{tikzpicture}
        \caption{The open path \( \gamma \). Note that for any \( \ell_1 \geq |x_2-x_1|\) and \( \ell_2 \geq 0 \) there is a rectangle \( R \) with side lengths \( \ell_1 \) and \( \ell_2 \) so that \( \gamma \) is a path along the boundary of \( R \), and hence \( \gamma \) satisfies the assumptions of Theorem~\ref{theorem: main result Z2}.} 
        \label{figure: straight Wilson line}
    \end{figure}
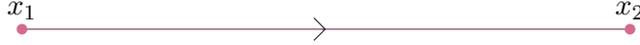
    The authors argue that the expectation \( \langle L_\gamma(\sigma,\phi) \rangle_{\beta,\kappa} \) should exhibit a phase transition, corresponding to binding versus unbinding of dynamical quarks in the field of a static colour source.
    In detail, they argue that \( \langle L_\gamma(\sigma,\phi) \rangle_{\beta,\kappa,\infty} \) should have exponential decay with polynomial corrections if \( \beta \) is large and \( \kappa \) is small, and exponential decay if either \( \beta \) is large and \( \kappa \) is not too small. 
    Since, under assumption~\ref{assumption: 3}, \( H_\kappa(\gamma) \) is uniformly bounded from below for all \( \gamma\), and \( \Theta'_{\beta,\kappa}(\gamma) \) has exponential decay in \( |\support \gamma|,\) we see that \( \langle L_\gamma(\sigma,\phi) \rangle_{\beta,\kappa,\infty} \) indeed has exponential decay in \( |\support \gamma| \) when \( \beta \) is large and \( \kappa \) is not too small.
\end{example}

\begin{example}[The Marcu-Fredenhagen parameter]\label{example: line loop ratio}
    
    Let \( \gamma \) and \( \gamma' \) be as in Figure~\ref{figure: U Wilson line}. In~\cite{fmf1986, m}, they consider the ratio
    \begin{equation}\label{eq: line loop ratio}
        \frac{\langle L_{\gamma'}(\sigma,\phi) \rangle_{\beta,\kappa,\infty}\langle L_{\gamma-\gamma'}(\sigma,\phi) \rangle_{\beta,\kappa,\infty} }{\langle W_{\gamma}(\sigma)\rangle_{\beta,\kappa,\infty}}.
    \end{equation}
    \begin{figure}[!htp]
        \centering
        \begin{subfigure}[b]{0.45\textwidth}
            \centering
            \begin{tikzpicture}
                \draw[detailcolor00] (0,2.5) -- (0,0) node[midway, left] {\color{black}\( h\)} -- (4,0) node[midway,anchor=north] {\color{black}\( |x_2-x_1|\)} -- (4,2.5); 
                \draw[detailcolor00,-{Straight Barb[length=1.5mm,color=detailcolor00!40!black]}] (1.95,0) -- (2.0,0);
                \fill[detailcolor04] (0,2.5) circle (2pt) node[anchor=south] {\color{black}\( x_1\)};
                \fill[detailcolor04] (4,2.5) circle (2pt) node[anchor=south] {\color{black}\( x_2\)};
            \end{tikzpicture}
            \caption{The open path \(  \gamma' \).} 
        \end{subfigure}
        \hfil
        \begin{subfigure}[b]{0.45\textwidth}
            \centering
            \begin{tikzpicture}
                \draw[detailcolor00] (0,5) -- (0,0) node[midway, left] {\color{black}\( 2h\)} --  (4,0) node[midway,anchor=north] {\color{black}\( |x_2-x_1|\)}  -- (4,5) -- (0,5); 
                \draw[detailcolor00,-{Straight Barb[length=1.5mm,color=detailcolor00!40!black]}] (1.95,0) -- (2.0,0);
            \end{tikzpicture}
            \caption{The loop \( \gamma \).} 
        \end{subfigure}
        \caption{The open path \( \gamma' \) and the generalized loop \( \gamma \) considered in Example~\ref{example: line loop ratio}.}
        \label{figure: U Wilson line}
    \end{figure}
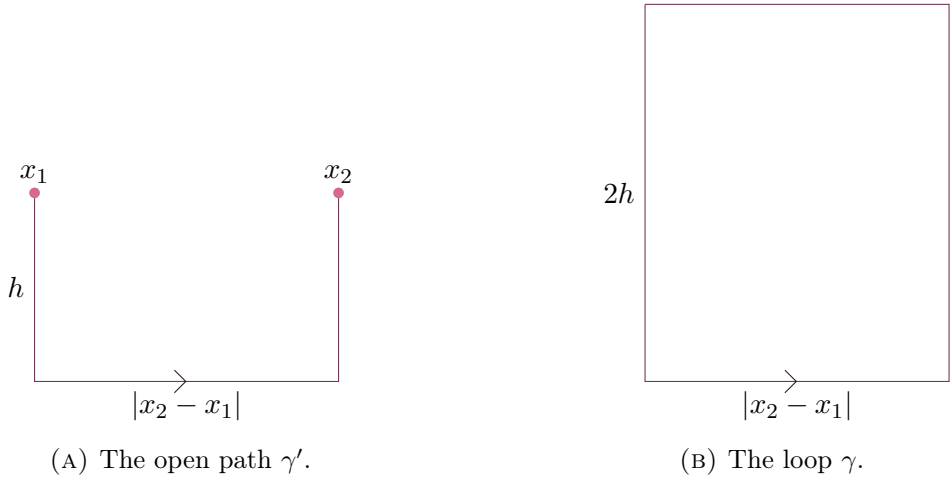
    The limit of this ratio, when \( |x_2-x_1| \) is proportional to \( h \) and \( h \to \infty, \) is often referred to as the \emph{Marcu-Fredenhagen order parameter}.
    If this limit is zero, the model in argued to have charged states, and no confinement, whereas if the limit is non-zero, then there should be no charged states and confinement.
    We mention that this ratio is also studied in, e.g.,~\cite{m,s1988,bf1987,fm1988,ghms2011}.
    
    As an immediate consequence of our Theorem~\ref{theorem: main result Z2}, if \(\kappa \) is not too small and \( \beta,\) \(\kappa,\) and \( \gamma \) are such that \( |\support \gamma|e^{-24\beta-4\kappa}\) is bounded away from infinity, then the right hand side of~\eqref{eq: line loop ratio} is equal to \( H_\kappa(\gamma')^2 + o_{|\support \gamma|}(1) + o_{\beta}(1).\) 
    However, since letting \( |\support \gamma| \) tend to infinity while keeping \( \beta \) and \( \kappa \) fixed violates that assumption that \( |\support \gamma|e^{-24\beta-4\kappa}\) is bounded away from infinity, we cannot use this approximate equation to make conclusions about the Marcu-Fredenhagen parameter itself. 
\end{example}

\begin{example}\label{example: closing box and almost closed}
    Let \( \gamma \) and \( \gamma' \) be as in Figure~\ref{fig: gliozzi 1}. 
    In~\cite{g2006}, Gliozzi considers the ratio
    \begin{equation}\label{eq: closing box and almost closed}
        \frac{ \langle L_{\gamma'}(\sigma,\phi) \rangle_{\beta,\kappa,\infty}
        L_{\gamma-\gamma'}(\sigma,\phi)\rangle_{\beta,\kappa,\infty} }{\langle W_{\gamma}(\sigma) \rangle_{\beta,\kappa,\infty}},
    \end{equation}
    and note that it, asymptotically, seem to only depend on the distance \( |x_2-x_1|. \) Indeed, from Theorem~\ref{theorem: main result Z2}, it follows that if \( |\support \gamma|e^{-24\beta-4\kappa}\) is bounded away from infinity and \( \kappa \) is not too small, then the expression in~\eqref{example: closing box and almost closed} is equal to \( H_\kappa(\gamma')^2 + o_{|\support \gamma|}(1) + o_{\beta}(1). \) Using Remark~\ref{remark: ising} to recognise \( H_\kappa(\gamma') \) as the spin-spin-correlation function for the Ising model, evaluated at the end-points of \( \gamma' \), this confirms the observation made in~\cite{g2006}.
    \begin{figure}[!htp]
        \centering
        \begin{subfigure}[b]{0.45\textwidth}
            \centering
            \begin{tikzpicture}
                \draw[detailcolor00] (0,2.5) -- (0,0) node[midway, left] {\color{black}\( h\)} -- (4,0) node[midway,anchor=north] {\color{black}\( |x_2-x_1|\)} -- (4,2.5); 
                \draw[detailcolor00,-{Straight Barb[length=1.5mm,color=detailcolor00!40!black]}] (1.95,0) -- (2.0,0);
                \fill[detailcolor04] (0,2.5) circle (2pt) node[anchor=south] {\color{black}\( x_1\)};
                \fill[detailcolor04] (4,2.5) circle (2pt) node[anchor=south] {\color{black}\( x_2\)};
            \end{tikzpicture}
            \caption{The open path \( \gamma'\).} 
        \end{subfigure}
        \hfil
        \begin{subfigure}[b]{0.45\textwidth}
            \centering
            \begin{tikzpicture}
                \draw[detailcolor00] (0,2.5) -- (0,0) node[midway, left] {\color{black}\( h\)} -- (4,0) node[midway,anchor=north] {\color{black} \( |x_2-x_1|\)} -- (4,2.5) -- (0,2.5) node[midway, anchor=south] {}; 
                \draw[detailcolor00,-{Straight Barb[length=1.5mm,color=detailcolor00!40!black]}] (1.95,0) -- (2.0,0);
            \end{tikzpicture}
            \caption{The generalized loop \( \gamma \).} 
            \end{subfigure}
        \caption{The open path \( \gamma' \) and generalized loop \( \gamma \) considered in Example~\ref{example: closing box and almost closed}.}
        \label{fig: gliozzi 1}
    \end{figure}
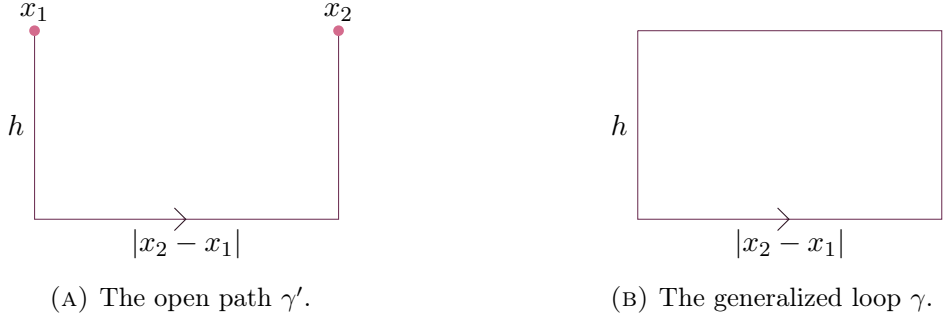

\end{example}

\begin{example}[Almost closed Wilson lines]\label{example: almost closed}
    Let \( \gamma \) and \( \gamma' \) be as in Figure~\ref{figure: almost closed Wilson line}, and let \( r \) be the distance between the endpoints of \( \gamma'. \) In~\cite{g2006}, when  \( r\) is much smaller that \( |\support \gamma|,\) the path \( \gamma' \) is referred to as an \emph{almost closed Wilson line}, and it was argued that the following functional equation should hold.
    \begin{equation}\label{eq: functional equation}
        \bigl\langle L_{\gamma'}(\sigma,\phi) \bigr\rangle_{\beta,\kappa,\infty} 
        \bigl\langle L_{\gamma-\gamma'}(\sigma,\phi)\bigr\rangle_{\beta,\kappa,\infty} 
        \simeq
        \bigl\langle W_{\gamma}(\sigma) \bigr\rangle_{\beta,\kappa,\infty} f(r)
    \end{equation}
    for some function \( f(r) \) that should neither depend on \( \gamma \) nor on the placement of the open path \( \gamma-\gamma' \) on \( \gamma. \) 
    Using our main result, it indeed see that if \( \kappa\) is not too small, then (assuming that the side lengths of the rectangle are proportional to \( |\support \gamma| \) is large), we have
    \begin{equation*}
        \bigl\langle L_{\gamma'}(\sigma,\phi) \bigr\rangle_{\beta,\kappa,\infty} 
        \bigl\langle L_{\gamma-\gamma'}(\sigma,\phi)\bigr\rangle_{\beta,\kappa,\infty} 
        =
        \bigl\langle W_{\gamma}(\sigma) \bigr\rangle_{\beta,\kappa,\infty}   H_\kappa(\gamma')^2 + o_{|\support \gamma|}(1) + o_{\beta+6\kappa}(1).
    \end{equation*} 
    In particular, using Remark~\ref{remark: ising}, this shows that the functional equation in~\eqref{eq: functional equation} indeed hold when \( |\support \gamma| \) and \( \beta \) are both large, and with \( f(r) \) given by the spin-spin-correlation function evaluated at the endpoints of \( \gamma'. \) 
    \begin{figure}[!htp]
        \centering
        \begin{subfigure}[b]{0.45\textwidth}
            \centering
            \begin{tikzpicture}[scale=0.8] 
                \draw[detailcolor00] (3.5,3) -- (0,3) -- (0,0) -- (6,0) -- (6,3) -- (5,3); 
                \fill[detailcolor04] (3.5,3) circle (2.5pt);
                \fill[detailcolor04] (5,3) circle (2.5pt);
                \draw[detailcolor00,-{Straight Barb[length=1.5mm,color=detailcolor00!40!black]}] (2.95,0) -- (3.0,0);
            \end{tikzpicture}
            \caption{The open path \(  \gamma'. \)}
        \end{subfigure}
        \hfil
        \begin{subfigure}[b]{0.45\textwidth}
            \centering
            \begin{tikzpicture}[scale=0.8]
                \draw[detailcolor00] (3.5,3) -- (0,3) -- (0,0) -- (6,0)  -- (6,3)  -- (5,3) -- (3.5,3);  
                \draw[detailcolor00,-{Straight Barb[length=1.5mm,color=detailcolor00!40!black]}] (2.95,0) -- (3.0,0);
            \end{tikzpicture}
            \caption{The generalized loop \(  \gamma .\)}
        \end{subfigure}
        \caption{The open path \( \gamma' \) and the loop \( \gamma \) considered considered in Example~\ref{example: almost closed}.} 
        \label{figure: almost closed Wilson line}
    \end{figure}
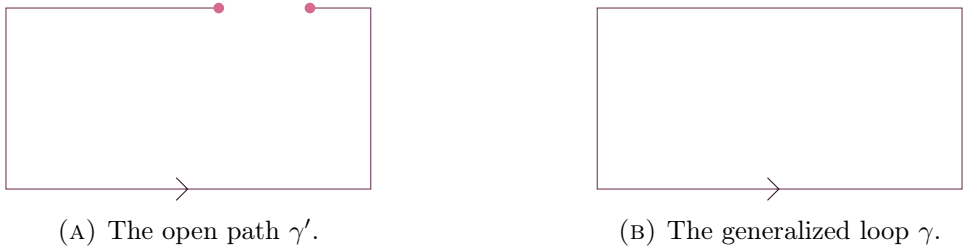
\end{example}

\subsection{Relation to other work}

Many of the ideas used in this paper are refined versions of analogue ideas used in~\cite{flv2021}, which in turn build upon the works~\cite{c2019,flv2020,sc2019}. However, since this paper deals with general paths \( \gamma, \) and not only generalized loops as in~\cite{c2019,flv2020,sc2019,flv2021}, the first main idea in these papers, which is to pass from a generalized loop to an oriented surface, does not work. One of the main contributions of this paper thus consists in dealing with this obstacle. Even in the case when the path \( \gamma \) in Theorem~\ref{theorem: main result Z2} is a generalized loop, our proof is different from that in~\cite{flv2021}, and we hence provides an alternative proof in this case. In addition, when \( \gamma \) is a generalized loop, we express the leading-order term in a more transparent way than in~\cite{flv2021}.

We mention that although the recent paper~\cite{gs2021} also calculate the first order term of Wilson loop observables in an abelian lattice gauge theory, they work with a continuous structure group, and thus their methods are fundamentally different from the ideas used here.

\subsection{Structure of paper}

In Section~\ref{sec: preliminaries}, we give a brief introduction to the cell complex of \( \mathbb{Z}^m \) and the discrete exterior calculus on this cell complex. We also define vortices and recall some of their properties from~\cite{flv2020}~and~\cite{flv2021}. Moreover, we recall the definition of generalized loops and oriented surfaces from~\cite{flv2021}, explain unitary gauge and define a corresponding measure, and discuss the existence of the infinite volume limit \(\mathbb{E}_{\beta,\kappa,\infty}[L_\gamma(\sigma,\phi)]. \)
In Section~\ref{sec: additional notation}, we introduce additional notation which will be useful throughout the paper.
In Section~\ref{sec: activity}, we recall the notion of activity of gauge field configurations from~\cite{flv2021}.
In Section~\ref{sec: couplings}, we describe a useful edge graph, and introduce two couplings, one between the abelian lattice Higgs model and a \( \mathbb{Z}_n \)-model, and one between two \( \mathbb{Z}_n \)-models. These will be important in the proof of our main result.
In Section~\ref{sec: general cluster events}, using the edge graph from Section~\ref{sec: couplings}, we give upper bounds on a number of events related to the couplings introduced in Section~\ref{sec: couplings}.
Next, in Section~\ref{sec: first version of main result}, we show how one of the couplings introduced in Section~\ref{sec: couplings} can be used to obtain a first version of our main result, which is useful when \( |\support \gamma| e^{-4(\kappa + 6\beta)}\) is small. This result is not needed for the proof of Theorem~\ref{theorem: main result Z2}, but illustrates the usefulness of the coupling.
In Section~\ref{sec: spin config decomposition}, we introduce a spin decomposition of two coupled configurations.
In Section~\ref{sec: 1forms}, we describe how different 1-forms affect the Wilson line observable.
Finally, in Section~\ref{sec: proof of main result}, we use the setup from the earlier sections to give a proof of our main result.

\subsection{Acknowledgements}
The author is grateful to Jonatan Lenells and Fredrik Viklund for many useful discussions.

\subsection{Funding}
The author acknowledges support from the Knut and Alice Wallenberg Foundation, from the Swedish Research Council, Grant Agreement No. 2015-05430, and from the European Research Council, Grant Agreement No. 682537.

\section{Preliminaries}\label{sec: preliminaries}

\subsection{The cell complex}
In this section, we introduce notation for the cell complexes of the lattices \( \mathbb{Z}^m \) and \( B_N \coloneqq [-N,N]^m \cap  \mathbb{Z}^m \) for \( m,N \geq 1 \). This section will closely follow the corresponding section in~\cite{flv2020}, where we refer the reader for further details.

To simplify notation, we define \( e_1 \coloneqq (1,0,\dots,0) \), \( e_2 \coloneqq (0,1,0,\dots,0) \), \ldots, \( e_m \coloneqq (0,\dots,0,1) \).

\subsubsection{Boxes and cubes}
A set \( B \) of the form \( \bigl( [a_1,b_1] \times \cdots \times [a_m,b_m] \bigr) \cap \mathbb{Z}^m\) where, for each \( j \in \{ 1,2, \ldots, m \} \), \( \{a_j, b_j\}  \subset \mathbb{Z} \) satisfies \(a_j < b_j\), will be referred to as a  \emph{box}. If all the intervals \( [a_j,b_j]\), \(1 \leq j \leq m\), have the same length, then the set \( \bigl( [a_1,b_1] \times \cdots \times [a_m,b_m] \bigr) \cap \mathbb{Z}^m\) will be referred to as a {\it cube}.

\subsubsection{Non-oriented cells}

When \( a \in \mathbb{Z}^m \), \( k \in \{ 0,1, \dots, m \} \), and \( \{ j_1,\dots, j_k \} \subseteq \{ 1,2, \dots, m \} \), we say that the set
\begin{equation*}
    (a; e_{j_1}, \dots, e_{j_k}) \coloneqq \bigl\{ x \in \mathbb{R}^m \colon \exists b_1, \dots, b_k \in [0,1] \text{ such that } x = a + \sum_{i=1}^k b_i e_{j_i}  \bigr\}
\end{equation*}
is a \emph{non-oriented \( k \)-cell}. Note that if \( \sigma \) is a permutation, then \( (a; e_{j_1}, \dots, e_{j_k}) \) and \( (a; \sigma(e_{j_1}, \dots, e_{j_k})) \) represent the same non-oriented \( k \)-cell.

\subsubsection{Oriented cells}\label{sec: oriented cells}
To each non-oriented $k$-cell \( (a; e_{j_1}, \dots, e_{j_k}) \) with \( a \in \mathbb{Z}^m \), \( k \geq 1 \), and \( 1\leq j_1 < \dots < j_k\leq m \), we associate two \emph{oriented \( k \)-cells}, denoted \(  \frac{\partial}{\partial x^{j_1}}\big|_a \wedge \dots \wedge \frac{\partial}{\partial x^{j_k}}\big|_a\) and \( -\frac{\partial}{\partial x^{j_1}}\big|_a \wedge \dots \wedge \frac{\partial}{\partial x^{j_k}}\big|_a \), with opposite orientation.  
When \( a \in \mathbb{Z}^m \), \( 1\leq j_1 < \dots < j_k\leq m \), and \( \sigma  \) is a permutation of \( \{ 1,2, \dots, k \} \), we define
\begin{equation*}
    \frac{\partial}{\partial x^{j_{\sigma(1)}}}\bigg|_a \wedge \dots \wedge \frac{\partial}{\partial x^{j_{\sigma(k)}}}\bigg|_a 
    \coloneqq 
    \sgn(\sigma) \, 
    \frac{\partial}{\partial x^{j_1}}\bigg|_a \wedge \dots \wedge \frac{\partial}{\partial x^{j_k}}\bigg|_a
\end{equation*}
If \( \sgn(\sigma)=1 \), then  \( \frac{\partial}{\partial x^{j_{\sigma(1)}}}\big|_a \wedge \dots \wedge \frac{\partial}{\partial x^{j_{\sigma(k)}}}\big|_a \) is said to be \emph{positively oriented}, and if \( \sgn(\sigma)=-1 \), then \( \frac{\partial}{\partial x^{j_{\sigma(1)}}}\big|_a \wedge \dots \wedge \frac{\partial}{\partial x^{j_{\sigma(k)}}}\big|_a  \) is said to be \emph{negatively oriented}. 
Analogously, we define \begin{equation*}
    -\frac{\partial}{\partial x^{j_{\sigma(1)}}}\bigg|_a \wedge \dots \wedge \frac{\partial}{\partial x^{j_{\sigma(k)}}}\bigg|_a 
    \coloneqq 
    -\sgn(\sigma) \, 
    \frac{\partial}{\partial x^{j_1}}\bigg|_a \wedge \dots \wedge \frac{\partial}{\partial x^{j_k}}\bigg|_a,
\end{equation*}
and say that \( -\frac{\partial}{\partial x^{j_{\sigma(1)}}}\bigg|_a \wedge \dots \wedge \frac{\partial}{\partial x^{j_{\sigma(k)}}}\bigg|_a  \) is positively oriented if \( -\sgn(\sigma) = 1 \), and negatively oriented if \( -\sgn(\sigma) = -1. \)

Let $\mathcal{L} = \mathbb{Z}^m$ or $\mathcal{L} = B_N \subseteq \mathbb{Z}^m$. An oriented cell \( \frac{\partial}{\partial x^{j_1}}\big|_a \wedge \dots \wedge \frac{\partial}{\partial x^{j_k}}\big|_a \) is said to be in \( \mathcal{L} \) if all corners of \( (a;e_{j_1},\dots, e_{j_k})\) belong to \( \mathcal{L} \); otherwise it is said to be {\it outside} $\mathcal{L}$. 
The set of all oriented \( k \)-cells in \( \mathcal{L} \) will be denoted by \( C_k(\mathcal{L}). \) The set of all positively and negatively oriented cells in \( C_k(\mathcal{L}) \) will be denoted by \( C_k^+(\mathcal{L})\) and \( C_k^-(\mathcal{L})\), respectively.
A set \( C \subseteq C_k(\mathcal{L}) \) is said to be \emph{symmetric} if for each \( c \in C \) we have \( -c \in C \).

A non-oriented 0-cell \( a\in \mathbb{Z}^m \) is simply a point, and to each point we associate two oriented \(0\)-cells \( a^+ \) and \( a^-  \) with opposite orientation. We let \( C_0(\mathcal{L}) \) denote the set of all oriented \( 0 \)-cells.

Oriented 1-cells will be referred to as~\emph{edges}, and oriented 2-cells will be referred to as~\emph{plaquettes}.

\subsubsection{\( k \)-chains}\label{sec: chains}

The space of finite formal sums of positively oriented \( k \)-cells with integer coefficients will be denoted by \( C_k(\mathcal{L},\mathbb{Z}) \). 
Elements of \( C_k(\mathcal{L},\mathbb{Z}) \) will be referred to as \emph{\( k \)-chains}. 
If \( q \in C_k(\mathcal{L},\mathbb{Z}) \) and \( c \in C^+_k(\mathcal{L}) \), we let \( q[c] \) denote the coefficient of \( c \) in \( q \).
If \( c \in C^-_k(\mathcal{L}) \), we let \( q[c]\coloneqq -q[-c]. \)
For \(q,q' \in  C_k(\mathcal{L},\mathbb{Z}) \), we define
\begin{equation*}
    q+q' \coloneqq \sum_{c \in C_k^+(\mathcal{L})} \bigl(q[c] + q'[c] \bigr) c.
\end{equation*}
Using this operation, \( C_k(\mathcal{L},\mathbb{Z}) \) becomes a group.

When \( q \in C_k(\mathcal{L},G) \), we let the \emph{support} of \( q \) be defined by
\begin{equation*}
    \support q \coloneqq \bigl\{ c \in C_k^+(\mathcal{L}) \colon q[c] \neq 0 \bigr\}.
\end{equation*}

To simplify notation, when \( q \in C_k(\mathcal{L},G) \) and \( c\in C_k(\mathcal{L}) \), we write \( c \in q \) if either
\begin{enumerate}
    \item \( c \in C_k^+(\mathcal{L}) \) and \( q[c]>0\), or
    \item \( c \in C_k^-(\mathcal{L}) \) and \( q[-c]<0. \)
\end{enumerate}

\subsubsection{The boundary of a cell}\label{sec: cell boundary}

When \( k \geq 2 \), we define the \emph{boundary} \(\partial c \in C_{k-1}(\mathcal{L}, \mathbb{Z})\) of  \( c = \frac{\partial}{\partial x^{j_1}}\big|_a \wedge \dots \wedge \frac{\partial}{\partial x^{j_k}}\big|_a \in C_k(\mathcal{L})\) by
\begin{align}
    \label{eq: boundary chain}
        \partial c \coloneqq \sum_{k' \in \{ 1,\dots, k \}}  \biggl(&
        (-1)^{k'}  \frac{\partial}{\partial x^{j_1}}\bigg|_a \wedge \dots \wedge \frac{\partial}{\partial x^{j_{k'-1}}}\bigg|_a \wedge  \frac{\partial}{\partial x^{j_{k'+1}}}\bigg|_a \wedge \dots \wedge \frac{\partial}{\partial x^{j_k}}\bigg|_a
        \\\nonumber
        & + (-1)^{k'+1}   
        \frac{\partial}{\partial x^{j_1}}\bigg|_{a + e_{j_{k'}}} \wedge \dots \wedge \frac{\partial}{\partial x^{j_{k'-1}}}\bigg|_{a + e_{j_{k'}}} \wedge  \frac{\partial}{\partial x^{j_{k'+1}}}\bigg|_{a + e_{j_{k'}}} \wedge \dots \wedge \frac{\partial}{\partial x^{j_k}}\bigg|_{a + e_{j_{k'}}} 
        \biggr). 
\end{align}
When \( c \coloneqq \frac{\partial}{\partial x^{j_1}}\big|_a \in C_1(\mathcal{L})\) we define the boundary \( \partial c \in C_0(\mathcal{L},\mathbb{Z}) \) by
\begin{equation*}
    \partial c = (-1)^1 a^+ + (-1)^{1+1} 
    (a+e_{j_1})^+ = (a+e_{j_1})^+ - a^+.
\end{equation*}

We extend the definition of \( \partial \) to \( k \)-chains \( q \in C_k(\mathcal{L},\mathbb{Z}) \) by linearity.
One verifies, as an immediate consequence of this definition, that if \( k \in \{ 2,3, \dots, m \} \), then \( \partial \partial c = 0 \) for any \( c \in \Omega_k(\mathcal{L}). \)

\subsubsection{The coboundary of an oriented cell}\label{sec: coboundary}

If \( k \in \{ 0,1, \ldots, n-1 \} \) and \( c \in C_k(\mathcal{L})\) is an oriented \( k \)-cell, we define the \emph{coboundary} \( \hat \partial c \in C_{k+1}(\mathcal{L})\) of \( c \) as the \( (k+1) \)-chain 
\begin{equation*}
	\hat \partial c \coloneqq \sum_{c' \in C_{k+1}(\mathcal{L})} \bigl(\partial c'[c] \bigr) c'.
\end{equation*}
Note in particular that if \( c' \in C_{k+1}(\mathcal{L}),\) then \( \hat \partial c[c'] = \partial c'[c]. \) We extend the definition of \( \hat{\partial} \) to \( k \)-chains \( q \in C_k(\mathcal{L},\mathbb{Z}) \) by linearity.

\subsubsection{The boundary of a box}

An oriented \( k \)-cell \( c = \frac{\partial}{\partial x^{j_1}}\big|_a \wedge \dots \wedge \frac{\partial}{\partial x^{j_k}}\big|_a \in C_k(B_N)\) is said to be a \emph{boundary cell} of a box \( B = \bigl( [a_1,b_1]\times \dots \times [a_m,b_m] \bigr) \cap \mathbb{Z}^m \subseteq B_N\), or equivalently to be in \emph{the boundary} of \( B \), if the non-oriented cell \( (a;e_{j_1}, \dots, e_{j_k}) \) is a subset of the boundary of 
\(  [a_1,b_1]\times \dots \times [a_m,b_m]. \)

When \( k \in C_k(B_N) \), we let \( \partial C_k(B_N) \) denote the set cells in \( C_k(B_N) \) which are boundary cells of \( B_N. \)

\subsection{Discrete exterior calculus}
In what follows, we give a brief overview of discrete exterior calculus on the cell complexes of \( \mathbb{Z}^m \) and \( B = [a_1,b_1] \times \dots \times [a_m,b_m]  \cap  \mathbb{Z}^m \) for \( m \geq 1 \). As with the previous section, this section will closely follow the corresponding section in~\cite{flv2020}, where we refer the reader for further details and proofs.

All of the results in this subsection are obtained under the assumption that an abelian group \( G \), which is not necessarily finite, has been given. In particular, they all hold for both \( G=\mathbb{Z}_n \) and \( G=\mathbb{Z} \). 

\subsubsection{Discrete differential forms}\label{sec: ddf}

A homomorphism from the group \( C_k(\mathcal{L},\mathbb{Z}) \) to the group \( G \) is called a \emph{\( k \)-form}. The set of all such \( k\)-forms will be denoted by \( \Omega^k(\mathcal{L},G) \). This set becomes an abelian group if we add two homomorphisms by adding their values in \( G \).

The set $C_k^+(\mathcal{L})$ of positively oriented $k$-cells is naturally embedded in  $C_k(\mathcal{L},\mathbb{Z})$ via the map  $c \mapsto 1 \cdot c$, and we will  frequently identify $c \in C_k^+(\mathcal{L})$ with the $k$-chain $1 \cdot c$ using this embedding. Similarly, we will identify a negatively oriented $k$-cell $c \in C_k^-(\mathcal{L})$ with the $k$-chain $(-1) \cdot (-c)$. 
In this way, a $k$-form $\omega$ can be viewed as a \( G \)-valued function on \( C_k(\mathcal{L}) \) with the property that \( \omega(c) = -\omega(-c) \) for all \( c \in C_k(\mathcal{L}) \). Indeed, if \( \omega \in \Omega^k(\mathcal{L},G) \) and \( q = \sum a_i c_i \in C_k(\mathcal{L},\mathbb{Z}) \), we have
\begin{equation*}
    \omega(q) = \omega \bigl(\sum a_i c_i \bigr) = \sum a_i \omega(c_i),
\end{equation*}
and hence a \( k \)-form is uniquely determined by its values on positively oriented \( k \)-cells.

If \( \omega \) is a \( k \)-form, it is useful to represent it by the formal expression 
\begin{equation*}
    \sum_{1 \leq j_1 < \dots < j_k \leq m} \omega_{j_1\dots j_k} dx^{j_1} \wedge \cdots \wedge dx^{j_k}.
\end{equation*}
where \( \omega_{j_1\dots j_k} \) is a \( G \)-valued function on the set of all \( a \in \mathbb{Z}^m \) such that \( \frac{\partial}{\partial x^{j_1}} \big|_a \wedge \dots \wedge \frac{\partial}{\partial x^{j_k}}\big|_a \in C_k(\mathcal{L})\), defined by
\begin{equation*}
    \omega_{j_1 \dots j_k}(a) = \omega \biggl( \frac{\partial}{\partial x^{j_1}} \bigg|_a \wedge \dots \wedge \frac{\partial}{\partial x^{j_k}} \bigg|_a \biggr).
\end{equation*} 

If \( 1\leq j_1 < \dots < j_k\leq m \) and \( \sigma  \) is a permutation of \( \{ 1,2, \dots, k \} \), we define
\begin{equation*}
    dx^{j_{\sigma(1)}}  \wedge \dots \wedge d x^{j_{\sigma(k)}}
    \coloneqq 
    \sgn(\sigma) \, 
    d  x^{j_1} \wedge \dots \wedge d x^{j_k},
\end{equation*} 
and if \( 1 \leq j_1,\dots, j_k \leq n \) are such that \( j_i = j_{i'} \) for some \( 1 \leq i < i' \leq k \), then we let
\begin{equation*}
    d  x^{j_1} \wedge \dots \wedge d x^{j_k} \coloneqq 0.
\end{equation*}

Given a \( k \)-form \( \omega \), we let \( \support \omega \) denote the support of \( \omega \), i.e., the set of all oriented \( k \)-cells \( c \) such that \( \omega(c) \neq 0 \). Note that $\support \omega$ always contains an even number of elements. 

\subsubsection{The exterior derivative}\label{sec: derivative}
Given \( h \colon \mathbb{Z}^m \to G \), \( a \in \mathbb{Z}^m \), and \( i \in \{1,2, \ldots, m \} \), we let 
\begin{equation*}
    \partial_i h(a) \coloneqq h(a+e_i) - h(a) .
\end{equation*}
If \( k \in \{ 0,1,2, \ldots, m \} \) and \( \omega \in \Omega^k(\mathcal{L},G) \), we define the \( (k+1) \)-form \( d\omega \in \Omega^{k+1}(\mathcal{L},G) \) by
\begin{equation*}
    d\omega = \sum_{1 \leq j_1 < \dots < j_k \leq m} \sum_{i=1}^m \partial_i \omega_{j_1,\dots,j_k} \,  dx^i \wedge (dx^{j_1} \wedge \dots \wedge dx^{j_k}).
\end{equation*} 
The operator \( d \) is called the \emph{exterior derivative.} 
Using \eqref{eq: boundary chain}, one can show that  \( \omega\in \Omega^k(\mathcal{L},G) \) and \( c \in  C_k(\mathcal{L},\mathbb{Z}) \), we have \( d\omega(c) = \omega(\partial c). \) This equality is known as the \emph{discrete Stokes' theorem.} 
Recalling that when \( k \in \{ 2,3,\dots, m-2\}  \) and \( c \in C_{k+2}(\mathcal{L}) \), then \( \partial \partial c = 0, \) it follows from the discrete Stokes theorem that for any \( \omega \in \Omega^{k}(\mathcal{L},G) ,\) we have \( dd\omega = 0 .\)

\subsubsection{Closed forms and the Poincar\'e lemma}

For \( k \in \{ 0,\ldots, m \} \), we say that a \( k \)-form \( \omega \in \Omega^k(\mathcal{L},G) \) is \emph{closed} if \( d\omega(c) = 0 \) for all \( c \in C_{k+1}(\mathcal{L}). \)
The set of all closed forms in \( \Omega^k(\mathcal{L},G) \) will be denoted by \( \Omega^k_0(\mathcal{L},G). \)

\begin{lemma}[The Poincar\'e lemma, Lemma 2.2 in~\cite{c2019}]\label{lemma: poincare}
    Let \( k \in \{ 1, \ldots, m\} \) and let \( B \) be a box in \( \mathbb{Z}^m \). Then the exterior derivative \( d \) is a surjective map from the set \( \Omega^{k-1}(B \cap \mathbb{Z}^m, G) \) to \( \Omega^k_0(B \cap \mathbb{Z}^m, G) \).
    Moreover, if \(G\) is finite, then this map is an \( \bigl| \Omega^{k-1}_0(B \cap \mathbb{Z}^m, G)\bigr|\)-to-\(1\) correspondence.
    Lastly, if \( k \in \{ 1,2, \ldots, m-1 \} \) and \(\omega \in \Omega^k_0(B \cap \mathbb{Z}^m, G)\) vanishes on the boundary of \(B\), then there is a \((k-1)\)-form \( \omega' \in \Omega^{k-1}(B \cap \mathbb{Z}^m, G)\) that also vanishes on the boundary of \(B\) and satisfies \(d\omega' = \omega\). 
\end{lemma}

\subsubsection{Non-trivial forms}

We say that a \( k \)-form \( \omega \in \Omega^k(\mathcal{L},G) \) is \emph{non-trivial} if there is at least one \( k \)-cell \( c \in C_k(\mathcal{L}) \) such that \( \omega(c) \neq 0 \).

\subsubsection{Restrictions of forms}

If \( \omega \in \Omega^k(\mathcal{L},G) \), \( C \subseteq C_k(\mathcal{L}) \) is symmetric, and \( c \in C \), we define 
\begin{equation*}
    \omega|_C(c) \coloneqq 
    \begin{cases}
        \omega(c) &\text{if } c \in C, \cr 
        0 &\text{else.}
    \end{cases}
\end{equation*}

\subsubsection{A partial ordering of \texorpdfstring{\(\Omega^k(\mathcal{L},G)\)}{k-forms}}\label{sec: the partial ordering} 

We now recall the partial ordering on differential forms, which was introduced in~\cite{flv2021}.

\begin{definition}[Definition~2.6 in~\cite{flv2021}]\label{def: partial order}
    When \( k \in \{ 0,1,\dots, m \} \) and \( \omega,\omega' \in \Omega^k(\mathcal{L},G) \), we write \( \omega' \leq  \omega \) if
    \begin{enumerate}[label=(\roman*)]
        \item \( \omega' = \omega|_{\support \omega'} \), and
        \item \( d\omega' = (d\omega)|_{\support d\omega'} \).
    \end{enumerate}
    If \( \omega' \neq \omega \) and \( \omega' \leq \omega \), we write \( \omega'< \omega \).
\end{definition}

The following lemma from~\cite{flv2021} collects some basic facts about the relation \(\leq\) on \( \Omega^k(\mathcal{L},G)\), and shows that \(\leq\) is a partial order on \( \Omega^k(\mathcal{L},G)\).
\begin{lemma}[Lemma~2.7 in~\cite{flv2021}]\label{lemma: the blue lemma}
    Let \( k \in \{ 0,1,\dots, m \} \) and \(\omega, \omega', \omega'' \in \Omega^k(\mathcal{L},G) \). The relation \( \leq\) on \( \Omega^k(\mathcal{L},G) \) has the following properties.
    \begin{enumerate}[label=\textnormal{(\roman*)}]
        \item Reflexivity: \( \omega \leq \omega \). \label{property 1}
        \item Antisymmetry: If \(\omega' \leq \omega \) and \(\omega \leq \omega'\), then \( \omega = \omega'\).\label{property 2}
        \item Transitivity: If \( \omega'' \leq \omega' \) and \( \omega' \leq \omega \), then \( \omega'' \leq \omega \).\label{property 3}
        \item If \( \omega' \leq \omega \), then \( \omega-\omega' = \omega|_{C_1(B_N) \smallsetminus (\support \omega')} \leq \omega \). \label{property 4}
        \item If \( \omega' \leq \omega \), then \( \support d\omega'\) and \( \support d(\omega-\omega') \) are disjoint. \label{property 5} 
    \end{enumerate}
\end{lemma} 

The next lemma guarantees the existence of minimal elements satisfying certain constraints.
\begin{lemma}[Lemma~2.8 in~\cite{flv2021}]\label{lemma: reduction V}
    Let \( k \in \{ 0,1,\dots, m \}\), let \( \Omega \subseteq \Omega^k(\mathcal{L},G) \), and let \( \omega \in \Omega \). 
    Then there is \(  \omega' \leq \omega \) such that
    \begin{enumerate}[label=\textnormal{(\roman*)}]
        \item \(   \omega' \in \Omega \), and \label{item: new fix lemma i}
        \item there is no \( \omega'' <  \omega' \) such that \( \omega'' \in \Omega \).\label{item: new fix lemma ii}
    \end{enumerate}
\end{lemma}

\subsubsection{Irreducible forms}\label{sec: irreducibility}

The partial ordering given in Definition~\ref{def: partial order} allows us to introduce a notion of irreducibility.

\begin{definition}[Definition~2.9 in~\cite{flv2021}]\label{def: irreducible kform}
    When \( k \in \{ 0,1,\dots, m-1 \} \), a \( k \)-form \( \omega \in \Omega^k(\mathcal{L},G) \) is said to be \emph{irreducible} if there is no non-trivial \( k\)-form \( \omega' \in \Omega^k(\mathcal{L},G)\) such that \( \omega' < \omega \).
\end{definition}

Equivalently, \( \omega \in \Omega^k(\mathcal{L},G) \) is irreducible if there is no  non-empty set \( S \subsetneq \support \omega \) such that \( \support d(\omega|_S) \) and \( \support d (\omega|_{S^c}) \) are disjoint.
Note that if \( \omega \in \Omega^k(\mathcal{L},G) \)  satisfies \( d\omega = 0 \), then \( \omega \) is irreducible if and only if there is no non-empty set \( S \subsetneq \support \omega \) such that \(   d(\omega|_S) = d(\omega|_{S^c}) = 0\).

\begin{lemma}[Lemma~2.10 in~\cite{flv2021}]\label{lemma: lemma sum of irreducible configurations} 
    Let \( k \in \{ 0,1,\dots, m-1 \} \), and let \(\omega\in \Omega^k(\mathcal{L},G) \) be non-trivial and have finite support.
    Then there is an integer \( j \geq 1 \) and $k$-forms \( \omega_1, \ldots, \omega_j \in \Omega^k(\mathcal{L},G) \) such that
    \begin{enumerate}[label=\textnormal{(\roman*)}]
        \item\label{lemma210property1} for each \( i \in \{ 1,2, \ldots, j \} \), \( \omega_i  \) is non-trivial and irreducible,
        \item\label{lemma210property2} for each \( i \in \{ 1,2, \ldots, j \} \), \( \omega_i \leq \omega \),
        \item\label{lemma210property3} \( \omega_1, \dots, \omega_j \) have disjoint supports,
        
        \item\label{lemma210property4} \( \omega = \omega_1 + \dots + \omega_j \), and

         \item \label{lemma210property5}
         \( d\omega_1, \dots, d\omega_j \) have disjoint supports. 
         
    \end{enumerate} 
\end{lemma}

A set \( \Omega \coloneqq \{ \omega_1,  \dots , \omega_j \} \subseteq \Omega^k(\mathcal{L},G) \) such that \( \omega_1, \dots, \omega_j \) satisfies~\ref{lemma210property1}--\ref{lemma210property5} of Lemma~\ref{lemma: lemma sum of irreducible configurations} will be referred to as a \emph{decomposition} of \( \omega \in \Omega^k(\mathcal{L},G).\)

We note that as an immediate consequence of the previous lemma, if \( \omega \in \Omega^2_0(\mathcal{L},G) \) has finite support, then there is a set \( \Omega \subseteq \Omega^2_0(\mathcal{L},G) \) which is a decomposition of \( \omega \) (see also Lemma~2.12 in~\cite{flv2021}).

\subsubsection{Minimal forms}\label{sec: minimal configurations}

In this section, we recall three lemmas from~\cite{flv2021} which gives lower bounds on the size of the support of differential forms. 
Throughout this section, we assume that \( m = 4 \). In other words, we assume that we are working on the \(\mathbb{Z}^4 \)-lattice.

\begin{lemma}[Lemma~2.16 in~\cite{flv2021}]\label{lemma: 6 plaquettes per edge}
    Let \( \sigma \in \Omega^1(\mathcal{L},G)\). Then
    \begin{equation*}
        |\support  \sigma| \geq |\support d \sigma|/6.
    \end{equation*}
\end{lemma}

\begin{lemma}\label{lemma: minimal vortex I}
    Let \( \omega \in \Omega^2_0(\mathcal{L},G) \) be non-trivial and have finite support, and assume that there is a plaquette \( p \in \support \omega \) such that \( \support \partial p \) contains no boundary edges of \( B_N. \)
    Then \( | (\support \omega)^+ | \geq 6,\) and if \(| (\support \omega)^+| = 6\), then there is an edge \( e_0 \in C_1(B_N) \) such that \( \support \nu = \support \hat \partial e_0 \cup \support \hat \partial (-e_0)\). 
\end{lemma}
For a proof of Lemma~\ref{lemma: minimal vortex I}, see, e.g., Lemma~3.4.6~in~\cite{sc2019}.

\begin{lemma}[Lemma~2.19 in~\cite{flv2021}]\label{lemma: small 1forms}
    Let \( \sigma \in \Omega^1_0(B_N,G) \) be non-trivial, and assume that there is an edge \( e\in \support \sigma \) such that the support of \(  \hat \partial e \) contains no boundary cells of \( B_N. \) 
    Then \(|(\support \sigma)^+ | \geq 8 \).
\end{lemma}


\subsection{Vortices}\label{sec: vortices}

In this section, we use the notion of irreducibility introduced in Section~\ref{sec: irreducibility} to define what we refer to as vortices. We mention that the definition of a vortex given in Definition~\ref{def: vortex} below is identical to the definitions used in~\cite{flv2020,flv2021}, but is different from the corresponding definitions in~\cite{sc2019}~and~\cite{c2019}.
    
\begin{definition}[Vortex]\label{def: vortex}
    Let \( \sigma \in \Omega^1(B_N,G) \). A non-trivial and irreducible \( 2 \)-form \( \nu \in \Omega^2_0(B_N,G) \) is said to be a \emph{vortex} in \( \sigma \) if \( \nu \leq d\sigma \), i.e., if \( d\sigma(p) = \nu(p)\) for all \( p \in \support \nu \). 
\end{definition}

We say that \( \sigma \in \Omega^1(B_N,G) \) has a vortex at \( V \subseteq C_2(B_N) \) if \( (d\sigma)|_V \) is a vortex in \( \sigma \).

\begin{lemma}[Lemma~3.6 in~\cite{flv2021}]\label{lemma: vortex transfer}
    Let \( \sigma' ,\sigma \in \Omega^1(B_N,G) \) be such that \( \sigma' \leq \sigma \), and let \( \nu \in \Omega^2_0(B_N,G)\) be a vortex in \( \sigma' \). Then  \( \nu \) is a vortex in \( \sigma \).
\end{lemma}

With Lemma~\ref{lemma: minimal vortex I} in mind, we say that a vortex \( \nu \) such that no plaquette in \( \support \nu \) is a boundary plaquette of \( \mathcal{L} \) is a \emph{minimal vortex} if \( |\support \nu| = 12\). 

\begin{lemma}[Lemma~3.2 in~\cite{flv2021}]\label{lemma: minimal vortex II}
    Let \( \sigma \in \Omega^1(B_N,G) \), and let \( \nu \in \Omega^2_0(B_N,G) \) be a minimal vortex in \( \sigma \). Then there is an edge \( \partial x_j  \in C_1(B_N) \) and a group element \( g \in G\smallsetminus \{ 0 \} \) such that 
    \begin{equation}\label{eq: minimal vortex as df}
        \nu = d\bigl(g \, d x_j \bigr).
    \end{equation}
    In particular, \( d\sigma(p) =  \nu(p) = g \) whenever \( p \in \hat \partial e_0 \).
\end{lemma}

If \( \sigma \in \Omega^1(B_N,G) \) and \( \nu \in \Omega^2_0(B_N,G) \) is a minimal vortex in \( \sigma \) which can be written as in~\eqref{eq: minimal vortex as df} for some \( e_0 \in C_1(B_N) \) and \( g \in G \backslash \{ 0 \} \), then we say that \( \nu \) is a \emph{minimal vortex centered at \( e_0 \).}


\subsection{Generalized loops and oriented surfaces}\label{sec: oriented surfaces}
In this section, we recall the definitions of generalized loops and oriented surfaces from~\cite{flv2020}, and outline their connection.

\begin{definition}[Definition~2.6 in~\cite{flv2020}]\label{def: generalized loop}
    A 1-chain \( \gamma \in C_1(\mathcal{L},\mathbb{Z}) \) with finite support is a \emph{generalized loop} if 
    \begin{enumerate}
        \item for all \( e \in \Omega^1(\mathcal{L}) \), we have \( \gamma[e] \in \{ -1,0,1 \} \), and
        \item  \( \partial \gamma = 0. \)
    \end{enumerate}
\end{definition}

\begin{definition}[Definition~2.7 in~\cite{flv2020}]
    Let \( \gamma \in C_1(\mathcal{L},\mathbb{Z}) \) be a generalized loop. A \( 2 \)-chain \( q \in C_2(\mathcal{L},\mathbb{Z}) \) is an \emph{oriented surface} with \emph{boundary} \( \gamma \) if \( \partial q = \gamma. \)
\end{definition}

We recall that by Stokes' theorem (see Section~\ref{sec: derivative}), for any \( q \in C_2(\mathcal{L},G) \) and any \( \sigma \in \Omega^1(\mathcal{L},G) \), we have
\begin{equation*} 
      \sigma(\partial q) = d\sigma(q).
  \end{equation*}

The following lemma gives a connection between generalized loops and oriented surfaces.

\begin{lemma}[Lemma~2.8 in~\cite{flv2020}]\label{lemma: oriented loops}
    Let \( \gamma \in C_1(\mathcal{L}, \mathbb{Z}) \) be a generalized loop, and let \( B \subseteq \mathcal{L} \) be a box containing the support of \( \gamma \). Then there is an oriented surface \( q \in C_2(\mathcal{L}, \mathbb{Z}) \) with support contained in \( B \) such that \( \gamma \) is the boundary of \( q \).
\end{lemma}

\subsection{Unitary gauge}\label{sec: unitary gauge}
In this section, we introduce gauge transforms, and the describe how these can be used to rewrite the Wilson line expectation as an expectation with respect to a slightly simpler probability measure.
 
Before we can state the main results of this section, we need to briefly discuss gauge transformations.  
To this end, for \( \eta \in \Omega_0(B_N,G) \), consider the bijection \( \tau \coloneqq \tau_\eta \coloneqq \tau_\eta^{(1)} \times \tau_\eta^{(2)} \colon \Omega^1(B_N,G)  \times \Omega^0(B_N,G)  \to  \Omega_1(B_N,G) \times \Omega^0(B_N,G) \), defined by
\begin{equation}\label{eq: gauge transform}
    \begin{cases}
     \sigma(e) \mapsto -\eta(x) +\sigma(e) + \eta(y), & e=(x,y)\in C_1(B_N), \cr
     \phi(x) \mapsto  \phi(x) + \eta(x), & x \in C_0(B_N).
    \end{cases}
\end{equation}
Any mapping \( \tau \) of this form is called a \emph{gauge transformation}. 
Any mapping \( \tau \) of this form is called a \emph{gauge transformation}, and functions \( f: \Omega^1(B_N,G) \times \Omega^0(B_N,G)  \to \mathbb{C} \) which are invariant under such mappings in the sense that \(f= f \circ \tau\) are said to be \emph{gauge invariant}. 

For \( \beta, \kappa \geq   0 \) and  \( \sigma \in \Omega^1(B_N,G) \), define 
\begin{equation}\label{eq: fixed length unitary measure}
    \mu_{N,\beta,\kappa}(\sigma) 
    \coloneqq
    Z_{N,\beta,\kappa}^{-1}\exp\pigl(\beta \sum_{p \in C_2(B_N)}      \rho\bigl(d  \sigma(p)\bigr) + \kappa \sum_{e \in C_1(B_N)}   \rho \bigl( \sigma(e)\bigr)\pigr)  ,
\end{equation} 
where \( Z_{N,\beta,\kappa}^{-1} \) is a normalizing constant which ensures that \( \mu_{N,\beta,\kappa} \) is a probability measure. We let \( \mathbb{E}_{N,\beta,\kappa} \) denote the corresponding expectation.

The main reason that gauge transformations are useful to us is the following result.

\begin{proposition}[Proposition~2.21 in~\cite{flv2021}]\label{proposition: unitary gauge one dim}
Let \( \beta,\kappa \geq 0 \), and let and assume that  the function \( f \colon \Omega^1(B_N,G) \times \Omega^0(B_N,G) \to \mathbb{C} \) is gauge invariant. Then 
\begin{equation*}
    \mathbb{E}_{N,\beta,\kappa,\infty}\bigl[f(\sigma,\phi)\bigr] = 
    \mathbb{E}_{N,\beta,\kappa}\bigl[f(\sigma,1)\bigr]. 
\end{equation*} 
\end{proposition}
The main idea of the proof of Proposition~\ref{proposition: unitary gauge one dim} is to perform a change of variables, where we for each pair \( (\sigma,\phi) \) apply the gauge transformation \( \tau_{-\phi}, \) thus mapping \( \phi \) to \( 0 \). After having applied this gauge transformation, we are said to be working in \emph{unitary gauge}.

Noting that for any path \( \gamma \), the function \( (\sigma,\phi) \mapsto  L_\gamma(\sigma,\phi) \) is gauge invariant, we obtain the following result as an immediate corollary of Proposition~\ref{proposition: unitary gauge one dim}.

\begin{corollary}\label{corollary: unitary gauge}
    Let \( \beta \in [0,\infty],\) \( \kappa \geq 0 \), and let \( \gamma \) be a path in \( C_1(B_N) \). Then
    \begin{equation*}
        \mathbb{E}_{N,\beta,\kappa,\infty}\bigl[L_\gamma(\sigma,\phi)\bigr] =
        \mathbb{E}_{N,\beta,\kappa}\bigl[L_\gamma(\sigma,1)\bigr] =
        \mathbb{E}_{N,\beta,\kappa}\pigl[\rho\bigl(\sigma(\gamma)\bigr)\pigr]. 
    \end{equation*} 
\end{corollary}

Results analogous to Proposition~\ref{proposition: unitary gauge one dim} are considered well-known in the physics literature.

By combining the previous result with Lemma~\ref{lemma: poincare}, we obtain the following result, which will help us interpret our main result.

\begin{corollary}\label{corollary: Ising}
    Let \( \kappa \geq 0 \),  and let \( \gamma \) be an open path from \( x_1 \in C_0(B_N) \) to \( x_2 \in C_0(B_N).\) Then 
    \begin{equation*}
        H_\kappa(\gamma) = \lim_{N \to \infty} Z_{N,\kappa}^{-1} \sum_{\eta \in \Omega^0(B_N,G)} \rho\bigl(\eta(-x_1)\bigr)\rho\bigl(\eta(x_2)\bigr)e^{-\kappa \sum_{e \in C_1(B_N)} \rho(\eta(\partial e))},
    \end{equation*} 
    where 
    \begin{equation*}
        Z_{N,\kappa} \coloneqq  \sum_{\eta \in \Omega^0(B_N,G)} e^{\kappa \sum_{e \in C_1(B_N)}\rho(\eta(\partial e))}.
    \end{equation*} 
    If particular, if \( G = \mathbb{Z}_2 ,\) then \( H_\kappa(\gamma) \) is the  spin-spin-correlation between for the spins at the endpoints of \( \gamma \) for the Ising model on \( B_N \) with coupling constant \( \kappa. \)
\end{corollary}

\begin{proof}
    By~Corollary~\ref{corollary: unitary gauge}, we have
    \begin{equation*}
        \begin{split}
            &
            H_\kappa(\gamma) 
            = \bigl\langle L_\gamma(\sigma,\phi)  \bigr\rangle_{\infty,\kappa,\infty} 
            = \lim_{N \to \infty} \bigl\langle L_\gamma(\sigma,\phi)\bigr\rangle_{N,\infty,\kappa,\infty} 
            \\&\qquad= \lim_{N \to \infty}\mathbb{E}_{N,\infty,\kappa,\infty} \bigl[ L_\gamma(\sigma,\phi)  \bigr]
            = \lim_{N \to \infty}\mathbb{E}_{N,\infty,\kappa} \bigl[ L_\gamma(\sigma,1)  \bigr]
            \\&\qquad
            = \lim_{N \to \infty} Z_{N,\infty,\kappa,\infty}^{-1} \sum_{\sigma \in \Omega^1_0(B_N,G)} \rho(\sigma( \gamma) ) e^{\kappa \sum_{e \in C_1(B_N)} \rho(\sigma( e))}.
        \end{split}
    \end{equation*}
    Since \( \beta = \infty,\) we only need to sum over \( \sigma \in \Omega^1_0(B_N,G). \) Now recall that by Lemma~\ref{lemma: poincare}, for each \( \sigma \in \Omega^1_0(B_N,G) \) there is \( \eta \in \Omega^0(B_N,G) \) such that \( d \eta = \sigma. \) Moreover, the mapping \( \eta \mapsto d\eta \) is a \( |\Omega_0^0(B_N,G)|\)-to-1 correspondence. From this the desired conclusion immediately follows.
\end{proof}
With the current section in mind, we will work with \( \sigma \sim \mu_{N,\beta, \kappa} \) rather than \( (\sigma,\phi) \sim \mu_{N,\beta,\kappa,\infty}\) throughout the rest of this paper, together with the observable
\begin{equation*}
    L_\gamma(\sigma) \coloneqq L_\gamma(\sigma,1) = \prod_{e \in \gamma} \rho\bigl(\sigma(e)\bigr) = \rho(\sigma(\gamma)).
\end{equation*}

\subsection{Existence of the infinite volume limit}\label{sec: ginibre}

In this section, we recall a result which shows existence and translation invariance of the infinite volume limit \( \langle L_\gamma(\sigma,\phi) \rangle_{\beta,\kappa} \) defined in the introduction. This result is well-known, and is often mentioned in the literature as a direct consequence of the Ginibre inequalities. A full proof of this result in the special case \( \kappa = 0 \) was included in~\cite{flv2020}, and the general case can be proven completely analogously, hence we omit the proof here. 

\begin{proposition}\label{proposition: limit exists}
    Let \( G = \mathbb{Z}_n \), \( M \geq 1 \), and let \( f \colon \Omega^1(B_M,G)  \to \mathbb{R}\).
    For \( M' \geq M \), we abuse notation and let \( f \) denote the natural extension of \( f \) to \( C_1(B_{M'}) \), i.e., the unique function such that \( f(\sigma) = f(\sigma|_{C_1(B_M)}) \) for all \( \sigma \in \Omega^1(B_{M'},G) \).
    Further, let \( \beta \in [0, \infty] \) and \( \kappa \geq 0 \). Then the following hold.
    \begin{enumerate}[label=\textnormal{(\roman*)}]
        \item The limit \( \lim_{N \to \infty} \mathbb{E}_{N,\beta,\kappa} \bigl[ f(\sigma) \bigr] \) exists.
        \item For any translation \( \tau \) of \( \mathbb{Z}^m \), we have \( \lim_{N \to \infty} \mathbb{E}_{N,\beta,\kappa}  \bigl[f \circ \tau(\sigma)\bigr] =  \lim_{N \to \infty} \mathbb{E}_{N,\beta,\kappa}\bigl[ f(\sigma) \bigr] \).
    \end{enumerate} 
\end{proposition}

\section{Additional notation and standing assumptions}\label{sec: additional notation}

Throughout the rest of this paper, we will assume that \( N \geq 1 \) is given, and that \( G = \mathbb{Z}_n \) for some \( n \geq 2 . \)

To simplify the notation, we now introduce some additional notation.

For \( r \geq 0\) and \( g \in G \), we define
\begin{equation}\label{eq: varphi}
    \varphi_r(g) \coloneqq e^{r \Re (\rho(g)-\rho(0))}.
\end{equation}
We extend this notation to \( r = \infty \) by letting
\begin{equation*}
    \varphi_\infty(g) 
    \coloneqq 
    \begin{cases} 
        1 &\text{if } g = 0 , \\
        0 & \text{if }g \in G \smallsetminus \{0\}.
    \end{cases}
\end{equation*}
Next, for \( \hat g \in G \) and \( \beta,\kappa \geq 0 \), we define 
\begin{equation}\label{eq: def of thetag}
    \theta_{\beta,\kappa}(\hat g) \coloneqq \frac{\sum_{g \in G} \rho(g)  \varphi_\beta (g)^{12} \varphi_\kappa(g+\hat g)^2}{\sum_{g \in G} \varphi_\beta(g)^{12} \varphi_\kappa (g+\hat g)^2}.
\end{equation} 
When \( \gamma \) is a path, or when \( E \subseteq C_1(B_N) \) is a finite set, we define
\begin{equation}\label{eq: ThetaN}
    \Theta_{N,\beta,\kappa}(\gamma) \coloneqq \mathbb{E}_{N,\infty,\kappa} \Bigl[ \prod_{e \in \gamma} \theta_{\beta,\kappa}\bigl(\sigma(e)\bigr)\Bigr] \quad \text{and} \quad
    \Theta_{N,\beta,\kappa}(E) \coloneqq \mathbb{E}_{N,\infty,\kappa} \Bigl[ \prod_{e \in E} \theta_{\beta,\kappa}\bigl(\sigma(e)\bigr)\Bigr].
\end{equation}

We next define a number of functions which will be used as error bounds. To this end, for \( r \geq 0 \), let
\begin{equation}\label{eq: alpha01def}
    \alpha_0(r) \coloneqq \sum_{g \in G\smallsetminus \{ 0 \}} \varphi_r(g)^2 \quad \text{and} \quad  \alpha_1(r) \coloneqq \max_{g \in G\smallsetminus \{ 0 \}} \varphi_r(g)^2.
\end{equation} 
Next, for \( \beta,\kappa \geq 0 \), define
\begin{equation}\label{eq: alpha234def}
    \alpha_2(\beta,\kappa) \coloneqq \alpha_0(\beta) \alpha_0(\kappa)^{1/6},
    \quad
    \alpha_3(\beta,\kappa) \coloneqq  \bigl| 1-  \theta_{\beta,\kappa}(0)\bigr|,
    \quad
    \alpha_4(\beta,\kappa) \coloneqq  \max_{g \in G} \bigl| \theta_{\beta,\kappa}(g)-  \theta_{\beta,\kappa}(0)\bigr|,
\end{equation} 
\begin{equation}\label{eq: alpha5}
    \alpha_5(\beta,\kappa) \coloneqq  \min_{g_1,g_2, \ldots, g_6 \in G}\bigg( 1 -   \biggl| \frac{\sum_{g \in G} \rho(g) \bigl(\prod_{k =1}^6 \varphi_\beta(g+g_k)^2\bigr) \varphi_\kappa(g)^2}{\sum_{g \in G} \bigl(\prod_{k =1}^6\varphi_\beta(g+g_k)^2\bigr) \varphi_\kappa(g)^2} \biggr|\bigg)
\end{equation}  
and
\begin{equation}\label{eq: alpha6}
    \alpha_6(\beta,\kappa) \coloneqq 
    \max_{g \in G}\bigl| 1 - \theta_{\beta,\kappa}( g) \bigr|.
\end{equation}

When \( \gamma \) is a path, an edge \( e \in \support \gamma \) is said to be a corner edge in \( \gamma \) if there is another edge \( e' \in \gamma \) and a plaquette \( p \in \hat \partial e \) such that \( p \in \pm \hat \partial e' \). We define the 1-form \( \gamma_c \in C_1(\mathcal{L},\mathbb{Z}) \) for \( c' \in C_1(B_N) \) by 
\begin{equation}\label{def: gammac}
    \gamma_c[c'] \coloneqq \begin{cases}
        \gamma[c'] &\text{if } c' \text{ is a corner edge of } \gamma, \cr
        0 &\text{else.}
    \end{cases}
\end{equation}

In the rest of this paper, we will often work under the following assumption.
\begin{enumerate}[label=\textnormal{[A]}] 
        \item \( 18^2 \alpha_0(\kappa)(2 + \alpha_0(\kappa))<1\). \label{assumption: 3}
\end{enumerate}
In essence, the purpose of this assumption is to guarantee that we are in the sub-critical regime of the model, where certain edge clusters are finite almost surely.

\section{Activity of gauge field configurations}\label{sec: activity}

In this section, we recall the useful notion of the activity of a gauge field configurations from~\cite{flv2021}.
To this end, recall the definition of \( \varphi_r \) from the previous section.
Since \( \rho \) is a unitary representation of \( G \), for any \( g \in G \) we have \( \rho(g) = \overline{\rho(-g)} \), and hence \( \Re \rho(g) = \Re \rho(-g) \). In particular, this implies that for any \( g \in G \) and any \( r \geq 0\), we have
\begin{equation} \label{eq: phi is symmetric}
    \varphi_r(g)
    =
    e^{ r (\Re \rho(g)-\rho(0))  }
    =
    e^{r \beta (\Re \rho(-g)-\rho(0)) }
    =
    \varphi_r(-g).
\end{equation}
Clearly, we also have \( \varphi_\infty(g) = \varphi_\infty(-g) \) for all \( g \in G \).
Moreover, if \( a \geq 0 \) and \( r \geq 0 \), then
\begin{equation*}
    \varphi_r(g)^a = \varphi_{ar}(g).
\end{equation*}
Abusing notation, for \( \sigma \in \Omega^1(B_N,G) \) and \( r \in [0,\infty] \), we define
\begin{equation*}
    \varphi_r(\sigma) \coloneqq \prod_{e \in C_1(B_N)} \varphi_r\bigl(\sigma(e)\bigr),
\end{equation*}
and for \( \omega \in \Omega^2_0(B_N,G) \), we define
\begin{equation*}
    \varphi_r(\omega) \coloneqq \prod_{p \in C_2(B_N)}\varphi_r\bigl(\omega(p)\bigr).
\end{equation*}
For \( \beta \in [0,\infty ] \) and \( \kappa \geq 0 \), we define the \emph{activity} of \( \sigma \in \Omega^1(B_N,G) \) by
\begin{equation*}
    \varphi_{\beta,\kappa}(\sigma) \coloneqq 
    \varphi_\kappa(\sigma) \varphi_\beta(d\sigma).
\end{equation*}
Note that with this notation, for \( \sigma \in \Omega^1(B_N,G) \), \( \beta \in [0,\infty] \), and \( \kappa \geq 0 \), we have
\begin{equation}\label{eq: mubetakappaphi}
    \mu_{N,\beta,\kappa}(\sigma) = \frac{\varphi_{\beta,\kappa} (\sigma)}{\sum_{\sigma' \in \Omega^1(B_N,G)} \varphi_{\beta,\kappa}(\sigma')}.
\end{equation}

Before ending this section, we recall two results from~\cite{flv2021} about the activity of gauge field configurations, which will be useful to us. 

\begin{lemma}[Lemma~4.1 in~\cite{flv2021}]\label{lemma: factorization property}
    Let \( \sigma,\sigma' \in \Omega^1(B_N,G) \) be such that \( \sigma' \leq \sigma \), let \( \beta \in [0,\infty] \), and let \( \kappa \geq 0 \). Then 
    \begin{equation}\label{eq: factorization property}
        \varphi_{\beta,\kappa}(\sigma) = \varphi_{\beta,\kappa}(\sigma')\varphi_{\beta,\kappa}(\sigma-\sigma').
    \end{equation}
\end{lemma}

\begin{proposition}[Proposition~5.1 in~\cite{flv2021}]\label{proposition: edgecluster flipping ii}
    Let  \( \sigma' \in \Omega^1(B_N,G) \), let \( \beta \in [0,\infty] \), and let \( \kappa \geq 0 \). Then
    \begin{equation*}
        \mu_{N,\beta,\kappa}\pigl(\bigl\{ \sigma \in \Omega^1(B_N,G) \colon
        \sigma' \leq \sigma \bigr\} \pigr) 
        \leq 
        \varphi_{\beta,\kappa}(\sigma').
    \end{equation*}
\end{proposition}

\section{Two couplings}\label{sec: couplings} 

The main purpose of this section is to introduce two couplings which will be useful to us throughout this paper. Both of these couplings use ideas from disagreement percolation, and will be constructed so that the two coupled configurations agree as often as possible, given certain constraints. Before we introduce the two couplings, we will recall the definition of a certain edge graph from~\cite{f2021}, and state and prove some of its properties, and introduce a set \( E_{E_0,\hat \sigma, \hat \sigma'} \) which will be used for the definitions of the two couplings.

\subsection{A useful edge graph}\label{sec: edge graph}


%
\begin{definition}
    Given \( \sigma,\sigma' \in \Omega^1(B_N,G) \), let  \( \mathcal{G}(\sigma, \sigma') \) be the graph with vertex set \( C_1(B_N) \), and with an edge between two distinct vertices  \( e,e' \in C_1(B_N)  \) if either 
    \begin{enumerate}[label=\textnormal{(\roman*)}]
        \item \( e' = -e \), or 
        \item \( e,e' \in \support \sigma \cup   \support \sigma'  \), and either \( \support \hat \partial e \cap \support \hat \partial e' \neq \emptyset \) or \( \support \hat \partial e \cap \support \hat \partial (-e') \neq \emptyset \).
    \end{enumerate} 
    Given \(\sigma \in \Omega^1(B_N),\) we let \( \mathcal{G}(\sigma) \coloneqq \mathcal{G}(\sigma,0). \) 
    Given \( \sigma,\sigma' \in \Omega^1(B_N), \) \( \mathcal{G} \coloneqq  \mathcal{G}(\sigma, \sigma'), \) and \( e \in C_1(B_N) ,\) we let \( \mathcal{C}_{\mathcal{G}}(e) \) be set of all edges \( e'\in C_1(B_N) \) which belong to the same connected component as \( e \) in \( \mathcal{G} \). For \( E \subseteq C_1(B_N) \), we let \( \mathcal{C}_{\mathcal{G}}(E) \coloneqq \bigcup_{e \in E} \mathcal{C}_{\mathcal{G}}(e).  \) 
\end{definition}

We now state and prove a number of lemmas, which describe different properties of the sets \( \mathcal{C}_{\mathcal{G}(\hat \sigma,\hat \sigma')}(E). \)

\begin{lemma}[Lemma~7.2 in~\cite{flv2021}]\label{lemma: cluster is subconfig}
    Let \( \sigma ,\sigma' \in \Omega^1(B_N,G)  \),  \( E \subseteq C_1(B_N) \), and \( E' \coloneqq \mathcal{C}_{\mathcal{G}(\sigma, \sigma')}(E) \). Then 
    \begin{enumerate}[label=\textnormal{(\roman*)}]
        \item \( \sigma |_{E'} \leq \sigma \), \label{item: subconfig i}
        \item \( \sigma |_{C_1(B_N)\smallsetminus E'} \leq \sigma \),\label{item: subconfig ii}
        \item \( \sigma' |_{E'} \leq \sigma' \), and\label{item: subconfig iii}
        \item \( \sigma' |_{C_1(B_N) \smallsetminus E'} \leq \sigma' \).\label{item: subconfig iv}
    \end{enumerate} 
\end{lemma}

\begin{lemma}\label{lemma: irreducible to connected}
  Let \( \sigma \in \Omega^1(B_N,G) \) be nontrivial and irreducible. Then the support of \( \sigma \) is a connected set in \( \mathcal{G}(\sigma). \)
\end{lemma}

\begin{proof}
    Let \( e \in \support \sigma \), and define \( \sigma' \coloneqq \sigma|_{\mathcal{C}_{\mathcal{G}(\sigma)}(e)}. \) Then, by definition, \( \sigma' \) is non-trivial, and by Lemma~\ref{lemma: cluster is subconfig}, we have \( \sigma' \leq \sigma. \) Since \( \sigma \) is irreducible, it follows that \( \sigma = \sigma', \) and hence the desired conclusion follows.
\end{proof}

\begin{lemma}\label{lemma: new technical lemma 1}
    Let \( \sigma,\sigma' \in \Omega^1(B_N,G)\). Assume that \( \sigma'' \leq \sigma \) is nontrivial and irreducible, and let \( e \in \support \sigma'' .\) Then
    \begin{equation}\label{eq: new technical lemma 1}
        \sigma'' \leq \sigma|_{\mathcal{C}_{\mathcal{G}(\sigma, \sigma')}(e)}.
    \end{equation}
\end{lemma}

\begin{proof}
    Since \( \sigma'' \) is irreducible, by Lemma~\ref{lemma: irreducible to connected}, the support of \( \sigma'' \) is a connected set  \( \mathcal{G}(\sigma'',0) \). 
    Since \( \sigma'' \leq \sigma \), we have \( \sigma|_{\support \sigma''} = \sigma'' \), and hence it follows that the support of \( \sigma'' \) is a connected set in \( \mathcal{G}(\sigma, \sigma'). \)
    Consequently, since \( e \in \support \sigma'', \) we have
    \begin{equation}\label{eq: irreducible is subset}
        \support \sigma'' \subseteq \mathcal{C}_{\mathcal{G}(\sigma, \sigma')}(e),
    \end{equation}
    and thus, since \( \sigma|_{\support \sigma''} = \sigma'', \) it follows that
    \begin{equation*}\label{eq: irreducible is subset ii}
        (\sigma|_{\mathcal{C}_{\mathcal{G}(\sigma,\sigma')}(e)})|_{\support \sigma''} = \sigma|_{\support \sigma''} = \sigma''.
    \end{equation*}
    For~\eqref{eq: new technical lemma 1} to follow, it thus remains to show that 
    \begin{equation}\label{eq: irreducible is subset iii}
        \bigl(d(\sigma|_{\mathcal{C}_{\mathcal{G}(\sigma,\sigma')}(e)})\bigr)\pigr|_{\support d\sigma''} = d\sigma''.
    \end{equation}
    If \( d\sigma'' = 0 \), then this immediately follows. Hence, assume that \( d\sigma'' \neq 0 \), and let \( p \in \support d\sigma''. \) %
    Since \( p \in \support d\sigma'', \) there must exist at least one \( e' \in \partial p \) with \( \sigma''(e') \neq 0.\) 
    Since \( e' \in \support \sigma'' \), it follows from~\eqref{eq: irreducible is subset} that \( e' \in \mathcal{C}_{\mathcal{G}(\sigma,\sigma')}(e). \) Since \( e' \in \partial p \), it follows from the definition of \( \mathcal{C}_{\mathcal{G}(\sigma,\sigma')}(e) \) that any edge  \( e'' \in \partial p \) with \( \sigma(e'') \neq 0 \) is also a member of \( \mathcal{C}_{\mathcal{G}(\sigma,\sigma')}(e). \) Consequently, we must have \( \sigma(e'') = \sigma|_{\mathcal{C}_{\mathcal{G}(\sigma,\sigma')}(e)}(e'') \) for all \( e'' \in \partial p, \) and hence
    \begin{equation*}
        d\sigma|_{\mathcal{C}_{\mathcal{G}(\sigma,\sigma')}(e)} (p) = d\sigma(p).
    \end{equation*}
    Since \( \sigma'' \leq \sigma \) and \( p \in \support d\sigma'' \), we also have \( d\sigma''(p) = d\sigma(p),  \) and hence we conclude that
    \begin{equation*}\label{eq: d cluster agreement}
        d\sigma|_{\mathcal{C}_{\mathcal{G}(\sigma,\sigma')}(e)} (p) = d\sigma''(p).
    \end{equation*}
    Since this holds for any \( p \in \support d\sigma'', \) we obtain~\eqref{eq: irreducible is subset iii}.
    This concludes the proof.
\end{proof}

\begin{lemma}\label{lemma: graph decomposition in coarser i}
    Let \( \sigma,\sigma' \in \Omega^1(B_N,G) \), and let \( E \subseteq C_1(B_N). \) Assume that \( \sigma'' \leq \sigma \) is non-trivial and irreducible, and that \( \support \sigma'' \cap \mathcal{C}_{\mathcal{G}(\sigma,\sigma')}(E) \neq \emptyset\). Then \( \sigma'' \leq \sigma|_{\mathcal{C}_{\mathcal{G}(\sigma,\sigma')}(E)}. \)
\end{lemma}

\begin{proof}
    Fix some \( e \in \support \sigma'' \cap \mathcal{C}_{\mathcal{G}(\sigma,\sigma')}(E). \)
    Since \( \sigma'' \) is irreducible and \( e \in \support \sigma'', \) it follows from  Lemma~\ref{lemma: new technical lemma 1} that \( \sigma'' \leq \sigma|_{\mathcal{C}_{\mathcal{G}(\sigma,\sigma')}(e)}. \)
    Next, since \( e \in \mathcal{C}_{\mathcal{G}(\sigma,\sigma')}(E) \), we have \( \mathcal{C}_{\mathcal{G}(\sigma,\sigma')}(e) \subseteq \mathcal{C}_{\mathcal{G}(\sigma,\sigma')} (E) \), and hence, by Lemma~\ref{lemma: cluster is subconfig}, it follows that
    \begin{equation*}
        \sigma|_{\mathcal{C}_{\mathcal{G}(\sigma,\sigma')}(e)} = (\sigma|_{\mathcal{C}_{\mathcal{G}(\sigma,\sigma')}(E)})|_{\mathcal{C}_{\mathcal{G}(\sigma,\sigma')}(e)} \leq \sigma|_{\mathcal{C}_{\mathcal{G}(\sigma,\sigma')}(E)}.
    \end{equation*}
    Since \( \sigma'' \leq \sigma|_{\mathcal{C}_{\mathcal{G}(\sigma,\sigma')}(e)} \), using Lemma~\ref{lemma: the blue lemma}\ref{property 3}, we thus conclude that \( \sigma'' \leq  \sigma|_{\mathcal{C}_{\mathcal{G}(\sigma,\sigma')}(E)}. \)
\end{proof}

\begin{lemma}\label{lemma: graph decomposition in coarser}
    Let \( \sigma,\sigma' \in \Omega^1(B_N,G) \), let \( E \subseteq C_1(B_N), \) and assume that \( \sigma'' \leq \sigma \) is non-trivial and irreducible. Then either \( \sigma'' \leq \sigma|_{\mathcal{C}_{\mathcal{G}(\sigma,\sigma')}(E)} \), or \( \sigma'' \leq \sigma|_{C_1(B_N)\smallsetminus \mathcal{C}_{\mathcal{G}(\sigma,\sigma')}(E)}. \)
\end{lemma}

\begin{proof}
    If \( \support \sigma'' \cap \mathcal{C}_{\mathcal{G}(\sigma,\sigma')}(E) \neq \emptyset \), then, by Lemma~\ref{lemma: graph decomposition in coarser i}, we have \( \sigma'' \leq \sigma|_{\mathcal{C}_{\mathcal{G}(\sigma,\sigma')}(E)} \), and hence the desired conclusion holds in this case.
    Now instead assume that \( \support \sigma'' \cap \mathcal{C}_{\mathcal{G}(\sigma,\sigma')}(E) = \emptyset, \) and note that this implies that \( \support \sigma'' \subseteq C_1(B_N)\smallsetminus \mathcal{C}_{\mathcal{G}(\sigma,\sigma')}(E). \)
    Define
    \begin{equation*}
        E' \coloneqq \support  \sigma|_{C_1(B_N)\smallsetminus \mathcal{C}_{\mathcal{G}(\sigma,\sigma')}(E)} \cup \support  \sigma'|_{C_1(B_N)\smallsetminus \mathcal{C}_{\mathcal{G}(\sigma,\sigma')}(E)}.
    \end{equation*}
    Then, since \( \sigma'' \leq \sigma \) and \(\support \sigma'' \subseteq C_1(B_N)\smallsetminus \mathcal{C}_{\mathcal{G}(\sigma,\sigma')}(E),\) we have \( \support \sigma'' \subseteq E'. \)
    Consequently, by Lemma~\ref{lemma: graph decomposition in coarser i}, we have \( \sigma'' \leq \sigma|_{\mathcal{C}_{\mathcal{G}(\sigma,\sigma')}(E')}. \) Since
    \begin{equation*}
        \sigma|_{C_1(B_N)\smallsetminus \mathcal{C}_{\mathcal{G}(\sigma,\sigma')}(E)} = 
        \sigma|_{\mathcal{C}_{\mathcal{G}(\sigma,\sigma')}(E')},
    \end{equation*}
    we obtain \( \sigma'' \leq \sigma|_{C_1(B_N)\smallsetminus \mathcal{C}_{\mathcal{G}(\sigma,\sigma')}(E)}, \) and hence the desired conclusion holds also in this case.
    This completes the proof.
\end{proof}

\subsection{\texorpdfstring{The set \( E_{E_0,\hat \sigma,\hat \sigma'}\)}{A useful set}}

We now define a set which we will need for the definitions of the couplings in Sections~\ref{sec: Z-Z coupling}~and~\ref{sec: LGT-Z coupling}
\begin{definition} 
    For \( \sigma,\sigma' \in \Omega^1(B_N,G) \) and \( E_0 \subseteq C_1(B_N) \), define
    \begin{equation}\label{eq: EE0sigmasigmadef}
        \begin{split}
            E_{E_0,\sigma,\sigma'} 
            \coloneqq \mathcal{C}_{\mathcal{G}(\sigma, \sigma')} \bigl(&E_0 \cup \{ e \in \support \sigma \colon  d\sigma|_{\pm \support \hat \partial e} \neq 0 \} \cup \{ e \in \support \sigma' \colon d\sigma'|_{\pm \support \hat \partial e} \neq 0 \}  \bigr).
        \end{split}
    \end{equation}
\end{definition}

\begin{lemma}\label{lemma: closed is closed new}
    Let \(  \sigma,\sigma'  \in \Omega^1(B_N,G) \), and let \( E_0 \subseteq C_1(B_N). \)
    Then  
    \begin{enumerate}[label=(\roman*)]
        \item \( d( \sigma|_{E_{E_0,\sigma,\sigma'}}) = d \sigma,\) and \label{item: coupling eqs i}
        \item \( d( \sigma|_{C_1(B_N)\smallsetminus E_{E_0,\sigma,\sigma'}}) = 0.\) \label{item: coupling eqs ii}
    \end{enumerate}
\end{lemma}

\begin{proof}
    To simplify notation, let \( E \coloneqq E_{E_0,\sigma,\sigma'} \).
    By Lemma~\ref{lemma: cluster is subconfig}, applied with \(  \sigma \), \(  \sigma' \), and  \( E \), we then have  \(  \sigma|_{E} \leq  \sigma \), \(  \sigma'|_{E} \leq  \sigma' \), \(  \sigma|_{C_1(B_N) \smallsetminus E} \leq  \sigma \) and \(  \sigma'|_{C_1(B_N) \smallsetminus E} \leq  \sigma' \).

    We now show that \( d(  \sigma|_{E} ) = d \sigma\).
    Since  \(   \sigma|_{E} \leq  \sigma \), it suffices to show that \( d(\sigma|_{E})(p) \neq 0 \) whenever \( d \sigma(p) \neq 0. \) 
    To this end, assume that \( d \sigma(p) \neq 0. \)
    Then the set \( \support \partial p \cap \support  \sigma \) must be non-empty. 
    Fix one edge \( e \in \support \partial p \cap \support  \sigma .\)
    Recalling the definition of \( E \), we see that \( e \in E \), and hence any edge \( e' \in \support \partial p\smallsetminus \{ e \} \) must satisfy either \(  \sigma'(e') =  \sigma(e') = 0 \), or \( e' \in E \).
    Consequently, \(  \sigma |_{E}(\partial p) =  \sigma (\partial p)  \), and hence
    \begin{equation*}
        d \sigma |_{E}(p) =  \sigma |_{E}(\partial p) =  \sigma (\partial p) = d \sigma ( p)
    \end{equation*}
    as desired. This concludes the proof of~\ref{item: coupling eqs i}.
    
    To see that~\ref{item: coupling eqs ii} holds, note simply that, using~\ref{item: coupling eqs i}, we  have
    \begin{equation*}
        d( \sigma|_{C_1(B_N)\smallsetminus E}) = 
        d( \sigma -  \sigma|_{E}) 
        =
        d \sigma-d( \sigma|_{E})
        =
        d \sigma - d \sigma
        =
        0,
    \end{equation*}
    and hence~\ref{item: coupling eqs ii} holds. This concludes the proof.
\end{proof}

\begin{lemma}\label{lemma: E in coupling}
    Let \(  \hat\sigma,\hat\sigma' \in \Omega^1(B_N,G) \),  and let
    \( E_0 \subseteq C_1(B_N). \)
    Further, either let 
    \begin{equation*}
        \begin{cases}
            \sigma \coloneqq  \hat \sigma|_{E_{E_0,\hat \sigma,\hat \sigma'}}
            +
            \hat \sigma'|_{C_1(B_N) \smallsetminus E_{E_0,\hat\sigma,\hat\sigma'}} \cr 
            \sigma' \coloneqq \hat \sigma'.
        \end{cases}
    \end{equation*}
    or let 
    \begin{equation*}
        \begin{cases}
            \sigma \coloneqq  \hat \sigma|_{E_{E_0,\hat \sigma,\hat \sigma'}}
            +
            \hat \sigma'|_{C_1(B_N) \smallsetminus E_{E_0,\hat\sigma,\hat\sigma'}} \cr 
            \sigma' \coloneqq \hat \sigma'|_{E_{E_0,\hat \sigma,\hat \sigma'}}
            +
            \hat \sigma|_{C_1(B_N) \smallsetminus E_{E_0,\hat\sigma,\hat\sigma'}}.
        \end{cases}
    \end{equation*}
    Then \(   E_{E_0,\sigma, \sigma'} = E_{E_0,\hat \sigma,\hat\sigma'} \). 
\end{lemma}

\begin{proof}
     By Lemma~\ref{lemma: closed is closed new}, we have  \( d\hat\sigma|_{E_{E_0,\sigma,\sigma'}} = d\hat \sigma, \) and \( d \hat \sigma'|_{C_1(B_N)\smallsetminus E_{E_0,\hat\sigma,\hat\sigma'}} = 0 \)  and hence
    \begin{equation*}
        d\sigma 
        = d(\hat \sigma|_{E_{E_0,\hat\sigma,\hat\sigma'}} 
        +
        \hat \sigma'|_{C_1(B_N)\smallsetminus E_{E_0,\sigma,\sigma'}}) 
        = d(\hat \sigma|_{E_{E_0,\hat\sigma,\hat\sigma'}}) + d(\hat \sigma'|_{C_1(B_N)\smallsetminus E_{E_0,\sigma,\sigma'}}) 
        = d\hat\sigma + 0 = d\hat\sigma.
    \end{equation*}
    If \( e \in \support \hat \sigma \) is such that \( d\hat \sigma|_{\pm \support \hat \partial e} \neq 0,\)  then \( e \in E_{E_0,\hat\sigma,\hat \sigma'} \), and thus \( \sigma(e) = \hat \sigma(e), \) implying in particular that \( e \in \support \sigma .\) 
    Since \( d\sigma = d\hat \sigma, \) it follows that \( d\sigma|_{\pm \support \hat \partial e} \neq 0, \) and hence
    \begin{equation*}
        \{ e \in \support \hat \sigma \colon d\hat \sigma|_{\pm \support \hat \partial e} \neq 0 \} \subseteq 
        \{ e \in \support \sigma \colon d\sigma|_{\pm \support \hat  \partial e} \neq 0 \}.
    \end{equation*} 
    Analogously, we obtain
    \begin{equation*}
        \{ e \in \support \hat \sigma' \colon d\hat \sigma'|_{\pm \support \hat \partial e} \neq 0 \} \subseteq 
        \{ e \in \support \sigma' \colon d\sigma'|_{\pm \support \hat  \partial e} \neq 0 \}.
    \end{equation*} 
    Noting that \( \mathcal{G}(\hat \sigma,\hat \sigma') = \mathcal{G}(\sigma,\sigma'), \) we thus obtain 
    \begin{equation*}
        \begin{split}
            &E_{E_0,\hat \sigma, \hat \sigma'} 
            =
            \mathcal{C}_{\mathcal{G}(\hat \sigma,\hat \sigma')} \bigl(E_0 \cup \{ e \in \support \hat \sigma \colon d\hat \sigma|_{\pm \support \partial e} \neq 0 \} \cup \{ e \in \support \hat \sigma' \colon d\hat \sigma'|_{\pm \support \partial e} \neq 0 \} \bigr)
            \\&\qquad\subseteq
            \mathcal{C}_{\mathcal{G}( \sigma, \sigma')} \bigl(E_0 \cup \{ e \in \support  \sigma \colon d \sigma|_{\pm \support \partial e} \neq 0 \} \cup \{ e \in \support  \sigma' \colon d \sigma'|_{\pm \support \partial e} \neq 0 \} \bigr)
            =
            E_{E_0, \sigma,  \sigma'}.
        \end{split}
    \end{equation*}  
    
    For the other direction, assume that \( e \in \support \sigma \) is such that \( d\sigma|_{\pm \support \hat \partial e} \neq 0. \) Then \( \sigma(e)\neq 0 \), and there must exist \( p \in \hat \partial e \) such that \( d\sigma(p) \neq 0 \). 
    Since \( d\hat \sigma = d\sigma \), it follows that \( d\hat \sigma(p) \neq 0 \). Consequently, there must exist \( e' \in \partial p \) such that \( \hat \sigma(e') \neq 0 \). For any such edge \( e' \), we have \(  d\hat \sigma|_{\pm \support \hat \partial e'} \neq 0, \) and 
    hence \( e' \in E_{E_0,\hat \sigma,\hat \sigma'}. \) In particular, this implies that \( \sigma(e') = \hat \sigma(e')  \neq 0, \) and hence \( e \in E_{E_0,\hat \sigma,\hat \sigma'}. \) Consequently, 
    \begin{equation*}
        \{ e \in \support \sigma \colon d\sigma|_{\pm \support \hat  \partial e} \neq 0 \}  \subseteq E_{E_0,\hat \sigma,\hat \sigma'}. 
    \end{equation*} 
    Analogously, we also obtain
         \begin{equation*}
        \{ e \in \support \sigma' \colon d\sigma'|_{\pm \support \hat  \partial e} \neq 0 \}  \subseteq E_{E_0,\hat \sigma,\hat \sigma'}. 
    \end{equation*}  
    Again recalling that \( \mathcal{G}(\hat \sigma,\hat \sigma') = \mathcal{G}(\sigma,\sigma'), \) we thus obtain 
    \begin{equation*}
        \begin{split}
            &E_{E_0, \sigma,  \sigma'} 
            =
            \mathcal{C}_{\mathcal{G}( \sigma, \sigma')} \bigl(E_0 \cup \{ e \in \support  \sigma \colon d \sigma|_{\pm \support \partial e} \neq 0 \} \cup \{ e \in \support  \sigma' \colon d \sigma'|_{\pm \support \partial e} \neq 0 \} \bigr)
            \\&\qquad\subseteq
            \mathcal{C}_{\mathcal{G}(\hat \sigma,\hat \sigma')} \bigl(E_0 \cup E_{E_0,\hat \sigma,\hat \sigma'} \cup \{ e \in \support \hat \sigma' \colon d\hat \sigma'|_{\pm \support \partial e} \neq 0 \} \bigr)
            =
            E_{E_0,\hat \sigma, \hat \sigma'} .
        \end{split}
    \end{equation*}
    This concludes the proof.
\end{proof}

\begin{lemma}\label{lemma: Edsigmalemma}
    Let \( \sigma, \sigma' \in \Omega^1(B_N,G) \), let \( E_0 \subseteq C_1(B_N)\), and let \( e \in C_1(B_N).\) Then \( e\in E_{E_0,\sigma, \sigma'}\) if and only if one of the following holds.
    \begin{enumerate}[label=(\roman*)]
        \item \(d(\sigma |_{\mathcal{C}_{\mathcal{G}(\sigma, \sigma')}(e)}) \neq 0\)
        \item \(d(\sigma' |_{\mathcal{C}_{\mathcal{G}(\sigma, \sigma')}(e)}) \neq 0\)
        \item \(\mathcal{C}_{\mathcal{G}(\sigma, \sigma')}(e) \cap E_0 \neq \emptyset\)
    \end{enumerate}
\end{lemma}

\begin{proof}
    Suppose first that \( e\in E_{E_0,\sigma, \sigma'}\).
    By the definition of \(\mathcal{C}_{\mathcal{G}(\sigma,\sigma')}(e)\), there exists an edge \( e' \in \mathcal{C}_{\mathcal{G}(\sigma,\sigma')}(e)\) such that 
    \begin{equation*}
        e' \in E_0 \cup \{ e'' \in \support \sigma \colon d\sigma|_{\pm \support \hat \partial e''} \neq 0\}\cup \{ e'' \in \support \sigma' \colon d\sigma'|_{\pm \support \hat \partial e''} \neq 0\}.
    \end{equation*} 
    If \( e' \in E_0 \), then \( e' \in \mathcal{C}_{\mathcal{G}(\sigma,\sigma')}(e) \cap E_0,\) and hence \( \mathcal{C}_{\mathcal{G}(\sigma,\sigma')}(e) \cap E_0 \neq \emptyset.\)
    If \( e' \notin E_0 \), then, by symmetry, we can assume that \( e' \in \{ e'' \in \support \sigma \colon d\sigma|_{\pm \support \hat \partial e''} \neq 0 \}. \) %
    In this case, we infer that there exists a plaquette \(p \in \hat{\partial} e' \) such that \(d\sigma(p) \neq 0.\)
    Since \( e' \in \mathcal{C}_{\mathcal{G}(\sigma,\sigma')}(e)\), we have \(\support \sigma \cap \support \partial p \subseteq \mathcal{C}_{\mathcal{G}(\sigma,\sigma')}(e)\), and so \(d(\sigma |_{\mathcal{C}_{\mathcal{G}(\sigma,\sigma')}(e)})(p) = d\sigma(p) \neq 0\).
    
    For the other direction, assume first that \(\mathcal{C}_{\mathcal{G}(\sigma, \sigma')}(e) \cap E_0 \neq \emptyset.\) Then there is \( e' \in E_0 \) such that \( e' \in \mathcal{C}_{\mathcal{G}(\sigma, \sigma')}(e) .\) Since \( e' \in \mathcal{C}_{\mathcal{G}(\sigma, \sigma')}(e) \) we must also have \( e \in \mathcal{C}_{\mathcal{G}(\sigma, \sigma')}(e'), \) which is a subset of \( E_{E_0,\sigma,\sigma'} \) since \( e' \in E_0. \)
    Next, assume instead that there is a plaquette \(p \in C_2(B_N) \) such that  \( d(\sigma|_{\mathcal{C}_{\mathcal{G}(\sigma,\sigma')}(e)})(p) \neq 0. \)
    Then there exists an edge \(e' \in \partial p\) with \(\sigma(e') \neq 0\) and \( e' \in \mathcal{C}_{\mathcal{G}(\sigma,\sigma')}(e)\). Thus  \(\support \sigma \cap \support \partial p \subseteq \mathcal{C}_{\mathcal{G}(\sigma,\sigma')}(e)\), and so \(d\sigma(p) = d(\sigma |_{\mathcal{C}_{\mathcal{G}(\sigma,\sigma')}(e)})(p) \neq 0.\) In particular, it follows that \(e' \in \support \partial p \subseteq \{ e'' \in \support \sigma \colon d\sigma|_{\pm \support \hat \partial e''} \neq 0 \} .\) Consequently, we must have \(e' \in \{ e'' \in \support \sigma \colon d\sigma|_{\pm \support \hat \partial e'' } \neq 0 \}\), and hence \( e' \in E_{\sigma,\sigma'}. \) Since \( e' \in \mathcal{C}_{\mathcal{G}_{\sigma,\sigma'}(e)} \), we thus have \( e \in \mathcal{C}_{\mathcal{G}_{\sigma,\sigma'}(e')} \subseteq E_{E_0,\sigma,\sigma'} \) as desired.
    Using symmetry, this concludes the proof.
\end{proof}

\begin{lemma}\label{lemma: technical coupling lemma}
    Let \( \beta_1,\beta_2 \in [0,\infty]\), \( \kappa \geq 0 \), \( E_0 \subseteq C_1(B_N),\) and \(  \sigma,\hat \sigma,\hat \sigma' \in \Omega^1(B_N,G) \).
    Then
	\begin{equation}\label{eq: goal of technical lemma}
	    \begin{split}
	        &\varphi_{\beta_1,\kappa}(\hat \sigma) \varphi_{\beta_2,\kappa}(\hat \sigma') \cdot \mathbb{1} \bigl( \hat \sigma|_{E_{E_0, \hat \sigma, \hat\sigma'}} +  \hat \sigma'|_{C_1(B_N)\smallsetminus E_{E_0, \hat \sigma, \hat\sigma'}} =  \sigma \bigr) 
	        \\&\qquad =
	        \varphi_{\beta_1,\kappa}(\sigma)\sum_{\sigma' \in \Omega^1(B_N,G)}\varphi_{\beta_2,\kappa}(\sigma') 
			\cdot \mathbb{1}\bigl( \hat \sigma =  \sigma|_{E_{E_0,  \sigma, \sigma'}}+ \sigma'|_{C_1(B_N)\smallsetminus E_{E_0,  \sigma, \sigma'}} \bigr)
			\\[-2ex]&\hspace{17em}\cdot 
			\mathbb{1}\bigl(  \hat \sigma' = \sigma' |_{E_{E_0,  \sigma, \sigma'}} +  \sigma |_{C_1(B_N)\smallsetminus E_{E_0,  \sigma, \sigma'}}  \bigr).
	    \end{split}
	\end{equation}
\end{lemma}
	
\begin{proof}  
	By Lemma~\ref{lemma: cluster is subconfig}, we have \( \hat\sigma |_{E_{E_0,  \hat\sigma, \hat\sigma'}} \leq  \hat\sigma \) and \(  \hat\sigma' |_{E_{E_0,  \hat\sigma, \hat\sigma'}} \leq  \hat\sigma'\) and hence, by Lemma~\ref{lemma: factorization property},
    \begin{equation}\label{eq: useful eq 1newa}
     \varphi_{\beta_1,\kappa}(\hat \sigma) = \varphi_{\beta_1,\kappa}\bigl( \hat \sigma |_{E_{E_0,  \hat\sigma, \hat\sigma'}} \bigr)
     \varphi_{\beta_1,\kappa}\bigl( \hat \sigma |_{C_1(B_N)\smallsetminus E_{E_0,  \hat\sigma, \hat \sigma'}} \bigr)
     \end{equation}
     and
     \begin{equation}\label{eq: useful eq 1newb}
        \varphi_{\beta_2,\kappa}(\hat \sigma') = \varphi_{\beta_2,\kappa}\bigl( \hat \sigma' |_{E_{E_0,  \hat\sigma, \hat\sigma'}} \bigr)
        \varphi_{\beta_2,\kappa}\bigl(  \hat\sigma' |_{C_1(B_N)\smallsetminus E_{E_0,  \hat\sigma, \hat\sigma'}} \bigr).
    \end{equation}
    Next, by Lemma~\ref{lemma: closed is closed new}, we have
    \begin{equation*}
        d(\hat \sigma |_{C_1(B_N)\smallsetminus E_{E_0,  \hat\sigma, \hat\sigma'}}) = 0,
    \end{equation*}
    and hence 
    \begin{equation*}
        \varphi_{\beta_1,\kappa} \bigl( \hat\sigma |_{C_1(B_N)\smallsetminus E_{E_0,  \hat\sigma, \hat\sigma'}}\bigr)
        =
        \varphi_{\kappa} \bigl( \hat\sigma |_{C_1(B_N)\smallsetminus E_{E_0,  \hat\sigma, \hat\sigma'}}  \bigr)
        =
        \varphi_{\beta_2,\kappa} \bigl( \hat\sigma |_{C_1(B_N)\smallsetminus E_{E_0,  \hat\sigma, \hat\sigma'}}\bigr).
    \end{equation*}
    Since \( \hat \sigma'|_{E_{E_0, \hat\sigma,\hat\sigma'}} \) and \( \hat\sigma|_{C_1(B_N)\smallsetminus E_{E_0, \hat\sigma,\hat\sigma'}} \) have disjoint supports, it also follows that
    \begin{equation*}
        \hat\sigma'|_{E_{E_0, \hat\sigma,\hat\sigma'}} \leq \hat\sigma'|_{E_{E_0, \hat\sigma,\hat\sigma'}} + \hat\sigma|_{C_1(B_N)\smallsetminus E_{E_0, \hat\sigma,\hat\sigma'}}.
    \end{equation*}
    Thus, by Lemma~\ref{lemma: factorization property}, it follows that
    \begin{equation}\label{eq: useful eq 2new}
        \varphi_{\beta_2,\kappa}\bigl(\hat \sigma'|_{E_{E_0, \hat \sigma,\hat\sigma'}}\bigr)
        \varphi_{\beta_1,\kappa}\bigl(\hat\sigma|_{C_1(B_N)\smallsetminus E_{E_0, \hat\sigma,\hat\sigma'}}\bigr)
        =
        \varphi_{\beta_2,\kappa}\bigl(\hat\sigma'|_{E_{E_0, \hat\sigma,\hat\sigma'}} + \hat\sigma|_{C_1(B_N)\smallsetminus E_{E_0, \hat\sigma,\hat\sigma'}}\bigr).
    \end{equation}
    By symmetry, we also have
    \begin{equation}\label{eq: useful eq 3new}
        \varphi_{\beta_1,\kappa}\bigl(\hat\sigma|_{E_{E_0, \hat\sigma,\hat\sigma'}}\bigr)
        \varphi_{\beta_2,\kappa}\bigl(\hat\sigma'|_{C_1(B_N)\smallsetminus E_{E_0, \hat\sigma,\hat\sigma'}}\bigr)
        =
        \varphi_{\beta_1,\kappa}\bigl(\hat\sigma|_{E_{E_0, \hat\sigma,\hat\sigma'}} + \hat\sigma'|_{C_1(B_N)\smallsetminus E_{E_0, \hat\sigma,\hat\sigma'}}\bigr).
    \end{equation}
    Combining~\eqref{eq: useful eq 1newa},~\eqref{eq: useful eq 1newb},~\eqref{eq: useful eq 2new}, and~\eqref{eq: useful eq 3new}, it follows that
    \begin{equation}\label{eq: coupling goal eq ii} 
		\begin{split}
			&
			\varphi_{\beta_1,\kappa}( \hat\sigma) \varphi_{\beta_2,\kappa}( \hat\sigma') \cdot \mathbb{1} \bigl( \hat\sigma|_{E_{E_0,  \hat\sigma, \hat\sigma'}} +  \hat\sigma'|_{C_1(B_N)\smallsetminus E_{E_0,  \hat\sigma, \hat\sigma'}} =  \sigma \bigr) 
			\\&\qquad 
			=
			\varphi_{\beta_1,\kappa}\bigl( \hat\sigma |_{E_{E_0,  \hat\sigma, \hat\sigma'}} + \hat\sigma' |_{C_1(B_N)\smallsetminus E_{E_0,  \hat\sigma, \hat\sigma'}} \bigr)
			\varphi_{\beta_2,\kappa}\bigl(  \hat\sigma' |_{E_{E_0,  \hat\sigma, \hat\sigma'}} + \hat\sigma |_{C_1(B_N)\smallsetminus E_{E_0,  \hat\sigma, \hat\sigma'}} \bigr)
			\\&\qquad\qquad \cdot 
			\mathbb{1} \bigl(  \hat\sigma|_{E_{E_0,  \hat\sigma, \hat\sigma'}} +  \hat\sigma'|_{C_1(B_N)\smallsetminus E_{E_0,  \hat\sigma, \hat\sigma'}} =  \sigma \bigr) 
			\\&\qquad 
			=
			\varphi_{\beta_1,\kappa}(\sigma)
			\varphi_{\beta_2,\kappa}\bigl(  \hat\sigma' |_{E_{E_0, \hat \sigma, \hat\sigma'}} + \hat\sigma |_{C_1(B_N)\smallsetminus E_{E_0,  \hat\sigma, \hat\sigma'}} \bigr)
			\cdot 
			\mathbb{1} \bigl( \hat\sigma|_{E_{E_0,  \hat\sigma, \hat\sigma'}} + \hat\sigma'|_{C_1(B_N)\smallsetminus E_{E_0,  \hat\sigma, \hat\sigma'}} =  \sigma \bigr) 
			\\&\qquad 
			=
			\varphi_{\beta_1,\kappa}(\sigma)\sum_{\sigma' \in \Omega^1(B_N,G)}\varphi_{\beta_2,\kappa}(\sigma') 
			\cdot \mathbb{1} \bigl(  \hat\sigma' |_{E_{E_0,  \hat\sigma, \hat\sigma'}} + \hat\sigma |_{C_1(B_N)\smallsetminus E_{E_0,  \hat\sigma, \hat\sigma'}} = \sigma'\bigr)
			\\&\qquad\qquad\cdot 
			\mathbb{1} \bigl( \hat\sigma|_{E_{E_0,  \hat\sigma, \hat\sigma'}} + \hat\sigma'|_{C_1(B_N)\smallsetminus E_{E_0,  \hat\sigma, \hat\sigma'}} =  \sigma \bigr).
			\end{split}	
	\end{equation} 

    Now fix \( \sigma' \in \Omega^1(B_N,G)\) and assume that 
    \begin{align*}
    	 \begin{cases}
    	 	\sigma = \hat\sigma|_{E_{E_0,  \hat\sigma, \hat\sigma'}} + \hat\sigma'|_{C_1(B_N)\smallsetminus E_{E_0,  \hat\sigma, \hat\sigma'}} \cr
    	 	\sigma'  = \hat\sigma' |_{E_{E_0,  \hat\sigma, \hat\sigma'}} + \hat\sigma |_{C_1(B_N)\smallsetminus E_{E_0,  \hat\sigma, \hat\sigma'}}  
		 .
    	 \end{cases}
    \end{align*}
    By Lemma~\ref{lemma: E in coupling}, we have \( E_{E_0,  \hat\sigma, \hat\sigma'} = E_{E_0, \sigma,\sigma'}. \) 
    Since 
    \( \sigma|_{E_{E_0,  \hat\sigma, \hat\sigma'}}  = \hat\sigma|_{E_{E_0,  \hat\sigma, \hat\sigma'}} \) and \( \sigma' |_{C_1(B_N)\smallsetminus E_{E_0,  \hat\sigma, \hat\sigma'}} = \hat\sigma |_{C_1(B_N)\smallsetminus E_{E_0,  \hat\sigma, \hat\sigma'}},  \) it follows that
    \begin{equation*}
    	\hat\sigma = \sigma|_{E_{E_0,  \hat\sigma, \hat\sigma'}}  +  \sigma'|_{C_1(B_N)\smallsetminus E_{E_0,  \hat\sigma, \hat\sigma'}} =  \sigma|_{E_{E_0,  \sigma, \sigma'}}  +  \sigma'|_{C_1(B_N)\smallsetminus E_{E_0,  \sigma, \sigma'}}.
    \end{equation*} 
    Analogously, since \( \sigma'|_{E_{E_0,  \hat\sigma, \hat\sigma'}}   = \hat\sigma' |_{E_{E_0,  \hat\sigma, \hat\sigma'}} \) and \( \sigma|_{C_1(B_N)\smallsetminus E_{E_0,  \hat\sigma, \hat\sigma'}} =\hat\sigma'|_{C_1(B_N)\smallsetminus E_{E_0,  \hat\sigma, \hat\sigma'}}, \) it follows that
    \begin{equation*}
    	\hat\sigma' = \sigma'|_{E_{E_0,  \hat\sigma, \hat\sigma'}}  +  \sigma|_{C_1(B_N)\smallsetminus E_{E_0,  \hat\sigma, \hat\sigma'}} = \sigma'|_{E_{E_0,  \sigma, \sigma'}}  +  \sigma|_{C_1(B_N)\smallsetminus E_{E_0,  \sigma, \sigma'}}.
    \end{equation*}
    This shows that for any \( \sigma' \in \Omega^1(B_N,G),\) we have
	\begin{equation}\label{eq: doublesum11}
		\begin{split}
			&
			\mathbb{1} \bigl(  \hat\sigma' |_{E_{E_0,  \hat\sigma, \hat\sigma'}} + \hat\sigma |_{C_1(B_N)\smallsetminus E_{E_0,  \hat\sigma, \hat\sigma'}} = \sigma'\bigr)
			\cdot 
			\mathbb{1} \bigl( \hat\sigma|_{E_{E_0,  \hat\sigma, \hat\sigma'}} + \hat\sigma'|_{C_1(B_N)\smallsetminus E_{E_0,  \hat\sigma, \hat\sigma'}} =  \sigma \bigr)
			\\&\qquad=
			\mathbb{1}\bigl(  \hat\sigma =  \sigma|_{E_{E_0,  \sigma, \sigma'}}+ \sigma'|_{C_1(B_N)\smallsetminus E_{E_0,  \sigma, \sigma'}} \bigr)
			\cdot 
			\mathbb{1}\bigl(   \hat\sigma' = \sigma' |_{E_{E_0,  \sigma, \sigma'}} +  \sigma |_{C_1(B_N)\smallsetminus E_{E_0,  \sigma, \sigma'}}  \bigr).
		\end{split}
	\end{equation}

	Combining~\eqref{eq: coupling goal eq ii}  and~\eqref{eq: doublesum11}, we obtain~\eqref{eq: goal of technical lemma} as desired.
\end{proof}

\subsection{A coupling between two \texorpdfstring{\( \mathbb{Z}_n \)}{Zn}-models }\label{sec: Z-Z coupling}

In this section, we define a coupling between two copies of \( \mu_{N,\infty,\kappa} \), constructed to always agree on a given set \( E_0 \subseteq C_1(B_N) \)

\begin{definition}[A coupling of two \texorpdfstring{\( \mathbb{Z}_n \)}{Zn}-models]\label{def: the ZZ coupling}
    For \( \kappa \geq 0, \)
    \( \sigma,\sigma' \in \Omega^1_0(B_N,G), \) \( E_0 \subseteq C_1(B_N), \) and \( E_{E_0, \sigma,\sigma'} = \mathcal{C}_{\mathcal{G}_{\sigma,\sigma'}}(E_0),\) we define
    \begin{align*}
        \mu^{E_0}_{N, (\infty,\kappa),(\infty,\kappa)}(\sigma, \sigma') 
        &\coloneqq \mu_{N, \infty,\kappa} \times \mu_{N,\infty,\kappa}
        \pigl(\bigl\{
        (\hat{\sigma}, \hat{\sigma}')\in \Omega^1_0(B_N,G) \times \Omega^1_0(B_N,G) \colon 
        \\[-0.5ex]&\hspace{11em}  \sigma=\hat{\sigma}|_{E_{E_0, \hat \sigma,\hat \sigma'}} + \hat \sigma'|_{C_1(B_N) \smallsetminus E_{E_0, \hat \sigma,\hat \sigma'}} \text{ and } \sigma'=\hat \sigma'
        \bigr\} \pigr).
    \end{align*}
    We let \( \mathbb{E}^{E_0}_{N,(\infty,\kappa),(\infty,\kappa)} \) denote the corresponding expectation. 
\end{definition}

\begin{figure}[htp]
    \centering
    \begin{subfigure}[t]{0.32\textwidth}\centering
        \includegraphics[width=\textwidth]{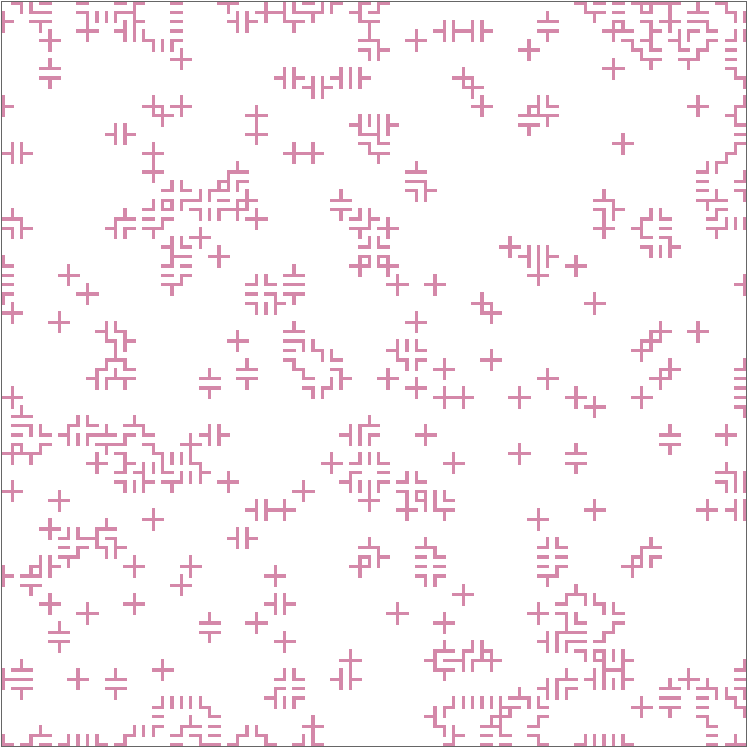}
        \caption{Red edges correspond to the support of \( \hat \sigma \in \Omega^1_0(B_N,\mathbb{Z}_2)\).}
    \end{subfigure}
    \hfil
    \begin{subfigure}[t]{0.32\textwidth}\centering
        \includegraphics[width=\textwidth]{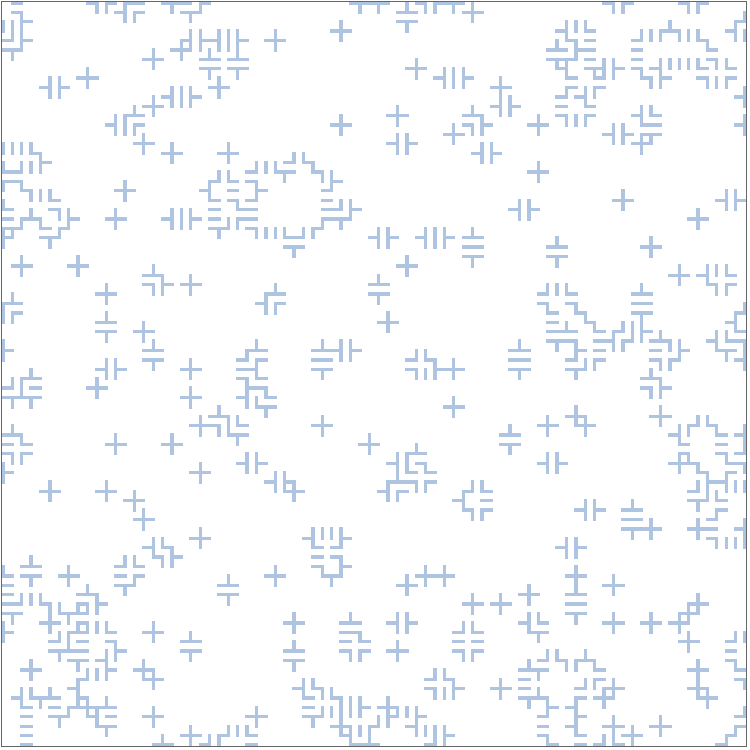}
        \caption{Blue edges correspond to the support of \( \hat \sigma'  \in \Omega^1_0(B_N,\mathbb{Z}_2)\).}
    \end{subfigure}
    \hfil
    \begin{subfigure}[t]{0.32\textwidth}\centering
        \includegraphics[width=\textwidth]{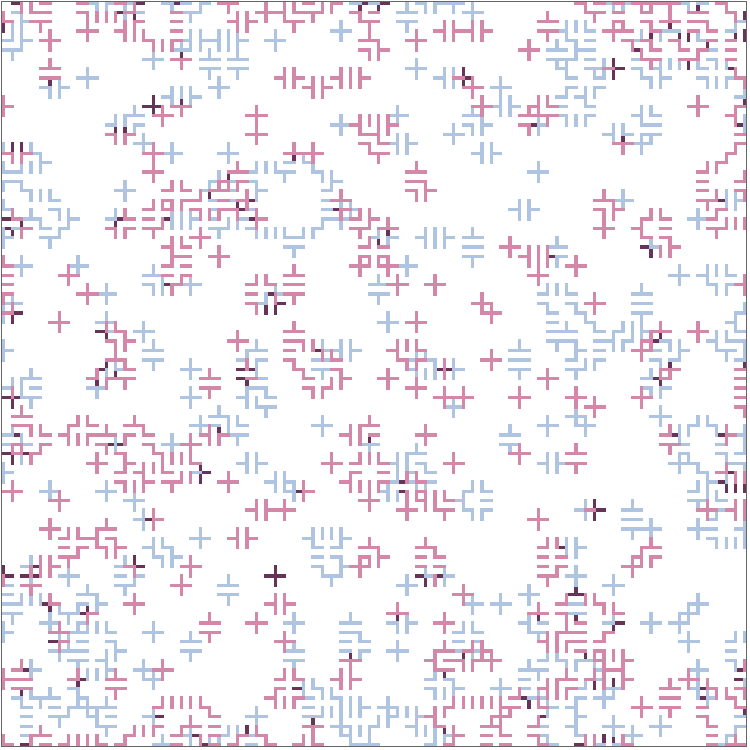}
        \caption{Red edges correspond to the support of \( \hat \sigma \), and blue edges correspond to the support of \( \hat \sigma'\).}
    \end{subfigure}

    \vspace{1ex}
    \begin{subfigure}[t]{0.32\textwidth}\centering
        \includegraphics[width=\textwidth]{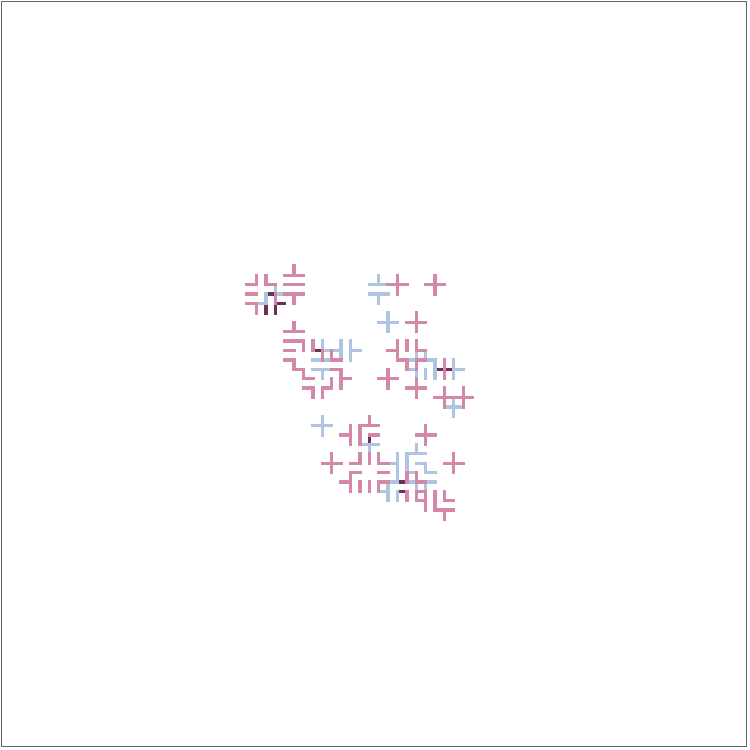}
        \caption{Red edges correspond to the support of  \(  \hat \sigma|_{E_{E_0,\hat \sigma,\hat \sigma'}}  \), and blue edges correspond to the support of \(  \hat \sigma'|_{E_{E_0,\hat \sigma,\hat \sigma'}}. \)}\label{subfig: the ZZ coupling d}
    \end{subfigure}
    \hfil
    \begin{subfigure}[t]{0.32\textwidth}\centering
        \includegraphics[width=\textwidth]{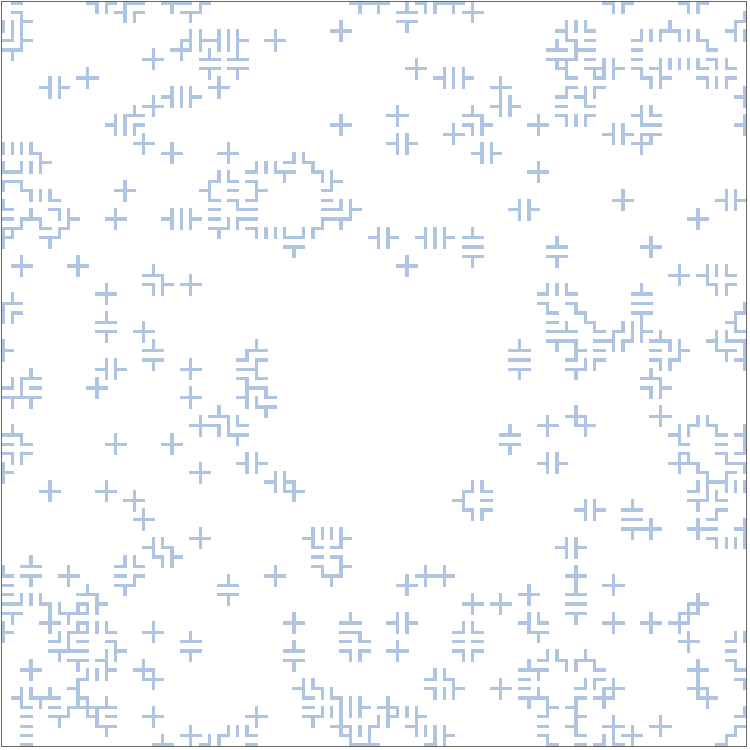}
        \caption{Blue edges correspond to the support of \( \hat \sigma'|_{C_1(B_N)\smallsetminus E_{\hat \sigma, \hat \sigma'}}\).}
    \end{subfigure}
    \hfil
    \begin{subfigure}[t]{0.32\textwidth}\centering
        \includegraphics[width=\textwidth]{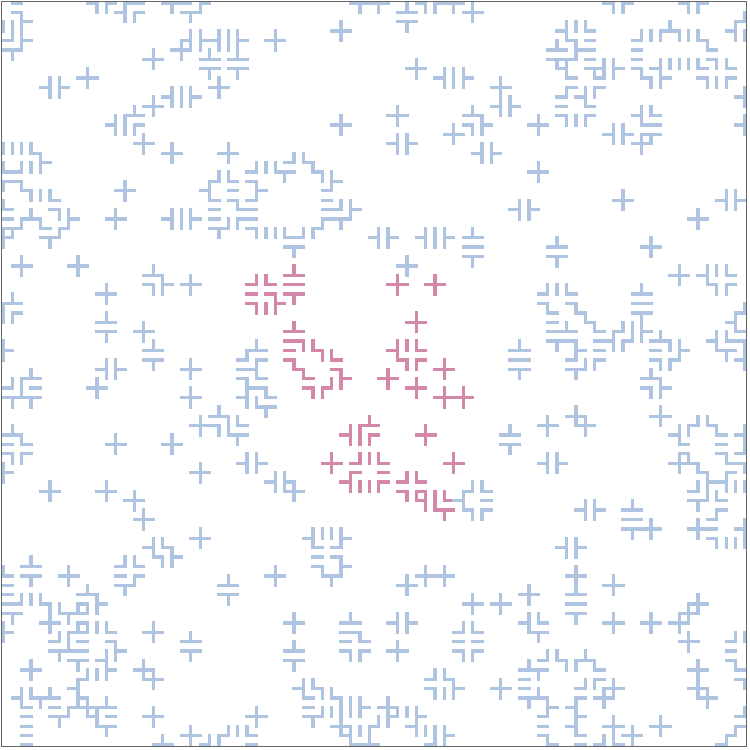}
        \caption{Red edges correspond to the support of \( \hat \sigma|_{E_{E_0,\hat \sigma,\hat \sigma'}} \), and purple edges correspond to the support of  \(  \hat \sigma'|_{E_N\smallsetminus E_{\hat \sigma,\hat \sigma'}}  \).}\label{subfig: the ZZ coupling e}
    \end{subfigure}
    \caption{Illustration of the coupling \( (\sigma, \sigma')  \sim \mu^{E_0}_{N,(\infty,\beta),(\infty,\beta)}\) defined in Definition~\ref{def: the coupling}, simulated on a 2-dimensional lattice, with \( G = \mathbb{Z}_2 \), and with \( E_0 = C_1(B_{N/4}).\)}
    \label{fig: the ZZ coupling}
\end{figure}

\begin{remark}
    When \(  \sigma,\sigma' \in \Omega^1_0(B_N,G) \), then \( d\sigma = d\sigma' = 0, \) and hence the definition of \( E_{E_0,\sigma,\sigma'} \) in Definition~\ref{def: the ZZ coupling} in consistent with~\eqref{eq: EE0sigmasigmadef}.
\end{remark}

\begin{remark}
    By definition, if \(  \hat \sigma, \hat  \sigma' \sim \mu_{N,\infty,\kappa} \) are independent, and we let \(\sigma \coloneqq  \hat  \sigma|_{E_{E_0,\hat \sigma, \hat \sigma'}}
        +
        \hat \sigma'|_{C_1(B_N) \smallsetminus E_{E_0, \hat \sigma, \hat \sigma'}} \)
    and \( \sigma' \coloneqq \hat \sigma' \), then \( (\sigma,\sigma') \sim \mu^{E_0}_{N,(\infty,\kappa),(\infty,\kappa)} \).
\end{remark}

The next result shows that the measure introduced in Definition~\ref{def: the ZZ coupling} is indeed a coupling.
\begin{proposition}\label{proposition: Ising coupling is a coupling}
    Let \( \kappa \geq 0 \), and let \( E_0 \subseteq C_1(B_N) \). Then \( \mu^{E_0}_{N,(\infty,\kappa),{(\infty,\kappa})} \) is a coupling of \( \mu_{N,\infty,\kappa} \) and \( \mu_{N,\infty,\kappa} \). 
\end{proposition}

\begin{proof} 
    It is immediate from the definition that if \( (\sigma,\sigma') \sim \mu^{E_0}_{N,(\infty,\kappa),(\infty,\kappa)}\), then \( \sigma' \sim\mu_{N,\infty,\kappa} \), and it is hence sufficient to show that \( \sigma \sim \mu_{N,\infty,\kappa} \).
	To this end, fix some \(  \sigma \in \Omega^1_0(B_N) \).
	We need to show that
	\begin{equation*}
	    \mu_{N,\infty,\kappa}\times \mu_{N,\infty,\kappa} \bigl( \bigl\{ (\hat \sigma, \hat \sigma') \in \Omega^1_0(B_N) \times \Omega^1_0(B_N) \colon \hat \sigma|_{E_{E_0,\hat \sigma, \hat\sigma'}} + \hat \sigma'|_{C_1(B_N)\smallsetminus E_{E_0, \hat \sigma,\hat  \sigma'}} = \sigma \bigr\} \bigr)
		=
		\mu_{N,\infty,\kappa}( \sigma),
	\end{equation*}
	or equivalently, that
	\begin{equation}\label{eq: Z-Z coupling goal eq}
		\sum_{\substack{\hat \sigma \in \Omega^1_0(B_N,G),\\ \hat \sigma' \in \Omega^1_0(B_N,G)}}
		\varphi_{\kappa}(\hat \sigma) \varphi_{\kappa}(\hat \sigma') \cdot \mathbb{1} \bigl( \hat \sigma|_{E_{E_0,\hat \sigma,\hat \sigma'}} + \hat \sigma'|_{C_1(B_N)\smallsetminus E_{E_0,\hat \sigma,\hat \sigma'}} =  \sigma \bigr)
		=
		\varphi_\kappa( \sigma) \sum_{ \sigma' \in \Omega^1_0(B_N,G)} \varphi_{\kappa} ( \sigma').
	\end{equation}
	We now rewrite the left-hand side of~\eqref{eq: Z-Z coupling goal eq} in order to see that this equality indeed holds.
	To this end, note first that by Lemma~\ref{lemma: technical coupling lemma}, applied with \( \beta_1=\beta_2=\infty \), we have
    \begin{equation}\label{eq: Z-Z coupling goal eq 2} 
		\begin{split}
			&\sum_{\substack{\hat \sigma \in \Omega^1_0(B_N,G),\\ \hat \sigma' \in \Omega^1_0(B_N,G)}}
			\varphi_{\kappa}(\hat \sigma) \varphi_{\kappa}(\hat \sigma')\mathbb{1} \bigl( \hat  \sigma|_{E_{E_0, \hat \sigma, \hat \sigma'}} +  \hat \sigma'|_{C_1(B_N)\smallsetminus E_{E_0, \hat \sigma, \hat \sigma'}} =  \sigma \bigr) 
			\\&\qquad =
			\varphi_{\kappa}( \sigma )
			\sum_{\sigma' \in \Omega^1_0(B_N,G)}
			\varphi_{\kappa}(\sigma')
			\cdot 
			\sum_{\substack{\hat \sigma \in \Omega^1_0(B_N,G),\\ \hat \sigma' \in \Omega^1_0(B_N,G)}}
			\mathbb{1}\bigl( \hat  \sigma =  \sigma'|_{E_{E_0, \sigma, \sigma'}}+ \sigma|_{C_1(B_N)\smallsetminus E_{E_0, \sigma, \sigma'}} \bigr)
			\\[-3.5ex]&\hspace{22em}\cdot 
			\mathbb{1}\bigl(  \hat  \sigma' = \sigma |_{E_{E_0, \sigma, \sigma'}} +  \sigma' |_{C_1(B_N)\smallsetminus E_{E_0, \sigma, \sigma'}}  \bigr).
		\end{split}
	\end{equation}
	Since \(  \sigma, \sigma' \in \Omega^1_0(B_N) \), we can apply Lemma~\ref{lemma: closed is closed new} to see that \( \sigma |_{E_{E_0, \sigma, \sigma'}} +  \sigma' |_{C_1(B_N)\smallsetminus E_{E_0, \sigma, \sigma'}} \in \Omega^1_0(B_N,G)\) and  \(  \sigma' |_{E_{E_0, \sigma, \sigma'}} +  \sigma|_{C_1(B_N)\smallsetminus E_{E_0, \sigma, \sigma'}}  \in \Omega^1_0(B_N,G).\)
	From this it follows that the double sum on the right-hand side of~\eqref{eq: Z-Z coupling goal eq 2} is equal to \( 1 \), and hence we obtain~\eqref{eq: Z-Z coupling goal eq} as desired. This completes the proof.
\end{proof}

\begin{lemma}\label{lemma: properties of Z-Z coupling measure} 
    Let \( \beta,\kappa \geq 0,\) let \( E_0 \subseteq C_1(B_N),\) and let \( (\sigma,\sigma')\in \Omega^1_0(B_N,G) \times \Omega^1_0(B_N,G) \) be such that  \( \mu^{E_0}_{N,(\infty,\kappa),(\infty,\kappa)}(\sigma , \sigma') \neq 0\). Then \( \sigma(e) = \sigma'(e) \) for all \( e \in C_1(B_N) \smallsetminus E_{E_0,\sigma,\sigma'} \).   
\end{lemma}

\begin{proof}
    Since \(\mu^{E_0}_{N,(\infty,\kappa),(\infty,\kappa)}(\sigma , \sigma') \neq 0\), by definition, there is \( (\hat \sigma, \hat \sigma') \in \Omega^1_0(B_N,G) \times \Omega^1_0(B_N,G) \) such that \( \sigma = \hat \sigma|_{E_{E_0,\hat \sigma,\hat \sigma'}} + \hat\sigma'|_{C_1(B_N) \smallsetminus E_{E_0,\hat{\sigma}, \hat{\sigma}'}} \) and \( \sigma' = \hat \sigma'\). Using Lemma~\ref{lemma: E in coupling}, we immediately obtain the desired conclusion.
\end{proof}

One application of the coupling introduced in Definition~\ref{def: the ZZ coupling}, which will be particularly useful to us, is the following proposition.
\begin{proposition}\label{proposition: coupling covariance}
    Let \( \kappa \geq 0 \), and let \( E_0,E_1 \subseteq C^1(B_N,\mathbb{Z}) \) have disjoint supports. Further, let \( f_0,f_1 \colon \Omega^1(B_N,G) \to \mathbb{R}\) be such that \( f_0(\sigma) = f(\sigma|_{E_0}) \) and \( f_1(\sigma) = f_1(\sigma|_{E_1}) \) for all \( \sigma \in \Omega^1(B_N). \) 
    Then
    \begin{align}
        \label{eq: coupling covariance} 
        &\pigl| 
        \mathbb{E}_{N,\infty,\kappa} \bigl[ f_0(\sigma)f_1(\sigma) \bigr]
        - 
        \mathbb{E}_{N,\infty,\kappa} \bigl[ f_0(\sigma) \bigr]
        \mathbb{E}_{N,\infty,\kappa} \bigl[ f_1(\sigma)\bigr]
        \pigr|
        \\\nonumber&\qquad \leq 
        2\| f_0 \|_\infty \| f_1 \|_\infty \sum_{e \in  E_1} \mu^{ E_0}_{N,(\infty,\kappa),(\infty,\kappa)} \bigl( \bigr\{ (\hat\sigma,\hat\sigma') \in \Omega^1_0(B_N,G) \times \Omega^1_0(B_N,G) \colon e \in  E_{E_0,\hat \sigma,\hat \sigma'} \bigr\} \bigr).
    \end{align}
\end{proposition}
We provide an upper bound on the right hand side of~\eqref{eq: coupling covariance} in Proposition~\ref{proposition: ZZ upper bound}.

\begin{proof}[Proof of Proposition~\ref{proposition: coupling covariance}]
    To simplify notation, for \( \sigma \in \Omega^1(B_N,G), \) let \( F(\sigma) \coloneqq f_1(\sigma)f_2(\sigma). \)
    
    Let \( \hat \sigma,\hat \sigma' \sim \mu_{N,\infty,\kappa}.\) Note that when \( d\hat \sigma = d\hat \sigma' = 0 \), we have \( E_{E_0,\hat \sigma,\hat \sigma'} = \mathcal{C}_{\mathcal{G}(\hat \sigma,\hat \sigma') }(E_0) .\)
    Define
    \begin{equation*}
        \begin{cases}
            \sigma \coloneqq  \hat \sigma|_{E_{E_0,\hat \sigma,\hat \sigma'}} + \hat \sigma'|_{C_1(B_N)\smallsetminus E_{E_0,\hat \sigma,\hat \sigma'}}. \cr 
            \sigma' \coloneqq \hat \sigma'.
        \end{cases}
    \end{equation*}
    Then \( (\sigma,\sigma') \sim \mu^{E_0}_{N,(\infty,\kappa),(\infty,\kappa)}, \) and hence \( \sigma,\sigma' \sim \mu_{N,\infty,\kappa}. \)
    Consequently, we have
    \begin{equation}\label{eq: coupling cov eq 1}
        \begin{split}
            & \mathbb{E}_{N,\infty,\kappa} \bigl[ F(\sigma) \bigr] 
            =
            \mathbb{E}^{E_0}_{N,(\infty,\kappa),(\infty,\kappa)} \bigl[ F(\sigma)  \bigr] 
            \\&\qquad=
            \mathbb{E}_{N,\infty,\kappa} \times \mathbb{E}_{N,\infty,\kappa} \bigl[ F (\hat \sigma|_{E_{E_0,\hat \sigma,\hat \sigma'}} + \hat \sigma'|_{C_1(B_N)\smallsetminus E_{E_0,\hat \sigma,\hat \sigma'}})   \bigr]
            \\&\qquad=
            \mathbb{E}_{N,\infty,\kappa} \times \mathbb{E}_{N,\infty,\kappa}\bigl[ F (\hat \sigma|_{E_{E_0,\hat \sigma,\hat \sigma'}} + \hat \sigma'|_{C_1(B_N)\smallsetminus E_{E_0,\hat \sigma,\hat \sigma'}}) \cdot \mathbb{1}( E_1 \cap E_{E_0,\hat \sigma,\hat \sigma'}) = \emptyset) \bigr]
            \\&\qquad\qquad+
            \mathbb{E}_{N,\infty,\kappa}\times\mathbb{E}_{N,\infty,\kappa} \bigl[ F (\hat \sigma|_{E_{E_0,\hat \sigma,\hat \sigma'}} + \hat \sigma'|_{C_1(B_N)\smallsetminus E_{E_0,\hat \sigma,\hat \sigma'}}) \cdot \mathbb{1}( E_1 \cap E_{E_0,\hat \sigma,\hat \sigma'}) \neq \emptyset) \bigr].
        \end{split}
    \end{equation}
    
    Since \( E_{E_0,\hat \sigma,\hat \sigma'} = \mathcal{C}_{\mathcal{G}(\hat \sigma,\hat \sigma')}(E_0), \) we have  \( ( \support \hat \sigma \cup \support \hat \sigma' )\cap E_0 \subseteq E_{E_0,\hat \sigma,\hat \sigma'}. \) This implies in particular that
    \begin{equation*}
        (\hat \sigma|_{E_{E_0,\hat \sigma, \hat \sigma'}} + \hat \sigma'|_{C_1(B_N)\smallsetminus E_{E_0,\hat \sigma,\hat \sigma'}})|_{E_0} 
        =
        (\hat \sigma|_{E_{E_0,\hat \sigma, \hat \sigma'}})|_{E_0}  + (\hat \sigma'|_{C_1(B_N)\smallsetminus E_{E_0,\hat \sigma,\hat \sigma'}})|_{E_0} 
        = 
        \hat \sigma|_{E_0},
    \end{equation*}
    and hence
    \begin{equation*}
        \begin{split}
            &f_0(\hat \sigma|_{E_{E_0,\hat \sigma, \hat \sigma'}} + \hat \sigma'|_{C_1(B_N)\smallsetminus E_{E_0,\hat \sigma,\hat \sigma'}})
            =
            f_0(\hat \sigma).
        \end{split}
    \end{equation*} 
    At the same time, on the event \( E_1 \cap E_{E_0,\hat \sigma,\hat \sigma'} = \emptyset, \) for all \( e \in E_1 \) we have 
    \begin{equation*}
        \begin{split}
            (\hat \sigma|_{E_{E_0,\hat \sigma,\hat \sigma'}} + \hat \sigma'|_{C_1(B_N)\smallsetminus E_{E_0,\hat \sigma,\hat \sigma'}})|_{E_1}
            =
            (\hat \sigma|_{E_{E_0,\hat \sigma,\hat \sigma'}})|_{E_1} + (\hat \sigma'|_{C_1(B_N)\smallsetminus E_{E_0,\hat \sigma,\hat \sigma'}})|_{E_1} = 0 + \hat \sigma'|_{E_1},
        \end{split}
    \end{equation*}
    and hence
    \begin{equation*}
        \begin{split} 
            &f_1(\hat \sigma|_{E_{E_0,\hat \sigma,\hat \sigma'}} + \hat \sigma'|_{C_1(B_N)\smallsetminus E_{E_0,\hat \sigma,\hat \sigma'}}) 
            =
            f_1(\hat \sigma').
        \end{split}
    \end{equation*}
    Consequently, on the event \( E_1 \cap E_{E_0,\sigma,\sigma'} = \emptyset, \) we have
    \begin{equation*}
        \begin{split}
            &F  (\hat \sigma|_{E_{E_0,\hat \sigma,\hat \sigma'}} + \hat \sigma'|_{C_1(B_N)\smallsetminus E_{E_0,\hat \sigma,\hat \sigma'}})
            =f_0 \bigl( \hat \sigma\bigr)
            f_1 \bigl(  \hat \sigma'\bigr),
        \end{split}
    \end{equation*}
    and hence
    \begin{equation*}
        \begin{split}
            &
            \mathbb{E}_{N,\infty,\kappa} \times \mathbb{E}_{N,\infty,\kappa}\bigl[ F (\hat \sigma|_{E_{E_0,\hat \sigma,\hat \sigma'}} + \hat \sigma'|_{C_1(B_N)\smallsetminus E_{E_0,\hat \sigma,\hat \sigma'}}) \cdot \mathbb{1}( E_1 \cap E_{E_0,\hat \sigma,\hat \sigma'} = \emptyset) \bigr]
            \\&\qquad= 
            \mathbb{E}_{N,\infty,\kappa} \times \mathbb{E}_{N,\infty,\kappa}\bigl[ f_0(\hat \sigma) f_1( \hat \sigma') \cdot \mathbb{1}( E_1 \cap E_{E_0,\hat \sigma,\hat \sigma'} = \emptyset) \bigr]
            \\&\qquad= 
            \mathbb{E}_{N,\infty,\kappa} \times \mathbb{E}_{N,\infty,\kappa}\bigl[ f_0(\hat \sigma) f_1( \hat \sigma') \bigr]
            -
            \mathbb{E}_{N,\infty,\kappa} \times \mathbb{E}_{N,\infty,\kappa}\bigl[ f_0(\hat \sigma) f_1( \hat \sigma') \cdot \mathbb{1}( E_1 \cap E_{E_0,\hat \sigma,\hat \sigma'} \neq \emptyset) \bigr]
            \\&\qquad= 
            \mathbb{E}_{N,\infty,\kappa} \bigl[ f_0(\hat \sigma) \bigr] \mathbb{E}_{N,\infty,\kappa} \bigl[ f_1( \hat \sigma') \bigr]
            -
            \mathbb{E}_{N,\infty,\kappa} \times \mathbb{E}_{N,\infty,\kappa}\bigl[ f_0(\hat \sigma) f_1( \hat \sigma') \cdot \mathbb{1}( E_1 \cap E_{E_0,\hat \sigma,\hat \sigma'} \neq \emptyset) \bigr].
        \end{split}
    \end{equation*}
    Inserting this into~\eqref{eq: coupling cov eq 1}, we see that
    \begin{equation*}
        \begin{split}
            & \mathbb{E}_{N,\infty,\kappa} \bigl[ f_0(\sigma)f_1(\sigma) \bigr] 
            \\&\qquad=
            \mathbb{E}_{N,\infty,\kappa} \bigl[ f_0( \hat \sigma)  \bigr] 
            \mathbb{E}_{N,\infty,\kappa}\bigl[ 
            f_1( \hat \sigma')  \bigr]
            -
            \mathbb{E}_{N,\infty,\kappa} \times \mathbb{E}_{N,\infty,\kappa}\bigl[ f_0 ( \hat \sigma)
            f_1( \hat \sigma')
            \cdot \mathbb{1}( E_1 \cap E_{E_0,\hat \sigma,\hat \sigma'}) \neq \emptyset) \bigr]
            \\&\qquad\qquad+
            \mathbb{E}_{N,\infty,\kappa}\times\mathbb{E}_{N,\infty,\kappa} \bigl[ F 
            (\hat \sigma'|_{E_{E_0,\hat \sigma,\hat \sigma'}} + \hat \sigma|_{C_1(B_N)\smallsetminus E_{E_0,\hat \sigma,\hat \sigma'}}) \cdot \mathbb{1}( E_1 \cap E_{E_0,\hat \sigma,\hat \sigma'} \neq \emptyset) \bigr].
        \end{split}
    \end{equation*}
    In particular, this implies that
    \begin{equation*}
        \begin{split}
            & \Bigl| \mathbb{E}_{N,\infty,\kappa} \bigl[ f_0( \sigma)f_1( \sigma) \bigr] 
            -
            \mathbb{E}_{N,\infty,\kappa} \bigl[ f_0( \hat \sigma)  \bigr] 
            \mathbb{E}_{N,\infty,\kappa}\bigl[ 
            f_1 ( \hat \sigma')  \bigr]
            \Bigr|
            \\&\qquad\leq
            2\| f_0\|_\infty \|f_1 \|_\infty\mu_{N,\infty,\kappa} \times \mu_{N,\infty,\kappa}
            \bigl( E_1 \cap E_{E_0,\hat \sigma,\hat \sigma'} \neq \emptyset \bigr)
            \\&\qquad\leq
            2\| f_0\|_\infty \|f_1 \|_\infty \, \sum_{e \in E_1} \mu_{N,\infty,\kappa} \times \mu_{N,\infty,\kappa}
            \bigl( e \in  E_{E_0,\hat \sigma,\hat \sigma'} \neq \emptyset \bigr).
        \end{split}
    \end{equation*}
    To obtain the desired conclusion, we note that by Lemma~\ref{lemma: E in coupling}, we have \( E_{E_0,\hat \sigma,\hat \sigma'} = E_{E_0,\sigma,\sigma'}. \) This concludes the proof.
\end{proof}

\subsection{A coupling between the Abelian Higgs model and the \texorpdfstring{\(\mathbb{Z}_n\)}{Zn}-model}\label{sec: LGT-Z coupling}

In this section, we recall the coupling between \( \mu_{N,\beta,\kappa} \) and \( \mu_{N,\infty,\kappa} \) introduced in~\cite{flv2021}. 

\begin{definition}[The coupling to the \( \mathbb{Z}_n \) model]\label{def: the coupling}
    For \( \beta,\kappa \geq 0, \) \( \sigma \in \Omega^1(B_N,G), \) and \( \sigma' \in \Omega^1_0(B_N,G),\) let
    let  
    \begin{equation}\label{eq: Esigmasigmadef}
        E_{\sigma, \sigma'} \coloneqq E_{\emptyset,\sigma, \sigma'}= \mathcal{C}_{\mathcal{G}(\sigma, \sigma')}\bigl( \{ e \in \support  \sigma \colon d \sigma|_{\pm \support \hat \partial e} \neq 0 \} \bigr).
    \end{equation}
    and define
    \begin{align*}
        &\mu_{N, (\beta,\kappa),(\infty,\kappa)}(\sigma, \sigma') 
        \\&\qquad\coloneqq \mu_{N, \beta,\kappa} \times \mu_{N,\infty, \kappa}\bigl(\bigl\{(\hat{\sigma}, \hat{\sigma}')\in \Omega^1(B_N,G) \times \Omega^1_0(B_N,G) \colon 
        \\&\hspace{16em} \sigma = \hat{\sigma}|_{E_{\hat \sigma,\hat \sigma'}} + \hat \sigma'|_{C_1(B_N) \smallsetminus E_{\hat \sigma,\hat \sigma'}} \text{ and } \sigma'=\hat{\sigma}' \bigr\} \bigr).
    \end{align*}
    We let \( \mathbb{E}_{N,(\beta,\kappa),(\infty,\kappa)} \) denote the corresponding expectation. 
\end{definition}
 
\begin{figure}[htp]
    \centering
    \begin{subfigure}[t]{0.32\textwidth}\centering
        \includegraphics[width=\textwidth]{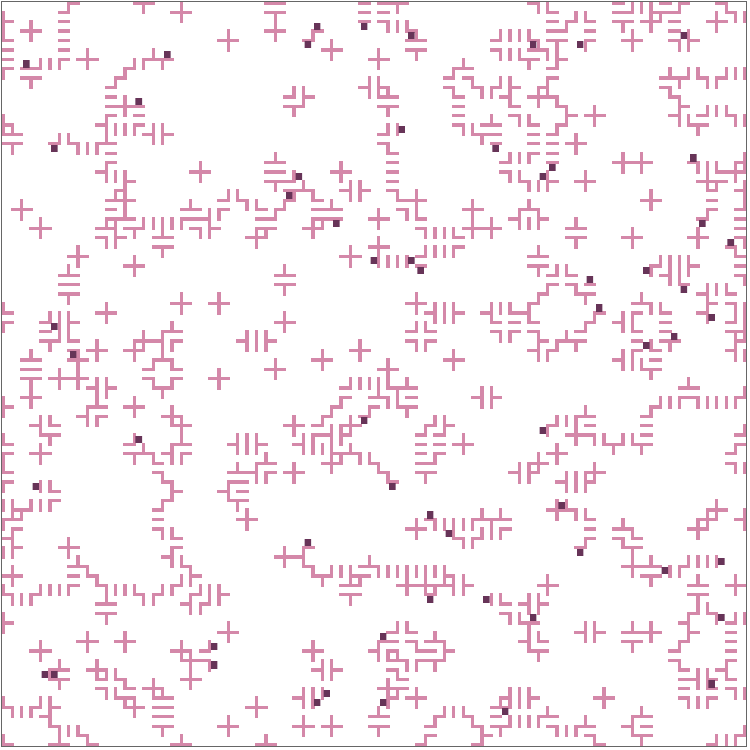}
        \caption{Red edges correspond to the support of \( \hat \sigma \in \Omega^1(B_N,\mathbb{Z}_2),\) and black squares to the support of \( d\hat \sigma \).}
    \end{subfigure}
    \hfil
    \begin{subfigure}[t]{0.32\textwidth}\centering
        \includegraphics[width=\textwidth]{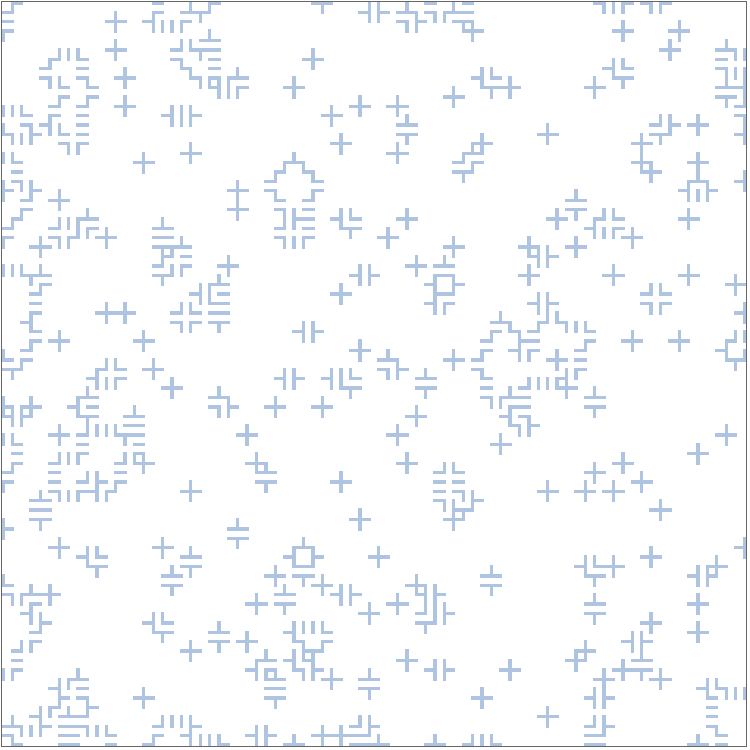}
        \caption{Blue edges correspond to the support of \( \hat \sigma' \in \Omega^1(B_N,\mathbb{Z}_2)\). In this case we automatically have \( d\hat \sigma' = 0\).}
    \end{subfigure}
    \hfil
    \begin{subfigure}[t]{0.32\textwidth}\centering
        \includegraphics[width=\textwidth]{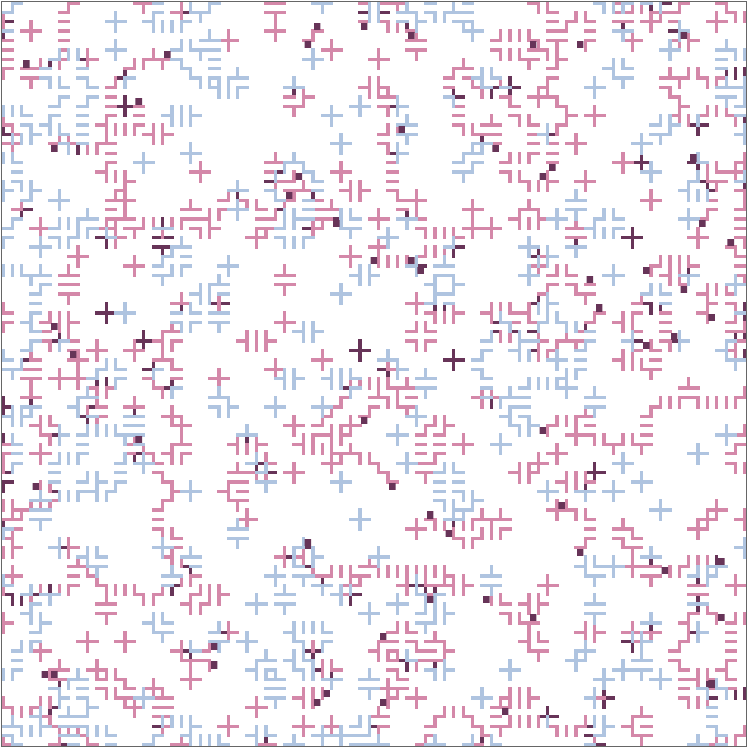}
        \caption{Red edges correspond to the support of \( \hat \sigma \), blue edges correspond to the support of \( \hat \sigma'\), and black squares to the support of \( d\hat \sigma \).}
    \end{subfigure}

    \vspace{1ex}
    \begin{subfigure}[t]{0.32\textwidth}\centering
        \includegraphics[width=\textwidth]{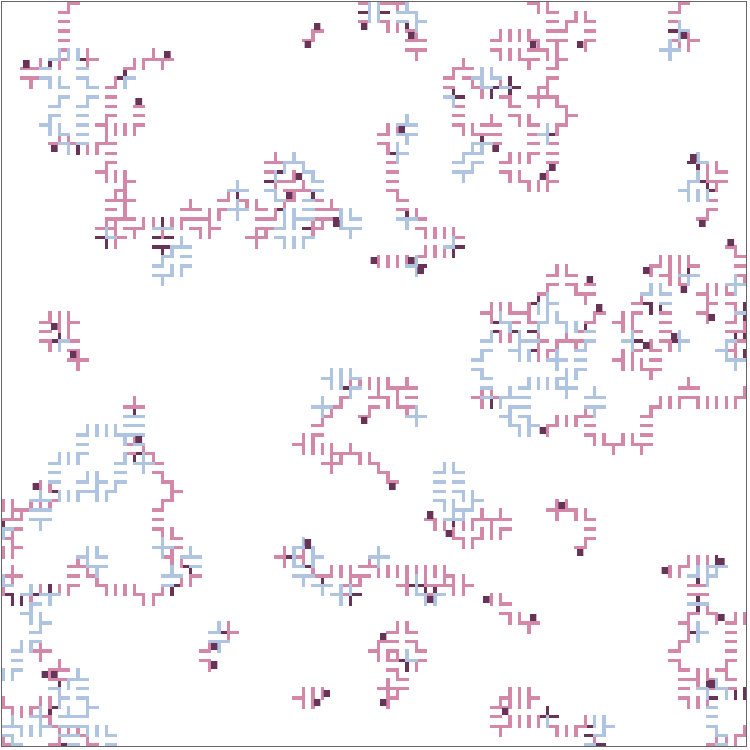}
        \caption{Red edges correspond to the support of  \(  \hat \sigma|_{E_{\hat \sigma,\hat \sigma'}}  \),  blue edges correspond to the support of \(  \hat \sigma'|_{E_{\hat \sigma,\hat \sigma'}} \), and  black squares correspond to the support of \( d\hat \sigma  \).}\label{subfig: the coupling d}
    \end{subfigure}
    \hfil
    \begin{subfigure}[t]{0.32\textwidth}\centering
        \includegraphics[width=\textwidth]{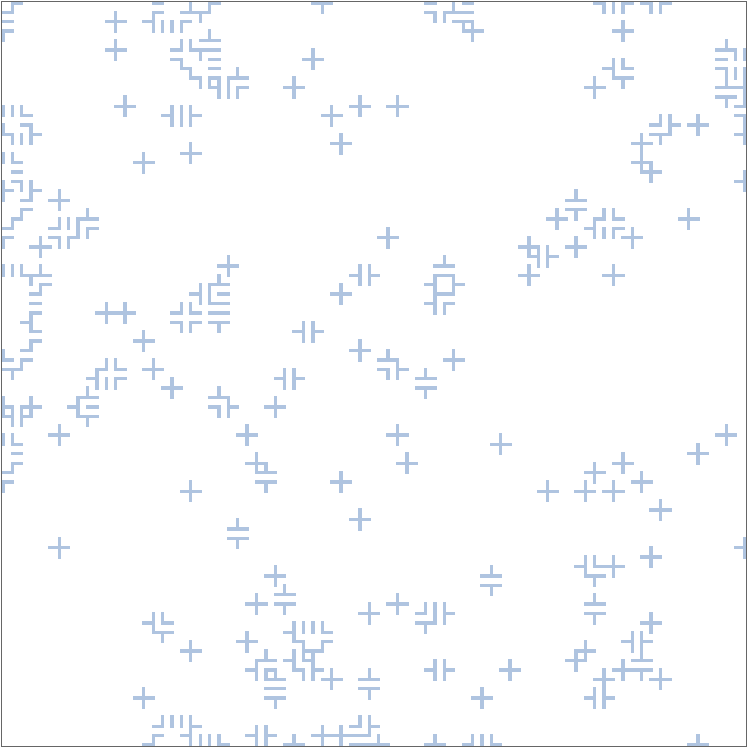}
        \caption{Blue edges correspond to the support of  \(  \hat \sigma'|_{C_1(B_N)\smallsetminus E_{\hat \sigma,\hat \sigma'}}  \).}\label{subfig: the coupling e}
    \end{subfigure}
    \hfil
    \begin{subfigure}[t]{0.32\textwidth}\centering
        \includegraphics[width=\textwidth]{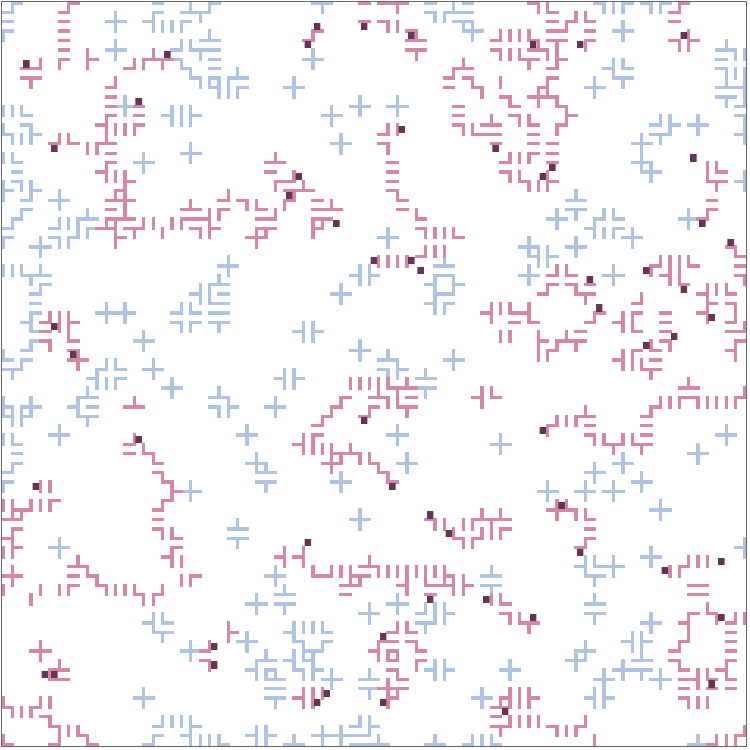}
        \caption{Red edges correspond to the support of \( \hat \sigma|_{E_{\hat \sigma, \hat \sigma'}}\), blue edges correspond to the support of \( \hat \sigma'|_{C_1(B_N)\smallsetminus E_{\hat \sigma, \hat \sigma'}}\), and  black squares correspond to the support of \( d\hat \sigma \).}
    \end{subfigure}
    \caption{Illustration of the coupling \( (\sigma, \sigma') \sim \mu_{N,(\beta,\kappa),(\infty,\kappa)} \) defined in Definition~\ref{def: the coupling} (simulated on a 2-dimensional lattice, with \( G = \mathbb{Z}_2 \)). }
    \label{fig: the coupling}
\end{figure}

\begin{remark}
    If~\cite{flv2021}, the measure \( \mu_{N,(\beta,\kappa),(\infty,\kappa)} \) in Definition~\ref{def: the coupling} above was defined slightly differently, but using the Proposition~\ref{proposition: coupling is a coupling} below, one easily shows that they are equivalent.
\end{remark}

The next result shows that this is indeed a coupling.
\begin{proposition}\label{proposition: coupling is a coupling}
    Let \( \beta,\kappa \geq 0 \). Then \( \mu_{N,(\beta,\kappa),{(\infty,\kappa})} \) is a coupling of \( \mu_{N,\beta,\kappa} \) and \( \mu_{N,\infty,\kappa} \). 
\end{proposition}

\begin{proof}
    It is immediate from the definition that if \( (\sigma,\sigma') \sim \mu_{N,(\beta,\kappa),(\infty,\kappa)}\), then \( \sigma' \sim\mu_{N,\infty,\kappa} \), and it is hence sufficient to show that \( \sigma \sim \mu_{N,\beta,\kappa} \).
	This is exactly equivalent to, for each \(  \sigma \in \Omega^1(B_N,G) \), showing that
	\begin{equation*}
	    \mu_{N,\beta,\kappa}\times \mu_{N,\infty,\kappa} \bigl( \bigl\{ (\hat \sigma, \hat \sigma') \in \Omega^1(B_N) \times \Omega^1_0(B_N) \colon  \hat \sigma|_{E_{ \hat \sigma, \hat \sigma'}} + \hat \sigma'|_{C_1(B_N)\smallsetminus E_{ \hat \sigma, \hat \sigma'}} = \sigma \bigr\} \bigr)
		=
		\mu_{N,\beta,\kappa}( \sigma).
	\end{equation*}
	or, equivalently, that
	\begin{equation}\label{eq: LGT-Z coupling goal eq}
		\sum_{\substack{\hat \sigma \in \Omega^1(B_N,G),\\ \hat \sigma' \in \Omega^1_0(B_N,G)}}
		\varphi_{\beta,\kappa}(\hat \sigma) \varphi_{\kappa}(\hat \sigma') \cdot \mathbb{1} \bigl( \hat \sigma|_{E_{\hat \sigma,\hat \sigma'}} + \hat \sigma'|_{C_1(B_N)\smallsetminus E_{\hat \sigma,\hat \sigma'}} =  \sigma \bigr)
		=
		\varphi_{\beta,\kappa}( \sigma) \sum_{ \sigma' \in \Omega^1_0(B_N,G)} \varphi_{\kappa} ( \sigma').
	\end{equation}
	
	We now show that~\eqref{eq: LGT-Z coupling goal eq} holds. To this end, fix some \(  \sigma \in \Omega^1(B_N,G) \). 
	By Lemma~\ref{lemma: technical coupling lemma}, applied with \( \beta_1= \beta \) and \(\beta_2=\infty \), we have
    \begin{equation}\label{eq: LGT-Z coupling goal eq ii} 
		\begin{split}
			&\sum_{\substack{\hat \sigma \in \Omega^1(B_N,G),\\ \hat \sigma' \in \Omega^1_0(B_N,G)}}
			\varphi_{\beta,\kappa}(\sigma) \varphi_{\kappa}(\sigma')\mathbb{1} \bigl(  \hat \sigma|_{E_{\hat \sigma,\hat \sigma'}} +  \hat \sigma'|_{C_1(B_N)\smallsetminus E_{ \hat\sigma, \hat\sigma'}} =  \sigma \bigr) 
			\\&\qquad =
			\varphi_{\kappa}( \sigma )
			\sum_{\sigma' \in \Omega^1(B_N,G)}
			\varphi_{\beta,\kappa}(\sigma')
			\sum_{\substack{\hat \sigma \in \Omega^1(B_N,G),\\ \hat \sigma' \in \Omega^1_0(B_N,G)}}
			\mathbb{1}\bigl(  \hat \sigma =  \sigma|_{E_{ \sigma, \sigma'}}+ \sigma'|_{C_1(B_N)\smallsetminus E_{ \sigma, \sigma'}} \bigr)
			\\[-3.5ex]&\hspace{22em}\cdot 
			\mathbb{1}\bigl( \hat \sigma' = \sigma' |_{E_{ \sigma, \sigma'}} +  \sigma |_{C_1(B_N)\smallsetminus E_{ \sigma, \sigma'}}  \bigr)
			\\&\qquad =
			\varphi_{\kappa}( \sigma )
			\sum_{\sigma' \in \Omega^1(B_N,G)}
			\varphi_{\beta,\kappa}(\sigma')
			\mathbb{1}\bigl( \sigma' |_{E_{ \sigma, \sigma'}} +  \sigma |_{C_1(B_N)\smallsetminus E_{ \sigma, \sigma'}} \in \Omega_0^1(B_N,G) \bigr)
			\\&\qquad =
			\varphi_{\kappa}( \sigma )
			\sum_{\sigma' \in \Omega^1(B_N,G)}
			\varphi_{\beta,\kappa}(\sigma')
			\mathbb{1}\Bigl( d\bigl( \sigma' |_{E_{ \sigma, \sigma'}} +  \sigma |_{C_1(B_N)\smallsetminus E_{ \sigma, \sigma'}}\bigr) =0  \Bigr).
		\end{split}
	\end{equation}
	By Lemma~\ref{lemma: closed is closed new}, we have 
	\begin{equation*}
		d\bigl( \sigma' |_{E_{ \sigma, \sigma'}} +  \sigma |_{C_1(B_N)\smallsetminus E_{ \sigma, \sigma'}} \bigr) 
		=
		d\bigl( \sigma' |_{E_{ \sigma, \sigma'}} \bigr)+  d\bigl(\sigma |_{C_1(B_N)\smallsetminus E_{ \sigma, \sigma'}} \bigr) 
		= d\bigl( \sigma' |_{E_{ \sigma, \sigma'}} \bigr) +0 = d\sigma',
	\end{equation*}
	and hence
	\begin{equation}\label{eq: LGT-Z coupling goal eq iii}
		\begin{split} 
		&
		\varphi_{\kappa}( \sigma )
		\sum_{\sigma' \in \Omega^1(B_N,G)}
		\varphi_{\beta,\kappa}(\sigma')
		\mathbb{1}\Bigl( d\bigl( \sigma' |_{E_{ \sigma, \sigma'}} +  \sigma |_{C_1(B_N)\smallsetminus E_{ \sigma, \sigma'}}\bigr) =0  \Bigr)
		\\&\qquad=
		\varphi_{\kappa}( \sigma )
		\sum_{\sigma' \in \Omega^1(B_N,G)}
		\varphi_{\beta,\kappa}(\sigma')
		\mathbb{1}( d\sigma' =0  )
		=
		\varphi_{\kappa}( \sigma )
		\sum_{\sigma' \in \Omega^1_0(B_N,G)}
		\varphi_{\beta,\kappa}(\sigma')  .
		\end{split}	
	\end{equation}
	Combining~\eqref{eq: LGT-Z coupling goal eq ii}  and~\eqref{eq: LGT-Z coupling goal eq iii} , we obtain~\eqref{eq: LGT-Z coupling goal eq} as desired. This concludes the proof.
\end{proof}

\begin{lemma}\label{lemma: properties of coupling measure} 
    Let \( \beta,\kappa \geq 0,\) let \( E_0 \subseteq C_1(B_N),\) and let \( (\sigma,\sigma')\in \Omega^1(B_N,G) \times \Omega^1_0(B_N,G) \) be such that  \( \mu_{N,(\beta,\kappa),(\infty,\kappa)}(\sigma , \sigma') \neq 0\). Then \( \sigma(e) = \sigma'(e) \) for all \( e \in C_1(B_N) \smallsetminus E_{\sigma,\sigma'} \). 
\end{lemma}

\begin{proof}
    Since \(\mu_{N,(\beta,\kappa),(\infty,\kappa)}(\sigma , \sigma') \neq 0\), by definition, there is \( (\hat \sigma, \hat \sigma') \in \Omega^1(B_N,G) \times \Omega^1_0(B_N,G) \) such that \( \sigma = \hat \sigma|_{E_{\hat \sigma,\hat \sigma'}} + \hat{\sigma}'|_{C_1(B_N) \smallsetminus E_{\hat{\sigma}, \hat{\sigma}'}} \) and \( \sigma' = \hat \sigma'\). Using Lemma~\ref{lemma: E in coupling}, we immediately obtain the desired conclusion.
\end{proof}

\section{Distribution of vortices and edge configurations}\label{sec: general cluster events}

In this section, we use the edge graph defined in Section~\ref{sec: edge graph} to give upper bounds on several useful events. Throughout this section, constants \( K_1,K_2,\dots, K_{15} \) will be introduced. We use distinct names for these to make it possible to find explicit upper bounds, but stress that under the assumptions of the main results these are all bounded from above, and will thus not affect the decay rate of the upper bounds obtained throughout this section.

For \( E_0 \subseteq C_1(B_N) \) and \( e \in C_1(B_N) \), we define
\begin{equation*}
    \begin{split}
        &\dist_1 (e,E_0) \coloneqq \frac{1}{2}\min \Bigl\{ \bigl|\mathcal{C}_{\mathcal{G}_{\hat \sigma,\hat \sigma'}}(e) \bigr| \colon \sigma,\sigma' \in \Omega^1(B_N,G),\,  \mathcal{C}_{\mathcal{G}_{\hat \sigma,\hat \sigma'}}(e)\cap E_0 \neq \emptyset  \Bigr\}
        \\&\qquad= 
        \frac{1}{2}\min \Bigl\{ \bigl|\mathcal{C}_{\mathcal{G}_{\hat \sigma}}(e) \bigr| \colon \sigma\in \Omega^1(B_N,G),\,  \mathcal{C}_{\mathcal{G}_{\hat \sigma}}(e)\cap E_0 \neq \emptyset  \Bigr\} ,\quad e \in C_1(B_N) ,
    \end{split}
\end{equation*}  
and
\begin{equation*}
    \dist_0 (e,E_0) \coloneqq \frac{1}{2}\min \Bigl\{ \bigl|\mathcal{C}_{\mathcal{G}_{\hat \sigma,\hat \sigma'}}(e) \bigr| \colon \sigma,\sigma' \in \Omega^1_0(B_N,G),\,  \mathcal{C}_{\mathcal{G}_{\hat \sigma,\hat \sigma'}}(e)\cap E_0 \neq \emptyset  \Bigr\} ,\quad e \in C_1(B_N).
\end{equation*}  
Note that, by Lemma~\ref{lemma: small 1forms}, if \( e \notin E, \) then \( \dist_0(e,E) \geq 8. \) 
We extend this definition to sets \( {E \subseteq C_1(B_N)} \) by letting \( \dist_1(E,E_0) \coloneqq \min_{e \in E} \dist_1(e,E_0) \) and \( \dist_0(E,E_0) \coloneqq \min_{e \in E} \dist_0(e,E_0). \) 

In this section, we will state and prove the following three propositions.

\begin{proposition}\label{proposition: new Z-LGT coupling upper bound}
    Let \( \beta,\kappa_1,\kappa_2 \in [0,\infty] \) be such that~\( 18^2\bigl( \alpha_0(\kappa_1) + \alpha_0(\kappa_2) + \alpha_0(\kappa_1)\alpha_0(\kappa_2)\bigr)<1 \), let \( e \in C_1(B_N) \) be such that \( \dist_0\bigl(e,\partial C_1(B_N)\bigr) \geq 8, \) and let \( M \geq 1 \) and \( M' \geq 0 .\)
    Then
    \begin{equation}\label{eq: coupling and conditions 3} 
        \begin{split}
            &
            \mu_{N,\beta,\kappa_1} \times \mu_{N,\infty,\kappa_2} \Bigl( \pigl\{ (\hat \sigma,\hat \sigma') \in \Omega^1(B_N,G) \times \Omega^1_0(B_N,G)\colon  
            \\[-1.5ex]&\hspace{13em}
            |\mathcal{C}_{\mathcal{G}(\hat \sigma,\hat \sigma')}(e)| \geq 2M, \text{ and } \bigl|\support d\bigl( \hat \sigma|_{ \mathcal{C}_{\mathcal{G}(\hat \sigma,\hat \sigma')}(e)} \bigr)\bigl|\geq 2 M' \pigr\}\Bigr)
            \\&\qquad\leq 
            \mathbb{1}_{M'>0} \cdot   \mathbb{1}_{M=1}  \cdot 
            \pigl(18^{2} \bigl(\alpha_0(\kappa_1) + \alpha_0(\kappa_2) + \alpha_0(\kappa_1)\alpha_0(\kappa_2) \bigr)\pigr)  \alpha_1(\beta)^{\max(6,M')}
            \\&\qquad\qquad+
            \mathbb{1}_{M'>0} \cdot K_1\pigl(18^{2} \bigl(\alpha_0(\kappa_1) + \alpha_0(\kappa_2) + \alpha_0(\kappa_1)\alpha_0(\kappa_2) \bigr)\pigr)^{\max(M,2)}   \alpha_1(\beta)^{\max(6,M')}
            \\&\qquad\qquad+
            \mathbb{1}_{M'\in \{ 1,2,3,4,5\}} \cdot K_1
            \pigl(18^{2} \bigl(\alpha_0(\kappa_1) + \alpha_0(\kappa_2) + \alpha_0(\kappa_1)\alpha_0(\kappa_2) \bigr)\pigr) ^{\dist_1(e,\partial C_1(B_N))} \alpha_1(\beta)
            \\&\qquad\qquad+ \mathbb{1}_{M'=0} \cdot K_1
            \pigl(18^{2} \bigl(\alpha_0(\kappa_1) + \alpha_0(\kappa_2) + \alpha_0(\kappa_1)\alpha_0(\kappa_2) \bigr)\pigr) ^{\max(M,8)},
        \end{split}
    \end{equation} 
    where 
    \begin{equation}\label{eq: K1}
        K_1 = K_1(\kappa_1,\kappa_2)\coloneqq 18^{-3}\pigl( 1 - 18^{2} \bigl(\alpha_0(\kappa_1) + \alpha_0(\kappa_2) + \alpha_0(\kappa_1)\alpha_0(\kappa_2) \bigr)\pigr)^{-1}.
    \end{equation}
\end{proposition}

\begin{remark}
    If \( \kappa_1 = \kappa_2 \eqqcolon \kappa \), then the assumption on \( \kappa_1 \) and \( \kappa_2\) in Proposition~\eqref{proposition: new Z-LGT coupling upper bound} is equivalent to~\ref{assumption: 3}.
\end{remark}

\begin{proposition}\label{proposition: alternative plaquette bound}
    Let \( \beta,\kappa \geq 0 \) be such that~\ref{assumption: 3} holds, and assume that \( p \in C_2(B_N) \) is such that \( \dist_0(\support \partial e,\partial C_1(B_N))) \geq 8. \) Then 
    \begin{equation}\label{eq: coupling and conditions 4}
        \begin{split}
            &
            \mu_{N,\beta,\kappa} \Bigl( \pigl\{ \hat \sigma \in \Omega^1(B_N,G)  \colon  
             d  \hat \sigma(p) \neq 0 \pigr\}\Bigr)
            \leq  
            K_2 \alpha_2(\beta,\kappa),
        \end{split}
    \end{equation}
    where 
    \begin{equation}\label{eq: K2}
        K_2 \coloneqq 4  
        \bigl( 18^{2}   
            +
            18 \alpha_0(\kappa) \bigl(1-18^2 \alpha_0(\kappa)\bigr)^{-1}     \bigr) \frac{\alpha_1(\beta)^6}{\alpha_0(\beta)^6}. 
    \end{equation}
\end{proposition}

\begin{proposition}\label{proposition: ZZ upper bound}
    Let \( \kappa \geq 0 \) be such that~\ref{assumption: 3} holds, let \( E_0 \subseteq C_1(B_N) \) be non-empty, and let \( e \in C_1(B_N) \) be such that \( \dist_0(e,\partial C_1(B_N))) \geq 8. \)
    Then
    \begin{equation}\label{eq: Z-Z coupling inequality}
        \mu^{E_0}_{N,(\infty,\kappa),(\infty,\kappa)} \bigl( \bigr\{ (\sigma,\sigma') \in \Omega^1_0(B_N,G) \times \Omega^1_0(B_N,G) \colon e \in E_{E_0,\sigma, \sigma'}  \bigr\} \bigr) \leq 
        K_3 \Bigl( 
        K_4 \alpha_0(\kappa) \Bigr)^{\dist_0(e,E_0)}
    \end{equation}
    where
    \begin{equation}\label{eq: K3 and K4}
        K_3 \coloneqq 18^{-3}(1 - 18^2  \bigl(2 + \alpha_0(\kappa) \bigr)\alpha_0(\kappa))^{-1},
        \quad \text{and} \quad
        K_4 \coloneqq 18^{2} \bigl(2  + \alpha_0(\kappa) \bigr).
    \end{equation} 
\end{proposition}

\begin{proposition}\label{proposition: before E3}
    Let \( \beta,\kappa \geq 0 \) be such that~\ref{assumption: 3} hold, let \( e \in C_1(B_N) \) be such that for all \( p \in \hat \partial e\) we have \( {\dist_0(\support \partial p,\partial C_1(B_N))\geq 8},\) and let \( E_0 \coloneqq \{ e' \in C_1(B_N) \colon \hat \partial e'  \cap \hat \partial e \neq \emptyset\}. \) Then
    \begin{equation*}\label{eq: coupling and conditions 5}
        \begin{split}
            &\mu_{N,\beta,\kappa}   \Bigl( \pigl\{ \sigma \in \Omega^1(B_N,G)  \colon   
            |\mathcal{C}_{\mathcal{G}(\sigma)}(E_0)| \geq 2M \text{ and }
            \bigl| \support d( \sigma|_{ \mathcal{C}_{\mathcal{G}(\sigma)}(E_0)} ) \bigr| \geq 2M' \pigr\}\Bigr)
            \\&\qquad\leq 
            K_5 \bigl(18^{2} \alpha_0(\kappa)\bigr)^M  \alpha_1(\beta)^{M'}.
        \end{split}
    \end{equation*}
    where
    \begin{equation}\label{eq: K5}
        K_5 \coloneqq  (18^{-3} + 18^{-1}) (1-18^2 \alpha_0(\kappa))^{-1}.
    \end{equation}
\end{proposition}

Before we give proofs of the above propositions, we introduce some additional notation and prove two useful lemmas.
To this end, we first define a graph \( \bar {\mathcal{G}} \) as follows. Fix some \( g \in G\smallsetminus \{ 0 \} \) and define \( \bar \sigma \in \Omega^1(B_N,G) \) by letting \( \sigma(e) = g \) for all \( e \in C_1(B_N)^+ \). Let \( \bar {\mathcal{G}} \coloneqq \mathcal{G}(\bar \sigma,0) \) and note that \( \bar {\mathcal{G}} \) does not depend on the choice of \( g \). 
Note also that if \( \sigma,\sigma' \in \Omega^1(B_N,G) \), then \( \mathcal{G}(\sigma, \sigma') \) is a subgraph of \( \bar {\mathcal{G}} \).

\begin{lemma}[See also Lemma~7.15 and~Lemma~7.16 in \cite{flv2021}]\label{lemma: from walks to sets}
   Let \( e \in C_1^+(B_N) \), and let \( m \geq 1. \) Then
   \begin{equation*}
       \pigl| \bigl\{ E \subseteq C_1^+(B_N) \colon e \in E,\, |E| = m,\, \text{and } \bar {\mathcal{G}}|_{E} \text{ is connected} \bigr\} \bigr|\leq 18^{\max(0,2m-3)}.
   \end{equation*} 
\end{lemma}
 
\begin{proof}
    Since the case \( m = 1 \) is trivial, we can assume that \( m \geq 2. \)

    Fix some set  \( E \subseteq C_1^+(B_N) \) such that \( e \in E, \) \( |E| = m,\) and \( \bar {\mathcal{G}}|_{E} \) is connected.
    Since the graph \( \bar {\mathcal{G}}|_{E}  \) is connected, it has a spanning tree. Let \( \mathcal{T} \) be such a spanning tree.
    By definition, \( \mathcal{T} \) must contain exactly \( m-1 \) edges. Since any spanning tree is connected, \( \mathcal{T} \) must have a spanning walk which uses each edge in \( \mathcal{T} \) exactly twice, and starts and ends at the same vertex. This walk must have length \( 2(m-1) = 2m-2 \).  By removing one of the edges adjacent to the vertex \( e \), we obtain a spanning walk of \( \bar {\mathcal{G}}|_{E}  \) which has length \( 2m-3 \), starts at the vertex \( e \) and visits every vertex of \( \bar {\mathcal{G}}|_{E}  \) at least once.
    
    Since for each \( e' \in E_N \), we have \( \bigl|\{ e'' \in C_1(B_N)\smallsetminus \{ e' \} \colon \hat \partial e'' \cap \hat \partial e' \neq \emptyset \}\bigr|= 6 \cdot 3 = 18,\) there can exists at most \( 18^{2m-3} \) walks in \( \bar{\mathcal{G}} \) which starts at \( e \) and has length \( 2m-3, \) and hence the desired conclusion follows.
\end{proof}

\begin{lemma}\label{lemma: Jke bound}
    Let \( \kappa_1,\kappa_2 \geq 0 \), and let \( E \subseteq C_1^+(B_N). \) Then
    \begin{equation}\label{eq: Jke}
         \sum_{\substack{ \hat{\hat\sigma} \in \Omega^1(B_N,G),\,  \hat{\hat\sigma}' \in \Omega^1(B_N,G) \mathrlap{\colon} \\
        (\support \hat{\hat\sigma} \cup \support \hat{\hat\sigma}')^+ = E}}
        \varphi_{\kappa_1}\bigl(\hat{\hat\sigma}\bigr) \varphi_{\kappa_2}\bigl(\hat{\hat\sigma}' \bigr)
        \leq \bigl( \alpha_0(\kappa_1) + \alpha_0(\kappa_2)
        + \alpha_0(\kappa_1)\alpha_0(\kappa_2)\bigr)^{|E|}.
    \end{equation} 
\end{lemma}

\begin{proof}
    If \( \hat{\hat\sigma} \in \Omega^1(B_N,G) \) and \( e' \notin \support \hat{\hat\sigma} \), then, for any \( \kappa \geq 0 \), we have \( \varphi_\kappa\bigl(\hat{\hat\sigma}(e')\bigr) = \varphi_\kappa(0)=1 \). Also, if \( \hat{\hat\sigma}, \hat{\hat\sigma}' \in \Omega^1(B_N,G) \) and \( e' \in (\support \hat{\hat\sigma} \cup \support \hat{\hat\sigma}')^+ = E  \), then either \( \hat{\hat\sigma}(e') \neq 0 \) and \( \hat{\hat\sigma}'(e') = 0 \), \( \hat{\hat\sigma}(e')=0 \) and \(\hat{\hat\sigma}'(e') \neq 0 \), or \( \hat{\hat\sigma}(e'),\hat{\hat\sigma}'(e') \neq 0 \). %
    Combining these observations, we find that
    \begin{equation*} 
    \begin{split}
        &\sum_{\substack{ \hat{\hat\sigma} \in \Omega^1(B_N,G),\,  \hat{\hat\sigma}' \in \Omega^1(B_N,G) \colon \\
            (\support \hat{\hat\sigma} \cup \support \hat{\hat\sigma}')^+ = E}}
        \prod_{e' \in C_1(B_N)^+}  \varphi_{\kappa_1}\bigl(\hat{\hat\sigma}(e')\bigr)^2 \prod_{e'' \in C_1(B_N)^+}
        \varphi_{\kappa_2}\bigl(\hat{\hat\sigma}'(e'')\bigr)
        \\&\qquad\leq  
        \prod_{e' \in E}
        \biggl\{ 
        \varphi_{\kappa_1}(0)^2
        \bigg(\sum_{\hat{\hat\sigma}'(e') \in G \smallsetminus \{0\}} \varphi_{\kappa_2}(\hat{\hat\sigma}'(e'))^2\bigg)
        + 
        \biggl(\sum_{\hat{\hat\sigma}(e') \in G \smallsetminus \{0\}}
        \varphi_{\kappa_1}\bigl(\hat{\hat\sigma}(e')\bigr)^2 \biggr)
        \varphi_{\kappa_2}(0)^2
        \\\nonumber
        &\qquad\qquad\qquad +\bigg(\sum_{\hat{\hat\sigma}(e') \in G \smallsetminus \{0\}} 
        \varphi_{\kappa_1}\bigl(\hat{\hat\sigma}(e')\bigr)^2 \bigg)
        \bigg(\sum_{\hat{\hat\sigma}'(e') \in G \smallsetminus \{0\}} \varphi_{\kappa_2}\bigl(\hat{\hat\sigma}'(e')\bigr)^2\bigg)
        \biggr\}
            \\ 
        & \qquad = \prod_{e' \in E}
        \bigl(\alpha_0(\kappa_1) + \alpha_0(\kappa_2)
        + \alpha_0(\kappa_1)\alpha_0(\kappa_2)\bigr) 
        =   \bigl(\alpha_0(\kappa_1) + \alpha_0(\kappa_2)
        + \alpha_0(\kappa_1)\alpha_0(\kappa_2)\bigr)^{|E|}. 
    \end{split}
    \end{equation*}     
    This concludes the proof.
\end{proof}

\begin{proof}[Proof of Proposition~\ref{proposition: new Z-LGT coupling upper bound}]

Since \( \mathcal{C}_{\mathcal{G}(\hat \sigma,\hat \sigma')}(e) \) is symmetric, induces a connected subgraph in \( \bar{\mathcal{G}} \), and contains \( e \) if it non-empty, we have
    \begin{align}
            &\mu_{N,\beta,\kappa_1} \times \mu_{N,\infty,\kappa_2} \Bigl( \pigl\{ (\hat \sigma,\hat \sigma') \in \Omega^1(B_N,G) \times \Omega^1_0(B_N,G)\colon  \nonumber
            \\[-1ex]&\hspace{13em}
            |\mathcal{C}_{\mathcal{G}(\hat \sigma,\hat \sigma')}(e)| \geq 2M, \text{ and }\bigl| \support d\bigl( \hat \sigma|_{ \mathcal{C}_{\mathcal{G}(\hat \sigma,\hat \sigma')}(e)} \bigr) \bigr| \geq 2M'  \pigr\}\Bigr)\nonumber
            \\&\qquad=          
            \sum_{\substack{E \subseteq C_1^+(B_N) \mathrlap{\colon} \\ \substack{e \in E,\, |E| \geq M,\\ \bar{\mathcal{G}|}_E \text{ is connected} }}}
            \mu_{N,\beta,\kappa_1} \times \mu_{N,\infty,\kappa_2} \Bigl( \pigl\{ (\hat \sigma,\hat \sigma') \in \Omega^1(B_N,G) \times \Omega^1_0(B_N,G)\colon  \label{eq: second line}
            \\[-6.75ex]&\hspace{13em}
            \mathcal{C}_{\mathcal{G}(\hat \sigma,\hat \sigma')}(e)^+ = E, \text{ and }\bigl| \support d\bigl( \hat \sigma|_{ \mathcal{C}_{\mathcal{G}(\hat \sigma,\hat \sigma')}(e)} \bigr) \bigr| \geq 2M'  \pigr\}\Bigr) \phantom{\mathrlap{\sum_{\substack{E \subseteq C_1^+(B_N) \mathrlap{\colon} \\ \substack{e \in E,\, |E| \geq M,\\ E \text{ connected in } \bar{\mathcal{G}}}}}}} \nonumber
    \end{align}
        Given \( (\hat \sigma,\hat \sigma') \in \Omega^1(B_N,G) \times \Omega^1_0(B_N,G) \), if we let \( \hat{\hat\sigma} \coloneqq  \hat \sigma|_{\mathcal{C}_{\mathcal{G}(\hat \sigma,\hat \sigma')}(e)} \) and \( \hat{\hat \sigma}' \coloneqq  \hat \sigma'|_{\mathcal{C}_{\mathcal{G}(\hat\sigma,\hat \sigma')}(e)} \), then the following statements hold.
    \begin{enumerate}
    
        \item By Lemma~\ref{lemma: closed is closed new}, we have  \( \hat{\hat\sigma}' \in \Omega^1_0(B_N,G) \). 
        
        \item If \( \bigl| \support d(\hat \sigma|_{\mathcal{C}_{\mathcal{G}(\hat \sigma, \hat \sigma')}(e)})\bigr| \geq 2M' \), then,  by definition, \( \bigl| \support d\hat{\hat \sigma} \bigr| \geq 2M'. \)
        
        \item If \( |\mathcal{C}_{\mathcal{G}(\hat \sigma,\hat \sigma')}(e) |\geq 2 \), then we have \( e \in \support \hat \sigma \cup \support \hat \sigma',\) and thus
        \( (\mathcal{C}_{\mathcal{G}(\hat \sigma,\hat \sigma')}(e))^+ = (\support \hat{\hat\sigma} \cup \support \hat{\hat \sigma}')^+\). Consequently, \( (\mathcal{C}_{\mathcal{G}(\hat \sigma,\hat \sigma')}(e))^+ = E \) if and only if \(   (\support \hat{\hat \sigma} \cup \support \hat{\hat\sigma}')^+ = E. \)
    \end{enumerate}
     As a consequence, ~\eqref{eq: second line} can be bounded from above by
    \begin{equation*} 
        \begin{split}
            &
            \sum_{\substack{E \subseteq C_1^+(B_N) \mathrlap{\colon} \\ \substack{e \in E,\, |E| \geq M,\\ \bar{\mathcal{G}|}_E \text{ is connected} }}}\,
            \sum_{\substack{\hat{\hat\sigma} \in \Omega^1(B_N,G),\,  \hat{\hat\sigma}' \in \Omega^1_0(B_N,G) \colon \\
            \substack{(\support \hat{\hat\sigma} \cup \support \hat{\hat\sigma}')^+ = E,\\ |\support d \hat{\hat\sigma} |\geq 2M'} }} 
            \mu_{N,\beta,\kappa_1} \times \mu_{N,\infty,\kappa_2} \bigl( \{ (\hat \sigma,\hat \sigma') \in \Omega^1(B_N,G) \times \Omega^1_0(B_N,G) \colon
            \\[-8ex]&\hspace{24em}
            \hat \sigma|_{\mathcal{C}_{\mathcal{G}(\hat \sigma, \hat \sigma')}(e)} =  \hat{\hat\sigma}
            , \text{ and }
            \hat \sigma'|_{\mathcal{C}_{\mathcal{G}(\hat \sigma, \hat \sigma'}(e)} =  \hat{\hat\sigma}'\} \bigr). \\ &\mbox{}
        \end{split}
    \end{equation*}
    For any   \(  \hat{\hat\sigma} \), and \(  \hat{\hat\sigma}' \) as in the sum above, by Lemma~\ref{lemma: cluster is subconfig}, we have 
    \begin{equation*}
        \begin{split}
            &\mu_{N,\beta,\kappa_1} \times \mu_{N,\infty,\kappa_2} \bigl( \{ (\hat \sigma,\hat \sigma') \in \Omega^1(B_N,G) \times \Omega^1_0(B_N,G) \colon
            \hat \sigma|_{\mathcal{C}_{\mathcal{G}(\hat \sigma, \hat \sigma')}(e)} =  \hat{\hat\sigma} \text{ and } 
            \hat \sigma'|_{\mathcal{C}_{\mathcal{G}(\hat \sigma, \hat \sigma'}(e)} = \hat{\hat\sigma}'\} \bigr) 
            \\&\qquad\leq \mu_{N,\beta,\kappa_1} \times \mu_{N,\infty,\kappa_2} \bigl( \{ (\hat \sigma,\hat \sigma') \in \Omega^1(B_N,G) \times \Omega^1_0(B_N,G) \colon
            \hat{\hat\sigma} \leq \hat \sigma  \text{ and } 
            \hat{\hat\sigma}' \leq \hat \sigma' \} \bigr)
            \\&\qquad= \mu_{N,\beta,\kappa_1} \bigl( \{ \hat \sigma  \in \Omega^1(B_N,G)  \colon
            \hat{\hat\sigma} \leq \hat \sigma   \} \bigr)
            \mu_{N,\infty,\kappa_2} \bigl( \{ \hat \sigma' \in  \Omega^1_0(B_N,G) \colon
            \hat{\hat\sigma}' \leq \hat \sigma' \} \bigr) \leq \varphi_{\beta,\kappa_1}( \hat{\hat\sigma}) \varphi_{\infty,\kappa_2} (\hat{\hat\sigma}'),
        \end{split}
    \end{equation*}
    where the last inequality follows by applying Proposition~\ref{proposition: edgecluster flipping ii} twice.
    Taken together, the above equations thus show that
    \begin{equation}\label{eq: first summary B2}  
        \begin{split}
            &\mu_{N,\beta,\kappa_1} \times \mu_{N,\infty,\kappa_2} \Bigl( \pigl\{ (\hat \sigma,\hat \sigma') \in \Omega^1(B_N,G) \times \Omega^1_0(B_N,G)\colon  
            \\[-1.5ex]&\hspace{13em}
            |\mathcal{C}_{\mathcal{G}(\hat \sigma,\hat \sigma')}(e)| \geq 2M\text{ and }  \bigl| \support d\bigl( \hat \sigma|_{ \mathcal{C}_{\mathcal{G}(\hat \sigma,\hat \sigma')}(e)} \bigr)\bigr| \geq 2M'  \pigr\}\Bigr)
            \\&\qquad \leq
            \sum_{\substack{E \subseteq C_1^+(B_N) \mathrlap{\colon} \\ \substack{e \in E,\, |E| \geq M,\\ \bar{\mathcal{G}|}_E \text{ is connected} }}}
            J_{\beta, \kappa_1, \kappa_2}(E),
        \end{split}
    \end{equation} 
    where
    \begin{equation*}
        J_{\beta, \kappa_1,\kappa_2}(E) \coloneqq 
        \sum_{\substack{ \hat{\hat\sigma} \in \Omega^1(B_N,G),\,  \hat{\hat\sigma}' \in \Omega^1_0(B_N,G) \colon \\
        \substack{(\support \hat{\hat\sigma} \cup \support \hat{\hat\sigma}')^+ = E,\\ \bigl| \support d \sigma \bigr| \geq 2M' 
        }}} \varphi_{\beta,\kappa_1}  ( \hat{\hat\sigma} )\,  \varphi_{\infty,\kappa_2} (  \hat{\hat\sigma}'  ).
    \end{equation*}  

    Fix some set \( E \subseteq C_1^+(B_N).\) Then 
    \begin{equation*}
        \begin{split}
            &J_{\beta, \kappa_1,\kappa_2}(E) = \sum_{\substack{ \hat{\hat\sigma} \in \Omega^1(B_N,G),\,  \hat{\hat\sigma}' \in \Omega^1_0(B_N,G) \colon \\
            \substack{(\support \hat{\hat\sigma} \cup \support \hat{\hat\sigma}')^+ = E,\\ |\support d \hat{\hat\sigma} | \geq 2M'
             }}} \varphi_{\kappa_1}\bigl(\hat{\hat\sigma}\bigr) \varphi_{\kappa_2}\bigl(\hat{\hat\sigma}'\bigr) \prod_{p \in C_2(B_N)}\varphi_\beta \bigl(d\hat{\hat\sigma}(p) \bigr)  .
        \end{split}
    \end{equation*} 
    Now recall that for any \( r \geq 0 \) and \( g \in G \), we have  \( \varphi_r(0) = 1\) and \( \varphi_r(g) = e^{r \Re(\rho(g) - 1)} \in (0,1]\). If \( g \neq 0\), then \( \varphi_r(g) < 1\) and hence \(\varphi_\beta(g)^2 \leq \alpha_1(\beta) < 1\).
    If \( \hat{\hat\sigma}, \) \( \hat{\hat\sigma}' \) and \( E \) are as above, then we must be in one of the following three cases.
    \begin{enumerate}
        \item If \( |(\support d\hat{\hat\sigma})^+| \geq 6\), then 
        \begin{equation*}
            \prod_{p \in C_2(B_N)}\varphi_\beta \bigl(d\hat{\hat\sigma}(p) \bigr) = \prod_{p \in C_2(B_N,G)^+}\varphi_\beta \bigl((d\hat{\hat\sigma})_p \bigr)^2 \leq \alpha_1(\beta)^{\max(M',6)}.
        \end{equation*} 
        
        \item If  \( |(\support d\hat{\hat\sigma})^+| \in \{ 1,2,3,4,5 \}\),  by Lemma~\ref{lemma: minimal vortex I}, \( \hat{\hat\sigma} \) must support a vortex with support at the boundary of \( B_N \), and hence we must have \( |E| \geq \dist_1(e,\partial C_1(B_N)).\) At the same time, by definition, we also have \( \prod_{p \in C_1(B_N,G)}\varphi_\beta \bigl(d\hat{\hat\sigma}(p) \bigr) \leq \alpha_1(\beta). \) 
        \item If  \( |(\support d\hat{\hat\sigma})^+| = 0, \) then \( \hat{\hat\sigma} \in \Omega^1_0(B_N,G). \) Since \( |\support (\hat{\hat\sigma})^+ \cup \support (\hat{\hat\sigma}')^+| > 0 \), it follows from Lemma~\ref{lemma: small 1forms} that \(|E| \geq \min(M,8,\dist_0(e, \partial C_1(B_N))). \) Moreover, we have \( \prod_{p \in C_1(B_N,G)}\varphi_\beta \bigl(d\hat{\hat\sigma}(p) \bigr) = 1. \) 
    \end{enumerate}
    Consequently, we have
    \begin{equation*}
        \begin{split}
            J_{\beta, \kappa_1,\kappa_2}(E)  
            &\leq   \bigl(
            \alpha_1(\beta)^{\max(6,M')}
            +
            \mathbb{1}_{M'\in \{ 1,2,3,4,5 \},\,|E| \geq \dist_1(e,\partial C_1(B_N))} \cdot \alpha_1(\beta)
            \\&\qquad\qquad+
            \mathbb{1}_{M'=0,\, |E| \geq \max(M,\min(8,\dist_0(e, \partial C_1(B_N))))} 
            \bigr)
            \!\!\!\!\sum_{\substack{ \hat{\hat\sigma} \in \Omega^1(B_N,G),\,  \hat{\hat\sigma}' \in \Omega^1_0(B_N,G) \colon \\
            \substack{(\support \hat{\hat\sigma} \cup \support \hat{\hat\sigma}')^+ = E, \\ |\support d\hat{\hat\sigma}| \geq 2M'  }}} \!\!\!\! \varphi_{\kappa_1}\bigl(\hat{\hat\sigma}\bigr) \varphi_{\kappa_2}\bigl(\hat{\hat\sigma}'\bigr).
        \end{split}
    \end{equation*} 
    By dropping the condition \( \bigl|\support d\hat{\hat\sigma}\bigr| \geq 2M'\), and replacing the condition \( \hat{\hat\sigma}' \in \Omega^1_0(B_N,G)\) with the condition that \( \hat{\hat\sigma}' \in \Omega^1(B_N,G) \), we make the sum larger. 
    Hence
    \begin{equation*}
        \begin{split}
            &
            \sum_{\substack{ \hat{\hat\sigma} \in \Omega^1(B_N,G),\,  \hat{\hat\sigma}' \in \Omega^1_0(B_N,G) \colon \\
            \substack{(\support \hat{\hat\sigma} \cup \support \hat{\hat\sigma}')^+ = E, \\ |\support d\hat{\hat\sigma}| \geq 2M'  }}} \varphi_{\kappa_1}\bigl(\hat{\hat\sigma}\bigr) \varphi_{\kappa_2}\bigl(\hat{\hat\sigma}'\bigr)
            \leq
            \sum_{\substack{ \hat{\hat\sigma} \in \Omega^1(B_N,G),\,  \hat{\hat\sigma}' \in \Omega^1(B_N,G) \colon \\
            (\support \hat{\hat\sigma} \cup \support \hat{\hat\sigma}')^+ = E }} \varphi_{\kappa_1}\bigl(\hat{\hat\sigma}\bigr) \varphi_{\kappa_2}\bigl(\hat{\hat\sigma}'\bigr).
        \end{split}
    \end{equation*}
    Using Lemma~\ref{lemma: Jke bound}, we thus obtain
    \begin{equation}\label{eq: innermost sum B2}
    \begin{split}
        J_{\beta, \kappa_1,\kappa_2}(E)
        &\leq 
        \bigl(
            \alpha_1(\beta)^{\max(6,M')}
            +
            \mathbb{1}_{M'\in \{ 1,2,3,4,5\},\,|E| \geq \dist_1(e,\partial C_1(B_N))} \cdot \alpha_1(\beta)
            \\&\qquad\qquad+
            \mathbb{1}_{M'=0,\, |E| \geq \max(M,\min(8,\dist_0(e, \partial C_1(B_N))))} 
            \bigr)
            \\&\qquad\cdot  \pigl( \bigl(\alpha_0(\kappa_1)+\alpha_0(\kappa_2)
        + \alpha_0(\kappa_1)\alpha_0(\kappa_2)\bigr)^{|E|} \pigr).
        \end{split}
    \end{equation}
    Combining~\eqref{eq: first summary B2} and~\eqref{eq: innermost sum B2} and applying Lemma~\ref{lemma: from walks to sets}, we now finally obtain
    \begin{equation*} 
        \begin{split}
            &\mu_{N,\beta,\kappa} \times \mu_{N,\infty,\kappa} \Bigl( \pigl\{ (\hat \sigma,\hat \sigma') \in \Omega^1(B_N,G) \times \Omega^1_0(B_N,G)\colon  
            \\[-1.5ex]&\hspace{13em}
           |\mathcal{C}_{\mathcal{G}(\hat \sigma,\hat \sigma')}(e)| \geq 2M\text{ and }  \bigl| \support d\bigl( \hat \sigma|_{ \mathcal{C}_{\mathcal{G}(\hat \sigma,\hat \sigma')}(e)} \bigr)\bigr| \geq 2M'  \pigr\}\Bigr)
            \\&\qquad \leq
            \sum_{m=M}^\infty   
            18^{\max(0,2m-3)} \bigl(\alpha_0(\kappa_1) + \alpha_0(\kappa_2) + \alpha_0(\kappa_1)\alpha_0(\kappa_2) \bigr)^m 
            \alpha_1(\beta)^{\max(6,M')}
            \\&\qquad\qquad+
            \mathbb{1}_{M'\in \{1,2,3,4,5\}}\sum_{m=\dist_1(e,\partial C_1(B_N))}^\infty  
            18^{\max(0,2m-3)} \bigl(\alpha_0(\kappa_1) + \alpha_0(\kappa_2) + \alpha_0(\kappa_1)\alpha_0(\kappa_2) \bigr)^m \alpha_1(\beta)
            \\&\qquad\qquad+
            \mathbb{1}_{M'=0}\!\!\!\! \sum_{m=\max(M,\min(8,\dist_0(e,\partial C_1(B_N))))}^\infty  \!\!\!\!
            18^{\max(0,2m-3)} \bigl(\alpha_0(\kappa_1) + \alpha_0(\kappa_2) + \alpha_0(\kappa_1)\alpha_0(\kappa_2) \bigr)^m.
        \end{split}
    \end{equation*} 
    Computing the above geometric sums, we obtain~\eqref{eq: coupling and conditions 4} as desired. 
\end{proof}

\begin{proof}[Proof of Proposition~\ref{proposition: alternative plaquette bound}]
    If \( \hat \sigma \in \Omega^1(B_N,G) \) satisfies \( d \hat \sigma(p) \neq 0 \), then there must exist \( e \in \partial p \) such that \( \sigma(e) \neq 0 .\) For any such \( e, \) we must have \( |(\mathcal{C}_{\mathcal{G}(\hat \sigma,0)}(e))^+| \geq 1 \). Moreover, since \( \sigma(e) \neq 0 \), for any \( e' \in \partial p \) such that \( \sigma(e') \neq0 \), by definition, we have \( e' \in \mathcal{C}_{\mathcal{G}(\hat \sigma,0)}(e). \) Consequently, we must have \( d(\hat \sigma|_{\mathcal{C}_{\mathcal{G}(\hat \sigma,0)}(e)})(p) = d\hat \sigma(p) \neq 0. \) Using Lemma~\ref{lemma: minimal vortex I}, it follows that \( |(\support d(\hat \sigma|_{\mathcal{C}_{\mathcal{G}(\hat \sigma,0)}(e)}))^+|\geq 6. \) Combining these observations with a union bound, it follows that
    \begin{equation*}
        \begin{split}
            &
            \mu_{N,\beta,\kappa} \Bigl( \pigl\{ \hat \sigma \in \Omega^1(B_N,G)  \colon  
             d  \hat \sigma(p) \neq 0 \pigr\}\Bigr)
            \\&\qquad\leq  
            \sum_{e \in \partial p}
            \mu_{N,\beta,\kappa} \Bigl( \pigl\{ \hat \sigma \in \Omega^1(B_N,G)  \colon  
            |(\mathcal{C}_{\mathcal{G}(\hat \sigma)}(e))^+|\geq 1 \text{ and } |(\support d(\hat \sigma|_{\mathcal{C}_{\mathcal{G}(\hat \sigma)}(e)}))^+| \geq 6 \pigr\}\Bigr).
        \end{split}
    \end{equation*}
    Applying Proposition~\ref{proposition: new Z-LGT coupling upper bound} with \( \kappa_1 = \kappa ,\) \( \kappa_2 = \infty ,\) \( M = 1, \) and \( M' = 6, \) we obtain~\eqref{eq: coupling and conditions 4} as desired.
\end{proof}

\begin{proof}[Proof of Proposition~\ref{proposition: ZZ upper bound}]
    Recall first that by the definition of \( \mu^{E_0}_{N,(\infty,\kappa),(\infty,\kappa)} \), using Lemma~\ref{lemma: E in coupling}, we have 
    \begin{equation*}  
        \begin{split}
            &\mu^{E_0}_{N,(\infty,\kappa),(\infty,\kappa)}   \bigl( \big\{ ( \sigma, \sigma') \in \Omega^1_0(B_N,G) \times \Omega^1_0(B_N,G) \colon  e \in E_{E_0,\sigma,  \sigma'}\big\}\bigr)
            \\&\qquad= 
            \mu_{N,\infty,\kappa} \times \mu_{N,\infty,\kappa}   \bigl( \big\{ (\hat \sigma,\hat \sigma') \in \Omega^1_0(B_N,G) \times \Omega^1_0(B_N,G) \colon  e \in E_{\hat \sigma,\hat \sigma'} \big\}\bigr).
        \end{split}
    \end{equation*}
    Next, since \( \hat \sigma,\hat \sigma' \in \Omega^1_0(B_N,G), \) we have \( d\hat \sigma = d\hat \sigma' = 0. \) Consequently, 
    \begin{equation*}  
        e \in E_{E_0,\hat \sigma,\hat \sigma'} \Leftrightarrow e \in \mathcal{C}_{\mathcal{G}(\hat \sigma,\hat \sigma')}(E_0) \Leftrightarrow E_0 \cap \mathcal{C}_{\mathcal{G}(\hat \sigma,\hat \sigma')}(e) \neq \emptyset.
    \end{equation*}
    Finally, note that if \( E_0 \cap \mathcal{C}_{\mathcal{G}(\hat \sigma,\hat \sigma')}(e) \neq \emptyset \), then, by definition, we must have \( |(\mathcal{C}_{\mathcal{G}(\hat \sigma,\hat \sigma')}(e))^+ | \geq \dist_0(e,E_0) .  \)
    Combining these observations, it follows that  
    \begin{equation*}  
        \begin{split}
            &\mu^{E_0}_{N,(\infty,\kappa),(\infty,\kappa)}   \bigl( \big\{ ( \sigma, \sigma') \in \Omega^1_0(B_N,G) \times \Omega^1_0(B_N,G) \colon  e \in E_{E_0,\sigma,  \sigma'}\big\}\bigr)
            \\&\qquad\leq 
            \mu_{N,\infty,\kappa} \times \mu_{N,\infty,\kappa}   \bigl( \big\{ (\hat \sigma,\hat \sigma') \in \Omega^1_0(B_N,G) \times \Omega^1_0(B_N,G) \colon  |\mathcal{C}_{\mathcal{G}(\hat \sigma,\hat \sigma')}(e) | \geq \dist_0(e,E_0) \big\}\bigr). 
        \end{split}
    \end{equation*}
    Applying Proposition~\ref{proposition: new Z-LGT coupling upper bound} with \( \kappa_1 = \kappa_2 = \kappa, \) \( \beta= \infty, \) \( M = \dist_0(e,E_0) ,\) and \( M' = 0 \), we obtain~\eqref{eq: Z-Z coupling inequality} as desired. 
\end{proof}

\begin{proof}[Proof of Proposition~\ref{proposition: before E3}]
    Without loss of generality, we can assume that \( e \in C_1(B_N)^+ \). To simplify notation, let 
    \begin{equation*}
        \mathcal{E} \coloneqq \bigl\{ \sigma \in \Omega^1(B_N,G) \colon |\mathcal{C}_{\mathcal{G}(\sigma)}(E_0)| \geq 2M, \text{ and }\bigl| \support d\bigl( \sigma|_{ \mathcal{C}_{\mathcal{G}(\sigma)}(E_0)} \bigr) \bigr| \geq 2M' \bigr\}.
    \end{equation*}
    
    Now note that \( \mathcal{C}_{\mathcal{G}(\hat \sigma,\hat \sigma')}(E_0) \) is symmetric, and that the set \( \mathcal{C}_{\mathcal{G}(\hat \sigma,\hat \sigma')}(E_0) \cup \{ e,-e \} \) induces a connected set in \( \bar {\mathcal{G}}. \)
    Consequently, we have
    \begin{equation*}
        \begin{split}
            &\mu_{N,\beta,\kappa} ( \mathcal{E})
            \\&\qquad= 
            \!\!\!\!\!\!\sum_{\substack{E \subseteq C_1^+(B_N) \colon  |E| \geq M,\\ \bar{\mathcal{G}|}_{E\cup \{ e \}} \text{ is connected} }}
            \!\!\!\!
            \mu_{N,\beta,\kappa}  \Bigl( \pigl\{  \sigma \in \Omega^1(B_N,G) \colon 
            \mathcal{C}_{\mathcal{G}(\sigma)}(E_0)^+ = E ,\, \bigl| \support d\bigl( \sigma|_{ \mathcal{C}_{\mathcal{G}(\sigma)}(E_0)} \bigr) \bigr| \geq 2M'
            \pigr\}\Bigr).
        \end{split}
    \end{equation*}

    Given \( \sigma \in \Omega^1(B_N,G) \), if we let \( \hat{\hat\sigma} \coloneqq  \sigma|_{\mathcal{C}_{\mathcal{G}(\sigma)}(E_0)} \), then the following statements hold.
    \begin{enumerate} 
    
        \item If \( \bigl| \support d( \sigma|_{\mathcal{C}_{\mathcal{G}( \sigma, 0)}(E_0)})\bigr| \geq 2M' \), then, by definition, \( \bigl| \support d\hat{\hat \sigma} \bigr| \geq 2M'. \)
        
        \item If \( d( \sigma|_{\mathcal{C}_{\mathcal{G}( \sigma, 0)}(E_0)})\neq 0 \), then \( E_0 \cap \support  \sigma \neq \emptyset\) and thus 
        \( (\mathcal{C}_{\mathcal{G}(\sigma)}(E_0))^+ = (\support \hat{\hat\sigma} )^+\). 
    \end{enumerate}
    As a consequence, the expression in the previous equation is bounded from above by 
    \begin{equation}\label{eq: muNbetainftysumsum4}
        \begin{split}
             &
            \sum_{\substack{E \subseteq C_1^+(B_N) {\colon} |E| \geq M,\\ \bar{\mathcal{G}|}_{E \cup \{ e \}} \text{ is connected} }}
            \,
            \sum_{\substack{\hat{\hat\sigma} \in \Omega^1(B_N,G)\colon \\
            \substack{(\support \hat{\hat\sigma})^+  = E,\\ |\support d \hat{\hat\sigma} |\geq 2M'} }} 
            \mu_{N,\beta,\kappa}  \Bigl( \pigl\{  \sigma \in \Omega^1(B_N,G) \colon  
            \hat \sigma|_{\mathcal{C}_{\mathcal{G}(\hat \sigma, 0)}(E_0)}  =  \hat{\hat\sigma}
            \pigr\}\Bigr).
        \end{split}
    \end{equation}
    For any \(  \hat{\hat\sigma} \) as in the sum above, by applying first Lemma~\ref{lemma: cluster is subconfig}, and then Proposition~\ref{proposition: edgecluster flipping ii}, we have 
    \begin{equation*}
        \begin{split}
            &\mu_{N,\beta,\kappa} \bigl( \{ \hat \sigma \in \Omega^1(B_N,G) \colon
            \hat \sigma|_{\mathcal{C}_{\mathcal{G}(\hat \sigma, 0)}(E_0)} =  \hat{\hat\sigma} \} \bigr)
            \leq 
            \mu_{N,\beta,\kappa}  \bigl( \{ \hat \sigma \in \Omega^1(B_N,G) \colon
            \hat{\hat\sigma} \leq \hat \sigma  \} \bigr)
            \leq \varphi_{\beta,\kappa}( \hat{\hat\sigma}).
        \end{split}
    \end{equation*}
    Taken together, the above equations show that
    \begin{equation}\label{eq: first summary B2 2}  
        \begin{split}
            &\mu_{N,\beta,\kappa}(\mathcal{E})
            \leq
            \sum_{\substack{E \subseteq C_1^+(B_N) \mathrlap{\colon} \\ \substack{e \in E,\, |E| \geq M,\\ \bar{\mathcal{G}|}_E \text{ is connected} }}} 
            J_{\beta, \kappa}(E),
        \end{split}
    \end{equation} 
    where
    \begin{equation*}
        J_{\beta, \kappa}(E) \coloneqq 
        \sum_{\substack{\hat{\hat\sigma} \in \Omega^1(B_N,G)\colon \\
        \substack{(\support \hat{\hat\sigma})^+ = \support E,\\ |\support d \hat{\hat\sigma} |\geq 2M'} }} \varphi_{\beta,\kappa}  ( \hat{\hat\sigma} ).
    \end{equation*} 
    Now recall that
    \begin{equation*} 
        \begin{split}
            &\varphi_{\beta,\kappa}( \hat{\hat\sigma})
            = \prod_{e' \in C_1(B_N)} \varphi_\kappa\bigl(\hat{\hat\sigma}(e')\bigr) \prod_{p \in C_2(B_N)}\varphi_\beta \bigl(d\hat{\hat\sigma}(p) \bigr).
        \end{split} 
    \end{equation*}
    Also, recall that for any \( r \geq 0 \) and \( g \in G \), we have  \( \varphi_r(0) = 1\) and \( \varphi_r(g) = e^{r \Re(\rho(g) - 1)} \in (0,1]\). If \( g \neq 0\), then \( \varphi_r(g) < 1\) and hence \(\varphi_\beta(g)^2 \leq \alpha_1(\beta) < 1\).
    If \( \hat{\hat\sigma} \) is as above, then \( |\support d\hat{\hat\sigma}| \geq 2M' \), and hence
    \begin{equation*}
        \prod_{p \in C_2(B_N,G)^+}\varphi_\beta \bigl((d\hat{\hat\sigma})_p \bigr)^2 \leq \alpha_1(\beta)^{M'}.
    \end{equation*}  
    Consequently, if \( E \) is as above, then
    \begin{equation*}
        \begin{split}
            &J_{\beta, \kappa}(E)  \leq \alpha_1(\beta)^{M'} 
             \sum_{\substack{ \hat{\hat\sigma} \in \Omega^1(B_N,G)  \colon \\
            \substack{(\support \hat{\hat\sigma})^+  = E, \\ |\support d\hat{\hat\sigma}| \geq 2M'  }}}
            \prod_{e' \in C_1(B_N)^+} \varphi_\kappa\bigl(\hat{\hat\sigma}(e')\bigr)^2 .
        \end{split}
    \end{equation*} 
    By dropping the condition \( \bigl|\support d\hat{\hat\sigma}\bigr| \geq 2M'\) we make the sum larger. Hence
    \begin{equation*}
        \begin{split}
            &J_{\beta, \kappa}(E)  \leq \alpha_1(\beta)^{M'} 
             \sum_{\substack{ \hat{\hat\sigma} \in \Omega^1(B_N,G)  \colon \\
            (\support \hat{\hat\sigma})^+  = \support W}}
            \prod_{e' \in C_1(B_N)^+} \varphi_\kappa\bigl(\hat{\hat\sigma}(e')\bigr)^2 .
        \end{split}
    \end{equation*}
    
    If \( \hat{\hat\sigma} \in \Omega^1(B_N,G) \) and \( e' \notin \support \hat{\hat\sigma} \), then \( \varphi_\kappa\bigl(\hat{\hat\sigma}(e')\bigr) = \varphi_\kappa(0)=1 \). Also, if \( \hat{\hat\sigma} \in \Omega^1(B_N,G) \) and \( e' \in (\support \hat{\hat\sigma})^+ \), then \( \hat{\hat\sigma}(e') \neq 0. \) 
    Using this observation, we obtain
    \begin{equation}
        \begin{split}
        &\sum_{\substack{ \hat{\hat\sigma} \in \Omega^1(B_N,G) \colon \\
        (\support \hat{\hat\sigma} ')^+  = E }}
        \prod_{e' \in C_1(B_N)^+} \varphi_\kappa\bigl(\hat{\hat\sigma}(e')\bigr)^2 
        \leq  
        \prod_{e' \in E } 
        \sum_{\hat{\hat\sigma}'(e') \in G \smallsetminus \{0\}} \varphi_\kappa(\hat{\hat\sigma}'(e'))^2  
        = \prod_{e' \in E }
        \alpha_0(\kappa)  
         = \alpha_0(\kappa)^{|E|}, 
        \end{split}
    \end{equation} 
    We thus have
    \begin{align}\label{eq: innermost sum B2 2}
        &J_{\beta, \kappa}(E)
        \leq  \alpha_1(\beta)^{M'}   \alpha_0(\kappa)
        ^{|E|}.
    \end{align} 
    
    Now note that, by Lemma~\ref{lemma: from walks to sets}, for any \( m \geq M, \) we have 
    \begin{equation*}
    \begin{split}
        &\pigl| \bigr\{ E \subseteq C_1^+(B_N) \colon |E|= m, \, \bar{\mathcal{G}}|_{E \cup \{ e \}} \text{ is connected} \bigr\} \pigr| 
        \\&\qquad=
        \pigl| \bigr\{ E \subseteq C_1^+(B_N) \colon |E|= m,\, e \in E, \, \bar{\mathcal{G}}|_{E} \text{ is connected} \bigr\} \pigr| 
        \\&\qquad\qquad+
        \pigl| \bigr\{ E \subseteq C_1^+(B_N) \colon |E|= m+1,\, e \in E, \, \bar{\mathcal{G}}|_{E} \text{ is connected} \bigr\} \pigr|
        \\&\qquad\leq 18^{2m-3} + 18^{2(m+1)-3}. 
    \end{split}
    \end{equation*}
    Combining this with~\eqref{eq: first summary B2 2} and~\eqref{eq: innermost sum B2}, we thus find that 
    \begin{equation*} 
        \mu_{N,\beta,\kappa} ( \mathcal{E})
        \leq
        \sum_{m=M}^\infty   
        (18^{2m-3} + 18^{2(m+1)-3}) \alpha_0(\kappa)^m  \alpha_1(\beta)^{M'}.
    \end{equation*} 
    Computing the above geometric sum, we obtain~\eqref{eq: coupling and conditions 4}.
\end{proof}

\section{A first version of our main result}\label{sec: first version of main result}

In this section, we present a first application of the coupling introduced in Section~\ref{sec: LGT-Z coupling}, by giving a first version of Theorem~\ref{theorem: main result}. This result provides an upper bound on \( \langle L_\gamma(\sigma,\phi) \rangle \) which is good when the probability is small that there is a cluster in \( \mathcal{G}(\hat \sigma,\hat \sigma') \) which both intersects \( \support \gamma \) and supports a vortex. We later present a strengthening of this result in Proposition~\ref{proposition: first version of main result}.

\begin{proposition}\label{proposition: short lines}
    Let \( \beta,\kappa \geq 0 \) be such that~\ref{assumption: 3} holds, and let \( \gamma \) be a path with finite support. Then
    \begin{equation}\label{eq: short lines}
        \Bigl| \bigl\langle L_\gamma(\sigma,\phi) \bigr\rangle_{\beta,\kappa,\infty} - \bigl\langle L_\gamma(\sigma,\phi)\bigr\rangle_{\infty,\kappa,\infty} \Bigr| \leq 2  
        K_4 \bigl( 1 + K_3 K_4
        \alpha_0(\kappa)\bigr)|\support \gamma|\alpha_0(\kappa)  \alpha_1(\beta)^{6},
    \end{equation}
    where \( K_3 \) and \( K_4 \) are defined by~\eqref{eq: K3 and K4}.
\end{proposition}

\begin{proof}
    Let \( N \geq 1 \) be large enough so that \( \support \gamma \subseteq C_1(B_N) \) and \( \dist_0(\support \gamma,\partial C_1(B_N))) \geq 8. \) 
    Then, by definition, we have
    \begin{equation*}
        \begin{split} 
        &\mathbb{E}_{N,\beta,\kappa} \bigl[L_\gamma(\sigma) \bigr]
        = 
        \mathbb{E}_{N,(\beta,\kappa),(\infty,\kappa)} \bigl[ L_\gamma(\sigma) \bigr]
        = 
        \mu_{N,\beta,\kappa} \times \mu_{N,\infty,\kappa} \bigl[ L_\gamma(\hat \sigma|_{E_{\hat \sigma,\hat \sigma'}} + \hat \sigma'|_{C_1(B_N)\smallsetminus E_{\hat \sigma,\hat \sigma'}}) \bigr].
        \end{split}
    \end{equation*} 
    On the event \( \support \gamma \cap E_{\hat \sigma,\hat \sigma'} = \emptyset, \) 
    we have
    \begin{equation*}
        \begin{split}
            &L_\gamma (\hat \sigma|_{E_{\hat \sigma,\hat \sigma'}} + \hat \sigma'|_{C_1(B_N)\smallsetminus E_{\hat \sigma,\hat \sigma'}})
            =
            L_\gamma( 0 + \hat \sigma'|_{C_1(B_N)\smallsetminus E_{\hat \sigma,\hat \sigma'}})
            =L_\gamma
            (\hat \sigma'|_{E_{\hat \sigma,\hat \sigma'}} + \hat \sigma'|_{C_1(B_N)\smallsetminus E_{\hat \sigma,\hat \sigma'}})
            =
            L_\gamma(\hat \sigma').
        \end{split}
    \end{equation*}
    As a consequence,
    \begin{equation*}
        \begin{split} 
            &
            \mu_{N,\beta,\kappa} \times \mu_{N,\infty,\kappa} \bigl[ L_\gamma(\hat \sigma|_{E_{\hat \sigma,\hat \sigma'}} + \hat \sigma'|_{C_1(B_N)\smallsetminus E_{\hat \sigma,\hat \sigma'}}) \bigr]
            \\&\qquad= 
            \mu_{N,\beta,\kappa} \times \mu_{N,\infty,\kappa} \bigl[ L_\gamma( \hat \sigma') \cdot \mathbb{1}_{\support \gamma \cap E_{\hat \sigma,\hat \sigma'}= \emptyset}\bigr]
            \\&\qquad\qquad+
            \mu_{N,\beta,\kappa} \times \mu_{N,\infty,\kappa} \bigl[ L_\gamma(\hat \sigma|_{E_{\hat \sigma,\hat \sigma'}} + \hat \sigma'|_{C_1(B_N)\smallsetminus E_{\hat \sigma,\hat \sigma'}}) \cdot \mathbb{1}_{\support \gamma \cap E_{\hat \sigma,\hat \sigma'}\neq \emptyset}  \bigr]
            \\&\qquad= 
            \mu_{N,\beta,\kappa} \times \mu_{N,\infty,\kappa} \bigl[ L_\gamma( \hat \sigma') \bigr]
            -
            \mu_{N,\beta,\kappa} \times \mu_{N,\infty,\kappa} \bigl[ L_\gamma( \hat \sigma') \cdot \mathbb{1}_{\support \gamma \cap E_{\hat \sigma,\hat \sigma'}\neq \emptyset}\bigr]
            \\&\qquad\qquad+
            \mu_{N,\beta,\kappa} \times \mu_{N,\infty,\kappa} \bigl[ L_\gamma(\hat \sigma|_{E_{\hat \sigma,\hat \sigma'}} + \hat \sigma'|_{C_1(B_N)\smallsetminus E_{\hat \sigma,\hat \sigma'}}) \cdot \mathbb{1}_{\support \gamma \cap E_{\hat \sigma,\hat \sigma'} \neq \emptyset } \bigr]
        \end{split}
    \end{equation*} 
    Since \( \rho \) is unitary and \( L_\gamma(\sigma) = \rho(\sigma(\gamma)) \) for any \( \sigma \in \Omega^1(B_N,G),\) it follows that
    \begin{equation*}
        \begin{split} 
        &\Bigl|\mathbb{E}_{N,\beta,\kappa} \bigl[L_\gamma(\sigma) \bigr]
        -
        \mathbb{E}_{N,\infty,\kappa} \bigl[ L_\gamma( \hat \sigma') \bigr] \Bigr|
        \\&\qquad \leq 
        2\mu_{N,\beta,\kappa} \times \mu_{N,\infty,\kappa} \pigl( \bigl\{ (\hat \sigma,\hat \sigma') \in \Omega^1(B_N,G)\times \Omega^1_0(B_N,G) \colon  \support \gamma \cap E_{\hat \sigma,\hat \sigma'}\neq \emptyset \bigr\} \pigr)  
        \\&\qquad \leq 
        2\sum_{e \in \gamma}\mu_{N,\beta,\kappa} \times \mu_{N,\infty,\kappa} \pigl( \bigl\{ (\hat \sigma,\hat \sigma') \in \Omega^1(B_N,G)\times \Omega^1_0(B_N,G) \colon  e \in  E_{\hat \sigma,\hat \sigma'}\neq \emptyset \bigr\} \pigr)
        \end{split}
        \end{equation*}
    
    By Lemma~\ref{lemma: Edsigmalemma}, if \( e \in C_1(B_N) \), \( \hat \sigma \in \Omega^1(B_N,G) \) and \( \hat \sigma' \in \Omega^1_0(B_N,G) \), then \( e \in E_{\hat \sigma,\hat \sigma'} \) if and only if \( d(\hat \sigma|_{\mathcal{C}_{\mathcal{G}(\hat \sigma,\hat \sigma')}(e)})\neq 0. \) On the other hand, if \( d(\hat \sigma|_{\mathcal{C}_{\mathcal{G}(\hat \sigma,\hat \sigma')}(e)})\neq 0, \) then we must have \( e \in \support \hat \sigma|_{\mathcal{C}_{\mathcal{G}(\hat \sigma,\hat \sigma')}(e)} \), implying in particular that \( -e \in \mathcal{C}_{\mathcal{G}(\hat \sigma,\hat \sigma')}(e) \), and hence \(|\mathcal{C}_{\mathcal{G}(\hat \sigma,\hat \sigma')}(e)| \geq 2. \) Applying Proposition~\ref{proposition: new Z-LGT coupling upper bound} with \( M =  M' = 1 \) and \( \kappa_1=\kappa_2 = \kappa,\) we thus obtain
    \begin{equation*}
        \begin{split}
            &\mu_{N,\beta,\kappa} \times \mu_{N,\infty,\kappa}\pigl( \bigl\{ (\hat \sigma,\hat \sigma') \in \Omega^1(B_N,G) \times \Omega^1_0(B_N,G) \colon e \in E_{\hat \sigma,\hat \sigma'} \bigr\} \pigr) 
            \\&\qquad\leq 
            K_4 \bigl( 1 + K_3 K_4\alpha_0(\kappa)\bigr) \alpha_0(\kappa)  \alpha_1(\beta)^{6}
            +
            K_3 \bigl( K_4 \alpha_0(\kappa) \bigr)^{\dist_1(e,\partial C_1(B_N))} \alpha_1(\beta).
        \end{split}
    \end{equation*}
    Combining the above equations and letting \( N \to \infty \) (using Proposition~\ref{proposition: limit exists} and Corollary~\ref{corollary: unitary gauge}), we obtain~\eqref{eq: short lines} as desired.
\end{proof}

\section{A decomposition of the coupled spin configuration}\label{sec: spin config decomposition}
 
The main result in this section is the following proposition, which gives a decomposition of \( \sigma \coloneqq \hat \sigma|_{E_{\sigma,\sigma'}} + \hat \sigma'|_{C_1(B_N)\smallsetminus E_{\hat \sigma,\hat \sigma'}}\) in terms of decompositions of \( \hat \sigma \) and \( \hat \sigma'. \)

\begin{proposition}\label{proposition: docomposition of the coupling}
    Let \( \hat \sigma  \in \Omega^1(B_N,G) \) and \( \hat \sigma' \in \Omega^1_0(B_N,G), \) let \( \hat \Sigma \) be a decomposition of \( \hat \sigma \) and \( \hat \Sigma' \) be a decomposition of \( \hat \sigma' \) (these are guaranteed to exist by Lemma~\ref{lemma: lemma sum of irreducible configurations}), and define
    \begin{equation*}
        \Sigma \coloneqq \bigl\{ \hathat \sigma  \in \hat \Sigma \colon \support  \hathat \sigma  \subseteq E_{\hat \sigma,\hat \sigma'} \bigr\}
    \end{equation*}
    and
    \begin{equation*}
        \Sigma' \coloneqq \bigl\{ \hathat \sigma  \in \hat \Sigma' \colon \support \hathat \sigma \subseteq C_1(B_N)\smallsetminus E_{\hat \sigma,\hat \sigma'} \bigr\}.
    \end{equation*}
    Then \( \Sigma \cup \Sigma' \) is a decomposition of  \( \sigma \coloneqq \hat \sigma|_{E_{\sigma,\sigma'}} + \hat \sigma'|_{C_1(B_N)\smallsetminus E_{\hat \sigma,\hat \sigma'}}.\) 
\end{proposition}

\begin{proof}
    We need to show that~\ref{lemma210property1}--\ref{lemma210property5} of Lemma~\ref{lemma: lemma sum of irreducible configurations} holds, i.e. that
\begin{enumerate}[label=\textnormal{(\roman*)}]
        \item\label{Plemma210property1} if \( \hathat \sigma \in \Sigma \cup  \Sigma' \), then \( \hathat \sigma \) is non-trivial and irreducible,
        
        \item\label{Plemma210property2}  if \( \hathat \sigma \in \Sigma \cup  \Sigma' \), then  \( \hathat \sigma \leq \sigma \),
        
        \item\label{Plemma210property3} if \( \hathat \sigma_1,\hathat \sigma_2 \in \Sigma \cup \Sigma' \), then \( \hathat \sigma_1 \) and \( \hathat \sigma_2 \) have disjoint supports,
        
        \item\label{Plemma210property4} \( \sigma = \sum_{\smallhathat \sigma \in \Sigma \cup \Sigma'} \hathat \sigma \), and

         \item \label{Plemma210property5}
         if \( \hathat \sigma_1,\hathat \sigma_2 \in \Sigma \cup \Sigma' \), then \( d\hathat \sigma_1 \) and \( d\hathat \sigma_2 \) have disjoint supports,
    \end{enumerate} 
    
    We now show that (i)--(v) holds.
    \begin{enumerate}[label=(\roman*)]
    
        \item Since \( \hat\Sigma \) and \( \hat \Sigma' \) are decompositions of \( \hat \sigma \) and \( \hat \sigma'\) respectively, \ref{Plemma210property1} holds with \( \hat \Sigma \cup \hat \Sigma' \) replaced with \(  \Sigma \cup  \Sigma'. \) Since \(  \Sigma \cup  \Sigma' \subseteq \hat \Sigma \cup \hat \Sigma'\), the desired conclusion follows.
        
        \item Fix some \( \hathat \sigma \in \Sigma \). By the definition of \( \Sigma \), we have \( \support \hathat \sigma \subseteq E_{\hat \sigma, \hat \sigma'}, \) and hence \( \hathat \sigma = \hathat \sigma|_{E_{\hat \sigma,\hat \sigma'}} \). At the same time, since \( \Sigma \subseteq \hat \Sigma \) and \( \hat \Sigma \) is a decomposition of \( \hat \sigma \), we have \( \hathat \sigma \leq \hat \sigma. \) Finally, note that, by Lemma~\ref{lemma: E in coupling}, we have \( E_{\hat \sigma,\hat \sigma'} = E_{\sigma,\hat \sigma'}. \) By applying Lemma~\ref{lemma: cluster is subconfig} twice, we obtain
        \begin{equation*}
            \hathat \sigma = \hathat \sigma|_{E_{\hat \sigma,\hat \sigma'}} \leq \hat \sigma|_{E_{\hat \sigma,\hat \sigma'}} = \sigma|_{E_{\hat \sigma,\hat \sigma'}}
            =
            \sigma|_{E_{\sigma,\hat \sigma'}} \leq \sigma,
        \end{equation*}
        and hence \( \hathat \sigma \leq \sigma. \)
        Since proof in the case \( \hathat \sigma \in \Sigma' \) is analogous, we omit it here.
        
        \item Since \( \hat \Sigma \) is a decomposition of \( \hat \sigma \), for any distinct \( \hathat \sigma_1,\hathat \sigma_2 \in \Sigma  \subseteq \hat \Sigma\), \( \hathat \sigma_1 \) and \( \hathat \sigma_2 \) have disjoint supports.
        Analogously, since \( \hat \Sigma' \) is a decomposition of \( \hat \sigma' \), for any distinct \( \hathat \sigma_1,\hathat \sigma_2 \in \Sigma' \subseteq \hat \Sigma'\), \( \hathat \sigma_1 \) and \( \hathat \sigma_2 \) have disjoint supports.
        Finally, if \( \hathat \sigma_1 \in \Sigma \) and \( \hathat \sigma_2 \in \Sigma'' \), then, since \( \support \hathat \sigma_1 \subseteq E_{\sigma,\sigma'} \), \( \support \hathat \sigma_2 \subseteq C_1(B_N)\smallsetminus E_{\sigma,\sigma'}, \) and the sets \( E_{\sigma,\sigma'} \) and \( C_1(B_N)\smallsetminus E_{\sigma,\sigma'} \) are disjoint, it follows that \( \hathat \sigma_1  \) and \( \hathat \sigma_2 \) have disjoint supports. This concludes the proof of~\ref{Plemma210property3}.
        
        \item Since \( \hat \Sigma \) is a spin decomposition of \( \hat \sigma \), each \( \hathat \sigma \in \hat \Sigma \) is non-trivial and irreducible. Consequently, using Lemma~\ref{lemma: graph decomposition in coarser}, it follows that for each \( \hathat \sigma \in \hat \Sigma \),  we have either \( \support \hathat \sigma \subseteq E_{\hat \sigma,\hat \sigma'} \) or \( \support \hathat \sigma \subseteq C_1(B_N)\smallsetminus E_{\hat \sigma,\hat \sigma'}, \) and hence
        \begin{equation*}
            \hat \sigma|_{E_{\hat \sigma,\hat \sigma'}} = \biggl( \, \sum_{\smallhathat \sigma \in \hat \Sigma} \hathat \sigma \biggr)\Big|_{E_{\hat \sigma,\hat \sigma'}}
             = \sum_{\smallhathat \sigma \in \hat \Sigma} \hathat \sigma|_{E_{\hat \sigma,\hat \sigma'}} 
             = \sum_{\smallhathat \sigma \in  \Sigma} \hathat \sigma.
        \end{equation*}
        Completely analogously, we find that 
        \begin{equation*}
            \hat \sigma'|_{C_1(B_N)\smallsetminus E_{\hat \sigma,\hat \sigma'}} = \biggl( \, \sum_{\smallhathat \sigma \in \hat \Sigma'} \hathat \sigma \biggr)\Big|_{C_1(B_N)\smallsetminus E_{\hat \sigma,\hat \sigma'}}
             = \sum_{\smallhathat \sigma \in \hat \Sigma'} \hathat \sigma|_{E_{\hat \sigma,\hat \sigma'}} 
             = \sum_{\smallhathat \sigma \in \Sigma'} \hathat \sigma.
        \end{equation*}
        Combining the previous equations and using the definition of \( \sigma \), we obtain~\ref{Plemma210property4}.
        
        \item If \( \hathat \sigma \in \Sigma' \), then, since \( \Sigma' \subseteq \hat \Sigma' \) and \( \hat \Sigma' \) is a decomposition of \( \hat \sigma' \), we have \( \hathat \sigma \leq \hat \sigma' \). Since \( \hat \sigma' \in \Omega^1_0(B_N,G) \), we have \( d\hat \sigma' = 0, \) and hence \( d\hathat \sigma = 0. \) Consequently, the desired conclusion will follow if we can show that~\ref{Plemma210property5} holds with \( \Sigma \cup \Sigma' \) replaced with \( \Sigma. \)
        To see that this holds, let \( \hathat \sigma_1 ,\hathat \sigma_2 \in \Sigma \). Then, since \( \Sigma \subseteq \hat \Sigma \), we also have \( \hathat \sigma_1 ,\hathat \sigma_2 \in \hat \Sigma. \) Since \( \hat \Sigma\) is a decomposition of \( \hat \sigma, \) the 2-forms \( d\hathat \sigma_1\) and \( d\hathat \sigma_2 \) must have disjoint support.
        This concludes the proof of~\ref{Plemma210property5}.
    \end{enumerate}
\end{proof}

\section{Disturbing 1-forms}\label{sec: 1forms}

The main purpose of this section is to introduce the following definition.
\begin{definition}\label{def: disturbing}
    Let \( \sigma \in \Omega^1(B_N,G) \), and let \( \gamma \in C^1(B_N)\) be a path. If there is no  path \( \hat \gamma \in C^1(B_N)\) with \( \partial \hat \gamma = - \partial \gamma\) and 1-form \( \hathat \sigma \in \Omega^1(B_N,G)  \) such that 
    \begin{enumerate}[label=(\roman*)]
        \item \label{item: disturbing 1}\( d\hathat \sigma \leq d\sigma \)
        \item \label{item: disturbing 2}\( \sigma(\hat \gamma) = 0 \),
        \item \label{item: disturbing 3}\( \hathat \sigma(\gamma + \hat \gamma) = 0 \),  
        \item \label{item: disturbing 4}any vortex \( \nu \) in \( \sigma-\hathat \sigma \) is a minimal vortex centered around an edge in~\( \gamma- \gamma_c \) (see~\eqref{def: gammac} for a definition of \( \gamma_c \)), and
        \item \label{item: disturbing v} if \( d\sigma(p) = d\sigma(p') \) for all \( p,p' \in \hat \partial e, \) then \( d\hathat\sigma(p) =0\) for all \(p \in \hat\partial e, \) 
        \end{enumerate}
        then we say that \( \sigma \) \emph{disturbs} \( \gamma \).
\end{definition}

Note that if \( \gamma \in C^1(B_N) \) is a generalized loop and \( \sigma \in \Omega^1(B_N,G), \) then we can pick \( \hat \gamma = 0 \) in Definition~\ref{def: disturbing}, and hence, in this case, (ii) automatically holds.

The main reason for introducing the previous definition is Lemma~\ref{lemma: removing nondisturbing configurations} below. To simplify the notation in this lemma, we define
\begin{equation}\label{eq: def gamma'}
    \gamma'[e] \coloneqq (\gamma - \gamma_c)[e] \cdot \mathbb{1}\bigl( \exists p,p' \in \hat \partial e \colon d\sigma(p) \neq d\sigma(p') \bigr),\quad e \in C_1^+(B_N).
\end{equation}

\begin{lemma}\label{lemma: removing nondisturbing configurations}
    Let \( \sigma \in \Omega^1(B_N,G) \) and let \( \gamma \in C^1(B_N) \) be a path. For each \( e \in  \gamma \), fix one plaquette \( p_e \in  \hat \partial e. \) 
    Then, if \( \sigma \) does not disturb \( \gamma \), we have
    \begin{equation*}
        \sigma(\gamma) = \sum_{e \in  (\gamma-\gamma_c)- \gamma'} d\sigma(p_e).
    \end{equation*}
\end{lemma}

\begin{proof}
    Assume that \( \sigma \) does not disturb \( \gamma \). Then, by definition, there is \( \hat \gamma \) and \( \hathat \sigma \) which satisfies~\ref{item: disturbing 1}--\ref{item: disturbing 4} of~Definition~\ref{def: disturbing}. 
    To simplify notation, define \( \bar \sigma \coloneqq \sigma - \hathat \sigma. \) Then
    \begin{equation*}
        \begin{split}
            &\sigma(\gamma) 
            = \sigma(\gamma) + 0
            \overset{\mathrm{\ref{item: disturbing 2}}}{=} \sigma(\gamma) + \sigma(\hat \gamma)
            = \sigma(\gamma + \hat \gamma)
            = (\sigma-\hathat \sigma+\hathat \sigma)(\gamma + \hat \gamma) 
            = (\sigma-\hathat \sigma)(\gamma + \hat \gamma) + \hathat \sigma(\gamma + \hat \gamma)
            \\&\qquad \overset{\mathrm{\ref{item: disturbing 3}}}{=} (\sigma-\hathat \sigma)(\gamma + \hat \gamma) + 0
             = (\sigma-\hathat \sigma)(\gamma + \hat \gamma)
            =  \bar\sigma(\gamma+\hat \gamma).
        \end{split}
    \end{equation*}
    
    Since \( \partial (\gamma + \hat \gamma) = \partial \gamma + \partial \hat \gamma = \partial \gamma - \partial \gamma = 0 ,\) \( \gamma + \hat \gamma \) is a generalized loop.
    Let \( B \) be a cube of  width \( |\support(\gamma+ \hat \gamma)| \) which contains \( \gamma + \hat \gamma\). Since \( \gamma + \hat \gamma \subseteq C^1(B_N) \), such a cube exists. 
    Next, let \( q \) be an oriented surface  inside \( B \) such that \( \gamma + \hat\gamma\) is the boundary of \( q \). The existence of such a surface is guaranteed by Lemma~\ref{lemma: oriented loops}.
    
    By Lemma~\ref{lemma: lemma sum of irreducible configurations}, there is a set \( \Omega \subseteq \Omega_2^0(B_N,G) \) which is a decomposition of \( d\bar\sigma. \)
    Fix such a set \( \Omega \), and note that, by definition,  each \( \omega \in \Omega \) is a vortex in \( \bar\sigma. \)
    Let \( \Omega^q \) be the set of all \( \omega \in \Omega \) with \( \omega(q) \neq 0. \)
    Then, by the discrete Stokes' theorem, we have
    \begin{equation*}
        \bar\sigma(\gamma + \hat \gamma) = d\bar\sigma(q) =  \sum_{\omega \in \Omega^q} \omega(q) .
    \end{equation*}

    Now fix some \( \omega \in \Omega^q \). Since \( \omega(q) \neq 0 \), by~\ref{item: disturbing 4} and Lemma~\ref{lemma: minimal vortex II}, there must exist \( e \coloneqq \frac{\partial}{\partial x^j}\big|_a \in \Omega_1^+(B_N) \) and \( g \in G\smallsetminus \{ 0\} \) such that \( \gamma[e] = 1 \) and \( \omega = d(g \mathbb{1}_{a}\, d x_j).\) 
    Then, by definition, we have \( \omega(p_e) = g \), and since \( \omega \leq d\bar\sigma \) and \( g \neq 0 \), it follows that \( d\bar\sigma(p_e) = \omega(p_e) = g. \) 
    Since \( q \) is an oriented surface with boundary \( \gamma \), we thus have 
    \begin{equation*}
        \omega(q) = d(g\mathbb{1}_{a}\, dx_j)(q) = (g\mathbb{1}_{a}\, dx_j)(\gamma) = g = \omega(p_e) = d\bar\sigma(p_e).
    \end{equation*}
     
    Define
    \begin{equation*}
        \gamma_5[e] \coloneqq (\gamma - \gamma_c)[e] \cdot \mathbb{1}\bigl( \exists \omega \in \Omega^q \text{ such that\ } \support \omega = \hat \partial e \cup \hat \partial (-e)\bigr), \quad e \in C_1^+(B_N).
    \end{equation*}
    Then, since  minimal vortices around distinct edges in \( \gamma-\gamma_c \) have disjoint supports, it follows that 
    \begin{equation*}
        \sum_{\omega \in \Omega^q} \omega(q) = \sum_{e \in \gamma_5} d\bar\sigma(p_e).
    \end{equation*}
    Since, by assumption, we have \( d\hathat \sigma \leq d\sigma \), and \( d\bar\sigma = d\sigma-d\hathat \sigma \), it follows from Lemma~\ref{lemma: the blue lemma}\ref{property 4} that \(  d\bar\sigma \leq d\sigma \). 
    Using the definition of \( \gamma_5 \), it follows that for any \( e \in \gamma_5 \), we have \( d\bar\sigma(p_e) = d\sigma(p_e). \)
    Consequently,
    \begin{equation*}
        \sum_{e \in \gamma_5} d\bar\sigma(p_e)
         =  \sum_{e \in \gamma_5} d\sigma(p_e).
    \end{equation*}

    Now note that by the definition of \( \gamma', \) we have
    \begin{equation*} 
            \bigl( (\gamma-\gamma_c) - \gamma')[e] = (\gamma-\gamma_c)[e] \cdot \mathbb{1} \bigl( d\sigma(p) = d\sigma(p') \text{ for all } p,p' \in \hat \partial e   \bigr), \quad e \in C_1^+(B_N).
    \end{equation*}
    Since \( d\bar\sigma \leq d\sigma \), it follows that if \( e \in \gamma_5 \) then  \( e \in  (\gamma-\gamma_c) - \gamma'. \) Finally, we note that if  \( e \in (\gamma-\gamma_c) - \gamma' ,\) then \( d\hathat \sigma(p_e)=0\), and hence  \( d\bar \sigma(p_e)= d\sigma(p_e)=0. \) As a consequence,
    \begin{equation*}
        \sum_{e \in \gamma_5} d\sigma(p_e) = \sum_{e \in (\gamma-\gamma_c) - \gamma'} d\sigma(p_e).
    \end{equation*}
    By combining the previous equations, we obtain the desired conclusion.
\end{proof}

\section{Proof of the main result}\label{sec: proof of main result}

In this section, we will first give a proof of the following result, which is more general that Theorem~\ref{theorem: main result Z2}, and then show how this proof, with very small adjustments, implies Theorem~\ref{theorem: main result Z2}.

\begin{theorem}\label{theorem: main result}
    Let \( G = \mathbb{Z}_n \) for some \( n \geq 2 \), let \( \beta,\kappa \geq 0 \) satisfy~\ref{assumption: 3}, let \( \gamma \) be a path, and let  \( \gamma_0 \in C^1(B_N)\) be any path with \( \partial \gamma_0 = -\partial \gamma. \) Then
    \begin{equation} 
        \begin{split}
            &
            \Bigl| \bigl\langle L_\gamma(\sigma,\phi) \bigr\rangle_{\beta,\kappa,\infty} - \Theta_{\beta,\kappa}(\gamma) H_\kappa(\gamma) \Bigr|
            \\&\qquad\leq 
            K_6 
            \Biggl( \alpha_2(\beta,\kappa) + \sqrt{\frac{\max(1,|\support \gamma_c|)}{|\support \gamma|}} \Biggr)^{|\support (\gamma-\gamma_c)|/(|\support (\gamma-\gamma_c)|+2|\support \gamma| )},
        \end{split}
    \end{equation}
    where
    \begin{equation*}
        \Theta_{\beta,\kappa}(\gamma) \coloneqq \pigl\langle \,\prod_{e \in \gamma}\theta_{\beta,\kappa}\bigl( \sigma(e)-\phi(\partial e)\bigr) \pigr\rangle_{\infty,\kappa,\infty}, 
    \end{equation*}
    \begin{equation*}
        H_\kappa(\gamma) \coloneqq 
        \pigl\langle L_{\gamma} (\sigma,\phi)\pigr\rangle_{\infty,\kappa,\infty},
    \end{equation*} 
    \begin{equation}\label{eq: K6}
        \begin{split}
            &K_6 \coloneqq
            2^{2|\support \gamma|/(2|\support \gamma|+|\support (\gamma-\gamma_c)| )} 
            \\&\qquad\qquad \cdot
            \Biggr[ 
            \mathbb{1}(\partial \gamma \neq 0) \cdot  2K_3 K_4^8 \alpha_0(\kappa)^7  \sum_{e \in \gamma}
            \Bigl( 18^2\bigl(2+\alpha_0(\kappa)\bigr) \alpha_0(\kappa) \Bigr)^{\max(0,\dist_0(e,\support \gamma_0)-8)} 
            \\&\qquad\qquad\qquad\qquad \cdot \biggl(\frac{\alpha_1(\beta)}{\alpha_0(\beta)}\biggr)^6 \cdot  \frac{\alpha_2(\beta,\kappa)^6}{\alpha_5(\beta,\kappa)}
            \\&\qquad\qquad\qquad+  
             K_2  \cdot \frac{ \alpha_2(\beta,\kappa)^{6}}{ \alpha_5(\beta,\kappa)} 
            +
            K_3 K_4^2\, \alpha_0(\kappa)^{5/6} \cdot \biggl(\frac{\alpha_1(\beta)}{\alpha_0(\beta)}\biggr)^7 \cdot \frac{ \alpha_2(\beta,\kappa)^{6}}{ \alpha_5(\beta,\kappa)} 
            \\&\qquad\qquad\qquad+ 
            \frac{18^4 K_5 \alpha_2(\beta,\kappa)^{5}}{2} \cdot \Bigl(\frac{\alpha_1(\beta)}{\alpha_0(\beta)} \Bigr)^{12} \cdot \frac{ \alpha_2(\beta,\kappa)^{6} }{\alpha_5(\beta,\kappa)} 
            \\&\qquad\qquad\qquad+   
            \sqrt{\frac{2 K_8 \,     \alpha_0(\kappa)^5 \alpha_4(\beta,\kappa) \max\bigl(   \alpha_0(\kappa),  \alpha_1(\beta)^6\bigr)}{ \alpha_5(\beta,\kappa)}} 
            \cdot
            \sqrt{\biggl(\frac{\alpha_1(\beta)}{\alpha_0(\beta)}\biggr)^6}
            \cdot \sqrt{\frac{\alpha_2(\beta,\kappa)^6}{\alpha_5(\beta,\kappa)}}
            \\&\qquad\qquad\qquad+
            2\sqrt{\frac{2K_7 \, \alpha_0(\kappa)^8\alpha_4(\beta,\kappa) }{ \alpha_5(\beta,\kappa)}}
            \cdot
            \sqrt{\biggl(\frac{\alpha_1(\beta)}{\alpha_0(\beta)}\biggr)^6}
            \cdot \sqrt{\frac{\alpha_2(\beta,\kappa)^6}{\alpha_5(\beta,\kappa)}}
            \\&\qquad\qquad\qquad+
            \sqrt{12 K_2} \cdot \sqrt{\frac{ \alpha_2(\beta,\kappa)^6}{ \alpha_5(\beta,\kappa)} } \cdot \sqrt{\frac{ \alpha_3(\beta,\kappa)}{ \alpha_5(\beta,\kappa)} }  
            +
            \sqrt{\frac{K_{10}\, \alpha_0(\kappa)^8 \alpha_4(\beta,\kappa)}{\alpha_5(\beta,\kappa)}} 
            \\&\qquad\qquad\qquad+
            \sqrt{\frac{\alpha_3(\beta,\kappa)}{\alpha_5(\beta,\kappa)}  } 
            \Biggr]^{|\support (\gamma-\gamma_c)|/(|\support (\gamma-\gamma_c)| + 2|\support \gamma|)},
        \end{split}
    \end{equation} 
    where \( K_2 \) is given by~\eqref{eq: K2}, where \( K_3 \) and \( K_4 \) are given by~\eqref{eq: K3 and K4}, \(  K_5 \) is given by~\eqref{eq: K5}, \( K_7 \) is given by~\eqref{eq: K7}, \( K_8 \) is given by~\eqref{eq: K8 and K9}, and \( K_{10} \) is given by~\eqref{eq: K10}.
\end{theorem}

\begin{remark}\label{eq: Z2 constant in gen thm}
    Using the equations in the beginning of Section~\ref{sec: proof of main theorem Z2}, together with~\eqref{eq: alpha ration estimates}, one easily shows that if \( G = \mathbb{Z}_2 \), then
\begin{equation*}
        \begin{split}
            &K_6 =
            2^{2|\support \gamma|/(2|\support \gamma|+|\support (\gamma-\gamma_c)| )} 
            \\&\qquad\qquad \cdot
            \Biggr[ 
            2K_3 K_4 ^8\alpha_0(\kappa)^7  \sum_{e \in \gamma}
            \bigl( K_4 \alpha_0(\kappa) \bigr)^{\max(0,\dist_0(e,\gamma_0)-8)} 
            \\&\qquad\qquad\qquad+  
             K_2 
            +
            K_3 K_4^2 
            +
            \sqrt{ K_7 } 
            +
            \sqrt{ K_8 } 
            +
            \sqrt{K_{10}  } 
            +
            \sqrt{12 K_2}  
            +
            1
            \Biggr]^{|\support (\gamma-\gamma_c)|/(|\support (\gamma-\gamma_c)| + 2|\support \gamma|)},
        \end{split}
    \end{equation*}
\end{remark}

\subsection{A first application of the coupling}\label{sec: Ising split}

In this section, we split the expected value we are interested in into two parts, later corresponding to the two functions \( \Theta_{\beta,\kappa}(\gamma) \) and \( H_\kappa(\gamma) \) in Theorem~\ref{theorem: main result}. In order to do this, we first define three useful events;
\begin{equation}\label{eq: E1}
    \mathcal{E}_1  \coloneqq
    \bigl\{ (\hat \sigma,\hat \sigma' ) \in \Omega^1(B_N,G) \times \Omega^1_0(B_N,G) \colon
    \exists \text{ irreducible } \bar \sigma \leq \hat \sigma'|_{E_{\hat \sigma,\hat \sigma'}} \text{ that disturbs } \gamma
    \bigr\},
\end{equation} 
\begin{equation}\label{eq: E2}
    \mathcal{E}_2 \coloneqq 
    \bigl\{ (\hat \sigma,\hat \sigma' ) \in \Omega^1(B_N,G) \times \Omega^1_0(B_N,G) \colon \exists \text{ irreducible }\bar \sigma \leq \hat \sigma|_{E_{\hat \sigma,\hat \sigma'}} \text{ that disturbs }  \gamma \bigr\},
\end{equation} 
and
\begin{equation}\label{eq: E3}
    \begin{split}
        &\mathcal{E}_3 \coloneqq 
        \bigl\{ \sigma \in \Omega^1(B_N,G) \colon
        \exists e \in \gamma,\,  \tilde\sigma \leq  \sigma,\,  \tilde\sigma' \leq \sigma -\tilde\sigma \text{ s.t. }
        d\tilde \sigma|_{\pm \support \hat \partial e} \neq 0 \text{ and }   d\tilde \sigma'|_{\pm \support \hat \partial e} \neq 0 \bigr\}.
    \end{split}
\end{equation} %
We provide upper bounds of the probabilities of these events occurring in Section~\ref{sec: upper bounds on bad events}.

Using this notation, we have the following result, which is the main result of this section.
 
\begin{proposition}\label{proposition: Ising LGT split}
    Let \( \beta,\kappa \geq 0 \), and let \( \gamma \in C^1(B_N) \) be a path. For each \( e \in \gamma \), let \( p_e \in \hat \partial e. \)
    Then
    \begin{equation*}
        \begin{split}
            &\biggl| \mathbb{E}_{N,\beta,\kappa}\bigl[  L_\gamma(\sigma) \bigr]
            -
            \mathbb{E}_{N,\infty,\kappa} \bigl[ L_\gamma(\sigma) \bigr]
            \,
            \mathbb{E}_{N,\beta,\kappa} \Bigl[ \,  \prod_{e \in (\gamma - \gamma_c)- \gamma'}   \rho\bigl(d \sigma(p_e)\bigr) \Bigr] \biggr| 
            \\&\qquad\leq 2\mu_{N,(\beta,\kappa),(\infty,\kappa)} ( \mathcal{E}_1) + 2\mu_{N,(\beta,\kappa),(\infty,\kappa)} ( \mathcal{E}_2) + 2\mu_{N,(\beta,\kappa),(\infty,\kappa)} ( \mathcal{E}_3).
        \end{split}
    \end{equation*} 
\end{proposition}

\begin{proof}
    Let \( \hat \sigma \in \Omega^1(B_N,G) \) and \( \hat \sigma' \in \Omega^1_0(B_N,G) \), and let
    \begin{equation*} 
            \sigma  \coloneqq \hat \sigma|_{E_{\hat \sigma,\hat \sigma'}}
            +
            \hat \sigma'|_{C_1(B_N) \smallsetminus E_{\hat \sigma,\hat \sigma'}}. 
    \end{equation*}
    Let \( \hat \Sigma \) be a decomposition of \( \hat \sigma, \) 
    let \( \hat \Sigma'\) be a decomposition of \( \hat \sigma' ,\) 
    and define 
    \begin{equation*}
        \Sigma \coloneqq \{ \hathat\sigma \in \hat \Sigma \colon \support \hathat\sigma \subseteq E_{\hat \sigma,\hat \sigma'} \},
    \end{equation*} 
    \begin{equation*}
        \Sigma' \coloneqq \{ \hathat\sigma \in \hat \Sigma' \colon \support \hathat\sigma \subseteq C_1(B_N) \smallsetminus E_{\hat\sigma,\hat\sigma'} \},
    \end{equation*} 
    \begin{equation*}
        \hat \Sigma_{bad} \coloneqq \{ \hathat\sigma \in \hat \Sigma \colon  \hathat\sigma \text{ disturbs } \gamma \},
    \end{equation*} 
    and
    \begin{equation*}
        \hat \Sigma_{bad}' \coloneqq \{ \hathat\sigma \in \hat \Sigma' \colon \hathat\sigma \text{ disturbs } \gamma \}.
    \end{equation*} 
    Note that these sets depend on \( \hat \sigma \) and \( \hat \sigma'. \)
    Note also that if  \( (\hat \Sigma' \smallsetminus \Sigma') \cap \hat \Sigma_{bad}' \neq \emptyset, \) then \( (\hat \sigma,\hat \sigma') \in \mathcal{E}_1 \), and if \( \Sigma  \cap \hat \Sigma_{bad} = \emptyset, \) then \( (\hat \sigma,\hat \sigma') \in \mathcal{E}_2 \).

    By Proposition~\ref{proposition: docomposition of the coupling}, \( \Sigma \cup \Sigma' \) is a decomposition of \( \sigma. \) This implies in particular that the 1-forms in \( \Sigma \cup \Sigma' \) have disjoint supports, and hence
    \begin{equation}\label{eq: spin decomposition equation}
        L_\gamma(\sigma)
        =
        L_\gamma\Bigl(\, \sum_{\smallhathat \sigma \in \Sigma}  \hathat \sigma+ \sum_{\smallhathat \sigma'\in \Sigma'} \hathat \sigma' \Bigr)
        =
        L_\gamma\Bigl(\, \sum_{\smallhathat \sigma \in \Sigma}  \hathat \sigma\Bigr) L_\gamma \Bigl( \, \sum_{\smallhathat \sigma'\in \Sigma'} \hathat \sigma' \Bigr).
    \end{equation}
    %
    %
    If \( (\hat \sigma,\hat \sigma') \notin\mathcal{E}_1 ,\) then \( (\hat \Sigma' \smallsetminus \Sigma') \cap \hat \Sigma_{bad}' = \emptyset. \) Since \( d\hathat \sigma = 0 \) for all \( \hathat \sigma \in \hat \Sigma',\) and hence, using Lemma~\ref{lemma: removing nondisturbing configurations}, it follows that, on this event, we have 
    
    \begin{equation*}
        L_\gamma \Bigl(  \sum_{\smallhathat \sigma \in \hat \Sigma'\smallsetminus \Sigma'} \hathat \sigma \Bigr) 
        =
        \rho\biggl( \Bigl(  \sum_{\smallhathat \sigma \in \hat \Sigma'\smallsetminus \Sigma'} \hathat \sigma \Bigr)(\gamma)\biggr)
        =
        \rho \Bigl(\sum_{\smallhathat \sigma \in \hat \Sigma'\smallsetminus \Sigma'} \hathat \sigma(\gamma)\Bigr)
        =
        \rho \Bigl(\sum_{\smallhathat \sigma \in \hat \Sigma'\smallsetminus \Sigma'} 0\Bigr)
        =
        1,
    \end{equation*}
    and hence
    \begin{equation*}
        L_\gamma\Bigl( \, \sum_{\smallhathat \sigma \in \Sigma'} \hathat \sigma \Bigr)
        =
        L_\gamma\Bigl( \, \sum_{\smallhathat \sigma \in \Sigma'} \hathat \sigma\Bigr) \cdot L_\gamma \Bigl(  \sum_{\smallhathat \sigma \in \hat \Sigma'\smallsetminus \Sigma'} \hathat \sigma \Bigr)
        =
        L_\gamma\Bigl( \, \sum_{\smallhathat \sigma \in \Sigma'} \hathat \sigma +  \sum_{\smallhathat \sigma \in \hat \Sigma'\smallsetminus \Sigma'} \hathat \sigma \Bigr)
        =
        L_\gamma(\hat \sigma').
    \end{equation*}
    In particular, this shows that
    \begin{equation}\label{eq: upper bound by event 1}
        \mathbb{E}_{N,(\beta,\kappa),(\infty,\kappa)}\Biggl( \biggl| L_\gamma\Bigl( \, \sum_{\smallhathat \sigma \in \Sigma'} \hathat \sigma \Bigr)
        -
        L_\gamma(\hat \sigma') \biggr| \Biggr) \leq 2 \mu_{N,(\beta,\kappa),(\infty,\kappa)}( \mathcal{E}_1).
    \end{equation}
    
    Next, note that since the 1-forms in \( \Sigma \) have disjoint supports, we have
    \begin{equation*}
        L_\gamma \Bigl(\,\sum_{\smallhathat \sigma \in \Sigma} \hathat \sigma \Bigr)
        =
        \rho \biggl(\Bigl(\,\sum_{\smallhathat \sigma \in \Sigma} \hathat \sigma \Bigr)(\gamma)\biggr)
        =
        \rho \Bigl(\,\sum_{\smallhathat \sigma \in \Sigma} \hathat \sigma(\gamma) \Bigr)  
        =
        \prod_{\smallhathat \sigma \in \Sigma} \rho\bigl(\hathat \sigma(\gamma) \bigr).
    \end{equation*}
    For \( \hathat \sigma \in \Omega^1(B_N,G), \) define
    \begin{equation*}
        \gamma'_{\smallhathat \sigma}[e] \coloneqq (\gamma-\gamma_c)[e] \cdot \mathbb{1}\bigl( \exists p,p' \in \hat \partial e \colon d\hathat \sigma(p) \neq d\hathat \sigma(p')\bigr),\quad e \in C_1^+(B_N).
    \end{equation*}
    %
    
    %
    If \( (\hat \sigma,\hat \sigma') \notin \mathcal{E}_2, \) then we have \( \Sigma \cap \hat \Sigma_{bad} = \emptyset. \) Consequently, for any \( \hathat \sigma \in \Sigma \) we can apply Lemma~\ref{lemma: removing nondisturbing configurations} to obtain 
    \begin{equation*} 
        \rho\bigl(\hathat \sigma (\gamma) \bigr)  
        =\prod_{e \in (\gamma-\gamma_c)- \gamma'_{\smallsmallhathat \sigma }} \rho\bigl(d\hathat \sigma (p_e)\bigr).
    \end{equation*} 
    If \( \hathat \sigma \in \hat \Sigma\smallsetminus \Sigma \), then \(d\hathat \sigma = 0. \)
    Consequently, if \( (\hat \sigma,\hat \sigma') \notin \mathcal{E}_2, \) then
    \begin{equation*}
        \prod_{\smallhathat \sigma \in \Sigma} \rho \bigl( \hathat \sigma(\gamma)\bigr) 
        = 
        \prod_{\smallhathat \sigma \in \Sigma}\prod_{e \in (\gamma-\gamma_c)- \gamma'_{\smallsmallhathat \sigma }} \rho\bigl(d\hathat \sigma (p_e)\bigr)
        = 
        \prod_{\smallhathat \sigma \in \hat \Sigma}\prod_{e \in (\gamma-\gamma_c)- \gamma'_{\smallsmallhathat \sigma }} \rho\bigl(d\hathat \sigma (p_e)\bigr).
    \end{equation*}
    We now make a few observations.
    \begin{itemize}
        \item  If \( \hathat \sigma \in \hat \Sigma \) satisfies \( d\hathat\sigma(p) \neq 0 \) for some \( e \in (\gamma-\gamma_c)- \gamma'_{\smallhathat \sigma} \) and \( p \in \hat \partial e \), then, since \( \hat \Sigma \) is a decomposition of \( \hat \sigma \), we must have \( d\hathat \sigma'(p) = 0 \) for all \( \hathat \sigma' \in \hat \Sigma \smallsetminus \{ \hathat \sigma\}. \)
        
        \item If \( \hathat \sigma \in \hat \Sigma, \) then, since \( \hat \Sigma \) is a decomposition of \( \hat \sigma \), we must have \( \hathat \sigma \leq \hat \sigma. \) Consequently, if \( d\hathat \sigma(p) \neq 0\) for some \( p \in C_2(B_N) \),  then  \( d\hathat \sigma(p) = d\hat \sigma(p) . \)
        \item If \( \hathat \sigma,\hathat \sigma' \in \hat \Sigma \) are distinct and \( e \in \gamma'_{\smallhathat \sigma}, \) then either \( e \in \gamma'_{\smallhathat \sigma'} \) or \( d\hathat\sigma(p) = 0 \) for all \( p \in \hat \partial e. \)
    \end{itemize}
    Define 
    \begin{equation*}
        \gamma'''[e] \coloneqq (\gamma-\gamma_c)[e] \cdot \mathbb{1}\bigl( \exists p,p' \in \hat \partial e,\, \hathat \sigma \in  \hat\Sigma \colon d\hathat\sigma(p) \neq d\hathat\sigma(p') \bigr),\quad e \in C_1^+(B_N).
    \end{equation*}
    Combining these observations, it follows that  
    \begin{equation*}
        \prod_{\smallhathat \sigma \in \hat \Sigma}\prod_{e \in (\gamma-\gamma_c)- \gamma'_{\smallsmallhathat \sigma }} \rho\bigl(d\hathat \sigma (p_e)\bigr)
        =
        \prod_{\smallhathat \sigma \in \hat \Sigma}\prod_{\substack{e \in (\gamma-\gamma_c)- \gamma'_{\smallsmallhathat \sigma }\mathrlap{\colon}\\ d\smallhathat\sigma(p_e) \neq 0}} \rho\bigl(d\hathat \sigma (p_e)\bigr)
        =
        \prod_{\smallhathat \sigma \in \hat \Sigma}\prod_{\substack{e \in (\gamma-\gamma_c)- \gamma'_{\smallsmallhathat \sigma }\mathrlap{\colon}\\ d\smallhathat\sigma(p_e) \neq 0}} \rho\bigl(d\hat \sigma (p_e)\bigr)
        =
        \prod_{e \in (\gamma-\gamma_c) - \gamma'''} \rho\bigl(d\hat \sigma (p_e)\bigr).
    \end{equation*}
    Combining the previous equations, it follows that if \( (\hat \sigma,\hat \sigma') \notin\mathcal{E}_2 \), we have
    \begin{equation*}
        L_\gamma\Bigl(\, \sum_{\smallhathat \sigma \in \Sigma}\hathat \sigma \Bigr)
        =
        \prod_{e \in (\gamma-\gamma_c)- \gamma'''}   \rho\bigl(d\hat\sigma(p_e)\bigr),
    \end{equation*} 
    and hence
    \begin{equation}\label{eq: upper bound by event 2}
        \mathbb{E}_{N,(\beta,\kappa),(\infty,\kappa)}\Biggl( \biggl| L_\gamma\Bigl(\, \sum_{\smallhathat \sigma \in \Sigma}\hathat \sigma \Bigr)
        -
        \prod_{e \in (\gamma-\gamma_c)- \gamma'''}   \rho\bigl(d\hat\sigma(p_e)\bigr)
        \biggr| \Biggr) \leq 2 \mu_{N,(\beta,\kappa),(\infty,\kappa)}( \mathcal{E}_2).
    \end{equation}

    We now argue that if \( \gamma'_{\hat \sigma} \neq \gamma''',\) then the event \( \mathcal{E}_3 \) must happen.
    To this end, first assume that \( e \in \gamma_{\hat \sigma}'. \) Then there is \( p,p' \in \hat \partial e \) with \( \hat \sigma (p) \neq \hat \sigma(p'). \) Without loss of generality, we can assume that \( \hat \sigma(p) \neq 0. \) 
    Since \( \hat \Sigma \) is a decomposition of \( \hat \sigma \), there is \( \hathat \sigma \in \hat \Sigma \) with \( \hathat \sigma \leq \hat \sigma \) such that \( d\hathat\sigma(p) = d\hat \sigma(p). \) Since \( \hathat \sigma \leq \hat \sigma, \) we must have either \( d\hathat \sigma(p') = d\hat \sigma(p')  \) or \( d\hathat \sigma(p') = 0. \) Using the assumption that \( d\hat \sigma(p) \neq 0, \) it follows that \( d\hathat \sigma(p) \neq d\hathat \sigma(p') \), and hence \( e \in \gamma''' .\)
    Now, instead assume that \( e \in \gamma'''-\gamma'. \) Then, since \( e \in \gamma''' \), there must exist \( p,p' \in \hat \partial e \) and \( \hathat \sigma \in \hat \Sigma \) such that \( d\hathat \sigma(p) \neq d\hathat\sigma(p').\) Without loss of generality, we can assume that \( d\hathat \sigma(p) \neq 0. \) Since \( \hathat \sigma \in \hat \Sigma, \) we have \( \hathat \sigma \leq \hat \sigma, \) and hence, since \( d\hathat \sigma(p) \neq 0, \) it follows that  \( d\hathat \sigma(p) = d\hat \sigma(p) \neq 0. \)
    Since \( \hathat \sigma \leq \hat \sigma, \) we must have either \( \hathat\sigma(p') = \hat\sigma(p') \) or \( \hathat\sigma(p') = 0. \) Since \( e \in \gamma', \) we  have \( d\hat \sigma(p) = d\hat \sigma(p'), \) and hence, since \( d\hat \sigma(p) = d \hathat\sigma(p)\) and \( d\hathat \sigma(p) \neq d \hathat\sigma(p'),\) we conclude that \( d\hathat\sigma(p') = 0. \)
    Since \( d\hat \sigma(p') \neq 0 \) and \( d\hathat \sigma(p') = 0, \) there must exist \( \hathat \sigma' \in \hat \Sigma \smallsetminus \{ \hathat \sigma \} \) such that \( d\hathat \sigma'(p') = d\hat \sigma(p')\neq 0. \)
    To sum up, we have showed that if \( e \in \gamma'''- \gamma, \) then there are distinct \( \hathat \sigma,\hathat\sigma' \in \hat \Sigma \) such that \( \hathat \sigma(p) \neq 0 \) and \( \hathat \sigma'(p') \neq 0. \) Since \( \hathat \sigma,\hathat \sigma' \in \hat\Sigma\) are distinct, using Lemma~\ref{lemma: minimal vortex I} we conclude that if \( \gamma'_{\hat \sigma} \neq \gamma''' \), then \( \mathcal{E}_3 \) holds.
    Consequently, 
    \begin{equation}\label{eq: upper bound by event 3}
        \mathbb{E}_{N,(\beta,\kappa),(\infty,\kappa)}\Biggl( \biggl| \prod_{e \in (\gamma-\gamma_c)- \gamma'''}   \rho\bigl(d\hat\sigma(p_e)\bigr)
        -
        \prod_{e \in (\gamma-\gamma_c)- \gamma'}   \rho\bigl(d\hat\sigma(p_e)\bigr)
        \biggr| \Biggr) \leq 2 \mu_{N,\beta,\kappa}( \mathcal{E}_3).
    \end{equation} 
     
    Combining~\eqref{eq: spin decomposition equation},~\eqref{eq: upper bound by event 1},~\eqref{eq: upper bound by event 2},~and~\ref{eq: upper bound by event 3} we obtain
    \begin{equation*}
        \begin{split}
            &\biggl| \mathbb{E}_{N,\beta,\kappa}\bigl[  L_\gamma(\sigma) \bigr]
            -
            \mathbb{E}_{N,\infty,\kappa} \bigl[ L_\gamma(\sigma) \bigr]
            \,
            \mathbb{E}_{N,\beta,\kappa} \Bigl[ \,  \prod_{e \in (\gamma-\gamma_c)- \gamma'}   \rho\bigl(d\sigma(p_e)\bigr) \Bigr] \biggr| 
            \\&\qquad= 
            \biggl| 
            \mathbb{E}_{N,(\beta,\kappa),(\infty,\kappa)}\Bigl[
            L_\gamma \pigl( \, \sum_{\smallhathat \sigma'\in \Sigma'} \hathat \sigma' \pigr)
            L_\gamma\pigl(\, \sum_{\smallhathat \sigma \in \Sigma}  \hathat \sigma\pigr) 
            -
            L_\gamma(\hat \sigma')  
            \prod_{e \in (\gamma-\gamma_c)- \gamma'}   \rho\bigl(d\hat\sigma(p_e)\bigr)
            \Bigr]
            \biggr|
            \\&\qquad\leq
            \mathbb{E}_{N,(\beta,\kappa),(\infty,\kappa)}\biggl[\Bigl| 
            L_\gamma \pigl( \, \sum_{\smallhathat \sigma'\in \Sigma'} \hathat \sigma' \pigr)
            L_\gamma\pigl(\, \sum_{\smallhathat \sigma \in \Sigma}  \hathat \sigma\pigr) 
            -
            L_\gamma(\hat \sigma')  
            \prod_{e \in (\gamma-\gamma_c)- \gamma'''}   \rho\bigl(d\hat\sigma(p_e)\bigr) 
            \Bigr|
            \biggr]
            \\&\qquad\leq
            \mathbb{E}_{N,(\beta,\kappa),(\infty,\kappa)}\biggl[\Bigl| 
            L_\gamma \pigl( \, \sum_{\smallhathat \sigma'\in \Sigma'} \hathat \sigma' \pigr)
            -
            L_\gamma(\hat \sigma')
            \Bigr|
            \biggr]
            \\&\qquad\qquad+ 
            \mathbb{E}_{N,(\beta,\kappa),(\infty,\kappa)}\biggl[\Bigl| 
            L_\gamma\pigl(\, \sum_{\smallhathat \sigma \in \Sigma}  \hathat \sigma\pigr) 
            -  
            \prod_{e \in (\gamma-\gamma_c)- \gamma'''}   \rho\bigl(d\hat\sigma(p_e)\bigr) 
            \Bigr|
            \biggr]
            \\&\qquad\qquad+ 
            \mathbb{E}_{N,(\beta,\kappa),(\infty,\kappa)}\biggl[\Bigl| 
            \prod_{e \in (\gamma-\gamma_c)- \gamma'''}   \rho\bigl(d\hat\sigma(p_e)\bigr) 
            -  
            \prod_{e \in (\gamma-\gamma_c)- \gamma'}   \rho\bigl(d\hat\sigma(p_e)\bigr) 
            \Bigr|
            \biggr]
            \\&\qquad\leq 2 \mu_{N,(\beta,\kappa),(\infty,\kappa)}( \mathcal{E}_1) + 2 \mu_{N,(\beta,\kappa),(\infty,\kappa)}( \mathcal{E}_2) + 2 \mu_{N,\beta,\kappa}( \mathcal{E}_3).
        \end{split}
    \end{equation*}
    This concludes the proof.
\end{proof}

\subsection{A resampling trick}\label{sec: resampling}

Recall that given a path \( \gamma \) and \( \sigma \in \Omega^1(B_N,G),\) we have let 
\begin{equation*}
    \gamma'[e] = (\gamma-\gamma_c)[e] \cdot \mathbb{1}\bigl( \exists p,p' \in \hat \partial e  \colon d\sigma(p) \neq d\sigma(p') \bigr),\quad e \in C_1(B_N),
\end{equation*} 
In this section, we describe a resampling trick, first introduced (in a different setting) in~\cite{c2019}.

\begin{proposition}[Proposition~10.1 in~\cite{flv2021}]\label{proposition: resampling in main proof}
    Let \( \beta,\kappa \geq 0 \), and let \( \gamma \in C^1(B_N)\) be a path such that \(  \dist_0(\gamma,\partial C_1(B_N))\geq 8\). For each \( e \in \gamma \), fix one plaquette \( p_e \in \hat \partial e \). Then 
    \begin{equation}\label{eq: resampling}
        \mathbb{E}_{N,\beta, \kappa}\biggl[ \rho \Bigl( \sum_{e \in (\gamma-\gamma_c) - \gamma'} d\sigma(p_e) \Bigr) \biggr] 
        =
        \mathbb{E}_{N,\beta, \kappa} \Bigl[ \, \prod_{ e \in (\gamma-\gamma_c) - \gamma'}  \theta_{\beta,\kappa} \bigl( \sigma(e) - d\sigma (p_e) \bigr) \Bigr]. 
    \end{equation}   
\end{proposition}

For a proof of Proposition~\ref{proposition: resampling in main proof}, we refer the reader to~\cite[Proposition~10.1]{flv2021}.  In~\cite[Proposition 10.1]{flv2021}, \( \gamma \) is assumed to be a generalized loop rather than a path as in Proposition~\ref{proposition: resampling in main proof}. However, since the proofs in the two cases are identical, we do not include a proof here.

\subsection{A second application of the coupling}\label{sec:applyingcoupling}

In this section, we take the next step towards the proof of Theorem~\ref{theorem: main result}, by giving an upper bound on the distance between the right hand side of~\eqref{eq: resampling} and \( \Theta'_{N,\beta,\kappa}(\gamma). \) To this end, using the notation of Section~\ref{sec: couplings}, we now introduce a few additional useful events which will be used to express the upper bound in the Proposition~\ref{proposition: E} below.

Given an edge \( e \in C_1(B_N) \), let
\begin{equation}\label{eq: E4}
    \begin{split} 
        &\mathcal{E}_4(e) \coloneqq \bigl\{ (\sigma,\sigma')\in\Omega^1(B_N,G)\times \Omega^1_0(B_N,G) \colon
        e \in E_{\sigma,\sigma'} \text{ and }\sigma'(e) \neq 0 \bigr\},
    \end{split}
\end{equation}
\begin{equation}\label{eq: E5}
    \begin{split}
        &\mathcal{E}_5(e) \coloneqq \bigl\{ (\sigma,\sigma') \in \Omega^1(B_N,G) \times \Omega^1_0(B_N,G) \colon 
        \exists e' \in E_{\sigma,\sigma'} \text{ s.t. } \hat \partial e' \cap \hat \partial e \neq \emptyset
        \\&\qquad\qquad\qquad 
         \text{ and } \exists g \in G\smallsetminus \{ 0 \} \text{ s.t. } \sigma(e) - d\sigma(p) = g \, \forall p \in \hat \partial e \bigr\},
    \end{split}
\end{equation}
\begin{equation}\label{eq: E6}
    \mathcal{E}_6(e) \coloneqq \bigl\{ \sigma' \in \Omega^1_0(B_N,G) \colon
     \sigma'(e) \neq 0 \bigr\},
\end{equation}
\begin{equation}\label{eq: E7}
    \mathcal{E}_7(e) \coloneqq \bigl\{ (\sigma,\sigma') \in \Omega^1(B_N,G) \times \Omega^1_0(B_N,G) \colon
     \exists p,p' \in \hat \partial e \text{ s.t. } d\sigma(p) \neq d\sigma(p') \}.
\end{equation}

\begin{proposition}\label{proposition: E}
    Let \( \beta,\kappa \geq 0 \), and let \( \gamma \in C^1(B_N)\) be a path such that for all \( e \in \gamma \),  \( p \in \hat \partial e \) and \( p' \in \partial C_2(B_N), \) we have \( \support \partial p \cap \support \hat \partial p' = \emptyset. \)
    Then
    \begin{equation} \label{eq: E}
        \begin{split}
            &\biggl| 
            \mathbb{E}_{N,\beta, \kappa} \Bigl[ \, \prod_{ e \in (\gamma-\gamma_c) - \gamma'}  \theta_{\beta,\kappa} \bigl( \sigma(e) - d\sigma(p_e) \bigr) \Bigr]
            - \Theta_{N,\beta,\kappa}(\gamma)
            \biggr| 
            \\&\qquad \leq 
            2\sqrt{2  \alpha_4(\beta,\kappa) \sum_{e \in \gamma} \mu_{N,(\beta,\kappa),(\infty,\kappa)}\bigl(\mathcal{E}_5(e)\bigr)}
            +
            4\sqrt{2  \alpha_4(\beta,\kappa) \sum_{e \in \gamma} \mu_{N,(\beta,\kappa),(\infty,\kappa)}\bigl(\mathcal{E}_4(e)\bigr)}
            \\&\qquad\qquad+
            2\sqrt{2   \alpha_4(\beta,\kappa) \sum_{e \in \gamma_c} \mu_{N,\infty,\kappa}\bigl(\mathcal{E}_6(e)\bigr)}
             + 
             2\sqrt{2|\support \gamma_c| \alpha_3(\beta,\kappa)} 
            \\&\qquad\qquad+ 
             2\sqrt{2 \alpha_3(\beta,\kappa) \sum_{e \in \gamma} \mu_{N,(\beta,\kappa),(\infty,\kappa)} \bigl(\mathcal{E}_7(e)\bigr)}.
        \end{split}
    \end{equation}
\end{proposition}

For the proof of Proposition~\ref{proposition: E} we need a two lemmas from~\cite{flv2021}, which we now recall. 

\begin{lemma}[Lemma~11.2 in~\cite{flv2021}]\label{lemma: Chatterjees inequality ii}
    Assume that \( z_1,z_2,z_1',z_2' \in \mathbb{C} \) are such that \( |z_1|,|z_2|,|z_1'|,|z_2'| \leq 1 \). Then 
    \begin{equation*}
        |z_1z_2-z_1'z_2'| \leq |z_1-z_1'| + |z_2-z_2'|.
    \end{equation*} 
\end{lemma}

\begin{lemma}[Lemma~11.3 in~\cite{flv2021}]  \label{lemma: Chatterjees trick abs} 
    Let $a,b>0$.  Assume that \( A \subseteq C_1(B_N) \) is a random set with \( \mathbb{E}[|A|] \leq a \),  and that
    \begin{enumerate}[label=\textnormal{(\roman*)}]
        \item \( X_e \in \mathbb{C} \) and \( |X_e| \leq 1 \) for all \( e \in C_1(B_N) \), and \label{item: Chatterjees trick abs assump 1}
        \item there exists a \( c \in [-1,1] \) such that \( |X_e-c| \leq b \) for all \( e \in C_1(B_N) \).\label{item: Chatterjees trick abs assump 2}
    \end{enumerate} 
    Then
    \begin{equation*} 
        \mathbb{E}\biggl[ \Bigl| \prod_{e \in A} c -  \prod_{e \in A} X_e \Bigr|  \biggr] 
        \leq
        2\sqrt{2ab} .
    \end{equation*}
\end{lemma}

\begin{proof}[Proof of Proposition~\ref{proposition: E}]
    Recall the coupling \( (\sigma,\sigma') \sim \mu_{N,(\beta,\kappa),(\infty,\kappa)} \) between \( \sigma \sim \mu_{N,\beta,\kappa} \) and \( \sigma' \sim \mu_{N,\infty,\kappa} \) described in Section~\ref{sec: LGT-Z coupling}, and the set \( E_{\sigma,\sigma'} \)  defined in~\eqref{eq: Esigmasigmadef}. Since \( \mu_{N,(\beta,\kappa),(\infty,\kappa)} \) is a coupling of \( \mu_{N,\beta,\kappa} \) and \( \mu_{N,\infty,\kappa} \), we have
    \begin{equation*}
        \begin{split}
        &\mathbb{E}_{N,\beta, \kappa} \Bigl[\, \prod_{ e \in (\gamma-\gamma_c) - \gamma'}  \theta_{\beta,\kappa} \bigl(  \sigma(e) - d\sigma(p_e) \bigr) \bigr]
        -
        \mathbb{E}_{N,\infty,\kappa} \Bigl[ \, \prod_{e \in \gamma } \theta_{\beta,\kappa}\bigl(\sigma(e)\bigr) \Bigr]
        \\&\qquad =  
        \mathbb{E}_{N,(\beta,\kappa),(\infty,\kappa)} \Bigl[\, \prod_{ e \in (\gamma-\gamma_c) - \gamma'}  \theta_{\beta,\kappa}\bigl( \sigma(e) - d\sigma(p_e) \bigr)
        -
        \prod_{e \in \gamma } \theta_{\beta,\kappa}\bigl(\sigma'(e)\bigr) \Bigr].
        \end{split}
    \end{equation*}
    
    Given \( (\sigma,\sigma') \in \Omega^1(B_N,G) \times \Omega^1_0(B_N,G) \), define 
    \begin{equation*}
        \gamma''[e] \coloneqq (\gamma-\gamma_c)[e] \cdot \mathbb{1} \bigl(\exists e' \in E_{\sigma,\sigma'} \colon \hat \partial e' \cap \hat \partial e \neq \emptyset \bigr), \quad e \in C_1(B_N)^+.
    \end{equation*}
    In other words, \( \gamma' \) is the indicator function for all edges is \( \gamma-\gamma_c \) that is adjacent to some edge in \( E_{\sigma,\sigma'}. \)
    By Lemma~\ref{lemma: properties of coupling measure}, if \( e \in (\gamma-\gamma_c) - \gamma'' \), then \( \sigma(e') = \sigma'(e') \) whenever \( e' \in \partial p \) for some \( p \in \hat \partial e \), and  hence  \( d\sigma(p_e) = d\sigma'(p_e) = 0 \). In particular, this implies that if \( e \in (\gamma-\gamma_c) - \gamma'', \) then 
    \begin{equation}\label{eq: uncoupled edges eq}
        \sigma(e)- d\sigma(p_e)  = \sigma'(e)- d\sigma'(p_e)   = \sigma'(e)- 0 =  \sigma'(e).
    \end{equation}
    By the definition of \(\gamma'\), if \(e' \in \gamma'\) then there exists \( p' \in \hat \partial e' \) and  \(e'' \in \partial p' \) such that \(e'' \in \{ e''' \in \support{\sigma} \colon d\sigma|_{\pm \support \hat \partial e'''} \neq 0 \} \subseteq E_{\sigma, \sigma'}\). Consequently, there is \( e'' \in E_{\sigma, \sigma'} \) such that \( \hat \partial e'' \cap \hat \partial e \neq \emptyset. \) 
    Hence, if \( e \in \gamma' \) then \( e\in  \gamma'', \) and it follows that
    \begin{equation*}
        \begin{split}
            &\prod_{ e \in (\gamma-\gamma_c) - \gamma'}  \theta_{\beta,\kappa} \bigl( \sigma(e)- d\sigma(p_e) \bigr)
            =
            \prod_{ e \in (\gamma-\gamma_c) - \gamma''}  \theta_{\beta,\kappa} \bigl(\sigma(e)- d\sigma(p_e) \bigr)
            \prod_{ e \in \gamma''  - \gamma'}  \theta_{\beta,\kappa} \bigl( \sigma(e)- d\sigma(p_e) \bigr)
            \\&\qquad \overset{\eqref{eq: uncoupled edges eq}}{=}
            \prod_{ e \in (\gamma-\gamma_c) - \gamma'}  \theta_{\beta,\kappa}\bigl(\sigma'(e)\bigr)
            \prod_{ e \in \gamma'' - \gamma'}  \theta_{\beta,\kappa} \bigl(\sigma(e)- d\sigma(p_e) \bigr).
        \end{split}
    \end{equation*} 
    Consequently, using Lemma~\ref{lemma: properties of coupling measure}, we have
    \begin{equation*}
        \begin{split}
            &
            \prod_{ e \in (\gamma-\gamma_c) - \gamma'}  \theta_{\beta,\kappa} \bigl( \sigma(e)- d\sigma(p_e) \bigr)
            -  
            \prod_{ e \in \gamma  }  \theta_{\beta,\kappa} \bigl( \sigma'(e) \bigr)
            \\&\qquad=
            \prod_{ e \in (\gamma-\gamma_c) - \gamma''}  \theta_{\beta,\kappa} \bigl(\sigma(e)\bigr)
            \prod_{ e \in \gamma''  - \gamma'}  \theta_{\beta,\kappa} \bigl( \sigma(e)- d\sigma(p_e) \bigr)  
            \\&\qquad\qquad-
            \prod_{ e \in (\gamma-\gamma_c) - \gamma''  }  \theta_{\beta,\kappa} \bigl(\sigma'(e)\bigr)
            \prod_{ e \in \gamma''  - \gamma'  }  \theta_{\beta,\kappa} (\sigma'(e))
            \prod_{ e \in \gamma_c + \gamma' }  \theta_{\beta,\kappa} (\sigma'(e)) 
            \\&\qquad=
            \prod_{ e \in (\gamma-\gamma_c) - \gamma''}  \theta_{\beta,\kappa} (\sigma'(e))
            \biggl( \, \prod_{ e \in \gamma''  - \gamma'}  \theta_{\beta,\kappa} \bigl( \sigma(e)- d\sigma(p_e) \bigr)  
            -
            \prod_{ e \in \gamma''  - \gamma'  }  \theta_{\beta,\kappa} (\sigma'(e)) \biggr)
            \\&\qquad\qquad
            +
            \prod_{ e \in (\gamma-\gamma_c) - \gamma'  }  \theta_{\beta,\kappa} \bigl(\sigma'(e)\bigr) 
            \Bigl( 1-\prod_{ e \in \gamma_c  }  \theta_{\beta,\kappa}(\sigma'(e))\Bigr)
            \\&\qquad\qquad+
            \prod_{ e \in (\gamma-\gamma_c) - \gamma'  }  \theta_{\beta,\kappa} \bigl(\sigma'(e)\bigr) \prod_{ e \in \gamma_c  }  \theta_{\beta,\kappa}\bigl(\sigma'(e)\bigr)
            \Bigl( 1-\prod_{ e \in  \gamma' }  \theta_{\beta,\kappa}\bigl(\sigma'(e)\bigr)\Bigr)
            .
        \end{split}
    \end{equation*}
    Combining the above equations, we find that
    \begin{equation*}
        \begin{split}
            &
            \mathbb{E}_{N,\beta, \kappa} \Bigl[\, \prod_{ e \in (\gamma-\gamma_c) - \gamma'}  \theta_{\beta,\kappa} \bigl(  \sigma(e) - d\sigma(p_e) \bigr) \bigr]   - \mathbb{E}_{N,\infty,\kappa} \Bigl[ \,  \prod_{ e \in \gamma  }  \theta_{\beta,\kappa} \bigl(\sigma(e)\bigr) \Bigr] 
            \\&\qquad=  \mathbb{E}_{N,(\beta,\kappa),(\infty,\kappa)}  \biggl[\, \prod_{ e \in (\gamma-\gamma_c) - \gamma''}  \theta_{\beta,\kappa} (\sigma'(e))
            \biggl( \, \prod_{ e \in \gamma''  - \gamma'}  \theta_{\beta,\kappa} \bigl( \sigma(e)- d\sigma(p_e) \bigr)  
            -
            \prod_{ e \in \gamma''  - \gamma'  }  \theta_{\beta,\kappa} (\sigma'(e)) \biggr)\biggr]
            \\&\qquad\qquad+
            \mathbb{E}_{N,(\beta,\kappa),(\infty,\kappa)}  \biggl[\, \prod_{ e \in (\gamma-\gamma_c) - \gamma'  }  \theta_{\beta,\kappa} (\sigma'(e)) 
            \Bigl( 1-\prod_{ e \in \gamma_c  }  \theta_{\beta,\kappa}(\sigma'(e))\Bigr) \biggr]
           \\&\qquad\qquad +
            \mathbb{E}_{N,(\beta,\kappa),(\infty,\kappa)}  \biggl[\,\prod_{ e \in (\gamma-\gamma_c) - \gamma'}  \theta_{\beta,\kappa} (\sigma'_e) \prod_{ e \in \gamma_c  }  \theta_{\beta,\kappa}(\sigma'(e))
            \Bigl( 1-\prod_{ e \in  \gamma' }  \theta_{\beta,\kappa}(\sigma'(e))\Bigr)\bigg].
        \end{split}
    \end{equation*}
    Now note that
    \begin{align*}
        & \prod_{ e \in \gamma''  - \gamma'}  \theta_{\beta,\kappa} \bigl( \sigma(e)- d\sigma(p_e) \bigr)  
            -
            \prod_{ e \in \gamma''  - \gamma'}  \theta_{\beta,\kappa} (\sigma'(e))
            \\&\qquad=
            \prod_{\substack{ e \in \gamma''  - \gamma' \mathrlap{\colon}\\ \sigma(e) - d\sigma(p_e)=0}}  \theta_{\beta,\kappa} ( 0 ) 
            \biggl( \, \prod_{\substack{ e \in \gamma''  - \gamma' \mathrlap{\colon}\\ \sigma(e) - d\sigma(p_e)\neq 0}}   \theta_{\beta,\kappa} \bigl( \sigma(e)- d\sigma(p_e) \bigr)  -\prod_{\substack{ e \in \gamma''  - \gamma' \mathrlap{\colon}\\ \sigma(e) - d\sigma(p_e)\neq 0}} \theta_{\beta,\kappa}(0)\biggr) 
            \\&\qquad\qquad+
            \prod_{\substack{ e \in \gamma''  - \gamma' \mathrlap{\colon}\\ \sigma'(e)=0  }}  \theta_{\beta,\kappa} (0) \biggl(\,\prod_{\substack{ e \in \gamma''  - \gamma' \mathrlap{\colon}\\ \sigma'(e)\neq 0  }}  \theta_{\beta,\kappa} (0)
            - \prod_{\substack{ e \in \gamma''  - \gamma' \mathrlap{\colon}\\ \sigma'(e)\neq 0  }}  \theta_{\beta,\kappa} (\sigma'(e))\biggr),
    \end{align*}
    and similarly, that
    \begin{equation*}
        1 - \prod_{e \in \gamma_c} \theta_{\beta,\kappa}(\sigma'(e))
        =
        \bigl( 1 - \theta_{\beta,\kappa}(0)^{|\support \gamma_c|}\bigr)
        + \prod_{\substack{e \in \gamma_c \mathrlap{\colon}\\ \sigma'(e)=0}} \theta_{\beta,\kappa}(0) \biggl( \prod_{\substack{e \in \gamma_c \mathrlap{\colon}\\ \sigma'(e)\neq0}} \theta_{\beta,\kappa}(0)-\prod_{\substack{e \in \gamma_c \mathrlap{\colon}\\ \sigma'(e)\neq0}} \theta_{\beta,\kappa}(\sigma'(e)) \biggr)
    \end{equation*}
    and
    \begin{equation*}
        1 - \prod_{e \in \gamma'} \theta_{\beta,\kappa}(\sigma'(e))
        =
        \bigl( 1 - \theta_{\beta,\kappa}(0)^{|\support \gamma'|}\bigr)
        + 
        \prod_{\substack{e \in \gamma' \mathrlap{\colon}\\ \sigma'(e)=0}} \theta_{\beta,\kappa}(0) \biggl( \prod_{\substack{e \in \gamma' \mathrlap{\colon}\\ \sigma'(e)\neq0}} \theta_{\beta,\kappa}(0)-\prod_{\substack{e \in \gamma' \mathrlap{\colon}\\ \sigma'(e)\neq0}} \theta_{\beta,\kappa}(\sigma'(e)) \biggr).
    \end{equation*}
    Consequently, by applying the triangle inequality and noting that, since \( \rho \) is unitary, we have \( |\theta_{\beta,\kappa}(g)| \leq 1 \) for all \( g \in G \), we obtain
    \begin{equation}\label{eq: eq before claims ii}
        \begin{split} 
            &\biggl| \mathbb{E}_{N,\beta, \kappa} \Bigl[\, \prod_{ e \in (\gamma-\gamma_c) - \gamma'}  \theta_{\beta,\kappa} \bigl(  \sigma(e) - d\sigma(p_e) \bigr) \bigr]    - \mathbb{E}_{N,\infty,\kappa} \Bigl[ \,  \prod_{ e \in \gamma  }  \theta_{\beta,\kappa} \bigl(\sigma(e)\bigr) \Bigr]  \biggr|
            \\&\qquad\leq 
            \mathbb{E}_{N,(\beta,\kappa),(\infty,\kappa)}  \biggl[ 
            \Bigl|  \prod_{\substack{ e \in \gamma''  - \gamma' \mathrlap{\colon} \\ \sigma(e)- d\sigma(p_e) \neq 0}}   \theta_{\beta,\kappa} \bigl(\sigma(e)- d\sigma(p_e) \bigr)   - \prod_{\substack{ e \in \gamma''  - \gamma' \mathrlap{\colon} \\ \sigma(e)- d\sigma(p_e) \neq 0}} \theta_{\beta,\kappa}(0) 
            \Bigr|\biggr]
            \\&\qquad\qquad
            +
          \mathbb{E}_{N,(\beta,\kappa),(\infty,\kappa)}  \biggl[  
            \Bigl|  \prod_{\substack{ e \in \gamma''\smallsetminus \gamma' \mathrlap{\colon} \\ \sigma'(e) \neq 0}}   \theta_{\beta,\kappa}(0)  - \prod_{\substack{ e \in \gamma''\smallsetminus \gamma' \mathrlap{\colon} \\ \sigma'(e) \neq 0}}  \theta_{\beta,\kappa} \bigl(\sigma'(e)\bigr)   \Bigr|\biggr]
             \\&\qquad\qquad+   
           \mathbb{E}_{N,(\beta,\kappa),(\infty,\kappa)}  \biggl[ 
            \Bigl|  \prod_{\substack{e \in \gamma_c \mathrlap{\colon}\\ \sigma'(e)\neq 0}} \theta_{\beta,\kappa}(0)-\prod_{\substack{e \in \gamma_c \mathrlap{\colon}\\ \sigma'(e)\neq0}} \theta_{\beta,\kappa}\bigl(\sigma'(e)\bigr) \Bigr|\biggr]
             + 
             \bigl( 1 - \theta_{\beta,\kappa}(0)^{|\support \gamma_c|}\bigr) 
             \\&\qquad\qquad+   
           \mathbb{E}_{N,(\beta,\kappa),(\infty,\kappa)}  \biggl[ 
            \Bigl|  \prod_{\substack{e \in \gamma' \mathrlap{\colon}\\ \sigma'(e)\neq 0}} \theta_{\beta,\kappa}(0)-\prod_{\substack{e \in \gamma' \mathrlap{\colon}\\ \sigma'(e)\neq 0}} \theta_{\beta,\kappa}\bigl(\sigma'(e)\bigr) \Bigr|\biggr]
             + 
             \bigl( 1 - \theta_{\beta,\kappa}(0)^{|\support \gamma'|}\bigr) .
        \end{split}
    \end{equation}
    We now use Lemma~\ref{lemma: Chatterjees trick abs} to obtain upper bounds for each of the terms on the right-hand side of the previous equation.

    \begin{sublemma}
        \begin{equation}\label{eq: gamma'' bound}
            \mathbb{E}_{N,(\beta,\kappa),(\infty,\kappa)} \Bigl[ \bigl| \{ e \in \gamma'' \colon \sigma'(e) \neq 0 \} \bigr| \Bigr] 
            \leq 
            \sum_{e \in\gamma} \mu_{N,(\beta,\kappa),(\infty,\kappa)} \bigl( \mathcal{E}_4(e) \bigr).
        \end{equation} 
    \end{sublemma}
    
    \begin{subproof}
        For any \( e \in \gamma'' \), by definition, 
        there is at least one \(  e' \in E_{\sigma,\sigma'} \) such that \( {\hat \partial e' \cap \hat \partial e \neq \emptyset.} \) Consequently, by the definition of \( E_{\sigma,\sigma'}, \) if \( \sigma'(e) \neq 0 \), we \( e \in E_{\sigma,\sigma'} \).  
        From this it follows that
        \begin{equation*}
            \begin{split}
                & \mathbb{E}_{N,(\beta,\kappa),(\infty,\kappa)} \Bigl[ \bigl| \{ e \in \gamma'' \colon \sigma'(e) \neq 0 \} \bigr| \Bigr] 
                \\&\qquad= 
                \sum_{e \in\gamma''} \mu_{N,(\beta,\kappa),(\infty,\kappa)} \Bigl( \bigl\{ (\sigma,\sigma')\in\Omega^1(B_N,G)\times \Omega^1_0(B_N,G) \colon \sigma'(e) \neq 0 \bigr\} \Bigr)
                \\&\qquad=
                \sum_{e \in\gamma''} \mu_{N,(\beta,\kappa),(\infty,\kappa)} \Bigl( \bigl\{ (\sigma,\sigma')\in\Omega^1(B_N,G)\times \Omega^1_0(B_N,G) \colon e \in E_{\sigma,\sigma'} \text{ and }\sigma'(e) \neq 0  \bigr\} \Bigr)
                \\&\qquad=
                \sum_{e \in\gamma''} \mu_{N,(\beta,\kappa),(\infty,\kappa)} \bigl( \mathcal{E}_4(e) \bigr)
                \leq
                \sum_{e \in\gamma} \mu_{N,(\beta,\kappa),(\infty,\kappa)} \bigl( \mathcal{E}_4(e) \bigr),
            \end{split}
        \end{equation*} 
        which is the desired conclusion.
    \end{subproof}
    
    \begin{sublemma}
        \begin{equation}
            \mathbb{E}_{N,(\beta,\kappa),(\infty,\kappa)}\Bigl[ \bigl| \{ e \in \gamma'' - \gamma' \colon \sigma(e) - d\sigma(p_e) \neq 0 \} \bigr|\Bigr] 
            \leq 
            \sum_{e \in \gamma}\mu_{N,(\beta,\kappa),(\infty,\kappa)}\bigl( \mathcal{E}_5(e)\bigr).
        \end{equation}
    \end{sublemma}
    
    \begin{subproof}
        Note fist that by definition, for any \( e \in \gamma '' - \gamma', \) we have
        \begin{equation*}
             \bigl\{ (\sigma,\sigma') \in \Omega^1(B_N,G) \times \Omega^1_0(B_N,G) \colon \sigma(e) - d\sigma(p_e) \neq 0  \bigr\} = \mathcal{E}_5(e).
        \end{equation*}
        Consequently,
        \begin{equation*}
            \begin{split}
                & \mathbb{E}_{N,(\beta,\kappa),(\infty,\kappa)}\Bigl[ \bigl| \{ e \in \gamma'' - \gamma' \colon \sigma(e) - d\sigma(p_e) \neq 0 \} \bigr|\Bigr] 
                \\&\qquad=
                \sum_{e \in \gamma''  - \gamma'}\mu_{N,(\beta,\kappa),(\infty,\kappa)}\Bigl( \bigl\{ (\sigma,\sigma') \in \Omega^1(B_N,G) \times \Omega^1_0(B_N,G) \colon \sigma(e) - d\sigma(p_e) \neq 0  \bigr\}\Bigr)
                \\&\qquad=
                \sum_{e \in \gamma''  - \gamma'}\mu_{N,(\beta,\kappa),(\infty,\kappa)}\Bigl( \mathcal{E}_5^4\Bigr) 
                \leq
                \sum_{e \in \gamma}\mu_{N,(\beta,\kappa),(\infty,\kappa)}\bigl( \mathcal{E}_5(e)\bigr),
            \end{split}
        \end{equation*} 
        which is the desired conclusion.
    \end{subproof}
    Next, by definition, we have
    \begin{equation*}
        \mathbb{E}_{N,(\beta,\kappa),(\infty,\kappa)} \Bigl[ \bigl| \{ e \in \gamma_c \colon \sigma'(e) \neq 0 \} \bigr| \Bigr] = \sum_{e \in \gamma_c} \mu_{N,\infty,\kappa} \bigl( \mathcal{E}_6(e) \bigr)
    \end{equation*}
    and
    \begin{equation*}
        \begin{split}
            &\mathbb{E}_{N,\beta,\kappa}\bigl[|\support \gamma'|\bigr]
            \\&\qquad= \sum_{e \in \gamma\smallsetminus \gamma_c} \mu_{N,(\beta,\kappa),(\infty,\kappa)} \Bigl( \bigl\{ (\sigma,\sigma') \in \Omega^1(B_N,G)\times \Omega^1_0(B_N,G) \colon \exists p,p' \in \hat \partial e \text{ s.t. } d\sigma(p) \neq d\sigma(p') \bigr\} \Bigr)
            \\&\qquad = \sum_{e \in \gamma\smallsetminus \gamma_c} \mu_{N,(\beta,\kappa),(\infty,\kappa)} \bigl( \mathcal{E}_7(e)\bigr)
            \leq \sum_{e \in \gamma} \mu_{N,
            (\beta,\kappa),(\infty,\kappa)} \bigl( \mathcal{E}_7(e) \bigr).
    \end{split}
    \end{equation*}
    Finally, since for any \( e \in \gamma' \) we also have \( e \in  \gamma'' \), we have
    \begin{equation*}
        \mathbb{E}_{N,(\beta,\kappa),(\infty,\kappa)}\Bigl[\bigl|\{e \in \gamma' \colon \sigma'(e) \neq 0\}\bigr|\Bigr] \leq \mathbb{E}_{N,(\beta,\kappa),(\infty,\kappa)}\Bigl[\bigl|\{e \in \gamma'' \colon \sigma'(e) \neq 0\}\bigr|\Bigr],
    \end{equation*}
    the right-hand side of which we have given an upper bound for in~\eqref{eq: gamma'' bound},

    Applying Lemma~\ref{lemma: Chatterjees trick abs} to the terms in~\eqref{eq: eq before claims ii}, we thus obtain
    \begin{equation*}
        \begin{split} 
            &\biggl| \mathbb{E}_{N,\beta, \kappa} \Bigl[\, \prod_{ e \in (\gamma-\gamma_c) - \gamma'}  \theta_{\beta,\kappa} \bigl(  \sigma(e) - d\sigma(p_e) \bigr) \bigr]    - \mathbb{E}_{N,\infty,\kappa} \Bigl[ \,  \prod_{ e \in \gamma  }  \theta_{\beta,\kappa} \bigl(\sigma(e)\bigr) \Bigr]  \biggr|
            \\&\qquad\leq 
            2\sqrt{2 \max_{g \in G} \bigl|\theta_{\beta,\kappa}(g) - \theta_{\beta,\kappa}(0)\bigr| \sum_{e \in \gamma} \mu_{N,(\beta,\kappa),(\infty,\kappa)}\bigl(\mathcal{E}_5(e)\bigr)}
            \\&\qquad\qquad
            +
            4\sqrt{2 \max_{g \in G} \bigl|\theta_{\beta,\kappa}(g) - \theta_{\beta,\kappa}(0)\bigr| \sum_{e \in \gamma} \mu_{N,(\beta,\kappa),(\infty,\kappa)}\bigl(\mathcal{E}_4(e)\bigr)}
            \\&\qquad\qquad+
            2\sqrt{2 \max_{g \in G} \bigl|\theta_{\beta,\kappa}(g) - \theta_{\beta,\kappa}(0)\bigr| \sum_{e \in \gamma_c} \mu_{N,\infty,\kappa}\bigl(\mathcal{E}_6(e)\bigr)}
            \\&\qquad\qquad
             + 
             2\sqrt{2|\support \gamma_c| \, \bigl|1 - \theta_{\beta,\kappa}(0) \bigr|}
             + 
             2\sqrt{2 \bigl| 1 - \theta_{\beta,\kappa}(0)| \sum_{e \in \gamma} \mu_{N,(\beta,\kappa),(\infty,\kappa)} \bigl(\mathcal{E}_7(e)\bigr)}.
        \end{split}
    \end{equation*}
    Recalling the definitions of \( \alpha_3(\beta,\kappa) \) and \( \alpha_4(\beta,\kappa) \), we obtain the desired conclusion.
\end{proof}

\subsection{Upper bounds on events}\label{sec: upper bounds on bad events}

In this section we provide upper bounds on the events \( \mathcal{E}_1, \)  \( \mathcal{E}_2, \) and \( \mathcal{E}_3, \) defined in Section~\ref{sec: Ising split}, and the events \( \mathcal{E}_4(e), \) \( \mathcal{E}_5(e), \) \( \mathcal{E}_6(e), \) and \( \mathcal{E}_7(e), \) from Section~\ref{sec:applyingcoupling}.

\begin{proposition}\label{proposition: E1}
    Let \( \beta,\kappa \geq 0 \) be such that ~\ref{assumption: 3} hold, let \( \gamma\in C^1(B_N)  \) be a path with \( \dist_0\bigl(\support \gamma,\partial C_1(B_N)\bigr) \geq 8\), let \( \gamma_0 \in C^1(B_N)\) be any path such that \( \partial \gamma_0 = -\partial \gamma, \)  and let \( \mathcal{E}_1 \) be given by~\eqref{eq: E1}. Then
    \begin{equation*}\label{eq: E1 bound}
        \begin{split} 
            \mu_{N,(\beta,\kappa),(\infty,\kappa)}( \mathcal{E}_1) 
            &\leq 
            \mathbb{1}(\partial \gamma \neq 0) \Bigl( K_3 K_4^8\alpha_0(\kappa)^8\alpha_1(\beta)^6 \sum_{e \in \gamma}
            \bigl( K_4 \alpha_0(\kappa) \bigr)^{\max(0,\dist_0(e,\support \gamma_0 )-8)}
            \\&\hspace{13em}+
            K_3 |\support \gamma|\bigl( K_4 \alpha_0(\kappa) \bigr)^{\dist_1(\support \gamma ,\partial C_1(B_N))}\Bigr) ,
        \end{split}
    \end{equation*}
    where \( K_3 \) and \( K_4 \) are given in~\eqref{eq: K3 and K4}.
\end{proposition}

\begin{proposition}\label{proposition: E2}
    Let \( \beta,\kappa \geq 0 \) be such that~\ref{assumption: 3} hold, let \( \gamma \in C^1(B_N) \) be a path such that \( e \in \gamma \) and \( p \in \hat \partial e \) we have \( \dist_0(\support \partial p,\partial C_1(B_N))) \geq 8,\) let \( \gamma_0 \in C^1(B_N)\) be any path such that \( \partial \gamma_0 = -\partial \gamma, \) and let \( \mathcal{E}_2 \) be given by~\eqref{eq: E2}. Then
    \begin{equation*}\label{eq: E2 bound}
        \begin{split}
            &\mu_{N,(\beta,\kappa),(\infty,\kappa)}( \mathcal{E}_2) 
            \\&\qquad\leq  
            \mathbb{1}(\partial \gamma \neq 0)\cdot K_3 K_4^8\alpha_0(\kappa)^8\alpha_1(\beta)^6 \sum_{e \in \gamma}
            \bigl( K_4 \alpha_0(\kappa) \bigr)^{\max(0,\dist_0(e,\support \gamma_0)-8)}
            \\&\qquad\qquad+
            K_2 |\support \gamma_c|\alpha_2(\beta,\kappa)^{6}
            +
            K_3 K_4^2 |\support \gamma| \alpha_0(\kappa)^2   \alpha_1(\beta)^{7}  
            \\&\qquad\qquad+
            4K_3 |\support \gamma|\bigl( K_4 \alpha_0(\kappa) \bigr)^{\dist_1(\support \gamma ,\partial C_1(B_N))}
            ,
        \end{split}
    \end{equation*}
    where \( K_2 \) is defined in~\eqref{eq: K2}, and \( K_3 \) and \( K_4 \) are given in~\eqref{eq: K3 and K4}.
\end{proposition}

\begin{proposition}\label{proposition: E3}
    Let \( \beta,\kappa \geq 0 \) be such that~\ref{assumption: 3} hold, let \( \gamma \in C^1(B_N) \) be a path such for all \( e \in \gamma \) and \( p \in \hat \partial e \) we have  \( \dist_0(\support \partial p,\partial C_1(B_N)) \geq 8 ,\)  and let \( \mathcal{E}_3 \) be defined by~\eqref{eq: E3}. Then
    \begin{equation}\label{eq: coupling and conditions 6}
        \mu_{N,\beta,\kappa}(\mathcal{E}_3) \leq 18^4 K_5 |\support \gamma| \alpha_0(\kappa)^2 \alpha_1(\beta)^{12}
    \end{equation}
    where \( K_5 \) is given by~\eqref{eq: K5}.
\end{proposition}

\begin{proposition}[Proposition~7.10 in~\cite{flv2021}]\label{proposition: E4}
    Let \( \beta,\kappa \geq 0 \) be such that~\ref{assumption: 3} holds, let \( e \in C_1(B_N) \) be such that the support of \(   \hat \partial e \) contains no boundary plaquettes of \( B_N, \) and let \( \mathcal{E}_4(e) \) be given by~\eqref{eq: E4}.
    Then
    \begin{equation*}
        \mu_{N,(\beta,\kappa),(\infty,\kappa)}\bigl( \mathcal{E}_4(e) \bigr) 
        \leq
        K_7  \alpha_0(\kappa)^9\alpha_1(\beta)^6
        +
        K_3  \bigl( K_4 \alpha_0(\kappa)  \bigr)^{\dist_1(e,\partial C_1(B_N))} , 
    \end{equation*}
    where
    \begin{equation}\label{eq: K7}
        K_7 
        \coloneqq 
        18^{10} K_3  \Bigl( 2^8 \alpha_0(\kappa)^{-1} \pigl( \bigl( 1+\alpha_0(\kappa)/2 \bigr)^{8} - 1 \pigr) + 2^8
        K_4 \Bigr),
    \end{equation}
    and \( K_3 \) and \( K_4 \) are given by~\eqref{eq: K3 and K4}.
\end{proposition}

\begin{proposition}[Proposition~7.12 in~\cite{flv2021}]\label{proposition: E5}
    Let \( \beta,\kappa \geq 0 \) be such that~\ref{assumption: 3} holds. Next, let \( e \in C_1(B_N) \) be such that the support of \(   \hat \partial e \) contains no boundary plaquettes of \( B_N, \) and let \( \mathcal{E}_5(e) \) be given by~\eqref{eq: E5}.
    Then
    \begin{equation*}
        \mu_{N,(\beta,\kappa),(\infty,\kappa)}\bigl( \mathcal{E}_5(e) \bigr) \leq 
        K_8  \alpha_1(\beta)^6 \alpha_0(\kappa)^6 \max\bigl(   \alpha_0(\kappa),  \alpha_1(\beta)^6\bigr)
            +
           K_9  \bigl( K_4 \alpha_0(\kappa)  \bigr)^{\dist_1(e,\partial C_1(B_N))} , 
    \end{equation*} 
    where
    \begin{equation}\label{eq: K8 and K9}
        K_8 \coloneqq   
        2K_3  \cdot 18^{8} (18^2+1) \bigl(2 + \alpha_0(\kappa) \bigr)^{7} 
        , \quad 
        K_9 \coloneqq K_3  \Bigl( 1 + \pigl( \bigl(2 + \alpha_0(\kappa) \bigr) \alpha_0(\kappa)\pigr)^{-1}\Bigr),
    \end{equation}
    and \( K_3 \) and \( K_4 \) are given by~\eqref{eq: K3 and K4}.
\end{proposition}

\begin{proposition}[Proposition~7.14 in~\cite{flv2021}]\label{proposition: E6}
    Let \( \beta,\kappa \geq 0 \) be such that~\ref{assumption: 3} holds, let \( e \in C_1(B_N) \) be such that \( \dist_0(\support \partial p,\partial C_1(B_N)) \geq 8 \) for all \( p \in \hat \partial e, \) and let \( \mathcal{E}_6(e) \) be given by~\eqref{eq: E6}.
    Then
    \begin{equation}\label{eq: minimal Ising}
        \mu_{N,\infty,\kappa}   \bigl( \mathcal{E}_6(e) \bigr) 
        \leq 
        K_{10} \alpha_0(\kappa)^8,
    \end{equation}
    where
    \begin{equation}\label{eq: K10}
        K_{10} \coloneqq 18^{13} \bigl(1-18^{2}  \alpha_0(\kappa)\bigr)^{-1}.
    \end{equation}
\end{proposition}

\begin{proposition}\label{proposition: E7}
    Let \( \beta,\kappa \geq 0 \) be such that~\ref{assumption: 3} holds, let \( e \in C_1(B_N) \) be such that for all \( p \in \hat \partial e \), the support of \( \partial p \) contains no boundary edges of \( B_N, \) and let \( \mathcal{E}_7(e) \) be given by~\eqref{eq: E7}. Then
    \begin{equation*}
        \mu_{N,(\beta,\kappa),(\infty,\kappa)}\bigl( \mathcal{E}_7(e) \bigr) \leq 6 K_2 \, \alpha_2(\beta,\kappa)^6,
    \end{equation*}
    where  \( K_2\) is given by~\eqref{eq: K2}.
\end{proposition}

Before we provide proofs of Propositions~\ref{proposition: E1},~\ref{proposition: E2},~\ref{proposition: E3},~and ~\ref{proposition: E7}, we state and prove the following lemma, which will be useful in these proofs.

\begin{lemma}\label{lemma: last new lemma}
    Let \( \gamma \in C^1(B_N) \) be a path with \( \partial \gamma \neq 0 \) and \( \dist_0(\support \gamma,\partial C_1(B_N))) \geq 8, \) let \( \sigma \in \Omega^1_0(B_N,G) \) be such that the support of any path \( \hat \gamma \in C^1(B_N) \) with \( \partial \hat \gamma = - \partial \gamma \) intersects \( \support \sigma. \) Then, for any \( e \in \support \sigma \) and any path \( \gamma_0 \in C^1(B_N)\) such that \( \partial \gamma_0 = -\partial \gamma, \) we have
    \begin{equation*}
        \bigl|(\support \sigma)^+ \bigr| \geq \max\bigl(\dist_0(e,\support \gamma_0),8\bigr).
    \end{equation*}
\end{lemma}

\begin{proof}
    Since \( \partial \gamma \neq 0 \) and the support of any path \( \hat \gamma \in C^1(B_N) \) with \( \partial \hat \gamma = - \partial \gamma \) intersects \( \support \sigma, \) we must have \( \support \sigma \neq \emptyset. \) Since \( \sigma \in \Omega^1_0(B_N,G), \) it thus follows from Lemma~\ref{lemma: small 1forms} that
    \begin{equation*}
        \bigl|(\support \sigma)^+ \bigr| \geq 8.
    \end{equation*}
    Using the definition of \( \gamma_0, \) the desired conclusion now follows.
\end{proof}

\begin{proof}[Proof of Proposition~\ref{proposition: E1}]
    Assume first that \( (\hat \sigma, \hat \sigma') \in \mathcal{E}_1. \) Then there exists an irreducible {\(1\)-form} \(  \hathat\sigma' \leq \hat \sigma'|_{E_{\hat \sigma,\hat \sigma'}} \) that disturbs \( \gamma. \)

    By Lemma~\ref{lemma: cluster is subconfig}, we have \( \hat \sigma'|_{E_{\hat \sigma,\hat \sigma'}}  \leq  \hat \sigma' \), and hence, since \(   \hathat \sigma' \leq \hat \sigma'|_{E_{\hat \sigma,\hat \sigma'}} \), it follows from Lemma~\ref{lemma: the blue lemma}\ref{property 3} that \(   \hathat \sigma' \leq \hat \sigma'. \) Since \( \hat \sigma' \in \Omega^1_0(B_N,G) \), we have \( d\hat \sigma' = 0, \) and hence we conclude that \( d \hathat \sigma' = 0. \)

    Since \( \hathat \sigma' \) disturbs \( \gamma, \) the set \( \support \gamma \cap \support \hathat \sigma' \) must be non-empty (since otherwise,  we could let \( \hat \gamma = -\gamma \) and \( \bar \sigma = \hathat \sigma' \) in Definition~\ref{def: disturbing}). Fix some edge \( e \in \support \gamma \cap \support \hathat \sigma'. \) 

    Since \(  \hathat\sigma' \leq \hat \sigma'|_{E_{\hat \sigma,\hat \sigma'}}, \) we have \( \support  \hathat \sigma' \subseteq E_{\hat \sigma,\hat \sigma'}.\)
    Using the definition of \( E_{\hat \sigma,\hat \sigma'} \), we conclude that \( d\bigl( \hat \sigma|_{ \mathcal{C}_{\mathcal{G}(\hat \sigma,\hat \sigma')}(\support \smallhathat \sigma')} \bigr)\neq 0. \)
    Since \( \hathat \sigma' \) is irreducible and \( e \in \support \hathat \sigma', \)  Lemma~\ref{lemma: irreducible to connected} implies that \( \mathcal{C}_{\mathcal{G}(\hat \sigma,\hat \sigma')}(\support \hathat \sigma') = \mathcal{C}_{\mathcal{G}(\hat \sigma,\hat \sigma')}(e), \) and hence 
    \( d\bigl( \hat \sigma|_{ \mathcal{C}_{\mathcal{G}(\hat \sigma,\hat \sigma')}(e)} \bigr)\neq 0. \)
    Since \( \hathat \sigma'\) is irreducible, satisfies \( d\hathat \sigma' = 0\), and disturbs \( \gamma \), the support of any path \( \hat \gamma \in C^1(B_N) \) with \( \partial \hat \gamma = -\partial \gamma\) must intersect the support of \( \hathat \sigma' \). 
    This implies in particular that we must have \( \partial \gamma \neq 0. \) Applying Lemma~\ref{lemma: last new lemma}, we thus obtain \( \bigl|(\support \hathat \sigma')^+ \bigr| \geq \max\bigl(\dist_0(e,\support \gamma_0),8\bigr),\) and  consequently \( \bigl|\mathcal{G}_{\mathcal{G}(\hat \sigma,\hat \sigma')}(e) \bigr| \geq 2\max\bigl(\dist_0(e,\support \gamma_0),8 \bigr). \)
    To sum up, we have showed that
    \begin{equation*}
        \begin{split}
            &\mu_{N,(\beta,\kappa),((\infty,\kappa))}(\mathcal{E}_1) 
            \\&\qquad \leq 
            \mathbb{1}(\partial \gamma \neq 0) \,
            \mu_{N,\beta,\kappa} \times \mu_{N,\infty,\kappa}  \Bigl( \pigl\{ (\hat \sigma,\hat \sigma') \in \Omega^1(B_N,G) \times \Omega^1_0(B_N,G) \colon \exists e \in \gamma \text{ such that }
            \\&\hspace{11em}
            d\bigl( \hat \sigma|_{ \mathcal{C}_{\mathcal{G}(\hat \sigma,\hat \sigma')}(e)} \bigr)\neq 0 \text{ and } |\mathcal{C}_{\mathcal{G}(\hat \sigma,\hat \sigma')}(e)| \geq 2 \max\bigl(\dist_0(e,\support \gamma_0),8 \bigr) \pigr\}\Bigr)
            \\&\qquad \leq 
            \mathbb{1}(\partial \gamma \neq 0) \,
            \sum_{e \in \gamma}
            \mu_{N,\beta,\kappa} \times \mu_{N,\infty,\kappa}  \Bigl( \pigl\{ (\hat \sigma,\hat \sigma') \in \Omega^1(B_N,G) \times \Omega^1_0(B_N,G)\colon  
            \\[-1.5ex]&\hspace{11em}
            d\bigl( \hat \sigma|_{ \mathcal{C}_{\mathcal{G}(\hat \sigma,\hat \sigma')}(e)} \bigr)\neq 0 \text{ and } |\mathcal{C}_{\mathcal{G}(\hat \sigma,\hat \sigma')}(e)| \geq 2 \max\bigl(\dist_0(e,\support \gamma_0),8 \bigr)\pigr\}\Bigr)
        \end{split}
    \end{equation*}
    Applying Proposition~\ref{proposition: new Z-LGT coupling upper bound} with \( M = \max\bigl(\dist_0(e,\support \gamma_0),8\bigr) \) and \( M' = 1 \) for each \( e \in \gamma \), we obtain~\eqref{eq: E1 bound} as desired. 
\end{proof}

\begin{proof}[Proof of Proposition~\ref{proposition: E2}]
    Assume first that \( (\hat \sigma, \hat \sigma') \in \mathcal{E}_2. \) Then there exists an irreducible 1-form \(  \hathat \sigma' \leq \hat \sigma|_{E_{\hat \sigma,\hat \sigma'}} \) that disturbs \( \gamma. \) 
    Since \( \hathat \sigma' \) disturbs \( \gamma, \) the set \( \support \gamma \cap \support \hathat \sigma' \) must be non-empty (since otherwise,  we could let \( \hat \gamma = -\gamma \) and \( \bar \sigma = \hathat \sigma' \) in Definition~\ref{def: disturbing}). Fix some edge \( e \in \support \gamma \cap \support \hathat \sigma'. \)  
    By definition, we must have \( \support \hathat \sigma' \subseteq E_{\hat \sigma,\hat \sigma'}.\)
    Using the definition of \( E_{\hat \sigma,\hat \sigma'} \), we see that \( d\bigl( \hat \sigma|_{ \mathcal{C}_{\mathcal{G}(\hat \sigma,\hat \sigma')}(\support \smallhathat \sigma')} \bigr)\neq 0. \) Since \( e \in \support \hathat \sigma' \) and \( \hathat \sigma' \) is irreducible, it follows from Lemma~\ref{lemma: irreducible to connected} that \( \mathcal{C}_{\mathcal{G}(\hat \sigma,\hat \sigma')}(\support  \smallhathat \sigma') = \mathcal{C}_{\mathcal{G}(\hat \sigma,\hat \sigma')}(e), \) and hence 
    \( d\bigl( \hat \sigma|_{ \mathcal{C}_{\mathcal{G}(\hat \sigma,\hat \sigma')}(e)} \bigr)\neq 0. \)

    Since \( \hathat \sigma' \) is irreducible and disturbs \( \gamma \), we must be in one of the following four cases. %
    \begin{enumerate}[label=(\arabic*)]
        \item \( \partial \gamma \neq 0, \) and all paths \( \hat \gamma \in C^1(B_N) \) with \( \partial \hat \gamma = - \partial \gamma \) intersects the support of \( \hathat \sigma' \).
        In this case, by Lemma~\ref{lemma: last new lemma}, we must have \( |\support \hathat \sigma' | \geq 2\max\bigl(\dist_0(e,\support \gamma_0),8\bigr),\) and consequently, \( \bigl|\mathcal{G}_{\mathcal{G}(\hat \sigma,\hat \sigma')}(e) \bigr| \geq 2\max\bigl(\dist_0(e,\support \gamma_0),8 \bigr). \)
        
        \item \( \hathat \sigma' \) contains a minimal vortex centered around some edge \( e' \in \gamma_c.\) Since \( \hathat \sigma' \leq \hat \sigma|_{E_{\hat \sigma,\hat \sigma'}} \) by definition, and \( \hat \sigma|_{E_{\hat \sigma,\hat \sigma'}} \leq \hat \sigma \) by Lemma~\ref{lemma: cluster is subconfig}, is follows from Lemma~\ref{lemma: vortex transfer}, applied twice, that \( \hat \sigma \) also contains a minimal vortex centered around some edge \( e' \in \gamma_c.\)
        
        \item \( |(\support d\hathat \sigma')^+|>6. \) In this case, by the same argument as above, we must have both \( {\bigl| \support d\hat \sigma|_{\mathcal{C}_{\mathcal{G}(\hat \sigma,\hat \sigma')}(e)} \geq 2 \cdot 7 }\) and \( |\mathcal{C}_{\mathcal{G}(\hat \sigma,\hat \sigma')}(e)| \geq 2 \cdot 2. \)
        
        \item \( \hathat \sigma' \) supports a vortex \( \nu \) with support at the boundary of \( B_N. \) In this case, by the same argument as above,  we must have  \( \bigl| \mathcal{C}_{\mathcal{G}(\hat \sigma,\hat \sigma')}(e) \bigr| \geq 2 \cdot \dist_1(e,\partial C_1(B_N)). \)
    \end{enumerate} 
    Consequently, we have showed that
    \begin{equation*}
        \begin{split}
            &\mu_{N,(\beta,\kappa),(\infty,\kappa)}( \mathcal{E}_1) 
            \\&\qquad \leq 
            \mathbb{1}(\partial \gamma \neq 0)\,\mu_{N,\beta,\kappa} \times \mu_{N,\infty,\kappa}  \Bigl( \pigl\{ (\hat \sigma,\hat \sigma') \in \Omega^1(B_N,G) \times \Omega^1_0(B_N,G) \colon \exists e \in \gamma \text{ such that }
            \\&\hspace{13em}
            d\bigl( \hat \sigma|_{ \mathcal{C}_{\mathcal{G}(\hat \sigma,\hat \sigma')}(e)} \bigr)\neq 0 \text{ and } |\mathcal{C}_{\mathcal{G}(\hat \sigma,\hat \sigma')}(e)| \geq 2 \max\bigl(\dist_0(e,\support \gamma_0),8 \bigr) \pigr\}\Bigr)
            \\&\qquad\qquad + \mu_{N,\beta,\kappa}\Bigl( \pigl\{ \hat \sigma \in \Omega^1(B_N,G) \colon \exists  e' \in \gamma_c \text{ and } \nu \leq \hat \sigma \text{ s.t. } \nu \text{ is a minimal vortex around } e' \pigr\}\Bigr)
            \\&\qquad\qquad+
            \mu_{N,\beta,\kappa} \times \mu_{N,\infty,\kappa}  \Bigl( \pigl\{ (\hat \sigma,\hat \sigma') \in \Omega^1(B_N,G) \times \Omega^1_0(B_N,G) \colon \exists e \in \gamma \text{ such that }
            \\&\hspace{13em}
            \bigl| d\bigl( \hat \sigma|_{ \mathcal{C}_{\mathcal{G}(\hat \sigma,\hat \sigma')}(e)} \bigr)\bigr| \geq 2 \cdot 7 \text{ and } |\mathcal{C}_{\mathcal{G}(\hat \sigma,\hat \sigma')}(e)| \geq 2\cdot 2 \pigr\}\Bigr)
            \\&\qquad\qquad + 
            \mu_{N,\beta,\kappa} \times \mu_{N,\infty,\kappa}  \Bigl( \pigl\{ (\hat \sigma,\hat \sigma') \in \Omega^1(B_N,G) \times \Omega^1_0(B_N,G) \colon \exists e \in \gamma \text{ such that }
            \\&\hspace{13em}
            d\bigl( \hat \sigma|_{ \mathcal{C}_{\mathcal{G}(\hat \sigma,\hat \sigma')}(e)} \bigr) \neq 0 \text{ and } |\mathcal{C}_{\mathcal{G}(\hat \sigma,\hat \sigma')}(e)| \geq 2 \dist_1(e,\partial C_1(B_N)) \pigr\}\Bigr).
        \end{split}
    \end{equation*}
    By first applying union bounds to all terms, and then using Proposition~\ref{proposition: new Z-LGT coupling upper bound} to upper bound the first, third, and fourth term and Proposition~\ref{proposition: alternative plaquette bound} to upper bound the second term, we obtain
    \begin{equation*}
        \begin{split}
            &\mu_{N,(\beta,\kappa),(\infty,\kappa)}( \mathcal{E}_2) 
            \\&\qquad\leq  
            \mathbb{1}(\partial \gamma \neq 0)\cdot K_3 K_4^8 \alpha_0(\kappa)^8\alpha_1(\beta)^6 \sum_{e \in \gamma}
            \bigl( K_4 \alpha_0(\kappa) \bigr)^{\max(0,\dist_0(e,\support \gamma_0)-8)}
            \\&\qquad\qquad+
            K_3 |\support \gamma|\bigl( K_4 \alpha_0(\kappa) \bigr)^{\dist_0(\support \gamma ,\partial C_1(B_N))}
            +
            K_2 |\support \gamma_c|\alpha_2(\beta,\kappa)^{6}
            \\&\qquad\qquad+K_3 K_4^2 |\support \gamma| \alpha_0(\kappa)^2   \alpha_1(\beta)^{7} 
            \\&\qquad\qquad+
            K_3 |\support \gamma|\bigl( K_4 \alpha_0(\kappa) \bigr)^{\dist_1(\support \gamma,\partial C_1(B_N))} \alpha_1(\beta)^{6}
            \\&\qquad\qquad+
            2K_3|\support \gamma|\bigl( K_4 \alpha_0(\kappa) \bigr)^{\dist_1(\support \gamma,\partial C_1(B_N))}.
        \end{split}
    \end{equation*}
    Simplifying this expression, and noting that by definition, we have \( \alpha_1(\beta)\leq 1 \), we obtain~\eqref{eq: E2 bound} as desired. This concludes the proof.
\end{proof}

\begin{proof}[Proof of Proposition~\ref{proposition: E3}]
    Assume first that \( \sigma \in \mathcal{E}_3, \) and let \( \Sigma \) be a decomposition of \( \sigma. \) Further, let \(e \in \gamma, \) \( p,p' \in \hat \partial e ,\) and \( \hathat\sigma,\hathat\sigma' \in \Sigma \) be such that \( \hathat \sigma \neq \hathat \sigma', \)  \( d\hathat \sigma(p) \neq 0,\) and \( d\hathat \sigma'(p') \neq 0 .\) Without loss of generality we can assume that \( e \in C_1(B_N)^+. \) By Lemma~\ref{lemma: minimal vortex I}, we must then have \( \bigl|(\support d\hathat \sigma)^+\bigr|\geq 6 \) and \( \bigl|(\support d\hathat \sigma')^+\bigr|\geq 6, \) and hence \( \bigl|\bigl(\support d(\hathat \sigma + \hathat\sigma')\bigr)^+\bigr| \geq 12. \)
    
    Define \( E_e \coloneqq \bigl\{ e' \in C_1(B_N) \colon \hat \partial e' \cap \hat \partial e \neq \emptyset \bigr\}. \) 
    Since  \( d\hathat \sigma (p) \neq 0 \), there must exist \( e' \in \partial p \) such that \( \hathat \sigma(e') \neq 0 \). Since \( \hathat \sigma \leq \sigma \), it follows that \( \sigma(e') \neq 0 \). Since \( e' \in \partial p \subseteq E_e \), it follows that \( \mathcal{C}_{\mathcal{G}(\sigma,0)}(E_e) \) is non-empty. Moreover, since \( \hathat \sigma \) is irreducible, using Lemma~\ref{lemma: graph decomposition in coarser i}, it follows that 
    \( \support \hathat \sigma \subseteq \mathcal{C}_{\mathcal{G}(\sigma,0)}(E_e). \) Completely analogously, we also obtain \( \support \hathat \sigma' \subseteq \mathcal{C}_{\mathcal{G}(\sigma,0)}(E_e). \) Since \( d\hathat \sigma(p),d\hathat\sigma'(p')\neq 0 \), by Lemma~\eqref{lemma: minimal vortex I}, we must have \( \bigl|(\support d\hathat \sigma)^+\bigr|,\bigl|(\support d\hathat\sigma')^+\bigr|\geq 6, \) and hence \( |\support d\sigma|_{\mathcal{C}_{\mathcal{G}(\sigma,0)}(E_e)} )^+|\geq 2\cdot 6 = 12. \) Using Lemma~\ref{lemma: 6 plaquettes per edge}, it follows that \( \bigl|(\mathcal{C}_{\mathcal{G}(\sigma,0)}(E_e) )^+\bigr|\geq 12/6=2. \) 
    Combining these observations and using a union bound, we see that
    \begin{equation*}
        \begin{split}
            &\mu_{N,\beta,\kappa} (\mathcal{E}_3)
            \\&\qquad \leq
            \sum_{e \in \gamma}
            \mu_{N,\beta,\kappa}   \Bigl( \pigl\{ \sigma \in \Omega^1(B_N,G)  \colon   
            |\mathcal{C}_{\mathcal{G}( \sigma,0)}(E_e)| \geq 2\cdot 2, \text{ and }\bigl| \support d\bigl( \sigma|_{ \mathcal{C}_{\mathcal{G}( \sigma,0)}(E_e)} \bigr) \bigr| \geq 2\cdot 12
            \pigr\}\Bigr).
        \end{split}
    \end{equation*} 
    Applying Proposition~\ref{proposition: before E3} with \( M = 2 \) and \( M' = 12, \) we obtain~\eqref{eq: coupling and conditions 6} as desired. 
\end{proof}

\begin{proof}[Proof of Proposition~\ref{proposition: E7}]
    Recall that 
    \begin{equation*}
        \begin{split} 
            &\mathcal{E}_7(e) = \bigl\{ (\sigma,\sigma') \in \Omega^1(B_N,G) \times \Omega^1_0(B_N,G) \colon
            \exists p,p' \in \hat \partial e \text{ s.t. } d\sigma(p) \neq d\sigma(p') \}.
        \end{split}
    \end{equation*}
    On this event, there must exist some \( p \in \hat \partial e \) with \( d\sigma(p) \neq 0. \) Since \( |\hat \partial e| = 6, \) together with a union bound, the desired conclusion follows from Proposition~\ref{proposition: alternative plaquette bound}.
\end{proof}

\subsection{A second version of our main result}\label{sec: second version of main result}

In this section, we prove a second version of our main result, by giving a refinement of Proposition~\ref{proposition: short lines}. While the error term in Proposition~\ref{proposition: short lines} corresponds to the probability of the event that no cluster in \( \mathcal{G}(\hat \sigma,\hat \sigma') \) both intersects \( \gamma \) and at the same time supports a vortex, the error term in Proposition~\ref{proposition: first version of main result} below essentially corresponds to the probability that no cluster in \( \mathcal{G}(\hat \sigma,\hat \sigma') \) both intersects \( \gamma \) and at the same time supports a non-minimal vortex.

\begin{proposition}\label{proposition: first version of main result}
    Let \( \beta,\kappa \geq 0 \) be such that~\ref{assumption: 3} hold, let \( \gamma \in C^1(B_N)\) be a path such that \( \dist_0(\support \gamma,\partial C_1(B_N))\geq 8\) and such that for each \( e \in \gamma \) the support of \( \hat \partial e \) contains no boundary plaquettes of \( B_N, \) and let \( \gamma_0 \in C^1(B_N)\) be any path such that \( \partial \gamma_0 = -\partial \gamma. \)
    Then
    \begin{equation}\label{eq: first part of main result}
        \begin{split}
            &\Bigl|\mathbb{E}_{N,\beta,\kappa} \bigl[L_\gamma(\sigma)\bigr]-\mathbb{E}_{N,\infty,\kappa} \bigl[L_\gamma(\sigma)\bigr] \,\Theta_{N,\beta,\kappa}(\gamma)
            \Bigr| 
            \\&\qquad\leq 
            2K_{11} |\support \gamma|\alpha_2(\beta,\kappa)^6 
            + 
           2K_{12}  \sqrt{2|\support \gamma| \alpha_2(\beta,\kappa)^6},
        \end{split}
    \end{equation}
    where
    \begin{align}
        & \label{eq: K11}
        K_{11} \coloneqq 
        \mathbb{1}(\partial \gamma \neq 0) \cdot \frac{2K_3 K_4^8\alpha_0(\kappa)^8\alpha_1(\beta)^6  \sum_{e \in \gamma}
        \bigl( K_4 \alpha_0(\kappa) \bigr)^{\max(0,\dist_0(e,\support \gamma_0)-8)}}{|\support \gamma|\alpha_2(\beta,\kappa)^6}
        \\\nonumber&\qquad+ 
        \frac{4K_3 \bigl( K_4 \alpha_0(\kappa) \bigr)^{\dist_1(\support \gamma ,\partial C_1(B_N))}}{\alpha_2(\beta,\kappa)^6}  
        \\\nonumber&\qquad+
        \frac{ K_2 \, |\support \gamma_c|}{|\support \gamma| }
        +
        \frac{K_3 K_4^2\, \alpha_0(\kappa)^2 \alpha_1(\beta)^{7}}{\alpha_2(\beta,\kappa)^6}
        +
        \frac{18^4 K_5 \alpha_0(\kappa)^2 \alpha_1(\beta)^{12}}{2\alpha_2(\beta,\kappa)^6}
        \\\nonumber&\qquad+
        \sqrt{\frac{2 K_8 \,     \alpha_0(\kappa)^6\alpha_1(\beta)^6 \alpha_4(\beta,\kappa) \max\bigl(   \alpha_0(\kappa),  \alpha_1(\beta)^6\bigr)}{|\support \gamma| \alpha_2(\beta,\kappa)^{12}}}
        +
        \sqrt{\frac{2 K_9 \, \alpha_4(\beta,\kappa)  
        \bigl( K_3 \alpha_0(\kappa)  \bigr)^{\dist_1(\support \gamma,\partial C_1(B_N))} }{|\support \gamma| \alpha_2(\beta,\kappa)^{12}}}
        \\ \nonumber
        &\qquad+
        2\sqrt{\frac{2K_7 \, \alpha_0(\kappa)^9\alpha_1(\beta)^6\alpha_4(\beta,\kappa) }{|\support \gamma| \alpha_2(\beta,\kappa)^{12}}}
        +
        2\sqrt{\frac{2 K_3\, \alpha_4(\beta,\kappa)   
        \bigl( K_4 \alpha_0(\kappa)  \bigr)^{\dist_1(\support \gamma,\partial C_1(B_N))}}{|\support \gamma| \alpha_2(\beta,\kappa)^{12}}}
        \\\nonumber
        &\qquad+ 
        \sqrt{\frac{12 K_2 \,  \alpha_2(\beta,\kappa)^6 \alpha_3(\beta,\kappa)}{|\support \gamma| \alpha_2(\beta,\kappa)^{12}} },
        \\[1ex]\label{eq: K12}
        &
        K_{12} \coloneqq 
        \sqrt{\frac{K_{10}\, |\support \gamma_c|\alpha_0(\kappa)^8 \alpha_4(\beta,\kappa)}{|\support \gamma|\alpha_2(\beta,\kappa)^6}}
        +
        \sqrt{\frac{|\support \gamma_c|\alpha_3(\beta,\kappa)}{|\support \gamma|\alpha_2(\beta,\kappa)^6}  }.
    \end{align}
    \( K_2 \) is given by~\eqref{eq: K2}, \( K_3 \) and \( K_4 \) are given by~\eqref{eq: K3 and K4}, \( K_5 \) is given by~\eqref{eq: K5}, \( K_7 \) is given by~\eqref{eq: K7}, \( K_8 \) and \( K_9 \) are given by~\eqref{eq: K8 and K9}, and \( K_{10} \) is given by~\eqref{eq: K10}.
\end{proposition}

\begin{proof}
    By using first the definition of \( \Theta_{N,\beta,\kappa}(\gamma), \) and then the triangle inequality, we see that
    \begin{equation*}
        \begin{split}
            &\Bigl|\mathbb{E}_{N,\beta,\kappa} \bigl[L_\gamma(\sigma)\bigr]-\mathbb{E}_{N,\infty,\kappa} \bigl[L_\gamma(\sigma)\bigr] \Theta_{N,\beta,\kappa}(\gamma)\Bigr|
            \\&\qquad=
            \Bigl|\mathbb{E}_{N,\beta,\kappa} \bigl[L_\gamma(\sigma)\bigr]-
            \mathbb{E}_{N,\infty,\kappa} \bigl[ L_\gamma(\sigma) \bigr]
            \mathbb{E}_{N,\infty,\kappa} \Bigl[ \, \prod_{e \in \gamma } \theta_{\beta,\kappa}(\sigma_e) \Bigr] \Bigr| 
            \\&\qquad\leq 
            \biggl| \mathbb{E}_{N,\beta,\kappa} \bigl[ L_\gamma(\sigma) \bigr]
            -
            \mathbb{E}_{N,\infty,\kappa} \bigl[ L_\gamma(\sigma) \bigr]
            \,
            \mathbb{E}_{N,\beta,\kappa} \Bigl[ \,  \prod_{e \in (\gamma-\gamma_c)- \gamma'}   \rho\bigl(d\sigma(p_e)\bigr) \Bigr] \biggr|
            \\&\qquad\qquad+
            \Bigl| \mathbb{E}_{N,\infty,\kappa} \bigl[ L_\gamma(\sigma) \bigr] \Bigr| 
            \cdot \biggl| \mathbb{E}_{N,\beta,\kappa} \Bigl[ \,  \prod_{e \in (\gamma-\gamma_c)- \gamma'}   \rho\bigl(d\sigma(p_e)\bigr) \Bigr]
            -
            \mathbb{E}_{N,\beta, \kappa} \Bigl[ \, \prod_{ e \in (\gamma-\gamma_c) - \gamma'}  \theta_{\beta,\kappa} \bigl( \sigma(e) - d\sigma (p_e) \bigr) \Bigr]\biggr|\\&\qquad\qquad+
            \Bigl| \mathbb{E}_{N,\infty,\kappa} \bigl[ L_\gamma(\sigma) \bigr] \Bigr| 
            \cdot \biggl| 
            \mathbb{E}_{N,\beta, \kappa} \Bigl[ \, \prod_{ e \in (\gamma-\gamma_c) - \gamma'}  \theta_{\beta,\kappa} \bigl( \sigma(e) - d\sigma (p_e) \bigr) \Bigr]
            -
            \mathbb{E}_{N,\infty,\kappa} \Bigl[ \, \prod_{e \in \gamma } \theta_{\beta,\kappa}\bigl(\sigma(e)\bigr) \Bigr]
            \biggr|.
        \end{split}
    \end{equation*}
    Since \( |L_\gamma(\sigma)| \leq 1,  \) we can apply Proposition~\ref{proposition: Ising LGT split}, Proposition~\ref{proposition: resampling in main proof} and Proposition~\ref{proposition: E} in order to obtain 
     \begin{equation*} 
        \begin{split}
            &\Bigl|\mathbb{E}_{N,\beta,\kappa} \bigl[L_\gamma(\sigma)\bigr]- \mathbb{E}_{N,\infty,\kappa} \bigl[ L_\gamma(\sigma) \bigr] \mathbb{E}_{N,\infty,\kappa} \Bigl[ \, \prod_{e \in \gamma } \theta_{\beta,\kappa}(\sigma_e) \Bigr] \Bigr| 
            \\&\qquad \leq  
            2\mu_{N,(\beta,\kappa),(\infty,\kappa)} ( \mathcal{E}_1) + 2\mu_{N,(\beta,\kappa),(\infty,\kappa)} ( \mathcal{E}_2)
            +
            2\mu_{N,\beta,\kappa}(\mathcal{E}_3)
            \\&\qquad\qquad+
            2\sqrt{2  \alpha_4(\beta,\kappa) \sum_{e \in \gamma} \mu_{N,(\beta,\kappa),(\infty,\kappa)}\bigl(\mathcal{E}_5(e)\bigr)}
            +
            4\sqrt{2  \alpha_4(\beta,\kappa) \sum_{e \in \gamma} \mu_{N,(\beta,\kappa),(\infty,\kappa)}\bigl(\mathcal{E}_4(e)\bigr)}
            \\&\qquad\qquad+
            2\sqrt{2   \alpha_4(\beta,\kappa) \sum_{e \in \gamma_c} \mu_{N,\infty,\kappa}\bigl(\mathcal{E}_6(e)\bigr)}
            + 
             2\sqrt{2|\support \gamma_c| \alpha_3(\beta,\kappa)} 
            \\&\qquad\qquad+ 
             2\sqrt{2 \alpha_3(\beta,\kappa) \sum_{e \in \gamma} \mu_{N,(\beta,\kappa),(\infty,\kappa)} \bigl(\mathcal{E}_7(e)\bigr)}.
        \end{split}
    \end{equation*}
    Inserting the upper bounds from Proposition~\ref{proposition: E1},  Proposition~\ref{proposition: E2},  Proposition~\ref{proposition: E3}, Proposition~\ref{proposition: E4}, Proposition~\ref{proposition: E5}, Proposition~\ref{proposition: E6}, and Proposition~\ref{proposition: E7}, and using the inequality \( \sqrt{a+b} \leq \sqrt{a} + \sqrt{b},\) we obtain~\eqref{eq: first part of main result} as desired.
\end{proof}

\subsection{An upper bound}

The following result generalizes \cite[Lemma~7.12]{c2019} and \cite[Lemma~3.3]{flv2020}, and is completely analogous to Lemma~12.3 in~\cite{flv2021}.
  
\begin{proposition}[Lemma~12.3 in~\cite{flv2021}]\label{proposition: upper bound} 
    Let \( \beta,\kappa \geq 0 \), and let \( \gamma \in C^1(B_N)\) be a path such that no edge in \( \gamma \) is in the boundary of \( B_N. \) Then
    \begin{equation*}
        \Bigl|\mathbb{E}_{N,\beta,\kappa} \bigl[ L_\gamma(\sigma) \bigr]\Bigr| \leq   
        \exp\bigl( - |\support (\gamma-\gamma_c)| \, \alpha_5(\beta,\kappa) \bigr).
    \end{equation*} 
\end{proposition}

\begin{remark}
    In Lemma~12.3 in~\cite{flv2021}, \( \gamma \) is assumed to be a generalized loop, rather than a path. However, since the proof is identical in the two cases, we do not include a proof here.
\end{remark}

\subsection{A proof of Theorem~\ref{theorem: main result}}

In this section, we give a proof of Theorem~\ref{theorem: main result}. Before we give this proof, we recall the following lemma from~\cite{flv2021}.

\begin{lemma}[Lemma 8.2~in \cite{flv2021}]\label{lemma: theta inequalities Dplusii}
    Let \( \beta,\kappa \geq 0 \), and for each \( g \in G \), let \( j_g >0 \) be given. Further, let \( j \coloneqq \sum_{g \in G} j_g \). Then
    \begin{equation*}
        \Bigl| \prod_{g \in G} \theta_{\beta,\kappa}(g)^{j_g} \Bigr| \leq e^{-j \alpha_5(\beta,\kappa)}.
    \end{equation*}
\end{lemma}

\begin{proof}[Proof of Theorem~\ref{theorem: main result}]
    Let \( N \) be sufficiently large, so that \( \dist_0(\support \gamma,\partial C_1(B_N))\geq 8, \) and so that for each \( e \in \gamma, \) the support of \( \hat \partial e \) contains no boundary plaquettes of \( B_N. \)
    Then the assumptions of Proposition~\ref{proposition: first version of main result} holds, and hence
    \begin{equation*}
        \begin{split}
            &\Bigl|\mathbb{E}_{N,\beta,\kappa} \bigl[L_\gamma(\sigma)\bigr]- \mathbb{E}_{N,\infty,\kappa} \bigl[ L_\gamma(\sigma) \bigr] \Theta_{N,\beta,\kappa}(\gamma) \Bigr|
            \\&\qquad\leq   
            B \cdot 2|\support \gamma|\alpha_5(\beta,\kappa) 
            + 
            B' \cdot 2\sqrt{2|\support \gamma| \alpha_5(\beta,\kappa)}
        \end{split}
    \end{equation*} 
    where
    \begin{equation}\label{eq: BB prime} 
            B \coloneqq K_{11}\alpha_2(\beta,\kappa)^6/\alpha_5(\beta,\kappa) \quad \text{and} \quad
            B' \coloneqq K_{12} \sqrt{\alpha_2(\beta,\kappa)^6/\alpha_5(\beta,\kappa)}, 
    \end{equation}
    where \( K_{11} \) and \( K_{12} \) are given in~\eqref{eq: K11} and~\eqref{eq: K11} respectively.
    Using that for \( x>0 \), we have  \( x \leq e^x, \) and \( 2\sqrt{x}\leq e^x\),  it follows that
    \begin{equation} \label{eq: combined bounds 002}    
        \Bigl|\mathbb{E}_{N,\beta,\kappa} \bigl[L_\gamma(\sigma)\bigr]-\mathbb{E}_{N,\infty,\kappa} \bigl[L_\gamma(\sigma)\bigr] \Theta_{N,\beta,\kappa} (\gamma)  \Bigr| 
        \leq 
        (B+B') e^{2 |\support \gamma| \alpha_5(\beta,\kappa) }.
    \end{equation} 
    
    Now recall that, by Proposition~\ref{proposition: upper bound}, we have \( \pigl|\mathbb{E}_{N,\beta,\kappa} \bigl[ L_\gamma(\sigma) \bigr] \pigr| \leq   e^{-   |\support \gamma - \gamma_c| \alpha_5(\beta,\kappa)}.\) 
    By using the triangle inequality and applying Lemma~\ref{lemma: theta inequalities Dplusii}, it follows that
    \begin{equation}\label{eq: second part of main result}
        \begin{split}
            &\Bigl|\mathbb{E}_{N,\beta,\kappa} \bigl[L_\gamma(\sigma)\bigr]-\mathbb{E}_{N,\infty,\kappa} \bigl[L_\gamma(\sigma)\bigr]  \Theta_{N,\beta,\kappa}(\gamma) \Bigr| \\&\qquad\leq
            \pigl|\mathbb{E}_{N,\beta,\kappa} \bigl[L_\gamma(\sigma)\bigr]\pigr|
            +
            \pigr|\mathbb{E}_{N,\infty,\kappa} \bigl[L_\gamma(\sigma)\bigr] \pigr|
            \cdot
            \bigl| \Theta_{N,\beta,\kappa}(\gamma)
             \bigr|\\&\qquad \leq   
            e^{-  |\support \gamma-\gamma_c | \alpha_5(\beta,\kappa)}+ 1\cdot e^{ -  |\support \gamma| \alpha_5(\beta,\kappa)}
            \leq
            2  e^{ -  |\support \gamma-\gamma_c | \alpha_5(\beta,\kappa)}.
        \end{split}
    \end{equation}
    Combining~\eqref{eq: combined bounds 002} and~\eqref{eq: second part of main result}, we obtain
    \begin{equation}\label{eq: almost last equation}  
        \begin{split}
            &\Bigl|\mathbb{E}_{N,\beta,\kappa}[ L_\gamma(\sigma)] -  \mathbb{E}_{N,\infty,\kappa}[ L_\gamma(\sigma)] \Theta_{N,\beta,\kappa}(\gamma) \Bigr|^{1 + 2|\support \gamma|/|\support (\gamma-\gamma_c)|}
             \\&\qquad\leq
             2^{2|\support \gamma|/|\support(\gamma-\gamma_c)|}(B+B').
        \end{split}
    \end{equation}
    Recalling Proposition~\ref{proposition: unitary gauge one dim} and Proposition~\ref{proposition: limit exists}, and letting \( N \to \infty \), the desired conclusion thus follows from~\eqref{eq: almost last equation} after simplification.  
\end{proof}

\subsection{Simplifications for rectangular paths and \texorpdfstring{\( G = \mathbb{Z}_2 \)}{G = Z2}} 

The purpose of this section is to establish the tools we need in order make the small adjustments to the proof of Theorem~\ref{theorem: main result} needed to instead obtain Theorem~\ref{theorem: main result Z2}.

In order to simplify notations, for \( \beta,\kappa \geq 0 \) and a path \( \gamma, \) we define
\begin{equation*}
    \Theta'_{N,\beta,\kappa}(\gamma) 
    \coloneqq
    e^{-2|\support \gamma|e^{-24\beta-4\kappa}\bigl(1+(e^{8\kappa}-1)|\support \gamma|^{-1}\sum_{e \in \gamma}\mathbb{E}_{N,\infty,\kappa}[\mathbb{1}_{\sigma(e)=1}]\bigr)}.
\end{equation*}

The main result in this section is the following proposition.
\begin{proposition}\label{proposition: rectangular Z2}
    Let \( \beta,\kappa \geq 0 \) be such that~\ref{assumption: 3} and \( 6\beta>\kappa \) both hold, let \( \gamma \) be a path along the boundary of a rectangle with side lengths \( \ell_1 \) and \( \ell_2 \) which is  such that \( \dist_0(\support \gamma,\partial C_1(B_N))\geq 8, \) and let \( G = \mathbb{Z}_2.\)
    Then
    \begin{equation}\label{eq: rectangular Z2}
        \begin{split}
            &\pigl| \Theta_{N,\beta,\kappa}(\gamma) 
            -
            \Theta'_{N,\beta,\kappa}(\gamma)
            \pigr| 
            \leq 
            2\sqrt[3]{\frac{K_{13}\alpha_2(\beta,\kappa)}{|\support \gamma|}} 
            + 
            K_{14}
            |\support \gamma|\alpha_2(\beta,\kappa)^{12}
            ,
        \end{split}
    \end{equation}
    where
         \begin{equation}\label{eq: K13}
        \begin{split}
            &K_{13} \coloneqq \biggl(
             4 + 4K_1 K_4^8  \alpha_0(\kappa)^{4} 
            + 
            4K_3 K_4^4 \alpha_2(\beta,\kappa)^6 \cdot |\support \gamma|    (K_4 \alpha_0(\kappa))^{\min(\ell_1,\ell_2)-4} 
            \\&\qquad\qquad+ 
            4K_3 K_4^4 \alpha_2(\beta,\kappa)^6 \cdot 32(K_4 \alpha_0(\kappa))^{4}
            + 
            4K_3 K_4^4 \alpha_2(\beta,\kappa)^6 \cdot
            \frac{4(K_4 \alpha_0(\kappa))^5}{1-K_4 \alpha_0(\kappa)}  \biggr),
        \end{split}
    \end{equation} 
    \begin{equation}\label{eq: K14}
        K_{14} \coloneqq 4\pigl( 
        1
        +
        K_1(\infty,\kappa) K_4^8
        \bigl(\alpha_0(\kappa ) \bigr)^4 \pigr),
    \end{equation}
    \( K_1 \) is given by~\eqref{eq: K1}, \( K_3 \) and  \( K_4 \) are given by~\eqref{eq: K3 and K4}.
\end{proposition}

The second result which we will state and prove in this section is the following proposition, which will be used to simplify the error term in the proof of Theorem~\ref{theorem: main result Z2}. 
\begin{proposition}\label{proposition: square sum term in constant}
    Let \( \beta,\kappa \geq 0 \) be such that~\ref{assumption: 3} holds, and let \( \gamma \) be an open path along the boundary of a rectangle with side lengths \( \ell_1 \) and \( \ell_2.\)  Then there is a path \( \gamma_0 \in C^1(B_N)\) such that \( \partial \gamma_0 = -\partial \gamma \) and
    \begin{equation*} 
        \sum_{e \in \gamma} \pigl( K_4 \alpha_0(\kappa) \pigr)^{\max(0,\dist_0(e, \support \gamma_0)-8)} \leq 
        16
        +
        \frac{2K_4 \alpha_0(\kappa)}{1-K_4 \alpha_0(\kappa)}
        +
        \frac{|\support \gamma|}{2} \pigl( K_4 \alpha_0(\kappa) \pigr)^{\max(0,\min(\ell_1,\ell_2)-7)},
    \end{equation*}
    where \( K_4 \) is given by~\eqref{eq: K3 and K4}.
\end{proposition}

Before we give a proofs of Proposition~\ref{proposition: rectangular Z2} and Proposition~\ref{proposition: square sum term in constant}, we will state and prove a few useful lemmas. For these lemmas, it will be useful to note that when \( G = \mathbb{Z}_2 \) and \( \rho(G) = \{1,-1\} \), then
\begin{equation}\label{eq: theta for Z2}
    \theta_{\beta,\kappa}(0) = \frac{1-e^{-24\beta-4\kappa}}{1+e^{-24\beta-4\kappa}}
    \quad \text{and} \quad\theta_{\beta,\kappa}(1) = \frac{1-e^{-24\beta+4\kappa}}{1+e^{-24\beta+4\kappa}}.
\end{equation}
From this, it in particular follows that when \( 6\beta> \kappa, \) then \( \theta_{\beta,\kappa}(0),\theta_{\beta,\kappa} >0. \)
Next, we recall from Section~12.2 in~\cite{flv2021}, that when \( G = \mathbb{Z}_2 \) and \( \rho(G) = \{1,-1 \}, \) we have
\begin{equation}\label{eq: alphas for Z2}
    \begin{split}
        &  \alpha_0(r) = \alpha_1(r) = \varphi_r(1)^2 = e^{-4r},
        \qquad    \alpha_2(\beta,\kappa) = e^{-4(\beta+\kappa/6)},
        \\
        &
        \alpha_3(\beta,\kappa) = 
        1 - \theta_{\beta,\kappa}(0)  =
        \frac{ 2e^{-24\beta-4\kappa}}{1 + e^{-24\beta-4\kappa}},
        \qquad
        \alpha_5(\beta,\kappa) = 
        1 - \theta_{\beta,\kappa}(0)  =
        \frac{ 2e^{-24\beta-4\kappa}}{1 + e^{-24\beta-4\kappa}},
        \\
        &    \alpha_4(\beta,\kappa) = 
        \theta_{\beta,\kappa}(0) -  \theta_{\beta,\kappa}(1) 
        = 
        \frac{2 e^{-24\beta} (e^{4\kappa}-e^{-4\kappa})}{(1 + e^{-24\beta-4\kappa})(1 + e^{-24\beta+4\kappa})}.
    \end{split}
\end{equation}

\begin{proof}[Proof of Proposition~\ref{proposition: square sum term in constant}]
    Choose \( \gamma_0 \) so that \( \gamma+ \gamma_0 \) is a generalized loop along the boundary of a rectangle with side lengths \( \ell_1 ,\ell_2 \geq 2. \) 
    Let \( e_1, e_2, \dots, e_{|\support \gamma|} \) be the edges in \( \gamma, \) labelled according to their order in the path \( \gamma. \) Then, for any \( j \in \{ 1,2, \dots, |\support \gamma|\}, \) one verifies that (see Figure~\ref{fig: distance to gap})
    \begin{equation*}
        \dist_0(e,\support \gamma_0) \geq \max(8,\dist_1(e,\support \gamma_0)) \geq \max\pigl(8,\min\bigl(j, |\support \gamma|-j+1,\ell_1,\ell_2\bigr)+1\pigr).
    \end{equation*}
    Using this inequality, we obtain
    \begin{equation*}
        \begin{split}
            &\sum_{e \in \gamma} \pigl( K_4 \alpha_0(\kappa) \pigr)^{\max(0,\dist_0(e, \support \gamma_0)-8)} 
            \leq 
            \sum_{e \in \gamma} \pigl( K_4 \alpha_0(\kappa) \pigr)^{\max(0,\min(j,|\support \gamma|-j+1,\ell_1,\ell_2)+1-8)} 
            \\&\qquad\leq 
            16
            +
            2\sum_{j=9}^\infty \pigl( K_4 \alpha_0(\kappa) \pigr)^{j-8} 
            +
            \frac{|\support \gamma|}{2} \pigl( K_4 \alpha_0(\kappa) \pigr)^{\max(0,\min(\ell_1,\ell_2)-7)}.
        \end{split}
    \end{equation*}
    Evaluating the geometric sum above, we obtain the desired conclusion.
\end{proof}

\begin{figure}[htp]
    \centering 
    \begin{subfigure}[b]{0.45\textwidth}
        \centering
        \begin{tikzpicture}[scale=0.5]
            
            \draw[detailcolor00,-{Straight Barb[length=1.5mm,color=detailcolor00!40!black]}] (5.45,0) -- (5.5,0);
            \draw[detailcolor00] (1,5) -- (0,5) -- (0,0) -- (11,0)  -- (11,1);
            \fill[detailcolor04] (1,5) circle (4pt);
            \fill[detailcolor04] (11,1) circle (4pt);
        \end{tikzpicture}
    \caption{The open path \( \gamma. \)}
    \end{subfigure}
    \hfil
    \begin{subfigure}[b]{0.45\textwidth}
        \centering
        \begin{tikzpicture}[scale=0.5]
            \draw[white] (0,0) -- (11,0);
            \draw[white,-{Straight Barb[length=1.5mm,color=white]}] (5.45,0) -- (5.5,0);
                
            \draw[detailcolor00,-{Straight Barb[length=1.5mm,color=detailcolor00!40!black]}] (5.55,5) -- (5.5,5);
            \draw[detailcolor00] (11,1) -- (11,5) -- (1,5);
            \fill[detailcolor04] (1,5) circle (4pt);
            \fill[detailcolor04] (11,1) circle (4pt);  
        \end{tikzpicture}
        \caption{The open path \( \gamma_0 \)}
    \end{subfigure}
        
    \vspace{2ex}
    \begin{subfigure}[b]{0.45\textwidth}
        \centering
        \begin{tikzpicture}[scale=0.5]
            \draw[detailcolor00, opacity = 0.3] (0,0) -- (11,0) -- (11,5) -- (0,5) -- (0,0); 
            \fill[detailcolor04, opacity=0.3] (1,5) circle (4pt);
            \fill[detailcolor04, opacity=0.3] (11,1) circle (4pt);  \draw[detailcolor00, opacity=0.4,-{Straight Barb[length=1.5mm,color=detailcolor00!40!black]}] (5.55,5) -- (5.5,5);\draw[detailcolor00, opacity=0.4,-{Straight Barb[length=1.5mm,color=detailcolor00!40!black]}] (5.45,0) -- (5.5,0);
                
            \foreach \y in {0,0.5,...,1} 
                \draw[detailcolor00!20!black] (9,\y) -- (9.5,\y);
            \foreach \x in {9.5,10,...,11} 
                \draw[detailcolor00!20!black] (\x,1) -- (\x,1.5);  
        \end{tikzpicture}
    \caption{Given an edge \( e_j \in \gamma,\) we draw the support of a 1-form \( \sigma  \) with \( {\support \gamma_0 \cap \mathcal{C}_{\mathcal{G}(\sigma)}(e_j) \neq \emptyset} \) which minimizes \( |\support \sigma|. \)}
    \label{subfig: not opposite}
    \end{subfigure}
    \hfil
    \begin{subfigure}[b]{0.45\textwidth}
        \centering
        \begin{tikzpicture}[scale=0.5]
            \draw[detailcolor00, opacity = 0.3] (0,0) -- (11,0) -- (11,5) -- (0,5) -- (0,0); 
            \fill[detailcolor04, opacity=0.3] (1,5) circle (4pt);
            \fill[detailcolor04, opacity=0.3] (11,1) circle (4pt);  \draw[detailcolor00, opacity=0.4,-{Straight Barb[length=1.5mm,color=detailcolor00!40!black]}] (5.55,5) -- (5.5,5);\draw[detailcolor00, opacity=0.4,-{Straight Barb[length=1.5mm,color=detailcolor00!40!black]}] (5.45,0) -- (5.5,0);
                
            \foreach \y in {0,0.5,...,5} 
                \draw[detailcolor00!20!black] (2,\y) -- (2.5,\y); 
        \end{tikzpicture}
        \caption{Given an edge \( e_{j'} \in \gamma,\) we draw the support of a 1-form \( \sigma \) with \( {\support \gamma_0 \cap \mathcal{C}_{\mathcal{G}(\sigma)}(e_{j'}) \neq \emptyset} \) which minimizes \( |\support \sigma|. \)}
        \label{subfig: opposite}
    \end{subfigure}
    \caption{In the figures above, we illustrate the setting of the proof of Proposition~\ref{proposition: square sum term in constant}. Note in particular that for the edge \( e_j \) in~\subref{subfig: not opposite}, we have \( \dist_1(e_j,\support \gamma_0) = |\support \gamma|-j+2,\) and for the edge \( e_{j'} \) in~\subref{subfig: opposite}, we have \( \dist_1(e_{j'},\support \gamma_0) = \min(\ell_1,\ell_2)+1.\)}
        \label{fig: distance to gap}
\end{figure}

We now proceed to the proof of Proposition~\ref{proposition: rectangular Z2}. Before we give this proof, we will state and prove a few lemmas.
To simplify the notation in these lemmas, when \( 6\beta> \kappa \) and \( \sigma \sim\mu_{N,\infty,\kappa},\) we define the following random variable.
\begin{equation}\label{def: Upsilon}
    \Upsilon_{\beta,\kappa}(\gamma) \coloneqq |\support \gamma|^{-1}\sum_{e \in \gamma} \log \theta_{\beta,\kappa}\bigl(\sigma(e)\bigr).
\end{equation}

\begin{lemma}\label{lemma: uniform WLLN}
    Let \( \beta,\kappa \geq 0 \) be such that~\ref{assumption: 3} and \( 6\beta> \kappa \) both hold, let \( \gamma \) be path along a rectangle with side lengths \( \ell_1 \) and \( \ell_2\) and such that \( \dist_0(\gamma,\partial C_1(B_N))\geq 8, \) and let \( G = \mathbb{Z}_2. \)
    Then, for any \( \varepsilon > 0 ,\) we have
    \begin{equation*}
        \begin{split}
            & \mu_{N,\infty,\kappa}\Bigl( \pigl|\Upsilon_{\beta,\kappa}(\gamma) - \mathbb{E}_{N,\infty,\kappa} \bigl[ \Upsilon_{\beta,\kappa}(\gamma) \bigr] \pigr| \geq \varepsilon\Bigr) 
            \leq  
            \frac{K_{13}\alpha_2(\beta,\kappa)}{\varepsilon^2|\support \gamma|},
        \end{split}
    \end{equation*} 
    where \( K_{13} \) is given by~\eqref{eq: K13}.
\end{lemma}

\begin{remark}
    The idea of the proof of Lemma~\ref{lemma: uniform WLLN} is essentially to use the weak law of large numbers for correlated random variables with exponential decay. For this to approach to work, we need the loop to be "smooth" enough for the sum of the covariances of all pairs of edges in \( \gamma \) to be finite. The reason for working with rectangular loops is that in this case, it is relatively easy to show that this holds. However, with small modifications, the conclusion of this lemma holds for more general classes of loops as well, as long as the path \( \gamma \) do not have too many corners.
\end{remark}
    
\begin{proof}[Proof of Lemma~\ref{lemma: uniform WLLN}]
    Fix some \( \varepsilon > 0. \) 
    By Chebyshev's inequality, we have
    \begin{equation*}
        \begin{split}
            &\varepsilon^2 \mu_{N,\infty,\kappa}\Bigl( \pigl|\Upsilon_{\beta,\kappa}(\gamma) - \mathbb{E}_{N,\infty,\kappa} \bigl[ \Upsilon_{\beta,\kappa}(\gamma) \bigr] \pigr| \geq \varepsilon\Bigr) \leq  \Var_{N,\infty,\kappa} \bigl[ \Upsilon_{\beta,\kappa}(\gamma) \bigr]
            \\&\qquad = \sum_{e,e' \in \gamma}\Cov_{N,\infty,\kappa} \bigl( \Upsilon_{\beta,\kappa}(e),\Upsilon_{\beta,\kappa}(e') \bigr)
            \\&\qquad =  |\support \gamma|^{-2}  \sum_{e,e' \in \gamma}\Cov_{N,\infty,\kappa} \pigl(  \log \theta_{\beta,\kappa}\bigl(\sigma(e)\bigr),\log \theta_{\beta,\kappa}\bigl(\sigma(e')\bigr) \pigr).
        \end{split}
    \end{equation*}
    By combining Proposition~\ref{proposition: coupling covariance}, applied with \( f_0 = \log \theta_{\beta,\kappa}\bigl( \sigma(e)\bigr) \) and \( f_1 = \log \theta_{\beta,\kappa}\bigl( \sigma(e')\bigr) , \) and Proposition~\ref{proposition: ZZ upper bound}, it follows that
    \begin{equation*}
        \begin{split}
            &\sum_{e,e' \in \gamma}\Cov_{N,\infty,\kappa} \pigl( \log \theta_{\beta,\kappa}\bigl(\sigma(e)\bigr),\log \theta_{\beta,\kappa}\bigl(\sigma(e')\bigr) \pigr)
            \\&\qquad\leq 
            \sum_{e \in \support \gamma} \Var_{N,\beta,\kappa}\pigl( \log \theta_{\beta,\kappa}\bigl( \sigma(e)\bigr) \pigr)
            + 
            2 \bigl\| \log \theta_{\beta,\kappa} \bigr\|_\infty^2 \sum_{e,e' \in \support \gamma\colon e \neq e'} K_3 (K_4 \alpha_0(\kappa))^{\dist_0(e,e')}.
        \end{split}
    \end{equation*} 
    Since \( 0 \leq \theta_{\beta,\kappa}(1) \leq \theta_{\beta,\kappa}(0) \leq 1 \) for all \( \beta,\kappa \geq 0, \) we have
    \begin{equation*}
        \begin{split}
            &\bigl\| \log \theta_{\beta,\kappa} \bigr\|_\infty
            \leq 
            \bigl| \log \theta_{\beta,\kappa}(1) \bigr|
            \leq 
            \bigl| \log \frac{1 - \varphi_\beta(1)^{12}\varphi_\kappa(1)^{-2}}{1 + \varphi_\beta(1)^{12}\varphi_\kappa(1)^{-2}}\bigr|
            \leq  
            \bigl| \log e^{ -2 \varphi_\beta(1)^{12}\varphi_\kappa(1)^{-2}} \bigr|
            \leq  
            2 \varphi_\beta(1)^{12}\varphi_\kappa(1)^{-2}.
        \end{split}
    \end{equation*}
    Next, recall that, by Proposition~\eqref{proposition: new Z-LGT coupling upper bound}, applied with \( M = 1, \) \( M' = 0,\) \( \beta = \kappa_1 =  \infty, \) and \( \kappa_2= \kappa, \) for any edge \( e \in \gamma, \) we have
    \begin{equation*}
        \begin{split}
            &
            \mu_{N,\infty,\kappa}   \pigl( \bigl\{ \hat \sigma  \in \Omega^1_0(B_N,G)\colon  
            \sigma(e) \neq 0 \bigr\}\pigr)
            \leq 
            K_1(\infty,\kappa)
            \bigl(K_4 \alpha_0(\kappa ) \bigr)^8.
        \end{split}
    \end{equation*} 
    Consequently, for any \( e \in \gamma, \) we have 
    \begin{equation*}
    \begin{split}
        &\Var_{N,\infty,\kappa}\pigl( \log \theta_{\beta,\kappa}\bigl( \sigma(e)\bigr) \pigr)
        \leq 
        \mathbb{E}_{N,\infty,\kappa}\Bigl[ \pigl( \log \theta_{\beta,\kappa}\bigl( \sigma(e)\bigr) \pigr)^2 \Bigr]
        \\&\qquad\leq 
        \pigl( \log \theta_{\beta,\kappa}\bigl( 0 \bigr) \pigr)^2
        +
        K_1 \bigl(K_4 \alpha_0(\kappa)\bigr)^8 \pigl( \log \theta_{\beta,\kappa}\bigl( 1\bigr) \pigr)^2
        \\&\qquad\leq 
        4\varphi_\beta(1)^{24} \varphi_\kappa(1)^{4}
        +
        K_1 \bigl(K_4\alpha_0(\kappa)\bigr)^8 \cdot 4 \varphi_\beta(1)^{24} \varphi_\kappa(1)^{-4}
        \\&\qquad\leq  \varphi_\beta(1)^{12} \varphi_\kappa(1)^{2} \bigl( 4 
        +
        4K_1 K_4^8  \alpha_0(\kappa)^{4}  \bigr).
        \end{split}
    \end{equation*}
    Finally, note that
    \begin{equation*}
        \begin{split}
            & \sum_{e,e' \in \support \gamma\colon e \neq e'} (K_4\alpha_0(\kappa))^{\dist_0(e,e')}
            \leq
            \sum_{e,e' \in \support \gamma_R \colon e \neq e'} (K_4\alpha_0(\kappa))^{\dist_0(e,e')}
            \\&\qquad\leq
            |\support \gamma| \Bigl( |\support \gamma| /2 \cdot  (K_4\alpha_0(\kappa))^{\min(\ell_1,\ell_2)} + 2\sum_{j=1}^\infty (K_4\alpha_0(\kappa))^{\min(j,8)} \Bigr) 
            \\&\qquad\leq
            |\support \gamma| \Bigl( |\support \gamma| /2 \cdot  (K_4 \alpha_0(\kappa))^{\min(\ell_1,\ell_2)} + 16(K_4 \alpha_0(\kappa))^{8}
            +
            \frac{2(K_4 \alpha_0(\kappa))^9}{1-K_4 \alpha_0(\kappa)}  \Bigr) .
        \end{split}
    \end{equation*}
    Combining the above equations and recalling that when \( G = \mathbb{Z}_2, \) we have \( \alpha_0(\kappa) = \varphi_\kappa(1)^2 \) and \( \alpha_2(\beta,\kappa)^6 = \alpha_0(\beta)^6 \alpha_0(\kappa)^6 = \varphi_\beta(1)^{12}\varphi_\kappa(1)^{2},\) we finally obtain 
    \begin{equation*}
        \begin{split}
            &\varepsilon^2|\support \gamma| \mu_{N,\infty,\kappa}\Bigl( \pigl|\Upsilon_{\beta,\kappa}(\gamma) - \mathbb{E}_{N,\infty,\kappa} \bigl[ \Upsilon_{\beta,\kappa}(\gamma) \bigr] \pigr| \geq \varepsilon\Bigr) 
            \\&\qquad \leq  
             \alpha_2(\beta,\kappa)^6 \bigl( 4 + 4K_1 K_4^8  \alpha_0(\kappa)^{4}  \bigr)
            \\&\qquad\qquad+ 
            4K_3 K_4^4 \alpha_2(\beta,\kappa)^{12} \Bigl( |\support \gamma|    (K_4 \alpha_0(\kappa))^{\min(\ell_1,\ell_2)-4} + 32(K_4 \alpha_0(\kappa))^{4}
            +
            \frac{4(K_4\alpha_0(\kappa))^5}{1-K_4\alpha_0(\kappa)}  \Bigr).
        \end{split}
    \end{equation*}
    Rearranging this equation, the desired conclusion now immediately follows. This concludes the proof.
\end{proof}

\begin{lemma}\label{lemma: WLLN application}
    Let \( \beta,\kappa \geq 0 \) be such that~\ref{assumption: 3} and \( 6\beta>\kappa\) both hold, let \( \gamma \) be path along the boundary of a rectangle with side lengths \( \ell_1 \) and \( \ell_2, \) and let \( G = \mathbb{Z}_2. \)
    Then 
    \begin{equation*}
        \begin{split}
            &\pigl| \Theta_{N,\beta,\kappa}(\gamma) -   e^{|\support \gamma|  \mathbb{E}_{N,\infty,\kappa} [ \Upsilon_{\beta,\kappa}(\gamma)]}  \pigr| 
            \leq 
            2\sqrt[3]{\frac{K_{13}\alpha_2(\beta,\kappa)}{|\support \gamma|}},
        \end{split}
    \end{equation*}
    where \( K_{13} = K_{13}(\ell_1,\ell_2) \) is given by~\eqref{eq: K13}.
\end{lemma}

\begin{proof}
    Recall the definition of \( \Upsilon_{\beta,\kappa} (\gamma) \) from~\eqref{def: Upsilon}, and note that
    \begin{equation*} 
        \begin{split}
            &\Theta_{N,\beta,\kappa}(\gamma) 
            = 
            \mathbb{E}_{N,\infty,\kappa} \Bigl[ \prod_{e \in \gamma} \theta_{\beta,\kappa}\bigl(\sigma(e)\bigr)\Bigr] 
            = 
            \mathbb{E}_{N,\infty,\kappa} \Bigl[  e^{\sum_{e \in \gamma} \log \theta_{\beta,\kappa}\bigl(\sigma(e)\bigr)}\Bigr] 
            = 
            \mathbb{E}_{N,\infty,\kappa} \Bigl[  e^{|\support \gamma| \Upsilon_{\beta,\kappa}(\gamma)}\Bigr]. 
        \end{split}
    \end{equation*} 
    Consequently, 
    \begin{equation*}
        \begin{split} 
            &\biggl| Theta_{N,\beta,\kappa}(\gamma)  -   e^{|\support \gamma|  \mathbb{E}_{N,\infty,\kappa}[\Upsilon_{\beta,\kappa}(\gamma)]} \biggr|
            \\&\qquad=
            \biggl| \mathbb{E}_{N,\infty,\kappa} \Bigl[  e^{|\support \gamma|  \Upsilon_{\beta,\kappa}(\gamma)} - e^{|\support \gamma|  \mathbb{E}_{N,\infty,\kappa}[\Upsilon_{\beta,\kappa}(\gamma)]}\Bigr] \biggr|
            \\&\qquad\leq
             \mathbb{E}_{N,\infty,\kappa} \biggl[  \Bigl| e^{|\support \gamma|  \Upsilon_{\beta,\kappa}(\gamma)} - e^{|\support \gamma|  \mathbb{E}_{N,\infty,\kappa}[\Upsilon_{\beta,\kappa}(\gamma)]}\Bigr| \biggr].
        \end{split}
    \end{equation*}

    Next, note that since \( \rho \) is unitary, we have \( |\theta_{\beta,\kappa}(g)|\leq 1 \) for all \( g \in G ,\) and hence \(  \Upsilon_{\beta,\kappa}(\gamma) \leq 0. \) 
    Now fix some \( \varepsilon > 0 . \)
    On the event \( \bigl|\Upsilon_{\beta,\kappa}(\gamma) - \mathbb{E}_{N,\beta,\kappa}[\Upsilon_{\beta,\kappa}(\gamma) ]\bigr|\geq\varepsilon, \) since \( \Upsilon_{\beta,\kappa}(\gamma)\leq 0 ,\) we must have
    \begin{equation*}
        \Bigl| e^{|\support \gamma|  \Upsilon_{\beta,\kappa}(\gamma)} - e^{|\support \gamma|  \mathbb{E}_{N,\infty,\kappa}[\Upsilon_{\beta,\kappa}(\gamma)]}\Bigr| \leq 1.
    \end{equation*}
    On the other hand, on the event \( \bigl|\Upsilon_{\beta,\kappa}(\gamma) - \mathbb{E}_{N,\beta,\kappa}[\Upsilon_{\beta,\kappa}(\gamma) ]\bigr|<\varepsilon, \) since \( \Upsilon_{\beta,\kappa}(\gamma)\leq 0 ,\) we have
    \begin{equation*}
        \begin{split}
            &\Bigl| e^{|\support \gamma|  \Upsilon_{\beta,\kappa}(\gamma)} - e^{|\support \gamma|  \mathbb{E}_{N,\infty,\kappa}[\Upsilon_{\beta,\kappa}(\gamma)]}\Bigr|
            \leq
            \Bigl| \Upsilon_{\beta,\kappa}(\gamma) - \mathbb{E}_{N,\infty,\kappa}[\Upsilon_{\beta,\kappa}(\gamma)]\Bigr|< \varepsilon.
        \end{split}
    \end{equation*}
    Using Lemma~\ref{lemma: uniform WLLN} with \( \varepsilon = \bigl( K_{13}\alpha_2(\beta,\kappa)|\support \gamma|^{-1} \bigr)^{1/3},\) we obtain the desired conclusion. 
\end{proof}

\begin{lemma}\label{lemma: theta estimates}
    Let \( \beta,\kappa \geq 0  \) be such that \( 6\beta > \kappa, \) and let \( G = \mathbb{Z}_2.\) Then
    \begin{equation*}
        \bigl| \theta_{\beta,\kappa}(0) - e^{-2e^{-24\beta-4\kappa}}  \bigr| \leq 4 \bigl(e^{-24\beta-4\kappa } \bigr)^2 \qquad \text{and} \qquad \pigl| \theta_{\beta,\kappa}(1) - e^{-2e^{-24\beta+4\kappa}} \bigr| \leq 4 \bigl(e^{-24\beta+4\kappa } \bigr)^2.
    \end{equation*}
\end{lemma}

\begin{proof}
    For the first inequality, note that
    \begin{equation*}
    \begin{split} 
        &\bigl| \theta_{\beta,\kappa}(0) - e^{-2e^{-24\beta-4\kappa}}  \bigr|
        =
        \bigl| \frac{1 - e^{-24\beta-4\kappa}}{1 + e^{-24\beta-4\kappa}} - e^{-2e^{-24\beta-4\kappa}}  \bigr|
        \\&\qquad\leq
        \bigl| \frac{1 - e^{-24\beta-4\kappa}}{1 + e^{-24\beta-4\kappa}} - (1 - 2e^{-24\beta-4\kappa}) \bigr|
        + 
        \bigl| (1 - 2e^{-24\beta-4\kappa}) - e^{-2e^{-24\beta-4\kappa}} \bigr) \bigr|
        \\&\qquad\leq
        2(e^{-24\beta-4\kappa})^2
        + 
        \bigl( 2e^{-24\beta-4\kappa})^2/2
        =
        4(e^{-24\beta-4\kappa})^2.
    \end{split}
    \end{equation*}
    The second inequality follows analogously. 
\end{proof}

\begin{proof}[Proof of Proposition~\ref{proposition: rectangular Z2}]
    By definition, we have
    \begin{equation*}
        \begin{split}
            &\mathbb{E}_{N,\infty,\kappa} [ \Upsilon_{\beta,\kappa}(\gamma)]
            =
            \mathbb{E}_{N,\infty,\kappa} \pigl[ |\support \gamma|^{-1}\sum_{e \in \gamma} \log \theta_{\beta,\kappa}\bigl(\sigma(e)\bigr) \pigr]
            =
            |\support \gamma|^{-1}\sum_{e \in \gamma} 
            \mathbb{E}_{N,\infty,\kappa} \pigl[ \log \theta_{\beta,\kappa}\bigl(\sigma(e)\bigr) \pigr]
            \\&\qquad=
            |\support \gamma|^{-1}\sum_{e \in \gamma} 
            \mathbb{E}_{N,\infty,\kappa} \pigl[ \sum_{g \in G} \log \mathbb{1}_{\sigma(e)=g}\theta_{\beta,\kappa}(g)  \pigr]
            =
            \sum_{g \in G} 
            \log\theta_{\beta,\kappa}(g)  |\support \gamma|^{-1}\sum_{e \in \gamma} 
            \mathbb{E}_{N,\infty,\kappa} \pigl[  \mathbb{1}_{\sigma(e)=g} \pigr].
        \end{split}
    \end{equation*}
    Consequently, by Lemma~\ref{lemma: WLLN application},
    \begin{equation*}
        \begin{split}
            &\pigl| \Theta_{N,\beta,\kappa}(\gamma) -  \theta_{\beta,\kappa}(0)^{\sum_{e \in \gamma} 
            \mathbb{E}_{N,\infty,\kappa} [ \mathbb{1}_{\sigma(e)=0} ]} \theta_{\beta,\kappa}(1)^{\sum_{e \in \gamma} 
        \mathbb{E}_{N,\infty,\kappa} [ \mathbb{1}_{\sigma(e)=1} ]}  \pigr| 
            \leq 
            2\sqrt[3]{\frac{K_{13}\alpha_2(\beta,\kappa)}{|\support \gamma|}}.
        \end{split}
    \end{equation*} 
    Next, by Lemma~\ref{lemma: Chatterjees inequality ii}, we have 
    \begin{equation*}
        \begin{split}
            &
        \Bigl| \theta_{\beta,\kappa}(0)^{\sum_{e \in \gamma} 
        \mathbb{E}_{N,\infty,\kappa} [ \mathbb{1}_{\sigma(e)=0} ]} \theta_{\beta,\kappa}(1)^{\sum_{e \in \gamma} 
        \mathbb{E}_{N,\infty,\kappa} [ \mathbb{1}_{\sigma(e)=1} ]} 
        \\&\qquad\qquad - e^{-2e^{-24\beta-4\kappa}\sum_{e \in \gamma}\mathbb{E}_{N,\infty,\kappa}[\mathbb{1}_{\sigma(e)=0}]-2e^{-24\beta+4\kappa}\sum_{e \in \gamma}\mathbb{E}_{N,\infty,\kappa}[\mathbb{1}_{\sigma(e)=1}]} \Bigr| 
        \\&\qquad\leq 
        \sum_{e \in \gamma} 
        \mathbb{E}_{N,\infty,\kappa} [ \mathbb{1}_{\sigma(e)=0} ]\cdot \bigl| \theta_{\beta,\kappa}(0) - e^{-2e^{-24-4\kappa}} \bigr|
        +
        \sum_{e \in \gamma} 
        \mathbb{E}_{N,\infty,\kappa} [ \mathbb{1}_{\sigma(e)=1} ]\cdot \bigl| \theta_{\beta,\kappa}(1) - e^{-2e^{-24+4\kappa}} \bigr|.
        \end{split}
    \end{equation*}
    By combining Lemma~\ref{lemma: theta estimates} with Proposition~\ref{proposition: new Z-LGT coupling upper bound}, applied with \( M = 1, \) \( M' = 0,\) \( \beta = \kappa_1 =  \infty, \) and \( \kappa_2= \kappa, \) we can bound the previous equation from above by
    \begin{equation*}
        \begin{split}
            &
            |\support \gamma| \cdot 4(e^{-24\beta-4\kappa})^2
            +
            K_1(\infty,\kappa)
            \bigl(K_4 \alpha_0(\kappa ) \bigr)^8 |\support \gamma| \cdot 4(e^{-24\beta+4\kappa})^2
            .
        \end{split}
    \end{equation*}
    Combining the previous equations, we thus obtain 
    \begin{equation*}
        \begin{split}
            &\pigl| \Theta_{N,\beta,\kappa}(\gamma) 
            -
            e^{-2e^{-24\beta-4\kappa}\sum_{e \in \gamma}\mathbb{E}_{N,\infty,\kappa}[\mathbb{1}_{\sigma(e)=0}]-2e^{-24\beta+4\kappa}\sum_{e \in \gamma}\mathbb{E}_{N,\infty,\kappa}[\mathbb{1}_{\sigma(e)=1}]}
            \pigr| 
            \\&\qquad\leq 
            2\sqrt[3]{\frac{K_{13}\alpha_2(\beta,\kappa)}{|\support \gamma|}}
            +
            |\support \gamma| \cdot 4(e^{-24\beta-4\kappa})^2
            +
            K_1(\infty,\kappa)
            \bigl(K_4 \alpha_0(\kappa ) \bigr)^8 |\support \gamma| \cdot 4(e^{-24\beta+4\kappa})^2
            .
        \end{split}
    \end{equation*}
    Rearranging this equation, and recalling from~\eqref{eq: alphas for Z2} that \( e^{-24\beta-4\kappa} = \alpha_2(\beta,\kappa)^6 \) and \( \alpha_0(\kappa) = e^{-4\kappa}, \) we obtain~\eqref{eq: rectangular Z2} as desired.
\end{proof}

\subsection{A proof of Theorem~\ref{theorem: main result Z2}}\label{sec: proof of main theorem Z2}

We now provide a proof of Theorem~\ref{theorem: main result Z2}. Since this proof is very similar to the proof of Theorem~\ref{theorem: main result}, we will refer to this proof in order to avoid repetition.

\begin{proof}[Proof of Theorem~\ref{theorem: main result Z2}]  
    Let \( N \) be sufficiently large so that \( \dist_0(\gamma,\partial B_N) \geq 8\) and so that for each \( e \in \gamma, \) \( \hat \partial e \) contains no boundary plaquettes of \( C_2(B_N). \)
    Using~\eqref{eq: theta for Z2}, it follows that if $\beta$ and $\kappa$ satisfy the assumptions of Theorem~\ref{theorem: main result Z2}, then~\ref{assumption: 3} hold. 
    By combining Proposition~\ref{proposition: first version of main result} and Proposition~\ref{eq: rectangular Z2}, using that \( \pigl| \mathbb{E}_{N,\infty,\kappa}\bigl[ L_\gamma(\sigma) \bigr]\pigr| \leq 1 ,\) we obtain
    \begin{equation*}
        \begin{split}
            &\Bigl|\mathbb{E}_{N,\beta,\kappa} \bigl[L_\gamma(\sigma)\bigr]- \mathbb{E}_{N,\infty,\kappa} \bigl[ L_\gamma(\sigma) \bigr] \Theta'_{N,\beta,\kappa}(\gamma) \Bigr|
            \\&\qquad\leq \Bigl|\mathbb{E}_{N,\beta,\kappa} \bigl[L_\gamma(\sigma)\bigr]- \mathbb{E}_{N,\infty,\kappa} \bigl[ L_\gamma(\sigma) \bigr] \Theta_{N,\beta,\kappa}(\gamma) \Bigr|
            +
            \Bigl| \Theta_{N,\beta,\kappa}(\gamma) - \Theta'_{N,\beta,\kappa}(\gamma)\Bigr|
            \\&\qquad\leq   
            \Bigl(B+\frac{K_{14}
             \alpha_2(\beta,\kappa)^{12}}{2\alpha_5(\beta,\kappa)}\Bigr) \cdot 2|\support \gamma|\alpha_5(\beta,\kappa) 
            + 
            B' \cdot 2\sqrt{2|\support \gamma| \alpha_5(\beta,\kappa)}
            \\&\qquad\qquad+
            \sqrt[3]{\frac{K_{13}\alpha_2(\beta,\kappa)}{|\support \gamma|^2 \alpha_5(\beta,\kappa)}} \cdot 2\sqrt[3]{|\support \gamma|\alpha_5(\beta,\kappa)}
             ,
        \end{split}
    \end{equation*}
    where \( B \) and \( B' \) are given in~\eqref{eq: BB prime}. 
    Using that for \( x>0 \), we have  \( x \leq e^x, \) \( 2\sqrt{x}\leq e^x,\) and \( 2\sqrt[3]{x}\leq e^x\), it follows that
    \begin{equation*}
        \begin{split}
            &\Bigl|\mathbb{E}_{N,\beta,\kappa} \bigl[L_\gamma(\sigma)\bigr]-\mathbb{E}_{N,\infty,\kappa} \bigl[L_\gamma(\sigma)\bigr] \Theta'_{N,\beta,\kappa}(\gamma) \Bigr| 
            \\&\qquad\leq 
            \Bigl(B
            +
            \frac{K_{14} \alpha_2(\beta,\kappa)^{12}}{2\alpha_5(\beta,\kappa)}
            +
            B'
            +
            \sqrt[3]{\frac{K_{13}\alpha_2(\beta,\kappa)}{|\support \gamma|^2 \alpha_5(\beta,\kappa)}}
            \Bigr) e^{2 |\support \gamma| \alpha_5(\beta,\kappa) }.
        \end{split}
    \end{equation*}
    Combining this inequality with~\eqref{eq: second part of main result}, we obtain
    \begin{equation}\label{eq: almost last equation ii}  
        \begin{split}
            &\Bigl|\mathbb{E}_{N,\beta,\kappa}[ L_\gamma(\sigma)] -  \mathbb{E}_{N,\infty,\kappa}[ L_\gamma(\sigma)] \mathbb{E}_{N,\infty,\kappa} \Bigl[ \, \prod_{e \in \gamma } \theta_{\beta,\kappa}(\sigma_e) \Bigr] \Bigr|^{1 + 2|\support \gamma|/|\support (\gamma-\gamma_c)|}
            \\&\qquad\leq
            2^{2|\support \gamma|/|\support(\gamma-\gamma_c)|}
            \Bigl(K_{11}
            +
            \frac{K_{14} \alpha_2(\beta,\kappa)^{12}}{2\alpha_5(\beta,\kappa)}
            +
            K_{12}
            +
             \sqrt[3]{\frac{K_{13}\alpha_2(\beta,\kappa)}{|\support \gamma|^2 \alpha_5(\beta,\kappa)}}
            \Bigr).
        \end{split}
    \end{equation}

    Now recall~\eqref{eq: alphas for Z2}, and note that these expression imply that
    \begin{equation}\label{eq: alpha ration estimates}
        \frac{\alpha_2(\beta,\kappa)^6}{\alpha_5(\beta,\kappa)}  \leq 1
        \quad  \text{and} \quad 
        \frac{  \alpha_4(\beta,\kappa) \alpha_0(\kappa)^2}{\alpha_5(\beta,\kappa)} \leq 1.
    \end{equation} 
    Using these equations and inequalities and Proposition~\ref{proposition: square sum term in constant}, it follows that
    \begin{equation*}
        \begin{split}
            &B+\frac{K_{14} \alpha_2(\beta,\kappa)^{12}}{2\alpha_5(\beta,\kappa)}+B' +\sqrt[3]{\frac{K_{13}\alpha_2(\beta,\kappa)}{|\support \gamma|^2 \alpha_5(\beta,\kappa)}}
            \\&\qquad\leq
            \mathbb{1}(\partial \gamma \neq 0)\cdot \frac{2K_3 K_4^8\alpha_0(\kappa)^7}{\sqrt{|\support \gamma|}} 
            \\&\qquad\qquad\qquad \cdot \pigl( 16
        +
        \frac{2K_4 \alpha_0(\kappa)}{1-K_4 \alpha_0(\kappa)}
        +
        \frac{|\support \gamma|}{2} \pigl( K_4 \alpha_0(\kappa) \pigr)^{\max(0,\max(0,\min(\ell_1,\ell_2)-7))} \pigr) \cdot \sqrt{\frac{1}{|\support \gamma|}}
            \\&\qquad\qquad+ 
            \frac{4K_3 \bigl( K_4 \alpha_0(\kappa) \bigr)^{\dist_1(\support \gamma ,\partial C_1(B_N))}}{\alpha_5(\beta,\kappa)}  
            +
             K_2  \cdot\frac{|\support \gamma_c|}{|\support \gamma|}
            \\&\qquad\qquad+
            K_3 K_4^2\, \alpha_0(\kappa)^{5/6}  \cdot \alpha_2(\beta,\kappa)
            +
            \frac{18^4 K_5 \alpha_2(\beta,\kappa)^{5}}{2} \cdot \alpha_2(\beta,\kappa)
            \\&\qquad\qquad+ 
            \sqrt{\frac{2 K_9 \, \alpha_4(\beta,\kappa) 
            \bigl( K_4 \alpha_0(\kappa)  \bigr)^{\dist_1(\support \gamma,\partial C_1(B_N))} }{|\support \gamma| \alpha_5(\beta,\kappa)^2}}
            + 
            2\sqrt{\frac{2 K_3 \, \alpha_4(\beta,\kappa)   
            \bigl( K_4 \alpha_0(\kappa)  \bigr)^{\dist_1(\support \gamma,\partial C_1(B_N))}}{|\support \gamma| \alpha_5(\beta,\kappa)^2}}
            \\&\qquad\qquad+ 
            \Biggl( 
            \sqrt{ 2 K_8 \,     \alpha_0(\kappa)^3  \max\bigl(   \alpha_0(\kappa),  \alpha_1(\beta)^6\bigr)}
            +
            2\sqrt{ 2K_7 \, \alpha_0(\kappa)^6}
            +
            \sqrt{ 12 K_2} 
            \Biggr)
            \cdot \sqrt{\frac{1}{|\support \gamma|}}
            \\&\qquad\qquad+
            2^{-1}K_{14} \alpha_2(\beta,\kappa)^{5}  \cdot \alpha_2(\beta,\kappa)
            +
            \bigr(\sqrt{K_{10}\, \alpha_0(\kappa)^6 } 
            +
            1 \bigr)\cdot \sqrt{\frac{|\support \gamma_c|}{|\support \gamma|}}
            +
            \sqrt[6]{\frac{K_{13}^2}{|\support \gamma|}} \cdot \sqrt{\frac{1}{|\support \gamma|}}.
        \end{split}
    \end{equation*}

    Now note that since \( \gamma \) is a path along the boundary of a rectangle with side lengths \( \ell_1,\ell_2 \geq 2, \) we must have \( |\support \gamma_c| \leq 8. \) Since~\ref{assumption: 3} holds, we must have \( 2\alpha_0(\kappa) \leq K_4 \alpha_0(\kappa)\leq 1, \) and since \( G = \mathbb{Z}_2 \), we have \( \alpha_2(\beta,\kappa),\alpha_0(\beta)\leq 1. \)
    Recalling Proposition~\ref{proposition: unitary gauge one dim} and Proposition~\ref{proposition: limit exists}, letting \( N \to \infty ,\) and simplifying, we thus obtain
    \begin{equation*}
        \begin{split}
            &\pigl|\bigl\langle L_\gamma(\sigma,\phi) \bigr\rangle_{\beta,\kappa,\infty} -  \bigl\langle L_\gamma(\sigma,\phi)\bigr\rangle_{\infty,\kappa,\infty}
            \Theta'_{\beta,\kappa}(\gamma) \bigr|^{1 + 2|\support \gamma|/|\support (\gamma-\gamma_c)|}
            \\[0.5ex]&\qquad\leq
            2^{2|\support \gamma|/|\support(\gamma-\gamma_c)|}
            K_{15}
            \pigl( \alpha_2(\beta,\kappa) + |\support \gamma|^{-1/2}\pigr),
        \end{split}
    \end{equation*}
    where
    \begin{equation}\label{eq: K15}
    \begin{split}
        &K_{15} \coloneqq
            \mathbb{1}(\partial \gamma \neq 0) \cdot \Bigl(32K_3 K_4
            +
            \frac{4K_3 K_4^9 \alpha_0(\kappa)^8 }{1-K_4 \alpha_0(\kappa)}  
            +
            K_3 K_4 |\support \gamma|^{1/2} \pigl( K_4 \alpha_0(\kappa) \pigr)^{\min(\ell_1,\ell_2)} \Bigr)
            \\&\qquad+
            8 K_2 
            +
            K_3 K_4^2
            +
            18^4 K_5/2
            +  
            \sqrt{ K_8}
            +
            \sqrt{ K_7}
            +
            \sqrt{ 12 K_2 }  
            +
            K_{14}/2
            +
            \sqrt{8K_{10}} 
            +
            \sqrt{8}
            +
            \sqrt[3]{K_{13}} .
    \end{split}
    \end{equation}
    Next, note that since \( \gamma \) is a path along the boundary of some rectangle with side lengths \( \ell_1,\ell_2 \geq 2, \) we have \( |\support \gamma_c| \leq 8. \) Since, by assumption, we have \( |\support \gamma| \geq 24, ,\) it follows that \( |\support \gamma_c|/|\support \gamma| \leq 1/3,\) and hence
    \begin{equation*}
        \frac{1}{4} \leq \frac{1}{1+2|\support \gamma|/(|\support \gamma|-|\support \gamma_c|)} \leq \frac{1}{3}.
    \end{equation*} 
    If we in addition have \( \alpha_2(\beta,\kappa)+ \sqrt{\max(1,|\support \gamma_c|)/|\support \gamma|}\leq 1 \), then it follows that
    \begin{equation}
            \Bigl|\bigl\langle L_\gamma(\sigma,\phi) \bigr\rangle_{\beta,\kappa,\infty} - \bigl\langle L_\gamma(\sigma,\phi) \bigr\rangle_{\infty,\kappa,\infty} \Theta'_{\beta,\kappa}(\gamma) \Bigr|
            \leq
            2^{1-\frac{1}{4}}
            \cdot 
            \label{eq: Z2 main result before simplification}  K_{15}^{1/3} 
            \cdot \Bigl( \alpha_2(\beta,\kappa) + \sqrt{1 /|\support \gamma|}\Bigr)^{\frac{1}{4}}.
    \end{equation}  
   Since \( |\rho(g)| = 1 \) for all $g \in G$, we always have
    \begin{equation*}
        \begin{split}
            &\pigl|\bigl\langle L_\gamma(\sigma,\phi) \bigr\rangle_{\beta,\kappa,\infty} - \bigl\langle L_\gamma(\sigma,\phi) \bigr\rangle_{\infty,\kappa,\infty} \Theta'_{\beta,\kappa}(\gamma) \pigr|\leq 2.
        \end{split}
    \end{equation*}  
    Consequently, if  \( \alpha_2(\beta,\kappa) + \sqrt{|\support \gamma_c|/|\support \gamma|}\geq 1 \), then~\eqref{eq: Z2 main result before simplification} automatically holds. If we let
    \begin{equation}\label{eq: constant in main theorem Z2}
        \begin{split}
            &K_0 \coloneqq 
            2^{\frac{3}{4}} K_{15}^{1/3},
        \end{split}
    \end{equation} 
    we thus obtain~\eqref{eq: main result Z2}. This completes the proof of Theorem~\ref{theorem: main result Z2}. 
\end{proof}


\begin{thebibliography}{99}

    \bibitem{b1974} Balian, R., Drouffe, J. M., Itzykson, C., {Gauge fields on a lattice. I. General outlook}, Phys.\ Rev.\ D 10(10), (1974), 3376--3395.

    \bibitem{b1975II} Balian, R., Drouffe, J. M., Itzykson, C., {Gauge fields on a lattice. II. Gauge-invariant Ising model}, Phys.\ Rev.\ D, 11(8), (1975), 2098--2103.

    \bibitem{b1975III} Balian, R., Drouffe, J. M., Itzykson, C., {Gauge fields on a lattice. III. Strong-coupling expansions and transition points}, Phys.\ Rev.\ D, 11(8), (1975), 2104--2119.

    \bibitem{b1984} Borgs, C., {Translation symmetry breaking in four dimensional lattice gauge theories}, Commun.\ Math.\ Phys.\ 96, (1984), 251--284.

    \bibitem{bf1983} Bricmont, J., Fr\"olich, J., An order parameter distinguishing between different phases pf lattice gauge theories with matter fields. Physics Letters B, 122(1), (1983), 73--77.

    \bibitem{bf1987} Bricmont, J., Fr\"olich, J., Statistical mechanical methods in particle structure analysis of lattice field theories (III). Confinement and bound states in gauge theories, Nuclear Physics B280 [FS18] (1987) 385--444.

    \bibitem{sc2019} Cao, S., Wilson loop expectations in lattice gauge theories with finite gauge groups, Commun.\ Math.\ Phys.\ 380, (2020), 1439--1505.

    \bibitem{c2021} Chatterjee, S., A probabilistic mechanism for quark confinement, Commun.\ Math.\ Phys.\ 385, (2021), 1007--1039.
(2021).

    \bibitem{c2019} Chatterjee, S., Wilson loop expectations in Ising lattice gauge theory, Commun.\ Math.\ Phys.\ 377, (2020), 307--340.

    \bibitem{cis2002} Chernodubab, M. N., Ilgenfritz, E., M., Schiller, A., String breaking and monopoles: a case study in the 3D abelian Higgs model, Phys.\ Lett.\ B, 547(3--4), (2002),  269--277.

    \bibitem{c1980} Creutz, M., Phase diagrams for coupled spin-gauge system, Phys.\ Rev.\ D 21(4), (1980), 1106--1112. 

    \bibitem{d2017} Duminil-Copin, H., Lectures on the Ising and Potts models on the hypercubic lattice, in "Random Graphs, Phase Transitions, and the Gaussian Free Field", Springer Proceedings in Mathematics \& Statistics (2017).

    \bibitem{hgjjkn1987} Evertz, H. G., Grösch, V., Jansen, K., Jersak, J., Kastrup, H. A., Neuhaus, T., Confined and free charges in compact scalar QED, Neucl. Phys., B285 [FS19], (1987), 559--589.

    \bibitem{ejjsln1987} Evertz, H. G., Jansen, K., Jers\'ak, J, Lang, C. B., Neuhaus, T., Photon and Bosonium masses is scalar lattice QED, Nucl.\ Phys.\ B 285, (1987), 590--605. 
    
    \bibitem{fmf1986} Filk, T., Marcu, M., Fredenhagen, K., Line of second-order phase transitions in the four-dimensional \( \mathbb{Z}_2 \) gauge theory with matter fields, Phys.\ Lett.\ B 169(4), (1986), 405--412. 
    
    \bibitem{f2021} Forsstr\"om, M. P., Decay of correlations in finite abelian lattice gauge theories, Commun.\ Math.\ Phys.\ 393, (2022), 1311--1346.

    \bibitem{flv2020} Forsstr\"om, M. P., Lenells, J., Viklund, F., Wilson loops in finite abelian lattice gauge theories, Ann.\ Inst.\ H.\ Poincaré Probab.\ Statist. 58(4), (2022), 2129--2164.
 
    \bibitem{flv2021} Forsstr\"om, M. P., Lenells, J., Viklund, F., Wilson loops in the abelian lattice Higgs model, Probab.\ Math Phys.\ 4(2), (2023), 257--329.
     
    \bibitem{fs1979} Fradkin, E., Shenker, S. H. , Phase diagrams of lattice gauge theories with Higgs fields, Phys.\ Rev.\ D 19(12), (1979), 3682--3697.

    \bibitem{fm1988} Fredenhagen, K., Marcu, M., Dual interpretation of order parameters for lattice gauge theories with matter fields, Nuclear Physics B (Proc. Suppl.) 4 (1988) 352-357. 

    \bibitem{gs2021} Garban, C., Supelveda, A., Improved spin-wave estimate for Wilson loops in \( U(1) \) lattice gauge gauge theory, International Mathematics Research Notices 3, (2023).
    
    \bibitem{g1970} Ginibre, J., General formulation of Griffiths' inequalities, Commun.\ Math.\ Phys.\ 16, (1970), 310--328.
    
    \bibitem{g2006} Gliozzi, F.,  The functional form of open Wilson lines in gauge theories coupled to matter, Nuclear Physics B - Proc. Suppl. 153(1), (2006), 120--127.
    
    \bibitem{gr2002} Gliozzi, F. and Rago, A.,
    Monopole clusters, center vortices, and confinement in a Z2 gauge-Higgs system, PHYSICAL REVIEW D 66,  (2002) 074511.
    
    \bibitem{ghms2011} Gregor, K., Huse, D. A., Moessner, R., Sondhi, S. L., \textit{Diagnosing Deconfinement and Topological Order}, New J.Phys. 13 (2011) 025009.

    \bibitem{jw} Jaffe, A., Witten, E., Quantum Yang--Mills theory, \url{https://www.claymath.org/sites/default/files/yangmills.pdf}

    \bibitem{jsj1980} Jongeward, G. A., Stack, J. D., Jayaprakash, C., Monte Carlo calculations on \( \mathbb{Z}_2 \) gauge-Higgs theories, Phys.\ Rev.\ D 21(12), (1980), 3360--3368.

    \bibitem{ks1984} Kanaya, K., Sugiyama, Y., Meanfield Study of \( \mathbb{Z}_2 \) Higgs Model with Radial Excitations and Mode Correlation Problem, Prog.\ Theor.\ Phys., 72(6), (1984), 1158--1175.
    
    \bibitem{m} Marcu, M., (Uses of) An order parameter for lattice gauge theories with matter fields.

    \bibitem{s1988} Shrock, E., Lattice Higgs models, Nucl.\ Phys.\ B (Proc.\ Suppl.) 4 (1988), 373--389.

    \bibitem{w1971} Wegner, F. J., Duality in Generalized Ising Models and Phase Transitions without Local Order Parameters, J.\ Math.\ Phys.\ 12(10) (1971), 2259--2272.
    
    \bibitem{w1974} Wilson, K., \textit{Confinement of quarks}, Phys.\ Rev.\ D 10, (1974), 2445--2459. 
\end{thebibliography}
\end{document}